\documentclass[11pt,reqno]{amsart}
\usepackage{graphicx,enumerate,amssymb,bbold}
\usepackage[notrig]{physics}
\usepackage[T1]{fontenc}
\usepackage[applemac]{inputenc}
\usepackage[font=footnotesize,labelfont=small,position=top]{subcaption}
\usepackage[export]{adjustbox}
\usepackage[ruled,vlined,linesnumbered]{algorithm2e}
\usepackage{booktabs,multirow}
\usepackage{pgfplots}
\usepackage{color}
\usepackage{colortbl}
\usepackage{tikz}
\usepackage{siunitx}
\usepackage{mathtools}
\usepackage{float}
\usepackage{bm}
\usepackage{tabularx}
\usepackage{pifont}
\usepackage{url}
\usetikzlibrary{spy,backgrounds}
\usepackage[colorlinks=true, pdfstartview=FitV, linkcolor=blue, 
citecolor=blue, urlcolor=blue]{hyperref}
\definecolor{Mea_Deep_Color}{rgb}{0.3, 0.35, 0.45}
\definecolor{Mea_Cov_Color}{rgb}{0.5, 0.58, 0.7}
\pgfplotsset{width=7cm,compat=1.13}
\usepackage[left=2.5cm,right=2.5cm,top=2.5cm,bottom=2.5cm,includeheadfoot]{geometry}

\usepackage{siunitx}

\usepackage{cleveref}
\newenvironment{delayedproof}[1]
{\begin{proof}[\raisedtarget{#1}Proof of \Cref{#1}]}
	{\end{proof}}
\newcommand{\raisedtarget}[1]{%
	\raisebox{\fontcharht\font`P}[0pt][0pt]{\hypertarget{#1}{}}%
}
\newcommand{\proofref}[1]{\hyperlink{#1}{\Cref{#1}}}

\newcommand{\cmark}{\ding{51}}%
\newcommand{\xmark}{\ding{55}}%

\usepackage[colorlinks=true, pdfstartview=FitV, linkcolor=blue, 
            citecolor=blue, urlcolor=blue]{hyperref}

\newbox{\myorcidaffilbox}
\sbox{\myorcidaffilbox}{\large\includegraphics[height=1.7ex]{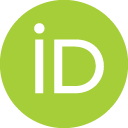}}

\newcommand{\orcidaffil}[1]{%
	\href{https://orcid.org/#1}{\usebox{\myorcidaffilbox}\,#1}}

\numberwithin{equation}{section}
\numberwithin{figure}{section}
\theoremstyle{plain}
\newtheorem{theorem}{Theorem}[section]
\newtheorem{lemma}[theorem]{Lemma}
\newtheorem{proposition}[theorem]{Proposition}

\theoremstyle{definition}
\newtheorem{definition}[theorem]{Definition}
\newtheorem{remark}[theorem]{Remark}

\newcommand{\bitem}{\begin{itemize}}
\newcommand{\eitem}{\end{itemize}}
\newcommand{\mc}[1]{\mathcal{#1}}

\newcommand{\N}{\mathbb{N}}
\newcommand{\R}{\mathbb{R}}

\newcommand{\Z}{\mathbb{Z}}

\newcommand{\bpm}{\begin{pmatrix}}
\newcommand{\epm}{\end{pmatrix}}
\newcommand{\bvm}{\begin{vmatrix}}
\newcommand{\evm}{\end{vmatrix}}
\newcommand{\bsm}{\left(\begin{smallmatrix}}
\newcommand{\esm}{\end{smallmatrix}\right)}
\newcommand{\T}{\top}

\newcommand{\ol}[1]{\overline{#1}}

\newcommand{\wt}[1]{\widetilde{#1}}
\newcommand{\la}{\langle}
\newcommand{\ra}{\rangle}
\newcommand{\mrm}[1]{\mathrm{#1}}

\newcommand{\veps}{\varepsilon}

\newcommand{\w}{\omega}

\newcommand{\gdw}{\Leftrightarrow}

\newcommand{\eins}{\mathbb{1}}

\DeclareMathSymbol{\mydiv}{\mathbin}{symbols}{"04}

\DeclareMathOperator{\Diag}{Diag}

\DeclareMathOperator{\intr}{int}
\DeclareMathOperator{\rint}{rint}

\DeclareMathOperator{\argmin}{arg min}
\DeclareMathOperator{\argmax}{arg max}
\DeclareMathOperator{\supp}{supp}

\DeclareMathOperator{\vvec}{vec}

\DeclareMathOperator{\ggrad}{grad}

\DeclareMathOperator{\Exp}{Exp}

\DeclareMathOperator{\KL}{KL}

\usepackage{lipsum}                     
\usepackage{xargs}                      
\usepackage{pgfplots}
\usepackage{changepage}
\usepackage[colorinlistoftodos,prependcaption,textsize=tiny]{todonotes}
\newcommandx{\unsure}[2][1=]{\todo[linecolor=red,backgroundcolor=red!25,bordercolor=red,#1]{#2}}
\newcommandx{\change}[2][1=]{\todo[linecolor=blue,backgroundcolor=blue!25,bordercolor=blue,#1]{#2}}
\newcommandx{\info}[2][1=]{\todo[linecolor=OliveGreen,backgroundcolor=OliveGreen!25,bordercolor=OliveGreen,#1]{#2}}
\newcommandx{\improvement}[2][1=]{\todo[linecolor=Plum,backgroundcolor=Plum!25,bordercolor=Plum,#1]{#2}}
\newcommandx{\thiswillnotshow}[2][1=]{\todo[disable,#1]{#2}}
\usetikzlibrary{positioning}

\usepackage{xcolor}
\definecolor{neigh_7}{rgb}{0.12156863, 0.46666667, 0.70588235}
\definecolor{neigh_9}{rgb}{0.1, 0.88039216, 0.90980392}
\definecolor{neigh_11}{rgb}{0.9.        , 0.49803922, 0.05490196}
\definecolor{neigh_13}{rgb}{0.3254902, 0.62745098, 0.17254902}
\usetikzlibrary{shapes}
\usepackage{color}
\usetikzlibrary{spy,arrows}
\newcommand{\markerone}{\raisebox{-2pt}{\tikz{\node[fill = 
			green,circle,minimum size=3pt,scale=0.7, label = 
			{[below = 
				0.5ex]}] () {};}}}
	
\newcommand{\markertwo}{\raisebox{-2pt}{\tikz{\node[fill = 
			red,circle,minimum size=1pt,scale=0.7, label = 
			{[below = 
				0.5ex]}] () {};}}}
\newcommand{\intreriornode}{\raisebox{1.5pt}{\tikz{\filldraw[radius = 
1pt,black] circle[];}}}

\newcommand{\boundarynode}{\raisebox{1.1pt}{\tikz{
			\filldraw[radius = 2pt,gray] circle[];
			\filldraw[radius = 1pt,black] circle[];}}}
\definecolor{Layer_1}{rgb}{0.12156863, 0.46666667, 0.70588235}
\definecolor{Layer_2}{rgb}{0.68235294, 0.78039216, 0.90980392}
\definecolor{Layer_3}{rgb}{1.        , 0.49803922, 0.05490196}
\definecolor{Layer_4}{rgb}{0.17254902, 0.62745098, 0.17254902}
\definecolor{Layer_5}{rgb}{0.59607843, 0.8745098 , 0.54117647}
\definecolor{Layer_6}{rgb}{0.83921569, 0.15294118, 0.15686275}
\definecolor{Layer_7}{rgb}{0.58039216, 0.40392157, 0.74117647}
\definecolor{Layer_8}{rgb}{0.77254902, 0.69019608, 0.83529412}
\definecolor{Layer_9}{rgb}{0.54901961, 0.3372549 , 0.29411765}
\definecolor{Layer_10}{rgb}{0.89019608, 0.46666667, 0.76078431}
\definecolor{Layer_11}{rgb}{0.96862745, 0.71372549, 0.82352941}
\definecolor{Layer_12}{rgb}{0.49803922, 0.49803922, 0.49803922}
\definecolor{Layer_13}{rgb}{0.7372549 , 0.74117647, 0.13333333}
\definecolor{Layer_14}{rgb}{0.85882353, 0.85882353, 0.55294118}
\definecolor{Stein_line}{rgb}{0,0.5,0,1}
\definecolor{Euc_line}{rgb}{1,0.3,0,0}
\definecolor{Riemann}{rgb}{0.58,0.08,0}

\newcommand{\markerfour}{\raisebox{-2pt}{\tikz{\node[fill = 
			Layer_4,circle,minimum size=3pt,scale=0.7, label = 
			{[below = 
				0.5ex]}] () {};}}}

\newcommand{\markersix}{\raisebox{-2pt}{\tikz{\node[fill = 
			Layer_10,circle,minimum size=3pt,scale=0.7, label = 
			{[below = 
				0.5ex]}] () {};}}}
\newcommand{\markerseven}{{\tikz{
		\filldraw[radius = 2pt,red] (0.2,0.5) circle[];
		\filldraw[radius = 1pt,black] 
		(0.2,0.5) circle[];}
		} }
\newcommand{\markereight}{\raisebox{1.5pt}{\tikz{
			\filldraw[radius = 1pt,black] 
			(0.2,0.5) circle[];}
} }

\newcommand{\markernine}{\raisebox{-2pt}{\tikz{\node[fill = 
			red,circle,minimum size=3pt,scale=0.7, label = 
			{[below = 
				0.5ex]}] () {};}}}
\newcommand{\markerten}{\raisebox{-2pt}{\tikz{\node[fill = 
			green,circle,minimum size=3pt,scale=0.7, label = 
			{[below = 
				0.5ex]}] () {};}}}
\newcommand{\markereleven}{\raisebox{-2pt}{\tikz{\node[fill = 
			blue,circle,minimum size=3pt,scale=0.7, label = 
			{[below = 
				0.5ex]}] () {};}}}
\newcommand{\alignedintertext}[1]{%
	\noalign{%
		\vskip\belowdisplayshortskip
		\vtop{\hsize=\linewidth#1\par
			\expandafter}%
		\expandafter\prevdepth\the\prevdepth
	}%
}

\usepackage{tikz}
\usetikzlibrary{backgrounds}
\usetikzlibrary{arrows,shapes}
\usetikzlibrary{tikzmark}
\usetikzlibrary{calc}

\usepackage{amsmath}
\usepackage{amsthm}
\usepackage{amssymb}
\usepackage{mathtools, nccmath}
\usepackage{wrapfig}
\usepackage{comment}

\usepackage{blindtext}

\usepackage{graphicx}

\usepackage{xspace}

\usepackage{array}
\usepackage{ragged2e}
\newcolumntype{P}[1]{>{\RaggedRight\hspace{0pt}}p{#1}}
\newcolumntype{X}[1]{>{\RaggedRight\hspace*{0pt}}p{#1}}

\usepackage{tcolorbox}

\usepackage{tikz}
\usetikzlibrary{arrows,shapes,positioning,shadows,trees,mindmap}
\usepackage[edges]{forest}
\usetikzlibrary{arrows.meta}
\colorlet{linecol}{black!75}
\usepackage{xkcdcolors} 

\usepackage{tikz}
\usetikzlibrary{backgrounds}
\usetikzlibrary{arrows,shapes}
\usetikzlibrary{tikzmark}
\usetikzlibrary{calc}

\title[A Nonlocal Graph-PDE and Higher-Order 
Geometric Integration for Image Labeling]{A Nonlocal Graph-PDE and \\ Higher-Order Geometric Integration  for Image Labeling}
\author[D.~Sitenko, B.~Boll, C.~Schn\"{o}rr]{Dmitrij Sitenko, Bastian Boll, 
Christoph Schn\"{o}rr}
\address{IPA Group, Institute of Applied Mathematics, Heidelberg 
	University, Germany}
\urladdr{\url{https://ipa.math.uni-heidelberg.de}}
\date{} 


\begin{document}
\begin{abstract}
This paper introduces a novel nonlocal 
partial difference equation (G-PDE) for labeling metric data on graphs. The G-PDE is derived as nonlocal  reparametrization of the assignment flow approach that was introduced in 
\textit{J.~Math.~Imaging \& Vision} 58(2), 2017. Due to this parameterization, solving the G-PDE numerically is shown to be equivalent to computing the Riemannian gradient flow with respect to a nonconvex potential. We devise an entropy-regularized difference-of-convex-functions (DC) decomposition of this potential and show that the basic geometric Euler scheme for integrating the assignment flow is equivalent to solving the G-PDE by an established DC programming scheme. Moreover, the viewpoint of geometric integration reveals a basic way to exploit higher-order information of the vector field that drives the assignment flow, in order to devise a novel accelerated DC programming scheme. A detailed convergence analysis of both numerical schemes is provided and illustrated by numerical experiments.

\end{abstract}


\maketitle
\tableofcontents

\section{Introduction}

\subsection{Overview, Motivation.} \textit{Nonlocal} iterative operations for data processing on graphs constitute a basic operation that underlies many major image and data processing frameworks, including variational methods and PDEs on graphs for denoising, morphological processing and other regularization-based methods of data analysis
\cite{Gilboa:2007aa,Elmoataz:2008aa,Gilboa:2009aa,Buades:2010aa,Elmoataz:2015aa}. This includes deep networks \cite{Goodfellow:2016aa} and time-discretized neural ODEs \cite{Chen:2018ab} whose layers generate sequences of nonlocal data transformations.

Among the extensions of such approaches to \textit{data labeling} on graphs, that is the assignment of an element of a finite set of labels to data points observed at each vertex, one may distinguish approaches whose mathematical structure is directly dictated by the labeling task,  and approaches that combine traditional data processing with as subsequent final discretization step: 
\begin{itemize}
\item
Examples of the former class are discrete graphical models \cite{Wainwright:2008aa,Kappes:2015aa} that encode directly the combinatorial label assignment task, as a basis for the design of various sequential nonlocal processing steps performing approximate inference, like belief propagation. However, the intrinsic non-smoothness of discrete graphical models constitutes a major obstacle for the design of hierarchical models and for efficient parameter learning. Graphical models, therefore, have been largely superseded by deep networks during the last decade. 
\item
Examples of the latter class include the combination of established PDE-based diffusion approaches and threshold operations \cite{Merriman:1994aa,Gennip:2014aa, Bertozzi:2016aa}. The mathematical formulations inherit the connection between total variation (TV) based variational denoising, mean curvature motion and level set evolution \cite{Osher:1988aa,Rudin:1992aa,Garcke:2013aa,Caselles:2015ws}, and they exhibit also connections to gradient flows in terms of the Allen-Cahn equation with respect to the Ginzburg-Landau functional \cite{Garcke:2013aa,Gennip:2014aa}. Regarding data labeling, however, a conceptual shortcoming of these approaches is that they do not provide a direct and natural mathematical problem formulation. As a consequence, this renders difficult to cope with the assignment of dozens or hundreds of labels to data, and to learn efficiently parameters in order to tailor regularization properties to the problem and the class of data at hand.
\end{itemize}
\textit{Assignment flows} \cite{Astrom:2017ac,Schnorr:2019aa} constitute a mathematical approach tailored to the data labeling problem, aimed to overcome the aforementioned shortcomings. The basic idea is to represent label assignments to data by a \textit{smooth} dynamical process, based on the Fisher-Rao geometry of discrete probability distributions and on a weighted (parametrized) coupling of local flows for label selection across the graph. As a result, no extrinsic thresholding or rounding is required since the underlying geometry enables to perform both spatial diffusion for assignment regularization and rounding to an integral solution just by integrating the assignment flow.

Stability and convergence to integral solutions of assignment flows hold under mild conditions \cite{Zern:2020aa}. A wide range of numerical schemes exist \cite{Zeilmann:2020aa} for integrating geometrically assignment flows with GPU-conforming operations. Generalized assignment flows for unsupervised and self-supervised scenarios \cite{Zern:2020ab,Zisler:2020aa} are more involved computationally but do not essentially change the overall \textit{mathematical} structure. 

Assignment flows \textit{regularize} the assignment of labels to data by 
parameters $\Omega$ that couple the local flows at edges across the graph. 
These parameters can be determined either directly in a data-driven way as 
demonstrated in Figure \ref{fig:Labeling_OCT_example} or learned offline in a 
supervised way. Learning the parameters of assignment flows from data can be 
accomplished using symplectic numerical integration \cite{Huhnerbein:2021th} 
or, alternatively and quite efficiently, using exponential integration of 
linearized assignment flows \cite{Zeilmann:2021wt,Zeilmann:2022ul}. In 
particular, deep parametrizations of assignment flows do not at all change the 
mathematical structure which enables to exploit recent progress on PAC-Bayes 
bounds in order to compute a statistical performance certificate of 
classifications performed by deep linearized assignment flows in applications 
\cite{Boll:2022wz}. The assignment flow approach is introduced 
in Section \ref{sec:Assignment-Flow} and illustrated by Figure 
\ref{fig:Assignment_Flow_Illustration}.

\subsection{Contribution, Organization}
This paper makes two contributions, illustrated by Figure 
\ref{fig:summary-of-results}:
\begin{enumerate}[(a)]
\item
Given an undirected weighted regular grid graph $\mc{G}=(\mc{V},\mc{E},\Omega)$, we show that solving a particular parametrization of the assignment flow is equivalent to solving the \textit{nonlocal nonlinear} partial \textit{difference} equation (\textit{G-PDE}) on the underlying graph $\mc{G}$,
	\begin{subequations}\label{eq:S-Non_Local_PDE}
		\begin{align}
		&\partial_{t}S(x,t) = R_{S(x,t)}\Big( 
		\frac{1}{2}\mathcal{D}^{\alpha}\big(\Theta 
		\mathcal{G}^{\alpha}(S)\big)+\lambda S \Big)(x,t),&\quad &\text{ on }  
		\mathcal{V}\times \R_{+},\\
		&S(x,t) = 0, &\quad &\text{ on } 
		\mathcal{V}^{\alpha}_{\mathcal{I}}\times 
		\R_{+} ,\\
		&S(x,0) = \overline{S}(x)(0), &\quad &\text{ on } 
		\ol{\mc{V}}\times\R_{+},
		\end{align}
	\end{subequations}
where the vector field $S$ takes values at $x\in\mc{V}$ in the relative interior of the probability simplex that is equipped with the Fisher-Rao metric. $\mc{D}^{\alpha}$ and $\mc{G}^{\alpha}$ are nonlocal divergence and gradient operators based on established calculus \cite{Du:2012aaD,Du:2013aa}. The linear mapping $R_{S(x),t}$ is the inverse metric tensor corresponding to the Fisher-Rao metric,  expressed in ambient coordinates. 

The G-PDE \eqref{eq:S-Non_Local_PDE} confirms and provides a generalized \textit{nonlocal} formulation of a PDE that was heuristically derived by \cite[Section 4.4]{Savarino:2019ab} in the continuous-domain setting. In particular, \eqref{eq:S-Non_Local_PDE} addresses the data labeling problem \textit{directly} without any further pre- or postprocessing step and thus contributes to the line of PDE-based research of image analysis initiated by Alvarez et al.~\cite{Alvarez:1993aa} and Weickert \cite{Weickert:1998aa}.
\item 
The particular parametrization of the assignment flow that we show in this paper to be equivalent to \eqref{eq:S-Non_Local_PDE}, constitutes a Riemannian gradient flow with respect to a non-convex potential \cite[Section 3.2]{Savarino:2019ab}. We consider a \textit{Difference-of-Convex (DC)} function decomposition \cite{Horst:1999aa} of this potential and show
\begin{enumerate}[(i)]
\item that the simplest first-order geometric numerical scheme for integrating the assignment flow can be interpreted as basic two-step iterative method of DC-programming \cite{Hoai-An:2005aa};
\item that a corresponding tangent-space parametrization of the assignment flow and second-order derivatives of the tangent vector field can be employed to \textit{accelerate} the basic DC iterative scheme.
\end{enumerate}
Due to result (a), both schemes (i) and (ii) also solve the G-PDE \eqref{eq:S-Non_Local_PDE}. In addition, we point out that while a rich literature exists about accelerated \textit{convex} optimization, see e.g.~\cite{Beck:2012aa,Krichene:2016ab,Fazylab:2018aa} and references therein, methods for accelerating \textit{nonconvex} iterative optimization schemes have been less explored.
\end{enumerate}
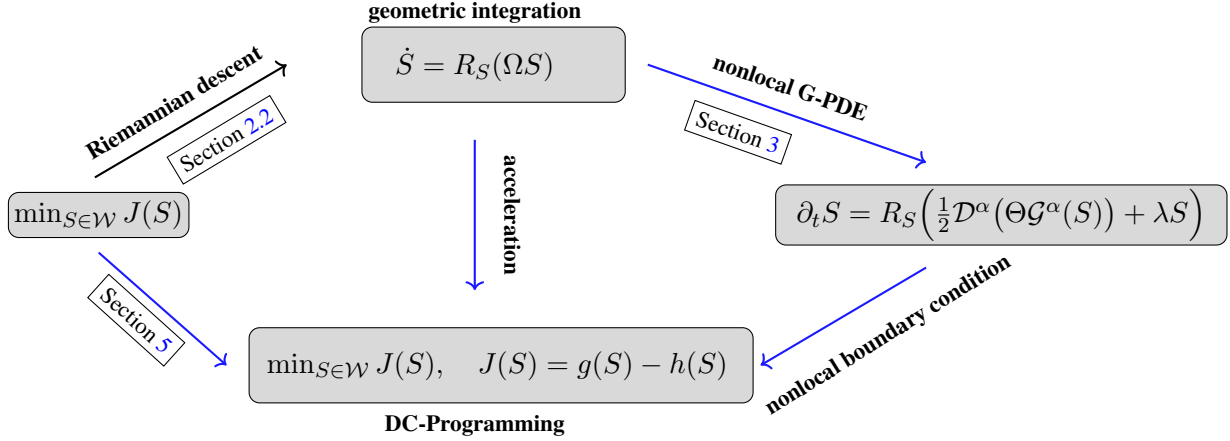
\begin{figure}
\begin{tikzpicture}
	\draw[fill = gray!30,draw = black,rounded corners] (3.5,1.5) 
	rectangle (7,2.5);
	\draw[fill = gray!30,draw = black,rounded corners] (2,-1.5) 
	rectangle (8.7,-2.5);
	\draw[fill = gray!30,draw = black,rounded corners] (-1.2,0.3) rectangle 
	(1.2,-0.3);
	\draw[fill = gray!30,draw = black,rounded corners] (9,-0.4) rectangle 
	(15,0.4);
	\node (A) at (0,0) {$ \min_{S\in \mathcal {W}} J(S)$};
	\node (C) at (12,0) {$\partial_{t}S = R_{S}\Big( 
		\frac{1}{2}\mathcal{D}^{\alpha}\big(\Theta 
		\mathcal{G}^{\alpha}(S)\big)+\lambda S \Big)$};
	\node (B) at (5,2) {$ \dot{S} = R_S(\Omega S)$};
	\node (D) at (5.3,-2) {$ \min_{S\in \mathcal {W}} J(S), \quad J(S) = 
	g(S)-h(S)$};
	\node[scale=0.8] at (5,2.7) {\textbf{geometric integration}};
	\node[scale=0.8] at (5,-2.8) {\textbf{DC-Programming}};
	\path[draw,->,thick] (-0.05,0.5) -- (2.5,2);
	\path[draw,->,thick,blue!90] (7.3,2) -- (11,0.7);
	\path[draw,->,thick,color=blue!90] (11,-0.7) -- (8.8,-2);
	\path[draw,->,thick,color=blue!90] (0,-0.5) -- (1.7,-2);
	\path[draw,->,thick,color=blue!90] (5,1) -- (5,-1);
	\node[scale=0.8,rotate=270] at (5.5,0) {\textbf{acceleration}};
	\node[scale=0.8,rotate=32] at (10.5,-1.6) 
	{\textbf{nonlocal boundary condition}};
	\node[draw,scale=0.8,rotate=-19.5] at (8.5,1.1) {Section 
	\ref{sec:Non_Local_PDE}};
	\node[draw,scale=0.8,rotate=-42] at (0.5,-1.4) {Section 
	\ref{sec:Non_Convex_Optimization}};
	\node[scale=0.8,rotate=31] at (1,1.5) {\textbf{Riemannian descent}};
	\node[scale=0.8,draw,rotate=31] at (1.7,1) {Section 
	\ref{sec:Assignment-Flow}};
	\node[scale=0.8,rotate=-19.5] at (9.2,1.7) {\textbf{nonlocal G-PDE}};
\end{tikzpicture}
\captionsetup{font=footnotesize}
\caption{\textbf{Summary of results.} Starting point (Section 
\ref{sec:Assignment-Flow}) is a particular formulation of the assignment flow 
ODE (\textbf{top}) that represents the Riemannian gradient descent of a 
functional $J$ (\textbf{left}). The first main contribution of this paper is an 
equivalent alternative representation of the assignment flow equation in terms 
of a partial difference equation on the underlying graph (\textbf{right}), with 
a nonlocal data-driven diffusion term in divergence form and further terms 
induced by the information-geometric approach to the labeling problem. The 
second major contribution concerns a DC-decomposition of the nonconvex 
functional $J$ (\textbf{bottom}) and a novel accelerated minimization algorithm 
using a second-order tangent space parametrization of the assignment flow. }
\label{fig:summary-of-results}
\end{figure}

\textbf{Organization.} Our paper is organized as follows. Section \ref{sec:Prel} introduces non-local calculus and the assignment flow, respectively. The equivalence of the assignment flow and the G-PDE \eqref{eq:S-Non_Local_PDE} is derived in Section \ref{sec:Non_Local_PDE}, together with a tangent space parametrization as basis for the development of iterative numerical solvers, and with a balance law that reveals how spatial diffusion interacts with label assignment by solving \eqref{eq:S-Non_Local_PDE}. Section \ref{sec:Related-Work-PDE} is devoted to explicitly working out common aspects and differences of \eqref{eq:S-Non_Local_PDE} to related work:
\begin{itemize}
\item[--]
continuous-domain nonlocal diffusion \cite{Andreu-Vaillo:2010aa}, 
\item[--]
nonlocal variational approaches to image analysis \cite{Gilboa:2009aa} and 
\item[--]
nonlocal G-PDEs on graphs \cite{Elmoataz:2008aa,Elmoataz:2015aa}. 
\end{itemize}
As summarized by Figure \ref{fig:ablation_study} and Table 
\ref{tab:ablation_study}, these approaches can be regarded as special cases 
from the mathematical viewpoint. They differ however regarding the scope and 
the class of problems to be solved: the approach \eqref{eq:S-Non_Local_PDE} is 
only devoted to the data \textit{labeling} problem which explains its 
mathematical form.
Finally, we show how our work extends the result of \cite{Savarino:2019ab}. 
Section \ref{sec:Non_Convex_Optimization} details contribution (b) on 
DC-programming from the viewpoint of geometric integration. The corresponding 
convergence analysis is provided in Section \ref{sec:Convergence}. Numerical 
results that illustrate our findings are reported in Section \ref{sec:Exp_Dis}. 
We conclude in Section \ref{sec:Conclusion}.


\section{Preliminaries}\label{sec:Prel}
This section contains basic material required in the remainder of this paper. A list of symbols and their meaning follows. 

\vspace{0.25cm}
	\scalebox{0.8}{
		\begin{tabularx}{\linewidth}{lll@{\hskip 1em}l@{\hskip 10em}l@{\hskip 
		1em}c}
			&\textbf{Symbol} & \hskip 5cm \textbf{Description} & 
			\\
			\hline
			\addlinespace
			&$\mc{G} = (\mc{V},\mc{E},\Omega)$ & A graph with vertex set $\mathcal{V}$, edge set $\mathcal{E}$ and weights $\Omega$. \\
			& $\mathcal{V}$  & Set of vertices representing the discrete 
			domain $\mc{V}\subset\Z^{d}$. &\\
			& $n$  & Total number $n=|\mc{V}|$ of nodes in the graph $\mathcal{G}$  & \\
			& $d$  & Dimension of the discrete domain associated with $\mc{V}$.  & \\
			& $\Omega$  & Weighted symmetric adjacency matrix of the graph $\mc{G}$. & 
			\\
			& $\mc{N}(x)$ & neighborhood of $x\in\mc{V}$ induced by $\Omega$. \\
			& $E$  & Subset of an Euclidean space.  & \\
			& $\mathcal{F}_{\mathcal{V}},\mathcal{F}_{\mathcal{V},E}$  & 
			Space of one-point functions defined on $\mathcal{V}$, taking values in $\R$ resp.~$E$.  
			& 
			\\
			& $\mathcal{F}_{\mathcal{V}\times 
				\mathcal{V}},\mathcal{F}_{\mathcal{V}\times \mathcal{V},E}$  
			& Space of two-point functions defined on $\mathcal{V}\times 
			\mathcal{V}$, taking values in $\R$ resp.~$E$.  & 
			\\
			& $\alpha \in \mathcal{F}_{\ol{\mathcal{V}}\times 
			\ol{\mathcal{V}}}$  & 
			Antisymmetric mapping that defines the interaction of 
			nodes  $x,y \in \Z^d$. & \\
			& $\Theta \in \mathcal{F}_{\ol{\mathcal{V}} \times 
			\ol{\mathcal{V}}}$  
			& 
			Nonnegative scalar-valued symmetric mapping that parametrizes the introduced 
			nonlocal diffusion process. & \\
			& $\mathcal{V}_{\mathcal{I}}^{\alpha}$ & Nonlocal interaction 
			domain 
			which represents the connectivity of nodes $x \in \mathcal{V}$ to 
			nodes 
			$y \in \Z^d\setminus \mathcal{V}$.  & \\
			& $\ol{\mathcal{V}}$  & Extension of the discrete domain associated with 
			$\mathcal{V}$ by the nodes in 
			$\mathcal{V}_{\mathcal{I}}^{\alpha}$.& \\
			& $\mathcal{D}^{\alpha},\mathcal{G}^{\alpha}$  & Nonlocal 
			divergence 
			and gradient operators parametrized by the mapping $\alpha$. & \\
			& $\mathcal{N}^{\alpha}$  & Nonlocal interaction operator parametrized by the mapping $\alpha$. & \\
			& $\mathcal{L}_{\omega}$  & Nonlocal Laplacian with weight function $\omega$. & \\
			& $\mathcal{X}_n$ & Data on the graph $\mathcal{G}$ taking values in a 
			metric space $\mathcal{X}$. & \\
			& $X(x)$ & Data point $X\in\mc{X}_{n}$ given at $x\in\mc{V}$. \\
			& $\mc{X}^{\ast}$ & set of labels $\{X_{j}^{\ast}\colon j\in\mc{J}\}\subset\mc{X}$. \\
			& $c$ & Number of labels $c=|\mc{J}|$, one of which is uniquely assigned to each data point. \\
			& $\Delta_{c}$ & Probability simplex in $\R^{c}$ of dimension $c-1$. \\
			& $\mathcal{S}$ & Relative interior of the probability simplex $\Delta_{c}$, forming the factors of the product manifold $\mc{W}$. \\
			& $T_{0}$ & Tangent space corresponding to $\mc{S}$. \\
			& $\mathcal{W}$, $\mathcal{T}_0$ & Assignment manifold and the 
			corresponding tangent space at the barycenter 
			$\mathbb{1}_{\mathcal{W}}$. & \\
			& $S,W \in \mathcal{W}$ & Points on the assignment manifold taking values $S(x), W(x)\in\mc{S}$ at $x\in\mc{V}$. & \\
			& $S^{\ast},W^{\ast} \in \ol{\mathcal{W}} \setminus \mathcal{W}$ & 
			Integral vectors on the boundary of $\mathcal{W}$. 
			\\
			& $V \in \mathcal{T}_0 $& Points in the tangent space taking values $V(x)\in T_{0}$ at $x\in\mc{V}$. & \\
			& $\Pi_0$ & Orthogonal projection onto the tangent space 
			$\mathcal{T}_0$. \\
			& $R_{S}$  & Replicator map at $S \in \mathcal W$. & \\
			\addlinespace
		\end{tabularx}
}
\subsection{Nonlocal Calculus}\label{sec:Nonlocal-Calculus}

Following 
\cite{Du:2012aaD}, we collect some basic notions of nonlocal calculus which will be used throughout this paper. 
See \cite{Du:2019us} for a detailed exposition.

Let $(\mathcal{V},\mathcal{E},\Omega)$ be an undirected weighted regular grid graph with 
\begin{equation}\label{def:weighted-graph}
n = |\mathcal{V}|,\qquad \mc{V}\subset \Z^{d},\qquad 2 \leq d \in\N
\end{equation}
nodes, with edge set $\mathcal{E} \subset \mathcal{V} \times \mathcal{V}$ that has no self-loops, and with the weighted adjacency matrix $\Omega$ that satisfies
\begin{equation}\label{eq:Omega_Mat}
0 \leq \Omega(x,y) \leq 1 , \qquad \Omega(x,y) = \Omega(y,x),\qquad \forall x, y\in\mc{V}.
\end{equation}
$\Omega$ defines the neighborhoods
\begin{equation}\label{eq:def-neighborhoods}
\mathcal{N}(x):= \{y \in \mathcal{V}\colon\Omega(x,y) > 0\},\qquad x\in\mc{V}
\end{equation}
and serves as a function $\Omega\colon\mc{V}\times\mc{V}\to\R$ measuring the similarity of adjacent nodes.

We define the function spaces
\begin{subequations}\label{eq:weight_func}
	\begin{align}
		\mathcal{F}_{\mathcal{V}} &:= \{f \colon \mathcal{V} \to \R\}, &\qquad 
	 	\mathcal{F}_{\mathcal{V} \times \mathcal{V}} &:= \{F \colon \mathcal{V} 
	 	\times 
	 	\mathcal{V} \to \R \},\\
	 	\mathcal{F}_{\mathcal{V}, E} &:= \{F \colon \mathcal{V} \to 
	 	E\}, &\qquad 	\mathcal{F}_{\mathcal{V}\times 
	 	\mathcal{V},E} &:= \{F \colon \mathcal{V}\times \mathcal{V} 
	 	\to 
	 E\},\label{eq:Vec_Val_Map}
	 \end{align}
\end{subequations}
where $E$ denotes a (possibly improper) subset of an Euclidean space. 
The spaces $\mathcal{F}_{\mathcal{V}}$ and 
$\mathcal{F}_{\mathcal{V}\times \mathcal{V}}$ respectively are equipped with the inner products 
\begin{align}\label{eq:Inner_Products}
\langle f,g \rangle_{\mathcal{V}} := \sum_{x \in 
	\mathcal{V}}  f(x)g(x), \qquad 	\langle F,G \rangle_{\mathcal{V}\times 
	\mathcal{V}} 
:= \sum_{ (x,y) \in 
	\mathcal{V}\times \mathcal{V}} F(x,y)G(x,y).
\end{align} 
We set
\begin{equation}\label{eq:def-olV}
\overline{\mathcal{V}} := \mathcal{V} \dot{\cup} 
\mathcal{V}_{\mathcal{I}}^{\alpha}\qquad 
\text{(disjoint union)}, 
\end{equation}
where the \emph{nonlocal 
interaction domain $\mathcal{V}_{\mathcal{I}}^{\alpha}$ with respect to} an \textit{antisymmetric} mapping 
\begin{equation}\label{eq:def-alpha}
\alpha \in \mathcal{F}_{\ol{\mathcal{V}}\times 
\ol{\mathcal{V}}},\qquad
\alpha(x,y) = -\alpha(y,x),\qquad
\forall x, y\in\ol{\mc{V}}
\end{equation}
is defined as 
\begin{align}\label{eq:Interaction_Dom}
\mathcal{V}_{\mathcal{I}}^{\alpha} := \{ x \in \Z^d \setminus \mathcal{V}: 
\alpha(x,y) \neq 0 \text{  for 
some } y \in \mathcal{V} \}.
\end{align}
$\mathcal{V}_{\mathcal{I}}^{\alpha}$ serves discrete formulations of conditions on nonlocal boundaries with positive measure in a Euclidean domain. Such conditions are distinct from traditional conditions imposed on boundaries that have measure zero. Figure 
\ref{fig:Non_Local_Boundary} displays a possible nonlocal boundary 
configuration. 

We state the following identity induced by \eqref{eq:def-alpha}
\begin{align}\label{eq:Zero_Flux_Identity}
		\sum_{x,y \in 
		\overline{\mathcal{V}}}\big(F(x,y)\alpha(x,y)-F(y,x)\alpha(y,x)\big) = 0, \qquad
		\forall F\in \mathcal{F}_{\overline{\mathcal{V}}\times 
			\overline{\mathcal{V}}}.
\end{align} 
The \emph{nonlocal 
divergence operator} $\mathcal{D}^{\alpha}$ and the \emph{nonlocal interaction 
operator} 
$\mathcal{N}^{\alpha}$  are defined by 
\begin{subequations}\label{eq:NL-Div_op}
\begin{align}
\mathcal{D}^{\alpha}\colon \mathcal{F}_{\overline{\mathcal{V}} 
\times \overline{\mathcal{V}}} \to 
\mathcal{F}_{\overline{\mathcal{V}}},\qquad
	\mathcal{D}^{\alpha}(F)(x) &:= \sum_{y\in \overline{\mathcal{V}}} 
	\big(F(x,y)\alpha(x,y)-F(y,x)\alpha(y,x)\big), &\text{ } &x \in 
	\overline{\mathcal{V}}, 
	\label{eq:NL-Div_op-a} \\ \label{eq:NL-Div_op-b}
	\mathcal{N}^{\alpha} \colon \mathcal{F}_{\overline{\mathcal{V}} \times 
\overline{\mathcal{V}}} 
\to \mathcal{F}_{\mathcal{V}_{\mathcal{I}}^{\alpha}},\qquad
\mathcal{N}^{\alpha}(F)(x) &:= -\sum_{y\in \overline{\mathcal{V}}} 
\big(F(x,y)\alpha(x,y)-F(y,x)\alpha(y,x)\big), &\text{ } & x \in 
	\mathcal{V}^{\alpha}_{\mathcal{I}}.  
	\end{align}
\end{subequations}  

Based on the mapping $\alpha$ given by \eqref{eq:def-alpha}, the 
operator \eqref{eq:NL-Div_op-b} is nonzero in general and accounts for the density of a \textit{nonlocal flux} from the entire domain  
$\overline{\mathcal{V}}$ to nodes $x \in 
\mathcal{V}_{\mathcal{I}}^{\alpha}$ \cite{Du:2019us}. This generalizes the 
notion \textit{local 
flux} density $\la q(x), n(x)\ra$ on continuous domains $\Omega \subset \R^d$ with outer normal vector field $n(x)\in \R^d$ on the 
boundary $\partial \Omega$, and with a vector-valued function $q(x)$ on $\partial \Omega$ that typically stems from an underlying constitutive 
physical relation. Due to the identity 
\eqref{eq:Zero_Flux_Identity}, the operators 
\eqref{eq:NL-Div_op} satisfy the \emph{nonlocal Gauss theorem} 
\begin{align}\label{eq:Non_local_Gauss}
\sum_{x \in \mathcal{V}}\mathcal{D}^{\alpha}(F)(x) = \sum_{y \in 
\mathcal{V}_{\mathcal{I}}^{\alpha}} \mathcal{N}^{\alpha}(F)(y).
\end{align}
The operator $\mathcal{D}^{\alpha}$ maps two-point 
	functions $F(x,y)$ to $\mathcal{D}^{\alpha}(F) \in 
	\mathcal{F}_{\overline{\mathcal{V}}}$, 
	whereas $\mathcal{N}^{\alpha}(F)$ is defined on the domain
	$\mathcal{V}_{\mathcal{I}}^{\alpha}$ given by \eqref{eq:Interaction_Dom} 
	where nonlocal boundary 
	conditions are imposed.
	 
The adjoint mapping $(\mathcal{D}^{\alpha})^{\ast}$ with respect to the inner product \eqref{eq:Inner_Products} is 
determined by the relation 
\begin{equation}
	\langle f, \mathcal{D}^{\alpha} (F) \rangle_{\overline{\mathcal{V}}} = 
	\langle 
	(\mathcal{D}^{\alpha})^{\ast}(f), F \rangle_{\overline{\mathcal{V}} \times 
	\overline{\mathcal{V}}}, \qquad \forall f \in 
	\mathcal{F}_{\overline{\mathcal{V}}},\qquad \forall F\in 
	\mathcal{F}_{\overline{\mathcal{V}}\times \overline{\mathcal{V}}},
\end{equation}
which yields the operator 
\begin{equation}\label{eq:NL-Grad_op}
(\mathcal{D}^{\alpha})^{\ast} \colon \mathcal{F}_{\overline{\mathcal{V}}} \to 
\mathcal{F}_{\overline{\mathcal{V}} \times \overline{\mathcal{V}}},\qquad
(\mathcal{D}^{\alpha})^{\ast}(f)(x,y):= -(f(y) - f(x))\alpha(x,y), \qquad 
\forall f\in\mc{F}_{\ol{\mc{V}}}.
\end{equation}
The \textit{nonlocal gradient operator} is defined as
\begin{equation}\label{eq:def-nonl-grad}
\mathcal{G}^{\alpha} \colon \mathcal{F}_{\overline{\mathcal{V}}} \to 
\mathcal{F}_{\overline{\mathcal{V}} \times \overline{\mathcal{V}}},\qquad
\mathcal{G}^{\alpha}(f)(x,y):=-(\mathcal{D}^{\alpha})^{\ast}(f)(x,y),\qquad \forall f\in\mc{F}_{\ol{\mc{V}}}.
\end{equation}
For \textit{vector-valued} mappings, the operators \eqref{eq:NL-Div_op} and \eqref{eq:NL-Grad_op} naturally extend to $ 
	\mathcal{F}_{\overline{\mathcal{V}}\times 
		\overline{\mathcal{V}},E}$ and $
	\mathcal{F}_{\overline{\mathcal{V}},E}$, respectively, 
 by acting \textit{componentwise}.

 Using the mappings 
 \eqref{eq:NL-Grad_op}, \eqref{eq:def-nonl-grad}, the nonlocal Gauss 
 theorem \eqref{eq:Non_local_Gauss} implies  \textit{Greens nonlocal first 
	identity} 
\begin{equation}\label{eq:Green_First_Identity}
\sum_{x \in \mathcal{V}} u(x)\mathcal{D}^{\alpha}(F)(x)-\hskip -0.1cm 
\sum_{x \in \overline{\mathcal{V}}}  \sum_{y \in 
	\overline{\mathcal{V}} }\mathcal{G}^{\alpha}(u)(x,y)F(x,y) = 
\sum_{x \in \mathcal{V}^{\alpha}_{\mathcal{I}}}u(x) \mathcal{N}^{\alpha}(F)(x),\qquad
\begin{aligned}
u &\in \mathcal{F}_{\overline{\mathcal{V}}}, \\
F &\in 
	\mathcal{F}_{\overline{\mathcal{V}}\times \overline{\mathcal{V}}}.
\end{aligned}
\end{equation}
Given a function $f\in\mc{F}_{\overline{\mathcal{V}}}$ and a symmetric mapping 
\begin{equation}\label{eq:def-Theta}
\Theta \in \mathcal{F}_{\overline{\mathcal{V}} \times \overline{\mathcal{V}}}
\quad\text{with}\quad 
\Theta(x,y) = \Theta(y,x),
\end{equation}
we define the \emph{linear nonlocal diffusion operator} 
\begin{equation}\label{eq:Non_Local_Dif}
\begin{aligned}
\mathcal{D}^{\alpha}\big(\Theta \mathcal{G}^{\alpha}(f)\big)(x) =  
2\sum_{y \in 
	\overline{\mathcal{V}}}\mathcal{G}^{\alpha}(f)(x,y)\Theta(x,y)\alpha(x,y),\qquad
f\in\mc{F}_{\overline{\mathcal{V}}}.
\end{aligned}
\end{equation}
For the particular case with no interactions, i.e.~$\alpha(x,y) = 0$ if $x \in 
\mathcal{V}$ and $y \in \mathcal{V}_{\mathcal{I}}^{\alpha}$, expression
\eqref{eq:Non_Local_Dif} reduces with $\Theta(x,y) = 1, x,y 
\in \mathcal{V}$ to 
\begin{equation}\label{eq:combinatorial_Laplacian}
	\mathcal{L}_{\omega}f(x) 
	\overset{\eqref{eq:def-neighborhoods}}{=} 
	\sum_{y \in 
	\mathcal{N}(x)}\omega(x,y)\big(f(y)-f(x) 
	\big),\qquad \omega(x,y) = 2\alpha(x,y)^2,
\end{equation} 
which coincides with the \emph{combinatorial 
	Laplacian} \cite{Chung:1996aa,Chung:1997aa} after reversing the sign. 
\begin{figure}
\captionsetup{font=footnotesize}
\begin{subfigure}{0.49\linewidth}
	\scalebox{0.9}{
		\begin{tikzpicture}[scale=1,every node/.style={minimum size=1cm},on 
		grid] 

		\draw[fill=gray,opacity = 0.8] plot[smooth 
		cycle,tension=1] 
		coordinates 
		{(0.5,0.2) (3,2) (2,2.4) (1,3.1)};
		\filldraw[radius = 2pt,gray] (1.5,3) circle[];
		\filldraw[radius = 2pt,gray] (0.5,2) circle[];
		\filldraw[radius = 2pt,gray] (1.5,0.5) circle[];
		\filldraw[radius = 1pt,black] (1.5,0.5) circle[];
		\foreach \x in {-1,-0.5,0}{
			\foreach \y in {-1,-0.5,...,4}{ 
				\draw[radius = 1pt,black] (\x,\y) circle[];}}
		\foreach \x in {3,3.5,4}{
			\foreach \y in {-1,-0.5,...,4}{ 
				\draw[radius = 1pt,black] (\x,\y) circle[];}}
		\foreach \x in {-1,-0.5}{
			\foreach \y in {-1,-0.5,...,4}{ 
				\draw[radius = 1pt,black] (\x,\y) circle[];}}
		\foreach \x in {-1,-0.5,...,4}{
			\foreach \y in {-1,-0.5,0,0.5,3.5,4}{
				\draw[radius = 1pt,black] (\x,\y) circle[];}}
		\filldraw[radius = 2pt,gray] (3,2) circle[];
		\filldraw[radius = 1pt,black] (3,2) circle[];
		\foreach \x in {0.5,1,...,2.5,3}{
			\foreach \y in {1,1.5,...,3}{
				\filldraw[radius = 1pt,black] (\x,\y) circle[];}}
		\filldraw[radius = 2pt,white] (0.5,3) circle[];
		\draw[radius = 1pt,black] (0.5,3) circle[];
		\filldraw[radius = 1.2pt,white] (0.5,2.5) circle[];
		\draw[radius = 1pt,black] (0.5,2.5) circle[];
		\foreach \x in {2,2.5,...,4}{
			\filldraw[radius = 1.2pt,white] (\x,3) circle[];
			\draw[radius = 1pt,black] (\x,3) circle[];}
		\foreach \x in {2.5,3}{
			\foreach \y in {2.5,1}{
				\filldraw[radius = 1.2pt,white] (\x,\y) circle[];
				\draw[radius = 1pt,black] (\x,\y) circle[];}}
		\filldraw[radius = 2pt,gray] (3,2) circle[];
		\filldraw[radius = 1pt,black] (3,2) circle[];
		\filldraw[radius = 2pt,gray] (3,2) circle[];
		\filldraw[radius = 1pt,black] (3,2) circle[];
		\filldraw[radius = 2pt,gray] (3,2) circle[];
		\filldraw[radius = 1pt,black] (3,2) circle[];
		\filldraw[radius = 1pt,black] (0.5,0.5) circle[];
		\filldraw[radius = 1pt,black] (1,0.5) circle[];
		\filldraw[radius = 1.2pt,white] (2,2.5) circle[];
		\draw[radius = 1pt,black] (2,2.5) circle[];
		\node at (1.2,1.3) {$\Omega$};
		\node[scale=0.7] at (0.75,2.2) {$\partial \Omega$};

		\filldraw[radius = 1pt,black] (4.3,4) circle[];
		\filldraw[radius = 2pt,gray] (4.3,3.5) circle[];
		\filldraw[radius = 1pt,black] (4.3,3.5) circle[];
		\draw[radius = 1pt,black] (4.3,3) circle[];
		\node at (4.9,4)   {$\in \mathcal{V} $};
		\node at (4.97,3.5) {$\in \partial \Omega$};
		\node at (5.2,3) {$\in \Z^2 \setminus \overline{\mathcal{V}}$};
		\end{tikzpicture}
	}
\end{subfigure}
\begin{subfigure}{.34\linewidth}
	\scalebox{0.9}{
\begin{tikzpicture}[scale=1,every node/.style={minimum size=1cm},on grid]  
\draw[double distance=6pt,fill=gray,opacity = 0.8] plot[smooth cycle,tension=1] 
coordinates 
{(0.5,0.2) (3,2) (2,2.4) (1,3.1)};
\filldraw[radius = 2pt,gray] (1,3) circle[];
\filldraw[radius = 2pt,gray] (1.5,3) circle[];
\filldraw[radius = 2pt,gray] (0.5,2) circle[];
\filldraw[radius = 2pt,gray] (2,2.5) circle[];
\filldraw[radius = 2pt,gray] (0.5,1.5) circle[];
\filldraw[radius = 2pt,gray] (2.5,1.5) circle[];
\filldraw[radius = 2pt,gray] (2,1) circle[];
\filldraw[radius = 2pt,gray] (1.5,0.5) circle[];
\filldraw[radius = 1pt,black] (1.5,0.5) circle[];
\foreach \x in {-1,-0.5,0}{
	\foreach \y in {-1,-0.5,...,4}{ 
		\draw[radius = 1pt,black] (\x,\y) circle[];}}
\foreach \x in {3,3.5,4}{
	\foreach \y in {-1,-0.5,...,4}{ 
		\draw[radius = 1pt,black] (\x,\y) circle[];}}
\foreach \x in {-1,-0.5}{
	\foreach \y in {-1,-0.5,...,4}{ 
		\draw[radius = 1pt,black] (\x,\y) circle[];}}
\foreach \x in {-1,-0.5,...,4}{
	\foreach \y in {-1,-0.5,0,0.5,3.5,4}{
		\draw[radius = 1pt,black] (\x,\y) circle[];}}
\filldraw[radius = 2pt,gray] (3,2) circle[];
\filldraw[radius = 1pt,black] (3,2) circle[];
\foreach \x in {0.5,1,...,2.5,3}{
	\foreach \y in {1,1.5,...,3}{
		\filldraw[radius = 1pt,black] (\x,\y) circle[];}}
\filldraw[radius = 2pt,white] (0.5,3) circle[];
\draw[radius = 1pt,black] (0.5,3) circle[];
\filldraw[radius = 1.2pt,white] (0.5,2.5) circle[];
\draw[radius = 1pt,black] (0.5,2.5) circle[];
\foreach \x in {2,2.5,...,4}{
	\filldraw[radius = 1.2pt,white] (\x,3) circle[];
	\draw[radius = 1pt,black] (\x,3) circle[];}
\foreach \x in {2.5,3}{
	\foreach \y in {2.5,1}{
		\filldraw[radius = 1.2pt,white] (\x,\y) circle[];
		\draw[radius = 1pt,black] (\x,\y) circle[];}}
\filldraw[radius = 2pt,gray] (3,2) circle[];
\filldraw[radius = 1pt,black] (3,2) circle[];
\filldraw[radius = 2pt,gray] (3,2) circle[];
\filldraw[radius = 1pt,black] (3,2) circle[];
\filldraw[radius = 2pt,gray] (3,2) circle[];
\filldraw[radius = 1pt,black] (3,2) circle[];
\filldraw[radius = 1pt,black] (0.5,0.5) circle[];
\filldraw[radius = 1pt,black] (1,0.5) circle[];
\node at (1.2,1.35) {$\Omega$};
\node[scale=0.7] at (0.9,2.2) {$\partial \Omega$};
	\filldraw[radius = 1pt,black] (4.3,4) circle[];
\filldraw[radius = 2pt,gray] (4.3,3.5) circle[];
\filldraw[radius = 1pt,black] (4.3,3.5) circle[];
\draw[radius = 1pt,black] (4.3,3) circle[];
\node at (4.9,4)   {$\in \mathcal{V} $};
\node at (4.97,3.5) {$\in \mathcal{V}_{\mathcal{I}}^{\alpha}$};
\node at (5.2,3) {$\in \Z^2 \setminus \overline{\mathcal{V}}$};

\draw[line width = 0.04cm,->] (2,1.5)--(2.5,1.5);
\draw[line width = 0.04cm,->] (2,1.5)--(1.5,1.5);
\draw[line width = 0.04cm,->] (2,1.5)--(2,2);
\draw[line width = 0.04cm,->] (2,1.5)--(2,1);
\node at (2.2,1) {$y$};
\node at (2.15,1.65) {$x$};
\end{tikzpicture}
}
\end{subfigure}
\caption{\textbf{Schematic visualization of a nonlocal boundary.} \textbf{Left:} A bounded open domain $\Omega \subset \R^2$ with  
\textit{local} boundary $\partial \Omega$ overlaid by the grid $\Z^2$. 
\textbf{Right:} 
A bounded open domain $\Omega$ with nonlocal boundary (light gray color). Nodes 
\protect\intreriornode \hspace{0.1cm} and \protect\boundarynode \hspace{0.1cm}, respectively, 
are vertices on the graph $\mathcal{V}$ and on the interaction domain $\mathcal{V}_{\mathcal{I}}^{\alpha}$ given by
\eqref{eq:Interaction_Dom}.}
\label{fig:Non_Local_Boundary}
\end{figure}
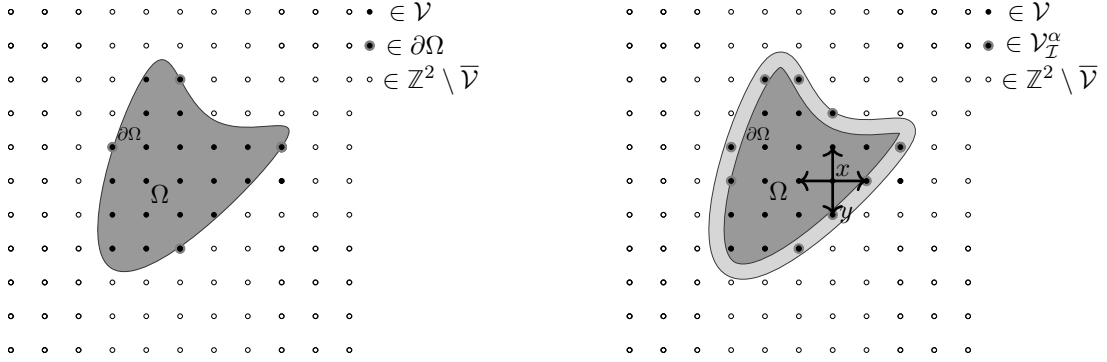

  The next remark 
provides an intuition for appropriate setup of parameters $\alpha,\Theta \in 
\mathcal{F}_{\ol{\mathcal{V}}\times \ol{\mathcal{V}}}$.
\begin{remark}(\textbf{Role of parameters in modeling nonlocal diffusion 
processes.}) In our work we differentiate the parameters $\alpha,\Theta$ by 
their 
role played in modeling 
nonlocal diffusion processes of the form
\eqref{eq:Non_Local_Dif}. More precisely, we use the
antisymmetric mapping $\alpha \in 
\mathcal{F}_{\ol{\mathcal{V}}\times \ol{\mathcal{V}}}$ for definition 
of first 
order derivative operators $\mathcal{D}^{\alpha},\mathcal{G}^{\alpha}, 
\mathcal{N}^{\alpha}$ and the symmetric 
mapping  
$\Theta \in \mathcal{F}_{\ol{\mathcal{V}}\times \ol{\mathcal{V}}}$ for 
specifying
the constitutive function at each $x \in \mathcal{V}$ that controls the 
smoothing 
properties of operator \eqref{eq:combinatorial_Laplacian}. 
Instances of 
$\alpha,\Theta$ along with an analytical ablation study will be presented in 
section \ref{sec:Related-Work-PDE}. 
\end{remark}

\subsection{The Assignment Flow Approach}\label{sec:Assignment-Flow}
We summarize the assignment flow approach introduced by \cite{Astrom:2017ac} 
and refer to \cite{Schnorr:2019aa} for more background and a 
review of related work.

\subsubsection{Assignment Manifold.}\label{subsec:Assign_Man}
Let $(\mc{X},d_{\mc{X}})$ be a metric space and 
\begin{equation}\label{eq:def-metric-data}
\mc{X}_{n} 
= \{X(x) \in \mathcal{X} \colon x \in \mathcal{V}\}
\end{equation}
be given data on a graph 
$(\mathcal{V},\mathcal{E},\Omega)$ as specified in Section \ref{sec:Nonlocal-Calculus}. 
We encode assignments of data $X(x),\; x \in \mc{V}$, to a set 
\begin{equation}\label{eq:def-prototypes}
\mc{X}^{\ast} 
= \{X^{\ast}_j \in \mc{X}, j \in \mc{J}\},\qquad
c:=|\mc{J}|
\end{equation}
of predefined 
prototypes by \textit{assignment vectors}
\begin{equation}\label{eq:def-Wx} 
W(x) = 
(W_{1}(x),\dotsc,W_{c}(x))^{\T} \in \mc{S}, 
\end{equation}
where $\mc{S}=\rint\Delta_{c}$ denotes
the relative interior of the probability simplex $\Delta_{c}\subset 
\R_{+}^{c}$ that we turn into a Riemannian 
manifold $(\mc{S},g)$ with the Fisher-Rao metric $g$ from 
information geometry \cite{Amari:2000aa,Ay:2017aa} at each $p \in 
\mc{S}$ 
\begin{equation}\label{eq:Fisher_Rau}
g_{p}(u,v) = \sum_{j \in \mc{J}} \frac{u_{j} {v}_{j}}{p_{j}} = \langle 
u,v\rangle_p,\qquad
u,v \in T_{0},
\end{equation}
with tangent space $T_{0}$ given by \eqref{eq:def-mcT0}. 
The \textit{assignment manifold} $(\mc{W},g)$ is defined as the 
product space $\mc{W}=\mc{S}\times\dotsb\times\mc{S}$ of $n = |\mc{V}|$ such 
manifolds. Points on the assignment manifold row-stochastic matrices with full 
support are denoted by
\begin{equation}\label{eq:def-mcW}
W = (\dotsc,W(x),\dotsc)^{\T}\in \mc{W}\subset\R_{++}^{n\times c},\qquad x\in\mc{V}.
\end{equation}
The assignment manifold has the trivial tangent bundle $T\mc{W}$ with $T_{W}\mc{W}=\mc{T}_{0},\,\forall W\in\mc{W}$ and  tangent space 
\begin{equation}\label{eq:def-mcT0}
\mc{T}_{0} = T_{0}\times\dotsb\times T_{0},\qquad\quad
T_{0}=\{v\in\R^{c}\colon \la\eins_{c},v\ra=0\}.
\end{equation}
The metric \eqref{eq:Fisher_Rau} naturally extends to
\begin{equation}
g_{W}(U,V) = \sum_{x\in\mc{V}} g_{W(x)}\big(V(x),U(x)\big),\qquad U,V \in \mc{T}_{0}.
\end{equation}
The orthogonal projection onto $T_{0}$ is given by
\begin{equation}\label{eq:def-Pi0}
\Pi_{0}\colon \R^{c}\to T_{0}, \qquad
\Pi_{0}(u) = u-\la\eins_{\mc{S}},u\ra\eins_{c},\qquad
\eins_{\mc{S}}:=\frac{1}{c}\eins_{c}.
\end{equation}
The orthogonal projection onto $\mc{T}_{0}$, also denoted by $\Pi_{0}$ for simplicity, is 
\begin{equation}
\Pi_{0}\colon \R^{n\times c}\to\mc{T}_{0},\qquad
\Pi_{0}D = \big(\dotsc,\Pi_{0}D(x),\dotsc\big)^{\T}.
\end{equation}

\subsubsection{Assignment Flows.}
Based on the given data and prototypes, we define the distance vector field on 
$\mc{V}$ by
\begin{equation}\label{eq:Dis_F}
	D_{\mc{X}}(x) = 
	\big(d_{\mc{X}}(X(x),X_{1}^{\ast}),\dotsc,d_{\mc{X}}(X(x),X_{c}^{\ast})\big)^{\T},\qquad
	x \in \mc{V}.
\end{equation}
This data representation is lifted to $\mc{W}$ 
to obtain the \textit{likelihood 
vectors}
\begin{equation}\label{schnoerr-eq:def-Li}
L(x) \colon \mc{S} \to \mc{S},\qquad
L(W)(x)
= \frac{W(x)\odot e^{-\frac{1}{\rho} D_{\mc{X}}(x)}}{\la W(x),e^{-\frac{1}{\rho} 
D_{\mc{X}}(x)} \ra},\qquad x \in \mc{V},\quad\rho >0,
\end{equation}
where the exponential function applies componentwise and $\odot$ denotes the componentwise multiplication
\begin{equation}\label{eq:def-odot}
(p\odot q)_{j} = p_{j}q_{j},\quad j\in[c],\qquad p, q\in\mc{S}
\end{equation}
of vectors $p, q$. Accordingly, we denote componentwise division of vectors by
\begin{equation}
\frac{v}{p} = \Big(\frac{v_{1}}{p_{1}},\dotsc,\frac{v_{c}}{p_{c}}\Big)^{\T},\qquad p\in\mc{S}
\end{equation}
for strictly positive vectors $p$.
 
The map \eqref{schnoerr-eq:def-Li} is based on the affine $e$-connection of information geometry \cite{Amari:2000aa,Ay:2017aa}. The 
scaling parameter $\rho > 0$ normalizes the a priori unknown scale 
of the components of $D_{\mc{X}}(x)$. 
Likelihood vectors are spatially regularized by the 
\textit{similarity map} and the \textit{similarity vectors}, respectively, given 
for each $x \in \mc{V}$ by
\begin{equation}\label{eq:def-Si}
S(x) \colon \mc{W} \to \mc{S},\qquad
S(W)(x) = \Exp_{W(x)}\Big(\sum_{y \in \mc{N}(x)}
\Omega(x,y) 
\Exp_{W(x)}^{-1}\big(L(W)(y)\big)\Big),
\end{equation}
where 
\begin{equation}\label{eq:def-Exp}
\Exp\colon \mc{S}\times T_{0} \to\mc{S},\qquad \Exp_{p}(v) = \frac{p \odot e^{\frac{v}{p}}}{\la p,e^{\frac{v}{p}}\ra},\qquad\quad
\frac{v}{p} = \Big(\frac{v_{1}}{p_{1}},\dotsc,\frac{v_{c}}{p_{c}}\Big)^{\T}
\end{equation}
is the exponential map 
corresponding to the $e$-connection. If the exponential map of the Riemannian 
(Levi Civita) connection were used instead, then the term in the round bracket 
of \eqref{eq:def-Si} would be the optimality condition for the weighted 
Riemannian mean of the vectors $\{L(W)(y)\colon y\in\mc{N}(x)\}$ \cite[Lemma 
6.9.4]{Jost:2017aa}. Using the exponential map of the e-connection enables to 
evalute the right-hand side of \eqref{eq:def-Si} in closed form and to define 
the similarity vectors as geometric means of the likelihood vectors 
\cite{Schnorr:2019aa}.

The weights $\Omega(x,y)$ 
determine the regularization properties of the similarity map, cf.~Remark \ref{rem:regularization} below. 
They satisfy \eqref{eq:Omega_Mat} and the additional constraint
\begin{align}\label{eq:Omega_AF}
	\sum_{y\in \mathcal{N}(x)} \Omega(x,y) = 1 ,\quad \forall x \in \mathcal{V}.
\end{align} 
The \textit{assignment flow} is induced on the assignment manifold $\mc{W}$ by solutions $W(t,x) = W(x)(t)$ of the system of nonlinear ODEs
\begin{equation}\label{eq:assignment-flow}
\dot W(x) = R_{W(x)} S(W)(x),\qquad W(0,x)=W(x)(0)\in \eins_{\mc{S}},\quad x 
\in 
\mc{V},
\end{equation}
where the map 
\begin{equation}\label{eq:def-Rp}
R_{p}=\Diag(p)-p p^{\T},\qquad p\in\mc{S} 
\end{equation}
corresponds to the inverse metric tensor expressed in the embedding coordinates of the ambient Euclidean space $\R^{c}$, which 
turns the right-hand side 
into the tangent vector field 
\begin{equation}
\mc{V}\ni x\mapsto R_{W(x)}S(W)(x)
= \Diag\big(W(x)\big)S(W)(x)-\la W(x),S(W)(x)\ra W(x) \;\in\; T_{0}. 
\end{equation}
Integrating the system \eqref{eq:assignment-flow} numerically 
\cite{Zeilmann:2020aa} yields integral assignment vectors 
$W(t,x),\;x\in\mc{V}$, for $t\to\infty$, that uniquely assign a label from the 
set $\mc{X}^{\ast}$ to each data point $X(x)$ \cite{Zern:2020aa}.
\begin{remark}[\textbf{Regularization}]\label{rem:regularization}
From the viewpoint of variational imaging, \textit{regularization} of the 
assignment flow has to be understood in a broad sense: The parameters $\Omega$ 
define by \eqref{eq:def-Si}, at each location $x$ and locally within 
neighborhoods $\mc{N}(x)$, what similarity of the collection of likelihood 
vectors $L(W)(y),\,y\in\mc{N}(x)$, which represent the input data, really means 
in terms of a corresponding geometric average, called similarity vector 
$S(W)(x)$. Unlike traditional variational approaches where regularization 
affects the primary variables directly, regularization of the assignment flow 
is accomplished more effectivly by affecting \textit{velocities} that 
\textit{generate} the primary assignment variables: the vector field $S(W)$ 
drives the assignment flow \eqref{eq:assignment-flow}. Figure 
\ref{fig:Labeling_OCT_example} illustrates two applications of the assignment 
flow approach using data-driven nonlocal regularization. Learning the 
regularization parameters $\Omega$ from data was studied by 
\cite{Huhnerbein:2021th,Zeilmann:2022ul}.
\end{remark}

\begin{figure}
	\captionsetup{font=footnotesize}
	\centering
	\scalebox{0.95}{
	\begin{tikzpicture}[scale=1.5, remember picture]
		\draw[ultra thick,loosely dotted] (2.5,3) -- (2.5,2.25);
		\draw[ultra thick,loosely dotted] (2.5,1.9) -- (2.5,-1.3);
		\node[rectangle, rounded corners] (L) at (4,2) {$L(W)(x)$};
		\node[rectangle, rounded corners] (S) at (5,1) 
		{$S(W)(x)$};
		\node[rectangle, rounded corners] (W) at (3,1) 
		{$W(t,x)$};
		\node[rectangle, rounded corners] (dotW) at (4,0) 
		{$\dot{W} = 
			R_{W(x)}S(W)(x)$};
		\node[rectangle, rounded corners] (D) at (1,2) {$D_{\mc{X}}(x)$};
		\node[rectangle, rounded corners] (G) at (1.8,1) {$X^{\ast}_j, j \in 
			\mathcal{J}$};
		\node[rectangle, rounded corners] (F) at (0,1) {$X(x), x  \in 
		\mathcal{V}$};
		\node (textD) at (1,2.75) [text width=5em,align=center] {distance 
			matrix};
		\node (textS) at (6,1) [text width=5em,align=center] {similarity 
			matrix};
		\node (textF) at (0.05,1.7) {data};
		\node (textFlow) at (4,-0.5) {assignment flow};
	\draw[fill = gray!20,draw = black,rounded corners,dashed] (-4.8,1.5) 
	rectangle (-1.8,-0.3);
	\node[rectangle, rounded corners] (feature) at (0.75,-1.2) 
	{\large metric space $\mathcal{X}$};
	\node[rectangle, rounded corners] (assignment) at (4.5,-1.2) 
	{\large assignment manifold $\mathcal{W}$};
	\path (F) edge[->,thick] (D);
	\path (G) edge[->,thick] (D);
	\path (D) edge[->,thick] node[anchor=south] (exp) {$\exp_{W(x)}$} (L);
	\path (L) edge[->,thick] node[anchor=south west] (mean) {} (S);
	\path (S) edge[->,thick] (dotW);
	\path (dotW) edge[->,thick] (W);
	\path (W) edge[->,thick] (L);
	\node[rectangle, rounded corners] (S0) at (-2.1,2) {$S(0)(x)$};
	\draw[ultra thick,loosely dotted] (-0.75,3) -- (-0.75,2.25);
	\draw[ultra thick,loosely dotted] (-0.75,1.9) -- (-0.75,-1.3);
	\node (exp_1) at (-0.5,2.23)
	{$\exp_{\mathbb{1}_{\mathcal{W}}}(-\Omega D_{\mathcal{X}})$} ;
	\path (D) edge[->,thick]  (-1.6,2);
	\node[rectangle, rounded corners] (S) at (-3.3,1) 
	{$\min \limits_{S \in \mathcal{W}} J(S) = -\frac{1}{2}\langle S, \Omega S 
		\rangle$};
	\node[rectangle, rounded corners] (S) at (-3.3,0.5) 
	{via};
	\node[rectangle, rounded corners] (S) at (-3.2,0) 
	{$\dot{S}(x) = R_{S}(\Omega S)(x)$};
	\node at (-3.35,-0.5) 
	{nonlocal geometric diffusion};
	\path (S0) edge[->,thick] (-3.2,1.7);
	\path[<->,thick,draw] (-1.45,0.5)--(3,0.5) -| (W);
	\node[scale=0.8,rectangle,rounded corners] at (0.8,0.2) {$W(t) = 
		\exp_{\eins_{\mc{W}}}\Big(\int_{0}^{t}\Pi_{0}S(\tau)\textup{d}\tau\Big)$};
			
	\path[<-,thick,draw] (-4.8,0.5)--(-5,0.5) -- (-5,-1.2) -- (-4.8,-1.2);
	\draw[fill = gray!20,draw = black,rounded corners,dashed] (-5.5,-0.7) 
	rectangle (-0.8,-1.4);
	\node[scale=0.9] at (-3.15,-1) {$\partial_{t}S(x,t) = R_{S(x,t)}\Big( 
		\frac{1}{2}\mathcal{D}^{\alpha}\big(\Theta 
		\mathcal{G}^{\alpha}(S)\big)+\lambda S \Big)(x,t)$};
	\node[draw,rotate=90] at (-5.25,-0.1) {Section \ref{sec:Non_Local_PDE}};
\end{tikzpicture}
}
	\caption{\textbf{Inference of label assignments via assignment 
	flows.}  \textbf{Center 
	column:} Application 
	task of assigning data to prototypes in a metric space. 
	\textbf{Right column:} Overview of the geometric approach \cite{Astrom:2017ac}. The data 
	are represented by the 
	distance matrix $D_{\mathcal{X}}$ and by the likelihood vector field $L(W)$ on the assignment manifold $\mathcal{W}$. 
	The similarity vectors $S(W)(x)$, determined through geometric averaging of the likelihood vectors, drive the assignment flow whose numerical geometric
	integration result in spatially coherent and unique label assignment to the data. \textbf{Left column:} 
	Alternative equivalent reformulation of the assignment flow 
	\cite{Savarino:2019ab} which separates (i) the influence of the data that 
	only determine the initial point of the flow (cf.~\eqref{eq:S-flow-S}), and 
	(ii) the influence of the parameters $\Omega$ that  parametrize the vector 
	field which drives the assignment flow. This enables to derive the novel 
	nonlocal geometric diffusion equation in Section \ref{sec:Non_Local_PDE}.}
	\label{fig:Assignment_Flow_Illustration}
\end{figure}
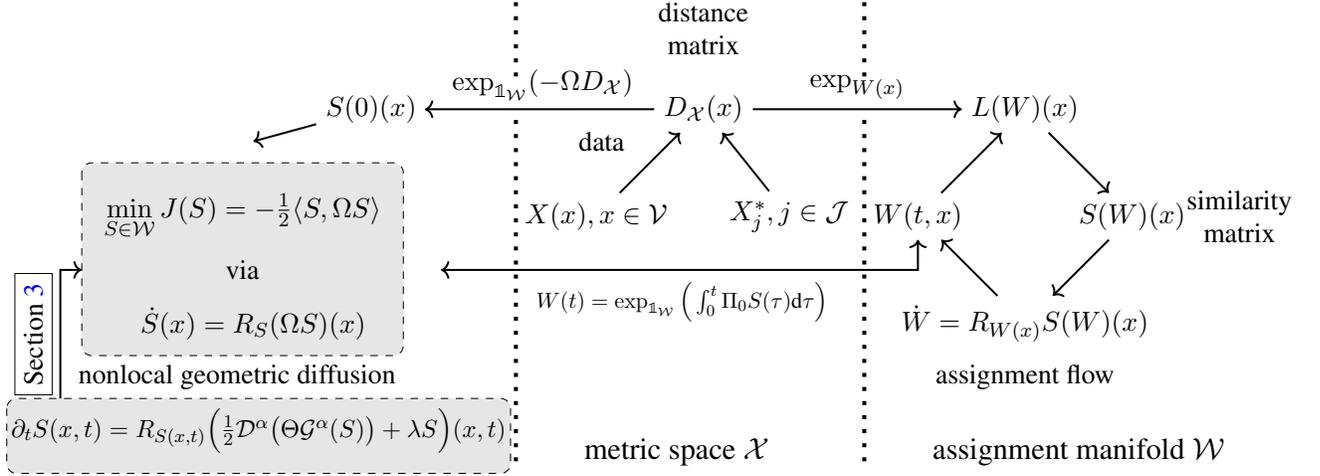

\subsubsection{S-Flow Parametrization.}
We adopt from \cite[Prop.~3.6]{Savarino:2019ab} the \textit{S-parametrization} of the assignment flow system \eqref{eq:assignment-flow}
\begin{subequations}\label{eq:S_parameterization}
\begin{align}
	\label{eq:S-flow-S}
	\dot S = R_{S}(\Omega S),&\qquad
	S(0) = \exp_{\mathbb{1}_\mathcal{W}}(-\Omega D_{\mc{X}}),\\
	\dot W = R_W(S),&\qquad W(0) = \mathbb{1}_{\mathcal{W}}, \quad \eins_{\mc{W}}(x)=\eins_{\mc{S}},\quad x\in\mc{V},
	\label{eq:W_fl_couples} 
\end{align}
\end{subequations}
where both $S$ and $W$ are points on $\mc{W}$ and hence have the format \eqref{eq:def-mcW} and
\begin{gather}
R_{S}(\Omega S)(x) = R_{S(x)}\big((\Omega S)(x)\big),\quad
(\Omega S)(x) = \sum_{y\in\mc{N}(x)}\Omega(x,y)S(y),
\label{eq:RS-Omega} \\
\exp_{\eins_{\mc{W}}}(-\Omega D_{\mc{X}}):=\big(\dotsc,\Exp_{\eins_{\mc{S}}}\circ 
R_{\eins_{\mc{S}}}(-(\Omega 
D_{\mc{X}})(x)),\dotsc\big)^{\T}\in\mc{W},\quad x\in\mc{V}, \label{eq:Small_exp}
\end{gather}
with the mappings $\Exp_{p}, R_{p},\,p\in\mc{S}$ defined by \eqref{eq:def-Exp} and \eqref{eq:def-Rp}, respectively. In view of \eqref{eq:Small_exp}, we define the \textit{lifting map}
\begin{equation}\label{eq:def-lifting-map}
\exp_{p}\colon T_{0}\to\mc{S},\qquad
\exp_{p}(v):= \Exp_{p}\circ R_{p} v
= \frac{p\odot e^{v}}{\la p,e^{v}\ra},\qquad
p\in\mc{S},\; v\in T_{0}
\end{equation}
which satisfies
\begin{subequations}\label{eq:exp-sum-constant}
\begin{align}\label{eq:exp-sum}
\exp_{\exp_{p}(v)}(v')
&=\exp_{p}(v+v'),\qquad
p\in\mc{S},\quad v,v'\in T_{0}.
\intertext{
In addition, one has (cf.~\eqref{eq:def-mcT0}, \eqref{eq:def-Pi0})
} \label{eq:exp-constant}
\exp_{p}(d) 
&= \exp_{p}(\Pi_{0}d),\quad \forall d\in\R^{c}.
\end{align}
\end{subequations}
Analogous to \eqref{eq:Small_exp}, the lifting map \eqref{eq:def-lifting-map} extends to
\begin{subequations}
\begin{align}
\exp_{S}\colon \mc{T}_{0}\to\mc{W},\qquad
\exp_{S}(V)=\big(\dotsc, \exp_{S(x)}\big(V(x)\big),\dotsc\big)
\end{align}
\end{subequations}
and the relations \eqref{eq:exp-sum-constant} extend to
\begin{subequations}
\begin{align}
\exp_{\exp_{S}(V)}(V') &= \exp_{S}(V+V'),\qquad
S\in\mc{W},\quad V,V'\in\mc{T}_{0},
\label{eq:exp-S-group} \\
\exp_{S}(D) &= \exp_{S}(\Pi_{0}D),\qquad
\forall D\in\R^{n\times c}.
\end{align}
\end{subequations}

Parametrization \eqref{eq:S_parameterization} has the advantage that $W(t)$ 
depends on $S(t)$, but not vice versa. As a consequence, it suffices to focus 
on \eqref{eq:S-flow-S} since its solution $S(t)$ determines the solution to 
\eqref{eq:W_fl_couples} by \cite[Prop.~2.1.3]{Zern:2020aa}
\begin{equation}\label{eq:S-flow-W}
	W(t) = 
	\exp_{\eins_{\mc{W}}}\Big(\int_{0}^{t}\Pi_{0}S(\tau)\textup{d}\tau\Big).
\end{equation}
In addition, \eqref{eq:S-flow-S} was shown in \cite{Savarino:2019ab} to be the \textit{Riemannian gradient descent flow} with respect to the potential 
\begin{equation}\label{eq:S-Flow_Pot}
J\colon\mc{W}\to\R,\qquad
	J(S) = -\frac{1}{2}\langle S,\Omega S \rangle =	\frac{1}{4} \sum_{x \in 
	\mc{V}}\sum_{y \in \mathcal{N}(x)} \Omega(x,y) 
	\|S(x)-S(y)\|^2 -\frac{1}{2}\|S\|_{F}^{2},
\end{equation}
where $\|\cdot\|_{F}$ denotes the Frobenius (matrix) norm and the vector field $\mc{V}\ni x\mapsto S(x)\in\mc{S}$ is identified with the matrix
\begin{equation}\label{eq:def-S-matrix}
S = (S_{j}(x))_{x\in\mc{V},\,j\in [c]} \in \R_{++}^{n\times c}
\end{equation} 
such that \eqref{eq:RS-Omega} can be written as
\begin{equation}\label{eq:def-Omega-S-matrix}
\big((\Omega S)(x)\big)_{j}
= \sum_{y\in\mc{N}(x)} \big(\Omega(x,y) S(y)\big)_{j}
= \sum_{y\in\mc{N}(x)} \Omega(x,y) S(y,j)
= (\Omega S)_{x,j}.
\end{equation}
Convergence and stability results for the gradient flow \eqref{eq:S-flow-S} have been established by 
\cite{Zern:2020aa}.

\begin{figure*}[ht]
		\scalebox{0.8}{\begin{tikzpicture}
			\foreach \i/\Layer/\Name in
			{1/Layer_12/BM,2/Layer_11/RPE,3/Layer_10/PR2,4/Layer_9/PR1,5/Layer_7/ONL,6/Layer_6/OPL,
				7/Layer_5/INL,8/Layer_4/IPL,9/Layer_3/GCL,10/Layer_2/RNFL,11/Layer_1/ILM}{
				\node[fill = \Layer,minimum size=1pt,scale=0.75, label =
				{[below
					= 0.5ex]}] at (-4.9+0.8*\i,-1){\Name};}
		\node at (9.2,-2) 
		{\includegraphics[height=3cm,width=8cm]{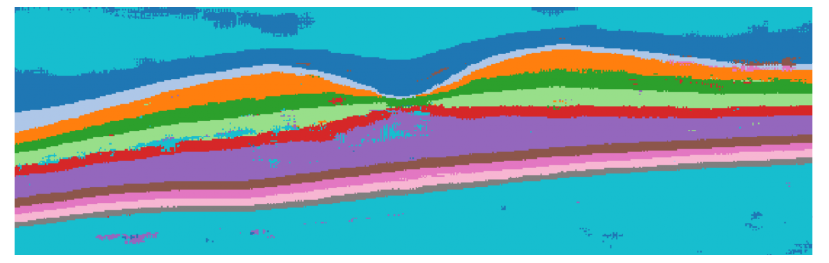}};
	\node at (9.2,-5.2) 
{\includegraphics[height=3cm,width=8cm]{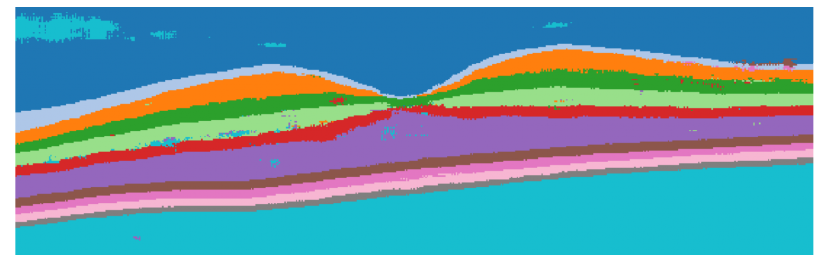}};
	\node at (9.2,-8.4) 
{\includegraphics[height=3cm,width=8cm]{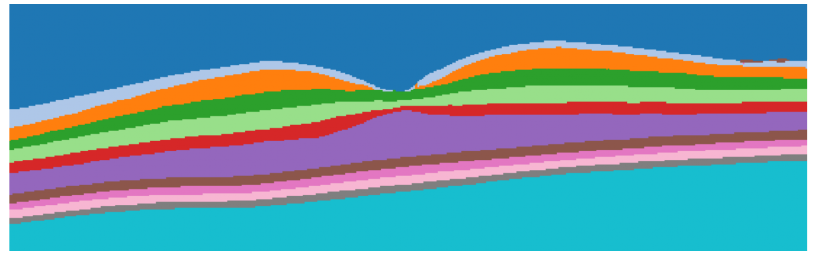}};	
	\node at (0,-4) 
{\includegraphics[height=3cm,width=8cm]{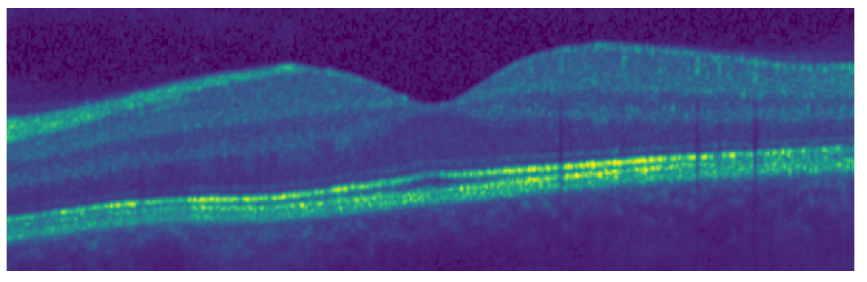}};
\node at (0,-8) 
{\includegraphics[height=3cm,width=8cm]{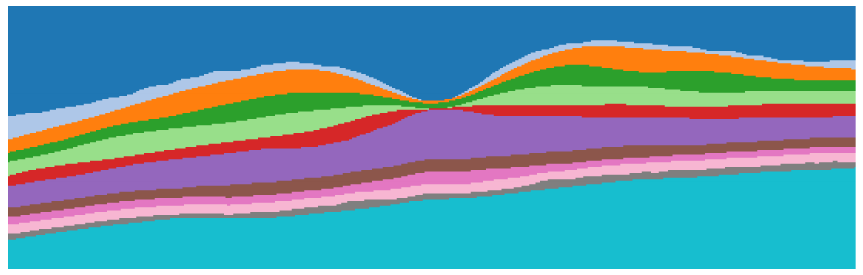}};	
\node[rotate=90] at (-4.35,-4) {\textbf{Noisy data}};
\node[rotate=90] at (-4.35,-7.95) {\textbf{Ground Truth}};
\node[rotate=90] at (5,-2) {\textbf{Resnet}};
\node[rotate=90] at (5,-5.2) {\textbf{Local}};
\node[rotate=90] at (5,-8.4) {\textbf{Nonlocal}};
\node at (-2.4,-12.3) 
{\includegraphics[height=4.3cm,width=4.3cm]{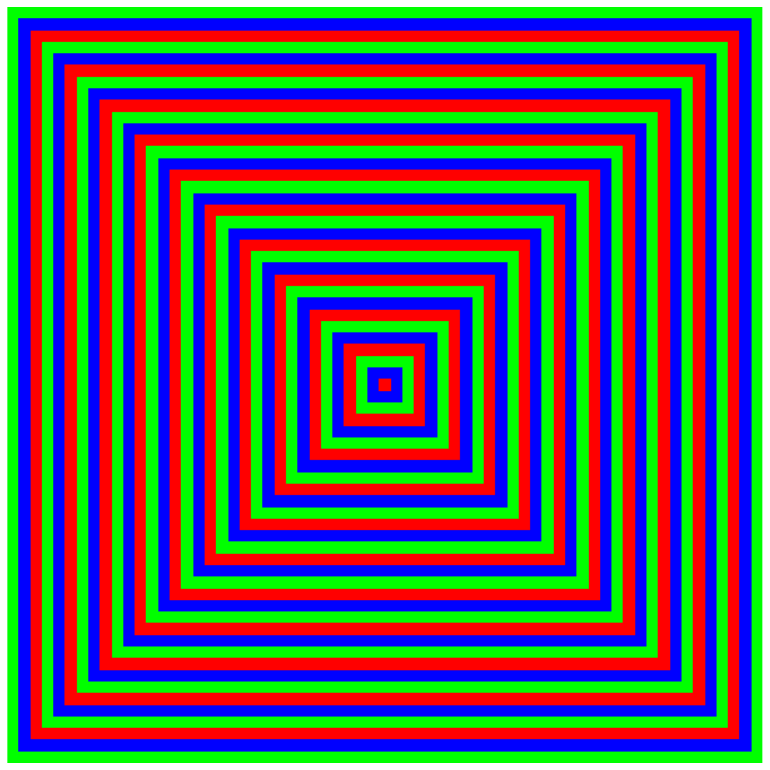}};
\node at (2.12,-12.3) 
{\includegraphics[height=4.3cm,width=4.3cm]{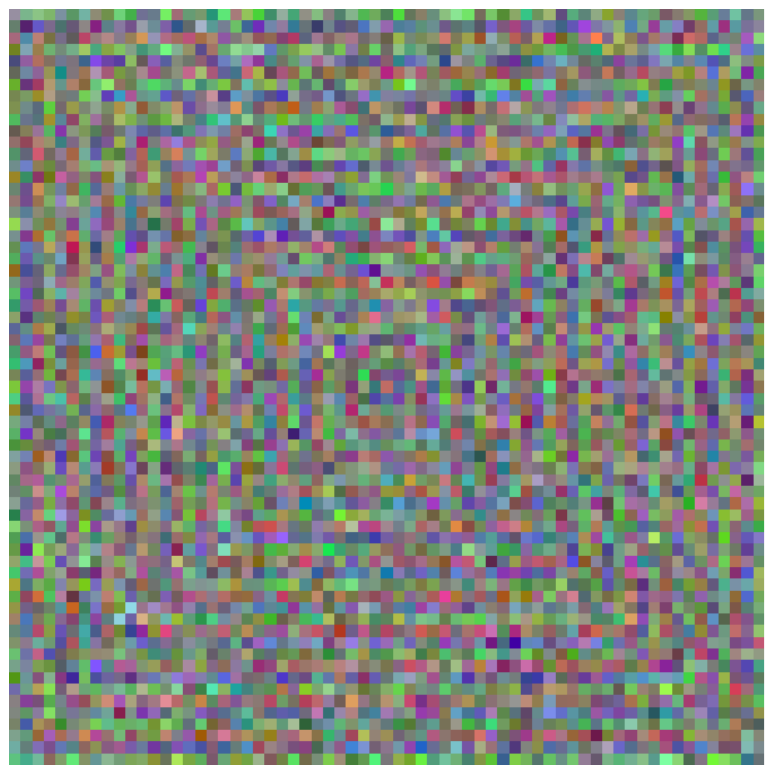}};
\node at (6.6,-12.3) 
{\includegraphics[height=4.3cm,width=4.3cm]{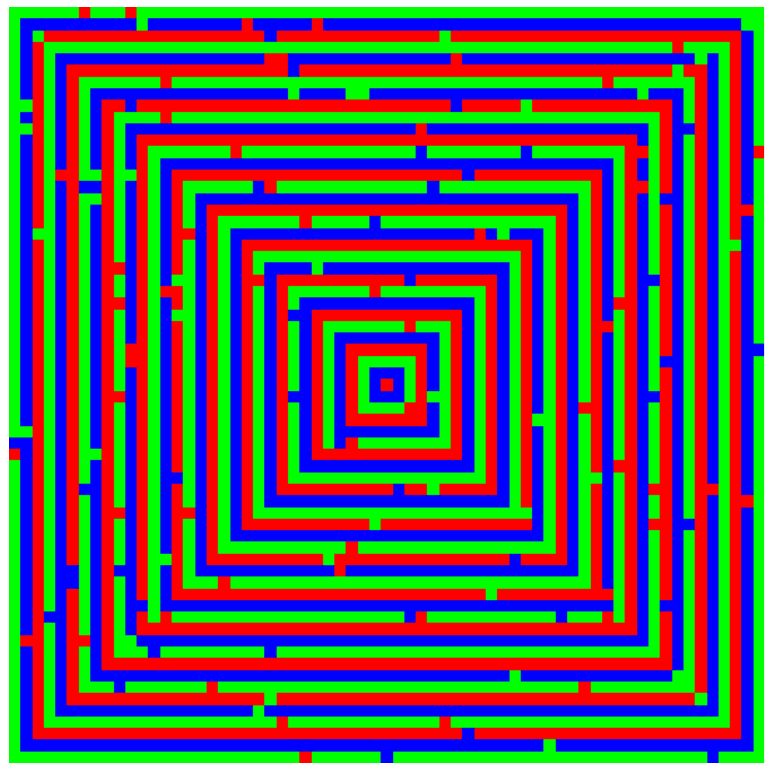}};
\node at (11.1,-12.3) 
{\includegraphics[height=4.3cm,width=4.3cm]{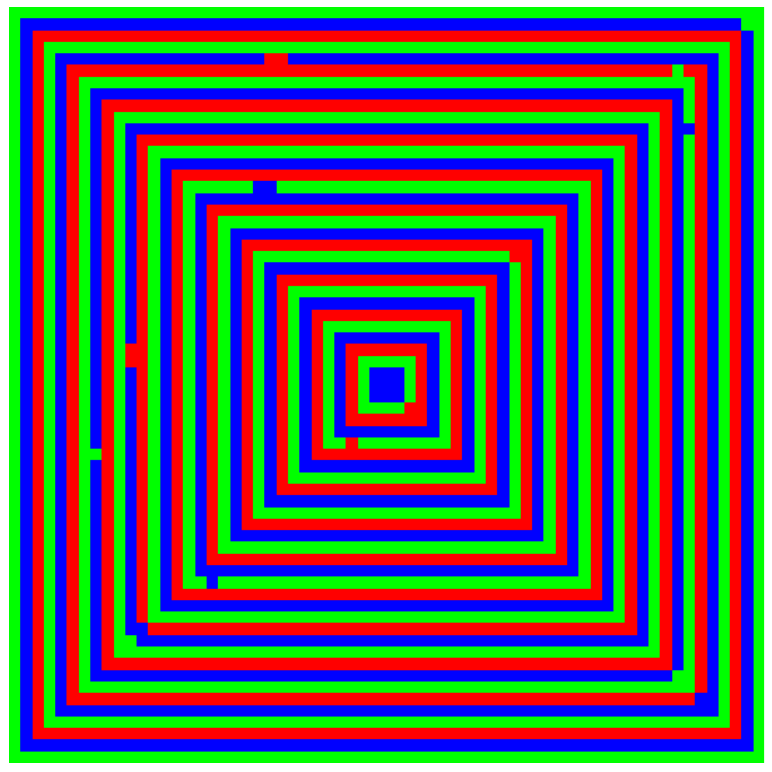}};
\node[scale=0.9] at (-2.35,-15) {\textbf{(f)}};
\node[scale=0.9] at (2.07,-15) {\textbf{(g)}};
\node[scale=0.9] at (6.1,-15) {\textbf{(h)}};
\node[scale=0.9] at (10.6,-15) {\textbf{(i)}};
\node[scale=0.9] at (0,-5.8) {\textbf{(a)}};
\node[scale=0.9] at (0,-9.8) {\textbf{(b)}};
\node[scale=0.9] at (9.2,-3.64) {\textbf{(c)}};
\node[scale=0.9] at (9.2,-6.8) {\textbf{(d)}};
\node[scale=0.9] at (9.2,-10) {\textbf{(e)}};
		\end{tikzpicture}}
	\captionsetup{font=footnotesize}
	\caption{Two image labeling scenarios demonstrating the 
	\textbf{influence 
	of nonlocal regularization}. \textbf{Top}:  
	Application of assignment flows to a 3D medical imaging problem for 
	segmenting the human retina (see \cite{Sitenko:2021vu} a detailed 
	exposition). \textbf{(a)}: A B-scan from a 3D OCT-volume showing a 
	section of the human retina that is corrupted by speckle noise. \textbf{(b)}: 
	The corresponding ground truth labeling with ordered retina layers. 
	\textbf{(c)}: Output from a Resnet 
	that serves as the  
	distance matrix \eqref{eq:Dis_F}. \textbf{(d)}: Result of applying 
	assignment flow with local neighborhoods given by a 3D seven point 
	stencil. \textbf{(e)}: Labeling obtained with \textit{nonlocal} uniform 
	neighborhoods of 
	size $|\mathcal{N}| = 11 \times 11 \times 11$. Increasing the connectivity 
	leads to more accurate labeling that satisfy the ordering constraint depicted in (b).
	\textbf{Bottom}: Labeling of noisy data by assignment flows with data-driven parameters $\Omega$ 
	determined by nonlocal means \cite{Buades:2010aa} using patches of size $7\times 7$ pixels. \textbf{(f)}: Synthetic image with thin repetitive 
	structure. \textbf{(g)}: Severly corrupted input image to be labeled with $\mathcal{X}^{\ast} = \{\protect\markernine, 
	\protect\markerten,\protect\markereleven \}$. \textbf{(h)},\textbf{(i)}: 
	Labeling by the assignment flow that was regularized with neighborhood sizes $|\mc{N}|=3\times 3$ and $|\mc{N}|=11\times 11$,  
	respectively. Enlarging the neighborhood size $|\mc{N}|$ increases  
	labeling accuracy. }         
	\label{fig:Labeling_OCT_example}
\end{figure*}
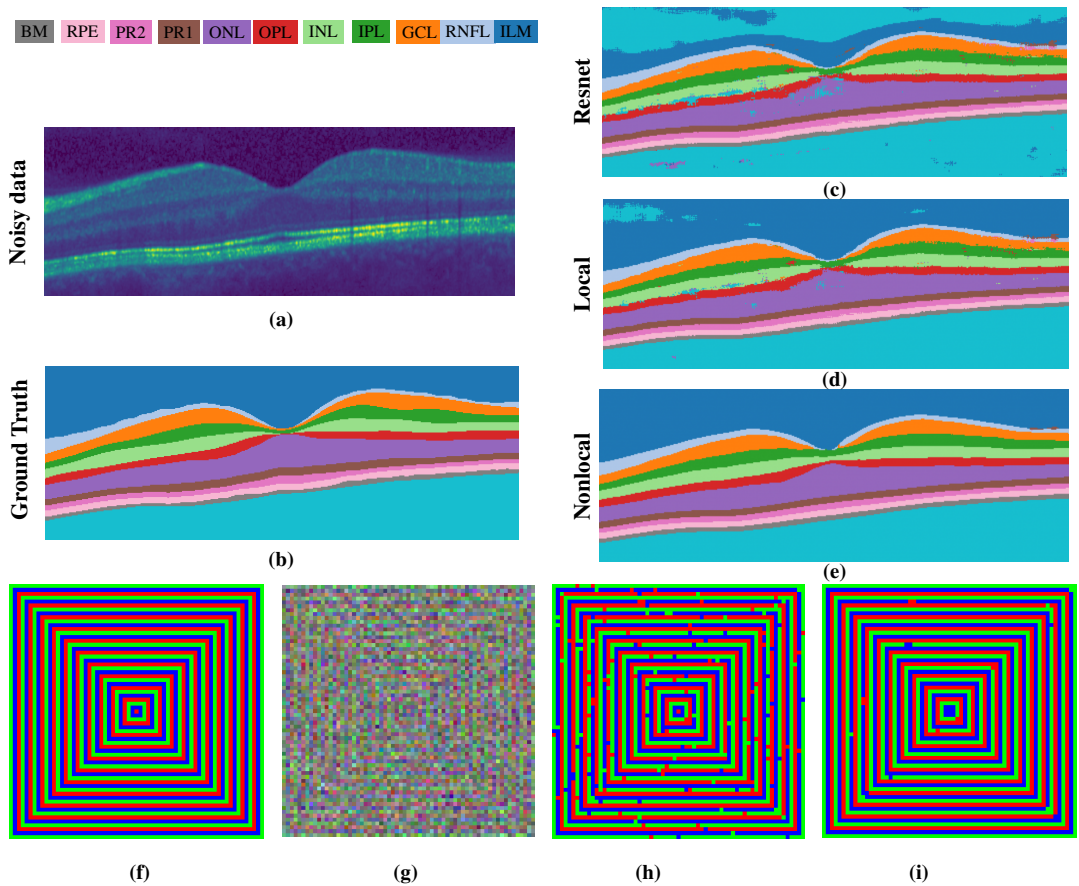

\vspace{0.2cm}

\section{Nonlocal Graph-PDE}\label{sec:Non_Local_PDE}

In this section, 
we show that the assignment flow corresponds to a 
particular nonlocal diffusion process. This results in an equivalent 
formulation of the Riemannian gradient flow \eqref{eq:S-flow-S} in terms of a suitable \textit{nonlinear} extension of the nonlocal \textit{linear} diffusion operator \eqref{eq:Non_Local_Dif}.

\subsection{$S$-Flow: Non-Local PDE 
Formulation}\label{sec:Nonlocal_PDE_Formulation}
We start with specifying 
a general class of parameter matrices $\Omega$ satisfying \eqref{eq:Omega_Mat} and 
\eqref{eq:Omega_AF}  
in terms of an anti-symmetric and symmetric mappings $\alpha 
\in \mathcal{F}_{\overline{\mathcal{V}}\times \overline{\mathcal{V}}}$ and 
$\Theta \in \mathcal{F}_{\overline{\mathcal{V}}\times \overline{\mathcal{V}}}$ 
respectively.
\begin{lemma}\label{Help_Lemma}
	Let 
\begin{equation}
\begin{aligned}\label{eq:alpha-Omega-constraints}
\alpha &\in \mathcal{F}_{\overline{\mathcal{V}}\times 
	\overline{\mathcal{V}}},
	\\
\Theta &\in \mathcal{F}_{\overline{\mathcal{V}}\times 
	\overline{\mathcal{V}}},
\end{aligned}
\qquad\qquad
\begin{aligned}
\alpha(y,x) &= -\alpha(x,y),\quad\qquad
\forall x, y\in\mc{F}_{\ol{\mc{V}}\times\ol{\mc{V}}}, \\
\Theta(x,y) &= \Theta(y,x)\geq 0,\qquad
\forall x, y\in\mc{F}_{\ol{\mc{V}}\times\ol{\mc{V}}},
\end{aligned}
\end{equation}
	be anti-symmetric and nonnegative symmetric 
	mappings, respectively. Assume further that $\alpha$ satisfies
\begin{equation}\label{eq:Mapping_assumption}
\alpha(x,y) = 0, \qquad  
\forall x,y \in \mathcal{V}_{\mathcal{I}}^{\alpha}.
\end{equation}
Then, for neighborhoods $\mc{N}(x)$ defined by  
\eqref{eq:def-neighborhoods} and with parameter matrix
\begin{equation}\label{eq:Avaraging_Matrix}
	\Omega(x,y) =	\begin{cases}
		\Theta(x,y)\alpha^2(x,y), \quad &\text{if } x \neq y,  \\
		\Theta(x,x), \quad &\text{if } x = 
		y, 
	\end{cases} \qquad
	x,y \in \mc{V},	 
\end{equation} for each function $f \in 
\mathcal{F}_{\overline{\mathcal{V}}}$ with 
$f|_{\mathcal{V}_{\mathcal{I}}^{\alpha}} = 
0$, 
the identity 
	\begin{equation}\label{eq:Reform_Av}
	\sum_{y \in \mathcal{V}}\Omega(x,y)f(y) = 
	\frac{1}{2}\mathcal{D}^{\alpha}\big(\Theta 
	\mathcal{G}^{\alpha}(f)\big)(x)+\lambda(x)f(x), \qquad \forall x 
	\in \mathcal{V},\quad
\forall f\in\mc{F}_{\ol{\mc{V}}}\colon
f\big|_{\mc{V}_{\mc{I}}^{\alpha}}=0
	\end{equation}
holds with $\mathcal{D}^{\alpha}, \mathcal{G}^{\alpha}$ given by 
\eqref{eq:NL-Div_op},\eqref{eq:def-nonl-grad} and  
\begin{equation}\label{eq:def-lambda-x}
	\lambda(x) = \sum_{y \in 
		\overline{\mathcal{V}}} \Theta(x,y)\alpha^{2}(x,y)+\Theta(x,x),\quad x\in\mc{V}.
\end{equation}
In addition, if $\lambda(x) \leq 1$ in \eqref{eq:def-lambda-x} for all $x 
\in \mathcal{V}$, then $\Omega$ given by \eqref{eq:Avaraging_Matrix} satisfies 
\eqref{eq:Omega_Mat}, and equality $\lambda(x) = 1, \forall x \in \mathcal{V}$ is achieved if property \eqref{eq:Omega_AF} holds.
\end{lemma}
\noindent
\textit{Proof.} Appendix \ref{app:Nonlocal_PDE_Formulation}.

\vspace{0.2cm}
\begin{remark}[\textbf{Comments}]
\label{rem:Omega_Restriction}
Lemma \ref{Help_Lemma} characterizes a class of parameter matrices $\Omega$ 
whose action \eqref{eq:Reform_Av} admits an representation using the nonlocal 
operators from Section \ref{sec:Nonlocal-Calculus}.

 Some comments follow on parameter matrices \textit{not} covered by  Lemma 
 \ref{Help_Lemma}, due to the imposed constraints.
	\begin{enumerate}[(i)]
		\item By ignoring the \textit{nonnegativity constraint} of \eqref{eq:alpha-Omega-constraints} imposed on 
		$\Omega$ through the mapping $\Theta$, Lemma \ref{Help_Lemma} additionally 
		covers a class of nonlocal graph Laplacians proposed in 
		\cite{Elmoataz:2015aa} and \cite{Gilboa:2009aa} for the aim of image 
		inpainting. We refer to Section \ref{sec:Related-Work-PDE} for a more detailed 
		discussion.        
		\item Due to assuming \textit{symmetry} of the mapping 
		$\Theta$, formulation \eqref{eq:Avaraging_Matrix} does \textit{not} cover 
		nonlocal 
		diffusion processes on \textit{directed} graphs 
		$(\mathcal{V},\mathcal{E},\Omega)$.  
		\item Imposing \textit{zero nonlocal Dirichlet 
		boundary conditions} is essential for relating assignment flows to the 
		specific class of nonlocal PDEs related to \eqref{eq:Reform_Av}, see 
		Proposition \ref{Prop:Non_Local_PDE} below.
	\end{enumerate}
As argued in \cite{Zern:2020aa} by a range of counterexamples, using  
nonsymmetric parameter matrices $\Omega$ compromises convergence of the assignment flow 
\eqref{eq:S-flow-S} to integral solutions (labelings) and is therefore not 
considered. The study of more general parameter matrices is left for future 
work, see Section \ref{sec:Conclusion} and Section 
\ref{sec:general_processes} for modifying the identity \eqref{eq:Reform_Av} 
in view of nonsymmetric parameter matrices $\Omega$.
\end{remark}

Next, we generalize the common \textit{local} boundary conditions for PDEs to nonlocal \textit{volume 
constraints} for \textit{nonlocal} PDEs on discrete domains. Following \cite{Du:2012aaD}, given an 
antisymmetric mapping $\alpha$ as in \eqref{eq:Interaction_Dom} and Lemma \ref{Help_Lemma}, the 
 natural domains $\mathcal{V}^{\alpha}_{\mathcal{I}_{N}},  
\mathcal{V}^{\alpha}_{\mathcal{I}_D}$ for imposing nonlocal
\emph{Neumann} and \emph{Dirichlet} constraints are given by a 
disjoint decomposition of the interaction domain \eqref{eq:Interaction_Dom}
\begin{equation}
\mathcal{V}^{\alpha}_{\mathcal{I}} = \mathcal{V}^{\alpha}_{\mathcal{I}_{N}} 
\dot \cup 
\mathcal{V}^{\alpha}_{\mathcal{I}_D}. 
\end{equation}
The following 
proposition reveals how the flow \eqref{eq:S-flow-S},  with $\Omega$ satisfying 
the assumptions of Lemma 
\ref{Help_Lemma}, can be reformulated as a 
nonlocal partial difference equation with zero nonlocal Dirichlet boundary  
condition imposed on the entire interaction domain, i.e. 
$\mathcal{V}^{\alpha}_{\mathcal{I}} = 
\mathcal{V}^{\alpha}_{\mathcal{I}_D}$. Recall the definition of the manifold $\mc{S}$ of discrete probability vectors with full support in 
connection with Eq.~\eqref{eq:def-Wx}. 
\begin{proposition}{\textbf{(S-flow as nonlocal 
G-PDE})}\label{Prop:Non_Local_PDE}
	Let $\alpha,\Theta \in 
	\mathcal{F}_{\overline{\mathcal{V}}\times\overline{\mathcal{V}}}$ be as in  
	\eqref{eq:Mapping_assumption}. Then the flow \eqref{eq:S-flow-S} 
	with 
	$\Omega$ given through \eqref{eq:Avaraging_Matrix} admits the 
	representation 
	\begin{subequations}\label{eq:Non_Local_PDE}
		\begin{align}
		&\partial_{t}S(x,t) = R_{S(x,t)}\Big( 
		\frac{1}{2}\mathcal{D}^{\alpha}\big(\Theta 
		\mathcal{G}^{\alpha}(S)\big)+\lambda S \Big)(x,t),&\quad &\text{ on }  
		\mathcal{V}\times \R_{+},\label{S_Flow_PDE}\\
		&\ol{S}(x,t) = 0, &\quad &\text{ on } 
		\mathcal{V}^{\alpha}_{\mathcal{I}}\times 
		\R_{+} \label{S_Flow_boundary1},\\
		&S(x,0) = \overline{S}(x)(0), &\quad &\text{ on } 
		\ol{\mc{V}}\times\R_{+},
		\label{S_Flow_boundary2}
		\end{align}
	\end{subequations}
	where $\lambda = \lambda(x)$ is given by \eqref{eq:Mapping_assumption} and
	$\ol{S}\in\mc{F}_{\ol{\mc{V}},\R_{+}^{c}}$ denotes the zero extension of the $\mc{S}$-valued 
	vector field $S \in 
	\mathcal{F}_{\mathcal{V},\mathcal{S}}$ to 
	the interaction domain 
	$\mc{V}_{\mc{I}}^{\alpha}$.
\end{proposition}
\noindent
\textit{Proof.} Appendix \ref{app:Nonlocal_PDE_Formulation}.

Proposition \ref{Prop:Non_Local_PDE} states the equivalence of the potential flow \eqref{eq:S-flow-S}, with $\Omega$ defined by 
\eqref{eq:Avaraging_Matrix}, and the nonlocal diffusion process 
\eqref{eq:Non_Local_PDE} with zero nonlocal Dirichlet boundary condition. We now explain 
that the system \eqref{S_Flow_PDE} can represent 
\textit{any} descent flow of the form \eqref{eq:S-flow-S} defined in terms of an 
\textit{arbitrary} nonnegative symmetric mapping $\Omega \in 
\mathcal{F}_{\mathcal{V}\times 
\mathcal{V}}$. Specifically, given such a mapping $\Omega$, 
let the mappings 
$\wt{\alpha},\wt{\Theta} \in 
	\mathcal{F}_{\mathcal{V}\times \mathcal{V}}$ be defined by
	\begin{equation}\label{eq:Omega-by-alpha}
	\wt{\Theta}(x,y) = \begin{cases}
	\Omega(x,y)&
	\text{if}\;y \in \mathcal{N}(x), \\
	0 &
	\text{else}
	\end{cases}, \qquad \wt{\alpha}^2(x,y) = 1, \qquad x,y \in \mathcal{V}.
	\end{equation}
	Further, denote by $\Theta,\alpha \in 
	\mathcal{F}_{\overline{\mathcal{V}} 
		\times \overline{\mathcal{V}}}$ the extensions of 
		$\wt{\alpha},\wt{\Theta}$ to $\ol{\mc{V}}\times\ol{\mc{V}}$ by $0$, 
		that is
	\begin{equation}\label{eq:Omega_Dec}
	\Theta(x,y)= \big(\delta_{\mathcal{V}\times \mathcal{V}}(
	\wt{\Theta})\big)(x,y), \quad \alpha(x,y):= 
	\big(\delta_{\mathcal{V}\times \mathcal{V}}(
	\wt{\alpha})\big)(x,y) \quad x,y \in \overline{\mathcal{V}},
	\end{equation}
	where $\delta_{\mathcal{V}\times \mathcal{V}}\colon \Z^{d}\times \Z^{d} 
	\to\{0,1\}$ is the indicator function of the set $\mc{V}\times\mc{V} 
	\subset \Z^{d}\times\Z^{d}$. Then the 
	potential flow \eqref{eq:S-flow-S} with $\Omega$ satisfying $\Omega(x,y) = 
	\Omega(y,x)$ is equivalently represented by 
	the system \eqref{eq:Non_Local_PDE} with an empty interaction domain 
	\eqref{eq:Interaction_Dom}. 
This shows how Proposition \ref{Prop:Non_Local_PDE} generalizes the assignment 
flow introduced in Section \ref{sec:Assignment-Flow} by ignoring the constraint 
\eqref{eq:Omega_AF} imposed on $\Omega$,  
and thus enables to use a broader class of 
parameter matrices $\Omega$ controlling the labeling process; see also Remark 
\ref{rem:Omega_Restriction}.

\subsection{Tangent-Space Parametrization of the $S$-Flow 
G-PDE}\label{sec:Tangent_Par}
Because $S(x,t)$ solving \eqref{eq:Non_Local_PDE} evolves on the non-Euclidean space $\mc{S}$, applying some standard discretization in order to evaluate \eqref{eq:Non_Local_PDE} numerically will not work. Therefore, motivated by the work \cite{Zeilmann:2020aa} on the geometric numerical integration of the original assignment flow system \eqref{eq:assignment-flow}, we devise a  parametrization of \eqref{eq:Non_Local_PDE} on the \textit{flat} tangent space \eqref{eq:def-mcT0} by means of the equation
\begin{equation}\label{eq:T0_parameterization}
	S(t) = \exp_{S^0}(V(t)) \in\mc{W},\qquad V(t)\in\mc{T}_{0},\qquad \quad S^0 
	= 
	S(0)\in\mc{W},
\end{equation}
where analogous to \eqref{eq:Small_exp}
\begin{equation}\label{eq:def-small-exp}
\exp_{S^0}(V(t)) 
= \big(\dotsc,\exp_{S^0(x)}(-V(x,t)),\dotsc\big)^{\T} \in\mc{W}
\end{equation}
with $\exp_{S^{0}(x)}$ given by \eqref{eq:def-lifting-map}. 
Applying $\frac{d}{dt}$ to both sides and using the expression of the 
differential of the mapping $\exp_{S^0}$ due to \cite[Lemma 
3.1]{Savarino:2019ab}, we get
\begin{equation}\label{eq:dt-S-exp-V}
\dot S(t) = R_{\exp_{S^0}(V(t))}\dot V(t)
\overset{\eqref{eq:T0_parameterization}}{=}
R_{S(t)}\dot V(t).
\end{equation}
Comparing this equation and \eqref{eq:S-flow-S}, and taking into account $R_{S}=R_{S}\Pi_{0}$, shows that $V(t)$ solving the nonlinear ODE
\begin{equation}\label{eq:T0_parameterization2}
\dot V(t) = \Pi_{0}\Omega \exp_{S^0}(V(t)),\qquad V(0) = 
0
\end{equation}
determines $S(t)$ by 
\eqref{eq:T0_parameterization} solving \eqref{eq:S-flow-S}. Hence it suffices to focus on 
\eqref{eq:T0_parameterization2} which evolves on the flat space $\mc{T}_{0}$. 
Repeating the derivation above that resulted in the G-PDE 
representation \eqref{eq:Non_Local_PDE} of the $S$-flow \eqref{eq:S-flow-S}, 
yields the  nonlinear PDE representation of \eqref{eq:T0_parameterization2}
	\begin{subequations}\label{eq:Non_Local_PDE_T0}
		\begin{align}
		&\partial_{t}V(x,t) = \Big( \frac{1}{2}\mathcal{D}^{\alpha}\big(\Theta 
		\mathcal{G}^{\alpha}(\exp_{S^0}(V))\big)+\lambda\exp_{S^0}(V) 
		\Big)(x,t) &  
		&\hskip -0.2cm \text{on }  
		\mathcal{V}\times \R_{+},\label{S_Flow_PDE_T0}\\
		&\ol{V}(x,t) = 0 & &\hskip -0.35cm \text{ on } 
		\mathcal{V}^{\alpha}_{\mathcal{I}}\times \R_{+}, 
		\label{S_Flow_T0_boundary1}\\
		&V(x,0) = \overline{V}(x)(0) & & \hskip -0.25cm \text{on } 
		\ol{\mc{V}}\times\R_{+},
		\label{S_Flow_T0_boundary2}
		\end{align}
	\end{subequations}
where $\ol{V}\in\mc{F}_{\ol{\mc{V}},\mc{T}_{0}}$ 
denotes the zero extension of the $\mc{T}_{0}$-valued vector field to 
the interaction domain $\mc{V}_{\mc{I}}^{\alpha}$. 
From the numerical point of view, this new formulation 
\eqref{eq:T0_parameterization}, \eqref{eq:Non_Local_PDE_T0} has the following expedient properties. Firstly, 
using a parameter matrix as specified by \eqref{eq:Avaraging_Matrix} and 
\eqref{eq:Omega_Dec} enables to define the entire system 
\eqref{eq:Non_Local_PDE_T0} on $\mc{V}$. Secondly, since $V(x,t)$ evolves on 
the flat space $T_{0}$, numerical techniques of geometric integration as 
studied by \cite{Zeilmann:2020aa} can here be applied as well. We utilize this 
fact in Section \ref{sec:Ral_Af_PDE} and in Section 
\ref{sec:Non_Convex_Optimization}.

\subsection{Nonlocal Balance Law}\label{sec:Non_Local_PDE_Bal_Law}
A key property of PDE-based models are balance laws implied by the model; see 
\cite[Section 7]{Du:2013aa} for a discussion of various scenarios. The 
following proposition reveals a \textit{nonlocal} balance law of the assignment 
flow based on the novel G-PDE-based parametrization 
\eqref{eq:Non_Local_PDE_T0}, that we express for this purpose in the form
\begin{subequations}\label{eq:balance-PDE}
\begin{align}\label{eq:balance-PDE-a}
\partial_t V(x,t) + \mathcal{D}^{\alpha}(F(V))(x,t) 
&= b(x,t), \qquad b(x,t) 
= \lambda(x) S(x,t),\qquad x \in \mathcal{V},
\\
F(V(t))(x,y) 
&= -\frac{1}{2}\Big( \Theta 
	\mathcal{G}^{\alpha}\big(\exp_{S^0}(V(t))\big)\Big)(x,y),
\end{align}
\end{subequations}
where $S(x,t) = \exp_{S^0}(V(x,t))$ is given by \eqref{eq:T0_parameterization} and 
$\lambda(x)$ is given by \eqref{eq:def-lambda-x}.
\begin{proposition}[\textbf{nonlocal balance law of assignment flows}]\label{Prop:Non_Local_Bal_Law}
	Under the assumptions of Lemma \ref{Help_Lemma}, let $V(t)$ solve \eqref{eq:Non_Local_PDE_T0}.
	Then, for each component $S_{j}(t) = \{S_{j}(x,t)\colon 
	x\in\mc{V}\},\,j\in[c]$, of $S(t) = \exp_{S^0}(V(t))$, the identity 
	\begin{equation}\label{eq:Non_Local_Balance_Law}
	\begin{aligned}
	\frac{1}{2}\frac{d}{d t} \langle S_j, \mathbb{1} \rangle_{\mathcal{V}} &+ 
	\frac{1}{2}\langle 
	\mathcal{G}^{\alpha}(S_j), \Theta \mathcal{G}^{\alpha}(S_j)   
	\rangle_{\overline{\mathcal{V}}\times \overline{\mathcal{V}}}   
	+ \langle S_j, \phi_{S} - \lambda S_j 
	\rangle_{\mathcal{V}}\\
	 &+ 
	\langle S_j, 
	\mathcal{N}^{\alpha}(\Theta \mathcal{G}^{\alpha}(S_j)) 
	\rangle_{\mathcal{V}_{\mathcal{I}^{\alpha}}} = 0
	\end{aligned}
	\end{equation}
	holds, where the inner products are given by \eqref{eq:Inner_Products} and \eqref{eq:def-olV}, and $\phi_S(\cdot) \in \mathcal{F}_{\mathcal{V}}$ is defined in terms of 
	$S(t) \in \mathcal{W}$ by 
	\begin{equation}\label{eq:Phi_Func}
		\phi_S:\mathcal{V} \to \R,\qquad
		x \mapsto \big\langle S(x),\Pi_{0} \big(\Omega S\big)(x)\rangle.
	\end{equation} 
	\end{proposition}   
\noindent
\textit{Proof.} Appendix \ref{app:sec:Non_Local_PDE_Bal_Law}.

\vspace{0.3cm}
The nonlocal balance law 
\eqref{eq:Non_Local_Balance_Law} comprises four terms. Since $\sum_{j\in[c]}S_{j}(x)=1$ at each vertex $x\in\mc{V}$, the first term of \eqref{eq:Non_Local_Balance_Law} measures the \textit{rate of `mass'} assigned to label $j$ over the entire image. This rate is governed by two interacting processes corresponding to the three remaining terms:
\begin{enumerate}[(i)]
\item \textit{spatial propagation of assignment mass} through the nonlocal diffusion process including nonlocal boundary conditions (second and fourth term);
\item \textit{exchange of assignment mass} with the remaining labels $\{l \in [c]\colon l\neq j\}$  (third term comprising the function $\phi_{S}$ \eqref{eq:Phi_Func}).
\end{enumerate}
We point out that other approaches to image labeling, including Markov random fields and deep networks, do \textit{not} reveal the flow of information during inference in such an explicit manner.

\subsection{Illustration: Parametrization and Nonlocal Boundary Conditions}

In this section, we illustrate two 
aspects of the mathematical results presented above by numerical results:
\begin{enumerate}[(1)]
	\item The use of \textit{geometric integration} for numerically 
	solving the nonlocal G-PDE \eqref{eq:Non_Local_PDE}. 
	Here we exploit a basic numerical scheme established for the 
	assignment flow \eqref{eq:S-flow-S} and the one-to-one 
	correspondence to the nonlocal G-PDE 
	\eqref{eq:Non_Local_PDE}, due to 
	Proposition \ref{Prop:Non_Local_PDE}.
	\item The effect of zero vs.~non-zero nonlocal 
	Dirichlet boundary conditions and uniform vs.~non-uniform 
	parametrizations \eqref{eq:Avaraging_Matrix}. Using non-zero 
	boundary conditions refers to the observation stated above in 
	connection with Equations \eqref{eq:Omega-by-alpha}, 
	\eqref{eq:Omega_Dec}: the nonlocal G-PDE 
	\eqref{eq:Non_Local_PDE} 
	generalizes the assignment flow when constraints are dropped. Here 
	specifically: the homogeneous Dirichlet boundary condition may be 
	non-homogeneous, and the constraint \eqref{eq:Omega_AF} is ignored; 
	see also Remark \ref{rem:Omega_Restriction}.
\end{enumerate}
Topic (1) is addressed here to explain how the results illustrating topic (2) were computed, and to set the stage for Section \ref{sec:Non_Convex_Optimization} that presents an advanced numerical scheme. Item (2) merely illustrates basic choices of the parametrization and boundary conditions. More advanced generalizations of the assignment flow are conceivable, but beyond the scope of this paper; see Section \ref{sec:Conclusion}.

\subsubsection{Numerically Solving the Nonlocal G-PDE By 
Geometric Integration} 
\label{sec:Ral_Af_PDE} 
According to 
Section \ref{sec:Tangent_Par}, imposing the 
homogeneous Dirichlet condition via the interaction domain 
\eqref{eq:Interaction_Dom} makes the right-hand side of \eqref{S_Flow_PDE_T0} 
equivalent to 
\eqref{eq:T0_parameterization2}. Applying to \eqref{S_Flow_PDE_T0} a
simple explicit time discretization with stepsize $h$ 
results in the 
iterative update formula 
\begin{align}\label{eq:Update_Tangent0}
V(x,t+h) \approx V(x,t) +h\Pi_{0}\exp_{S^0(x)}(\Omega V(x,t)),\qquad h>0.
\end{align}
By virtue of the parametrization \eqref{eq:T0_parameterization}, one recovers 
with any nonnegative symmetric mapping $\Omega$ as in Lemma \ref{Help_Lemma} 
the 
\emph{explicit geometric Euler} scheme on $\mc{W}$ 
\begin{subequations}\label{eq:geometric_euler}
\begin{align}
S(t+h) 
&\approx \exp_{S^0}\big( V(t) + h\dot V(t) \big)
\overset{\substack{\eqref{eq:exp-sum} \\\eqref{eq:T0_parameterization}}}{=} 
\exp_{S(t)}\big(h\dot V(t)\big)
\\
&\overset{\substack{\eqref{eq:exp-constant} \\\eqref{eq:T0_parameterization2}}}{=}
\exp_{S(t)}\big( h\Omega S(t)\big).
\end{align}
\end{subequations}
Higher order geometric integration methods \cite{Zeilmann:2020aa} generalizing 
\eqref{eq:geometric_euler} can be applied in a similar way.
This provides new  
perspective on solving a certain class of nonlocal G-PDEs 
numerically, conforming 
to the
underlying geometry, as we demonstrate in Section \ref{sec:higher-order-geometric}.  

\subsubsection{Basic Parametrizations, Effect of 
Nonlocal Dirichlet Boundary Conditions} 
\label{sec:Num_Ex}
We consider two different parametrizations as well as zero and non-zero nonlocal Dirichlet boundary conditions.  
\begin{description}
\item[Uniform parametrization] Mappings $\Theta,\alpha\in\mc{F}_{\ol{\mc{V}}\times\ol{\mc{V}}}$ are given by
\begin{subequations}
\begin{align}
|\mc{N}(x)|&=\mc{N},\;\forall x,\qquad
|\mc{N}|=(2 k+1)\times (2 k+1),\; k\in\N
\label{eq:mcN-uniform} \\ \label{eq:uniform-params}
\alpha^{2}(x,y) &= \begin{cases}
\frac{1}{(2 k+1)^{2}} &\text{if}\;y\in\mc{N}(x) \\
0 &\text{otherwise}
\end{cases},
\qquad\qquad
\Theta(x,y) = \begin{cases}
\frac{1}{(2 k+1)^{2}} &\text{if}\;x=y \\
1 &\text{otherwise}
\end{cases}.
\end{align}
\end{subequations}
\item[Nonuniform parametrization] Uniform 
neighborhoods as in 
\eqref{eq:mcN-uniform} and mappings 
$\Theta,\alpha\in\mc{F}_{\ol{\mc{V}}\times\ol{\mc{V}}}$ by
\begin{equation}\label{eq:nonuniform-params}
	\begin{aligned}
\alpha^{2}(x,y) &= 
	\begin{cases} e^{-\frac{\|x-y\|^{2}}{2\sigma_{s}^{2}}} \quad 
	&\text{if } y \in \mc{N}(x)\\
	0 \quad &\text{otherwise}
\end{cases}, 
\qquad \qquad &\sigma_{s}>0, \\
\Theta(x,y) &= \begin{cases}
	e^{-G_{\sigma_{p}}\ast\|s(x)-s(y)\|^{2}} \quad &\text{if }y \in 
	\mathcal{N}(x)\\
	0 \quad &\text{otherwise}
\end{cases},
	\qquad \qquad  &\sigma_{p}>0,
\end{aligned}
\end{equation}
where the nonlocal function $\Theta$ is designed using a patchwise 
similarity measure analogous to the basic nonlocal means approach 
\cite{Buades:2010aa}: $s(x)=\{s(x,z)\colon z \in \ol{\mc{V}},s(x,z) = 
\ol{X}(z)\}$ with $\ol{X} \in \mc{F}_{\ol{\mc{V}},\R^c}$ denoting the 
zero extension of data $X \in \mc{F}_{\mc{V},\R^c}$ to 
$\mc{V}^{\alpha}_{\mc{I}}$. $G_{\sigma_{p}}$ is 
the Gaussian kernel at scale $\sigma_{p}$ and $\ast$ denotes spatial convolution. 
\end{description}
We iterated \eqref{eq:geometric_euler} with step size $h = 1$ until 
assignment states  \eqref{eq:W_fl_couples} of low average entropy $10^{-3}$ 
were reached. 
To ensure a fair comparison and to assess solely the effects of the boundary 
conditions through nonlocal regularization, we initialized 
\eqref{eq:Non_Local_PDE} in the same way as \eqref{eq:S-flow-S} and adopted an 
uniform encoding of the 31 labels as described by \cite[Figure 
6]{Astrom:2017ac}. 

Figure \ref{fig:Regularization_comparison} depicts labelings computed using the uniform parametrization with zero and non-zero nonlocal Dirichlet boundary conditions, respectively. Inspecting panels (c) (zero boundary condition) and (d) 
(non-zero boundary condition) shows that using the latter may improve labeling
near the boundary (cf.~close-up views), whereas the labelings almost agree in the interior of the domain. 

Figure \ref{fig:Entropy_Plot} 
shows how the average entropy values of label assignments decrease as the iteration proceeds (left panel) and 
the number of iterations required to converge (right panel), for different 
neighborhood sizes. Moreover, a closer look on the right 
panel 
of Figure \ref{fig:Entropy_Plot} reveals besides a slightly slower convergence 
of the scheme \eqref{eq:Update_Tangent0} applied to the nonlocal G-PDE
\eqref{eq:Non_Local_PDE_T0} (red curve), the dependence of number of iterations required until convergence is 
comparable to the $S$-flow (green curve). Consequently, generalizing the $S$-flow 
by the nonlocal model \eqref{eq:Non_Local_PDE} does not have a detrimental effect 
on the overall numerical behavior.  
  We observe, in particular, that integral label assignments 
corresponding to zero entropy are achieved no matter 
which boundary condition is used, at 
comparable computational costs.

\begin{figure}
	\includegraphics[width = \textwidth]{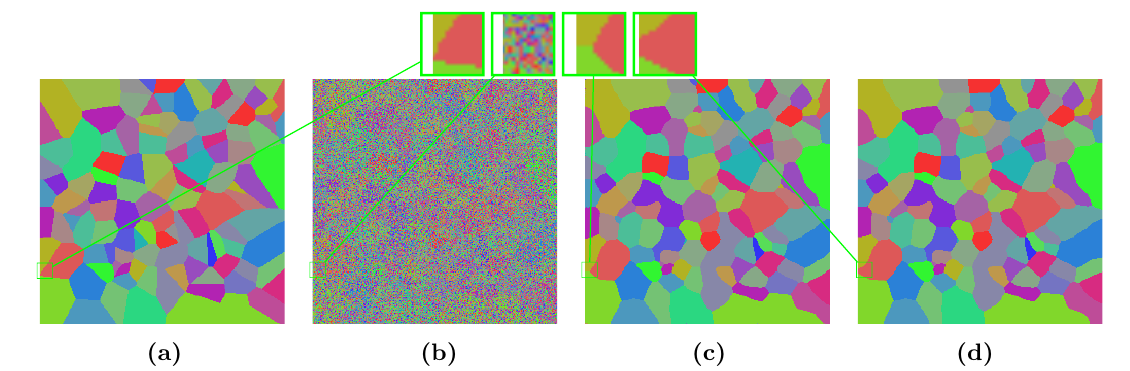}
	\vspace{-0.8cm}
	\captionsetup{font=footnotesize}
	\caption{Labeling through the nonlocal 
	geometric assignment flow with uniform 
	parametrization \eqref{eq:uniform-params} and 
	neighborhood size $|\mc{N}|=7$. \textbf{(a)} Ground 
		truth with 31 labels. 
		\textbf{(b)} Noisy input data  
		used to evaluate \eqref{eq:S-flow-S} and \eqref{eq:Non_Local_PDE}, respectively. 
		\textbf{(c)} Labeling returned when using the zero nonlocal Dirichlet boundary condition. \textbf{(d)} Labeling returned when using the non-zero nonlocal Dirichlet boundary condition (uniform extension to the interaction domain). 
The close-up views 
		show differences close to the boundary, whereas the results in the 
		interior domain are almost equal. }
	\label{fig:Regularization_comparison}
\end{figure}
\begin{figure}
	\centering
	\scalebox{0.8}{
		\begin{subfigure}[t]{0.6\textwidth}
			\scalebox{0.7}{\begin{tikzpicture}
\begin{axis}[
		every other node near coord/.append style={font=\tiny},
	legend style={row sep=0.05cm},
	height=8cm,
	width = 10cm,
	legend entries={{\small $|\mathcal{N}| = 7\times7$},{\small $|\mathcal{N}| 
	= 
	9\times9$},{\small $|\mathcal{N}| = 11\times11$},{\small $|\mathcal{N}| = 
	13\times13$}},
	legend pos=south west,
		mark repeat=2,
		xlabel = Number of iterations,
		ylabel = Average entropy,
		xmin=0,
		xmax=120,
		scaled ticks=false,
		/pgf/number format/.cd,
		use comma,
		1000 sep={}
		]
	\addplot[line join=round,mark = square*,color = neigh_7,line width=0.7pt]
	coordinates {(10,0.9997) (20,0.9994) (30,0.9988) (40,0.9972) (50,0.9912) 
		(60,0.9372) (70,0.5762) (80,0.1676) (90,0.0313) (100,0.0102) 
		(110,0.0063) (120,0.0047)};
	]			
	\addplot[line join=round,mark = trianlge*,color = neigh_9,line width=0.7pt]
	coordinates {(10,0.9998) (20,0.9995) (30,0.9990) (40,0.9978) (50,0.9939) 
		(60,0.9680) (70,0.7154) (80,0.2597) (90,0.0561) (100,0.0155) 
		(110,0.0088) (120,0.0064)};
	]
	\addplot[line join=round,mark = *,color = neigh_11,line width=0.7pt]
	coordinates {(10,0.9998) (20,0.9996) (30,0.9992) (40,0.9982) (50,0.9954) 
		(60,0.9810) (70,0.8233) (80,0.3901) (90,0.1006) (100,0.0276) 
		(110,0.0124) (120,0.0085)};
	]		
	\addplot[line join=round,mark = star,color = neigh_13,line width=0.7pt]
coordinates {(10,0.9998) (20,0.9997) (30,0.9993) (40,0.9986) (50,0.9964) 
	(60,0.9875) (70,0.8929) (80,0.5351) (90,0.1771) (100,0.0483) (110, 0.0189) 
	(120,0.0113)};

		\end{axis}
\end{tikzpicture}}
	\end{subfigure}}
	\scalebox{0.8}{
		\begin{subfigure}[t]{0.5\textwidth}
			\scalebox{0.7}{\begin{tikzpicture}
\begin{axis}[
every other node near coord/.append style={font=\tiny},
legend columns=4,
legend style={row sep=1.58cm},
height=8cm,
width = 10cm,
legend entries={{\small $\mathcal{V}^{\alpha}_{\mathcal{I}} \neq 
\emptyset$},{\small $S$-Flow}},
legend pos=south east,
mark repeat=2,
xlabel = Neighbourhood size,
ylabel = Iterations,
xmin=0,
xmax=19,
scaled ticks=false,
/pgf/number format/.cd,
use comma,
1000 sep={}
]
\addplot[line join=round,mark = square*,color = red]
coordinates {(3,93) (5,96) (7,101) (9,107) (11,116) 
	(13,124) (15,129) (17,136)};
]			
\addplot[line join=round,mark = square*,color = green]
coordinates {(3,92) (5,94) (7,99) (9,105) (11,112) 
	(13,120) (15,124) (17,130)};
]
\end{axis}
\end{tikzpicture}}
	\end{subfigure}}
	\captionsetup{font=footnotesize}
	\caption{\textbf{Left: }Convergence rates of the scheme 
		\eqref{eq:geometric_euler} solving \eqref{eq:Non_Local_PDE} with 
		nonzero nonlocal Dirichlet boundary condition. The 
		convergence behavior is rather insensitive with respect to the 
		neighborhood size $|\mc{N}|$. 
		\textbf{Right:} Number of iterations until convergence for 
		\eqref{eq:Non_Local_PDE} (\protect\markertwo) and
		\eqref{eq:S-flow-S} 
		(\protect\markerone), with zero nonlocal boundary condition in the 
		latter case. The result shows that different nonlocal 
		boundary conditions have only a minor influence on the required number of geometric integration steps.}
	\label{fig:Entropy_Plot}
\end{figure}
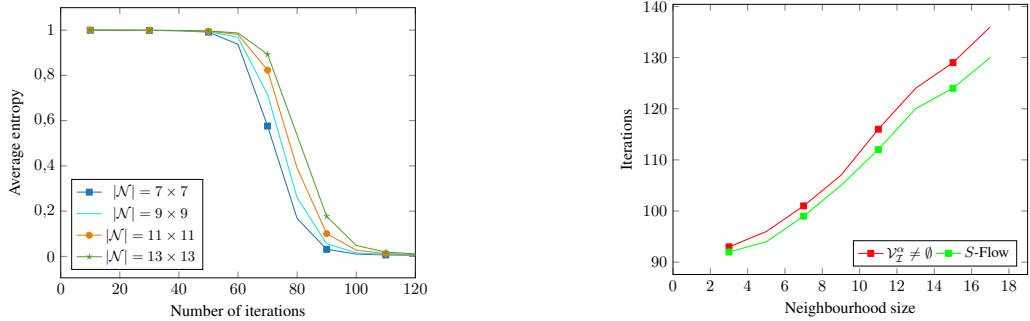

\begin{figure}
	\centering
	\includegraphics[width=\textwidth]{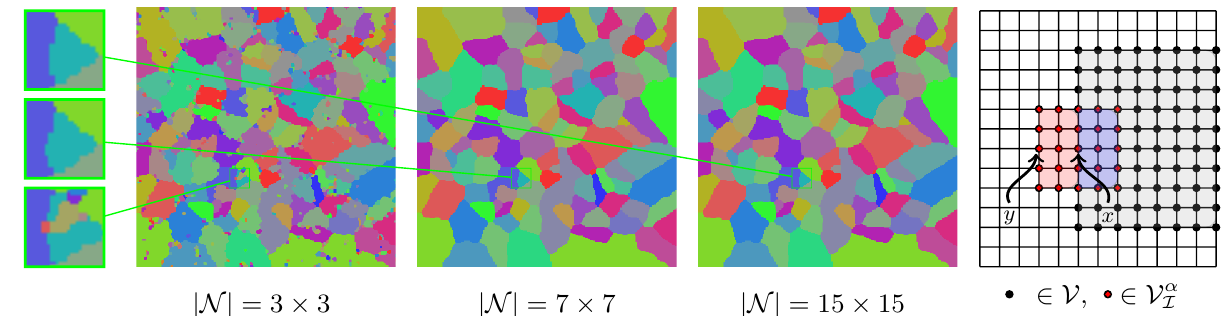}
	\captionsetup{font=footnotesize}
	\vspace{-0.5cm}
	\caption{\textbf{From left to 
			right:} Labeling results using \eqref{eq:Non_Local_PDE} 
			with the non-uniform parametrization 
			\eqref{eq:nonuniform-params}, zero non-local Dirichlet 
			boundary conditions and neighborhood sizes
		$|\mathcal{N}| \in\{3\times 3, 7\times 7, 15 \times 15\}$. Schematic 
		illustration of the 
		nonlocal interaction domain $y \in \mathcal{V}_{\mathcal{I}}^{\alpha}$ 
		(red 
		area) induced by nodes (blue area) in 
		$\mc{N}(x)$ with $|\mc{N}|=5\times 
		5$. Using nonuniform 
		weights \eqref{eq:nonuniform-params} improves labeling accuracy.}
	\label{fig:Segmentation_Non_Local_Mean}
	\vspace{-0.5cm}
\end{figure}

Iterating    
\eqref{eq:geometric_euler} with step size $h = 0.1$ and $\sigma_s = 1,\sigma_p 
= 5$ in \eqref{eq:nonuniform-params}  
 yields labeling results for different patch sizes as depicted by 
 Figure \ref{fig:Segmentation_Non_Local_Mean}. As 
 opposed to 
 segmentation results obtained with uniform parametrization 
 \eqref{eq:uniform-params} for $\mc{N} 
 = 7$ depicted in 
 Figure
 \ref{fig:Regularization_comparison}(d), a direct comparison with  
 Figure \ref{fig:Segmentation_Non_Local_Mean} (close up views) 
 indicates more accurate labelings when using
 regularization as given by the nonuniform 
 parametrization \eqref{eq:nonuniform-params}. 

\section{Related Work}\label{sec:Related-Work-PDE}
 In this section, we discuss how the system 
\eqref{eq:Non_Local_PDE} relates to 
approaches based on PDEs and variational models in the literature. Specifically, we conduct an \textit{analytical ablation} 
study of the 
nonlocal model \eqref{eq:Non_Local_PDE} in order to clarify the impact of omitting 
operators of the nonlocal model and the connection to existing methods. We exhibit 
both structural similarities from the viewpoint of diffusion processes and 
differences that account for the \textit{different scope} of our approach: \textit{labeling} 
metric data on graphs.

\subsection{General Nonlocal Processes on Graphs}\label{sec:general_processes}
We consider again the identity \eqref{eq:Reform_Av} that defines the 
nonlocal 
G-PDE \eqref{eq:Non_Local_PDE} in terms of \textit{symmetric} 
parameter mapping 
\eqref{eq:Avaraging_Matrix} and show next how \eqref{eq:Reform_Av} is 
generalized when a \textit{nonsymmetric} parameter matrix $\Omega \in \mathcal{F}_{\Z^d\times 
\Z^d}$ is used. Specifically, suppose a kernel $k \in 
\mathcal{F}_{\Z^d \times \Z^d}$ is given and the induced nonlocal functional
	\begin{equation}\label{eq:rem_nonlocal_diffusion_general}
		\mathcal{L}_{k}f(x) = \sum \limits_{y \in 
			\Z^d}\big(f(y)k(y,x)-f(x)k(x,y)\big).  
	\end{equation}
	Then, for a mapping
	$\alpha$ that 
	satisfies 
	$\alpha^2(x,y)  = 1$ whenever $k(x,y) \neq 0$, the decomposition
	\begin{equation}
		k = k^s+k^a \quad \text{with} \quad  k^s 
		= \frac{k + 
			k^{'}}{2} ,\quad  k^a = \frac{k - k^{'}}{2}, \qquad k^{'}(x,y) 
		\coloneqq 
		k(y,x), \quad x,y \in \Z^d,
	\end{equation}  
	results in the representation 
	\begin{equation}\label{eq:rem_kernel_repr}
	k(x,y) = \begin{cases}
			2 \Theta(x,y) \alpha^2(x,y) + \alpha(x,y) \nu (x,y) \quad &x \neq 
			y, \\
			2\Theta(x,x) \quad &x = y
		\end{cases} 
	\end{equation}	
	of the kernel $k$ in terms of 
	$\alpha,\Theta \in \mathcal{F}_{\Z^d \times \Z^d}$ and $\nu \in 
	\mathcal{F}_{\Z^d \times \Z^d}$ given by 
	\begin{equation}
		\Theta(x,y) \coloneqq \frac{1}{2} k^s(x,y), \qquad 
		\nu(x,y) \coloneqq k^a(x,y)\alpha(x,y),
	\end{equation}
	where the mapping $\nu$ is 
	a symmetric due to the antisymmetry of $\alpha$. Inserting 
	\eqref{eq:rem_kernel_repr} 
	into \eqref{eq:rem_nonlocal_diffusion_general} yields 
	\begin{equation}
		\mathcal{L}_k f(x) = 2 \sum \limits_{y \in \Z^d} 
		\Theta(x,y)\alpha^2(x,y)\big(f(y)-f(x)\big)-\sum \limits_{y \in \Z^d} 
		\alpha(x,y)\nu(x,y)\big(f(y)-f(x) \big).
	\end{equation}
	and applying nonlocal calculus of Section \ref{sec:Nonlocal-Calculus} 
	along with Lemma \eqref{Help_Lemma}, we arrive at an equivalent 
	representation 
	of $\mathcal{L}_k$ through nonlocal divergence and gradient operators 
	\begin{equation}\label{eq:rem_interaction_domain}
		\mathcal{L}_k f (x) \overset{\eqref{eq:rem_kernel_repr}}{=} 
		\underbrace{\mathcal{D}^{\alpha}\big( \Theta 
		\mathcal{G}^{\alpha}(f) \big)(x)}_\text{diffusion} - 
	\underbrace{\mathcal{D}^{\alpha}(\nu f)(x)}_\text{convection} + 
	\underbrace{\lambda(x) f(x)}_\text{fidelity},
	\end{equation} 
where $\nu$ plays the role of the convection parameter. Consequently, on a grid 
graph $\mc{G}$ with $\mathcal{V} \subset \Z^d$ and setting $\Omega$ by \eqref{eq:rem_kernel_repr}, 
we get  
	\begin{subequations}\label{eq:Nonlocal_Generall}
	\begin{align}	
		&\partial_t S(x,t) = R_{S(x,t)}\Big(\mathcal{D}^{\alpha}\big( \Theta 
		\mathcal{G}^{\alpha}(S) \big) - 
		\mathcal{D}^{\alpha}(\nu S)\Big)(x,t) + \lambda(x) S(x,t) &\qquad 
		&\text{on} \quad \mathcal{V}\times \R_{+}, \\
		&\ol{S}(x,t) = 0 &\qquad &\text{on} \quad
		\mathcal{V}_{\mathcal{I}}^{\alpha}\times \R_{+}, \\
		&\ol{S}(x,0) = S(x)(0) &\qquad &\text{on} \quad 
		\ol{\mc{V}}\times\R_{+},
	\end{align}
	\end{subequations}
	with the interaction domain \eqref{eq:Interaction_Dom} directly 
expressed through the connectivity of kernel $k$ by
\begin{equation}
	\mathcal{V}_{\mathcal{I}}^{\alpha} = \{ x \in \Z^d\setminus \mathcal{V} 
	\colon k(x,y) \neq 0 \text{ for some }  y \in \mathcal{V} \}.
\end{equation}  
In view of \eqref{eq:Nonlocal_Generall}, we therefore recognize 
the system \eqref{eq:Non_Local_PDE} as \textit{specific} nonlocal process that is 
induced by a \textit{nonnegative} 
\textit{symmetric} kernels $k$ with nonzero fidelity parameter $\lambda$, that 
account 
for nontrivial steady state solutions and zero 
convection $(\nu(x,y) = 0)$.

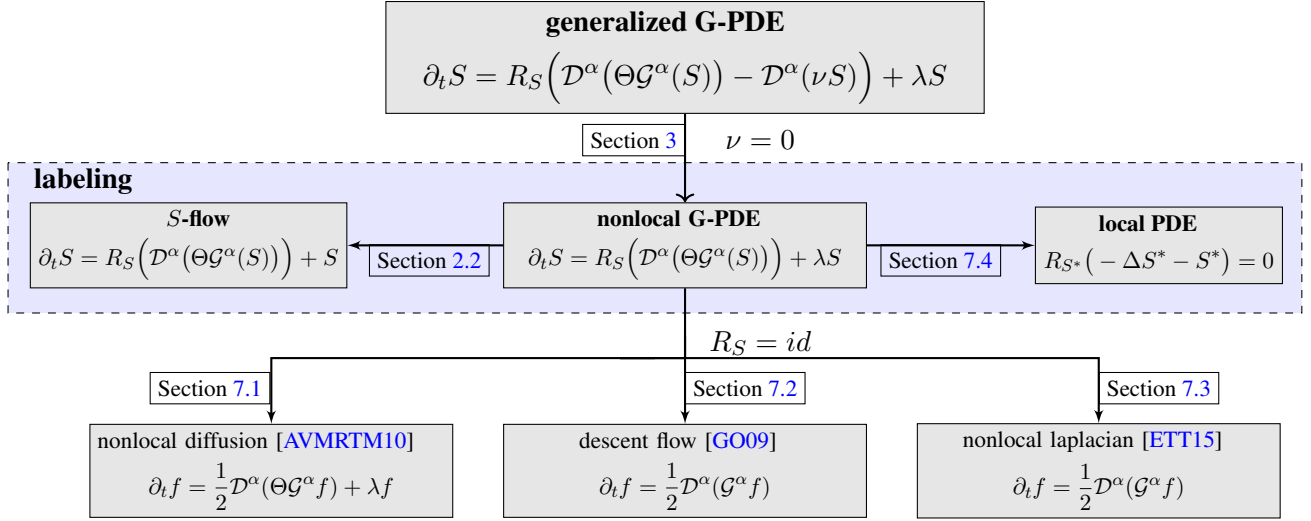
\begin{figure}
\begin{adjustwidth}{-0.1cm}{}
\begin{tikzpicture}[
	box/.style={rectangle,draw,fill=blue!10,node distance=1cm,text 
		width=15em,text centered,rounded corners,minimum height=2em,thick},
	arrow/.style={draw,-latex',thick},
	]
	\draw[draw,fill = blue!10,dashed] (-6,-1.4) rectangle ++(17.2,-2);
	\node[fill=gray!20,draw,text width=20 em] (G-G-PDE) at (3,0) 
	{\hspace{2cm}\textbf{generalized 
	G-PDE} \begin{equation*}
		\partial_t S = R_{S}\Big(\mathcal{D}^{\alpha}\big( \Theta 
		\mathcal{G}^{\alpha}(S) \big) - 
		\mathcal{D}^{\alpha}(\nu S)\Big) + \lambda S
\end{equation*}};
\node (nu) at (4,-1.1) {$\nu = 0$};
\node[fill=gray!20, draw,text width=15em,scale=0.8] (G-PDE) at 
(3,-2.5){\hspace{1.3cm} 
\textbf{nonlocal} 
\textbf{G-PDE} 
\begin{equation*}
		\partial_t S = R_{S}\Big(\mathcal{D}^{\alpha}\big( \Theta 
		\mathcal{G}^{\alpha}(S) \big)\Big) + \lambda S
	\end{equation*}};
\node[draw,left= 0.4cm of nu,scale=0.8] {Section \ref{sec:Non_Local_PDE}};
\path[->,draw,thick] (G-G-PDE) -- (G-PDE);
\node[fill = gray!20,draw,text width=15 em,scale=0.8] (NLD) at (-2.5,-5.5) 
{nonlocal 
diffusion \cite{Andreu-Vaillo:2010aa} \begin{equation*}
		\partial_{t} f = \frac{1}{2}\mathcal{D}^{\alpha}(\Theta 
		\mathcal{G}^{\alpha}f)+\lambda f
\end{equation*}}; 
\node (Left) at (-2.5,-4) {};  
\node (Middle) at (3,-4) {};
\node (Right) at (2,-4) {};
\node[fill = gray!20,draw,text width=13em,scale=0.8] (AF) at (-3.6,-2.5) 
{\hspace{2cm} \textbf{$S$-flow} 
\begin{equation*}
	\partial_t S = R_{S}\Big(\mathcal{D}^{\alpha}\big( \Theta 
	\mathcal{G}^{\alpha}(S) \big)\Big) + S
\end{equation*}};
\node[draw,scale=0.8] at (-0.4,-2.7) {Section \ref{sec:Assignment-Flow}};
\node[fill = gray!20,draw,text width=10em,scale=0.8] (LAF) at (9.3,-2.5) 
{ \hspace{0.8cm} \textbf{local PDE} 
\begin{equation*}
		R_{S^{\ast}}\big(-\Delta S^{\ast} - S^{\ast} \big) = 0 
\end{equation*}};
\node[draw,scale=0.8] at (6.4,-2.7) {Section 
	\ref{sec:Exp-subsec-4}};
\node[draw,scale=0.8] at (-3.315,-4.4) {Section 
	\ref{sec:Exp-subsec-1}};
\node[draw,scale=0.8] at (3.8,-4.4) {Section 
	\ref{sec:Exp-subsec-2}};
\node[draw,scale=0.8] at (9.3,-4.4) {Section 
	\ref{sec:Exp-subsec-3}};
\node[fill = gray!20,draw,text width=15em,scale=0.8] (NP) at (8.5,-5.5) 
{\hspace{0.5cm} nonlocal 
laplacian 
	\cite{Elmoataz:2015aa} \begin{equation*}
		\partial_{t} f = \frac{1}{2}\mathcal{D}^{\alpha}(\mathcal{G}^{\alpha}f)
\end{equation*}};
\path[draw,thick,arrow] (G-PDE) |- (Left)-| (NLD);
\node[fill = gray!20,draw,text width=15em , scale=0.8] (DF) at (3,-5.5) 
{\hspace{1cm} descent 
flow \cite{Gilboa:2009aa} 
\begin{equation*}
	\partial_{t} f = \frac{1}{2}\mathcal{D}^{\alpha}(\mathcal{G}^{\alpha}f)
\end{equation*}};
\path[draw,thick,arrow] (3,-3.7) -- (DF);
\path[draw,thick,arrow] (G-PDE) -- (AF);
\path[draw,thick,arrow] (G-PDE) -- (LAF);
\path[draw,thick,arrow] (2.6,-4) -| (Right) -| (NP);
\node at (4,-3.8) {$R_S = id$};
\node at (-5,-1.65) {\textbf{labeling}};
\end{tikzpicture}
\end{adjustwidth}
\captionsetup{font=footnotesize}
\caption{\textbf{Overview of nonlocal diffusion processes} proposed in 
related work  
\cite{Elmoataz:2015aa,Gilboa:2009aa,Andreu-Vaillo:2010aa} and their 
interrelations to the nonlocal G-PDE \eqref{eq:Nonlocal_Generall}. The approaches highlighted by the blue region only model the image labeling problem. Edge labels 
refer to the corresponding sections of the analytical ablation study.}
\label{fig:ablation_study}
\end{figure}

\begin{center}
	\begin{table}
		\begin{tabularx}{\linewidth}{llc@{\hskip 1em}c@{\hskip 1em}c@{\hskip 
		1em}c}
			& & & \hskip -3cm \textbf{Labeling} & & \hskip -5cm 
			\textbf{Denoising and 
				Inpainting}   
			\\
			\addlinespace
			&\textbf{Parameters} &\textbf{G-PDE \eqref{eq:Non_Local_PDE}} & 
			\textbf{Local 
				PDE} 
			\cite{Savarino:2019ab}   & \textbf{Nonl.~Laplacian} 
			\cite{Elmoataz:2015aa} &\textbf{Descent Flow}  
			\cite{Gilboa:2009aa} \\
			\hline
			\addlinespace
			& $\Theta \geq 0$ &\cmark  &\xmark         & \xmark        & \xmark 
			\\ 
			& $\lambda$&$\lambda > 0$  &$\lambda$ = 1  & $\lambda = 0$ 
			&$\lambda = 
			0$  \\
			& $R_S$    &\cmark  &\cmark         &\xmark         &\xmark  \\
			& $\nu$    &\xmark  &\xmark         &\xmark         &\xmark\\
			&$\mathcal{V}_{\mathcal{I}}^{\alpha}$ 
			&$\subseteq \Z^d \setminus 
			\mathcal{V}$ &$\partial \mathcal{V}^h$ &$\partial \mathcal{A} 
			\subset \mathcal{V}$ 
			&$\emptyset$ \\
			&$S^{\ast} (t \to \infty)$  &\cmark  &\cmark  &\xmark &\xmark \\
		\end{tabularx}
		\captionsetup{font=footnotesize}
		\caption{ \textbf{Summary of the analytical 
				ablation study.} Key differences of our approach to existing nonlocal diffusion 
				models are inclusion of the replicator 
				operator 
				$R_S$ and a nonzero fidelity term $\lambda S$ 
				that 
				results in nontrivial solution at the steady state 
				$S^{\ast}=S(t = 
				\infty)$. }
			\label{tab:ablation_study}
	\end{table}
\end{center}

 In the following sections, we relate different established nonlocal models to 
 the proposed G-PDE \eqref{eq:Non_Local_PDE} by adapting 
 the parameter mappings $\Theta,\alpha \in \mathcal{F}_{\ol{\mathcal{V}} \times 
 \ol{\mathcal{V}} }$ that parametrize the G-PDE and determine the interaction domain \eqref{eq:Interaction_Dom}. Figure \ref{fig:ablation_study} provides an overview of the analytical ablation study by specifying the model and the corresponding section where it is derived from the generalized G-PDE \eqref{eq:Nonlocal_Generall}. Table 
 \ref{tab:ablation_study} lists the involved parameters for each model.

\subsection{Relation to a Local PDE that Characterizes Labelings}
We focus on the connection of the system \eqref{eq:Non_Local_PDE} and the 
\textit{continuous-domain local} formulation of \eqref{eq:S-flow-S} on an open  
simply connected 
bounded domain $\mathcal{D} \subset \R^2$, as introduced by
\cite{Savarino:2019ab}. The \textit{variational} formulation has been 
rigorously derived in \cite{Savarino:2019ab} along with a PDE that formally 
characterizes solutions  $S^{\ast} = 
\lim_{t\to\infty}S(t) \in 
\ol{\mathcal{W}}$ only under strong regularity assumptions. This nonlinear PDE 
reads 
\begin{align}\label{eq:Continuous_Formulation_S_Flow}
	R_{S^{\ast}(x)}\big(-\Delta S^{\ast}(x) - S^{\ast}(x) \big) = 0, \quad x 
	\in 
	\mathcal{D}.
\end{align} 

We next show that our novel 
approach \eqref{eq:Non_Local_PDE} includes, as a special case, a natural discretization of 
\eqref{eq:Continuous_Formulation_S_Flow} on the spatial discrete grid 
$\mathcal{V}^h = h \Z^d \cap \mathcal{D}$ with boundary 
$\partial \mathcal{V}^h$ specified by a small spatial scale parameter $h > 0$. 
\eqref{eq:Continuous_Formulation_S_Flow} is complemented by \textit{local zero 
Dirichlet }
boundary conditions imposed on $S^{\ast}$ 
on $\partial \mathcal{V}^h$. Adopting the sign convention 
	$L_{\vartheta}^h = -\Delta_{\vartheta}^h$ for different discretizations 
	of the 
	\textit{continuous negative Laplacian} on 
	$\mathcal{V}^h$, by a nine-point stencil \cite{Welk:2021aa} 
	parametrized by $\vartheta \in [0,1]$, lead to strictly 
	positive entries 
	$L_{\vartheta}^h(x,x) > 0$ 
	on 
	the diagonal. 

 We introduce the weighted undirected graph $(\mathcal{V}^h,\Omega^h)$ and identify nodes $x = (k,l) \in 
 \mathcal{V}^h$ with interior grid points $(h k,h 
 l) \in \mathcal{V}^h$ (grid graph).  
	Let the parameter matrix $\Omega^h$ be given by 
	\eqref{eq:Avaraging_Matrix}
	and the mappings 
	$\alpha,\Theta \in 
	\mathcal{F}_{\overline{\mathcal{V}}\times \overline{\mathcal{V}}}$ defined 
	by 
	\begin{align}\label{eq:Local_Laplace_Dec}
		\alpha^2(x,y) = 
		\begin{cases}
			1 ,&\; y \in 
			\wt{\mathcal{N}}(x),\\
			0 ,&\; \text{else},
		\end{cases}, \qquad
		\Theta(x,y) = 
		\begin{cases}
			-L_{\vartheta}^h(x,y), & y \in \wt{\mathcal{N}}(x),\\
			1 - L_{\vartheta}^h(x,x), 
			& 
			x = y,\\
			0 & \text{ else },
		\end{cases}
	\end{align}
	where the neighborhoods $\wt{\mathcal{N}}(x) = \mathcal{N}(x)\setminus 
	\{x\}$ represent the connectivity of the stencil of the discrete Laplacian 
	$L^h_{\vartheta}$ on the mesh 
	$\mathcal{V}^h \dot{\cup} \partial \mathcal{V}^h$. Recalling the 
	definitions 
	from Section \ref{sec:Nonlocal-Calculus} with respect to undirected graphs 
	and setting $\alpha$ by \eqref{eq:Local_Laplace_Dec}, the 
	interaction domain \eqref{eq:Interaction_Dom} agrees for 
	parameter choices $\vartheta \neq 0$ with the discrete local 
	boundary, i.e. $\mathcal{V}_{\mathcal{I}}^{\alpha} = \partial 
	\mathcal{V}^h$; see Figure 
	\ref{fig:Local_Boundary_Recovery} and the caption for further explanation. 
Then, for each $x \in \mathcal{V}^h$, the action of $\Omega^{h}$ on $S$ reads
\begin{align}
	\hskip -0.25cm (\Omega^h S)(x)= \sum_{y \in 
		\wt{\mathcal{N}}(x)}\hskip 
	-0.2cm-L^h_{\vartheta}(x,y)S(y)+\big(1-L_{\vartheta}^h(x,x)\big)S(x)  = 
	-\big(-\Delta^h_{\vartheta}(S)-S\big)(x),
\end{align} 
which is the discretization of \eqref{eq:Continuous_Formulation_S_Flow} by 
$L_{\vartheta}^h$ multiplied by the minus sign.
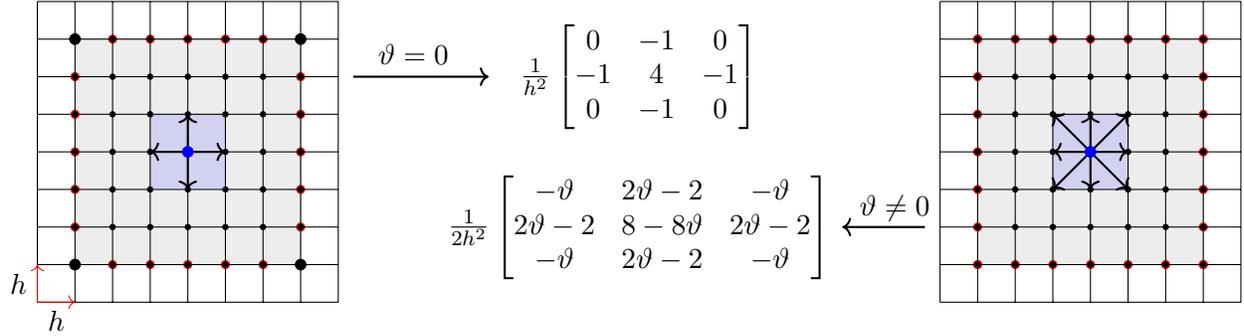
\begin{figure}
	\begin{tikzpicture}	
		\draw[rounded corners=3pt,fill=gray!70,opacity = 0.2]
		(0.5,0.5)--(0.5,3.5)--(3.5,3.5)--(3.5,0.5) ;
		\filldraw[radius=1.5pt,red]
		\foreach \x in {0.5, 3.5} {
			\foreach \y in {1,1.5,...,3} {
				(\x, \y) circle[]		
			}
		};
		\filldraw[radius=1.5pt,red]
		\foreach \x in {1,1.5,...,3} {
			\foreach \y in {0.5,3.5} {
				(\x, \y) circle[]		
			}
		};
		\filldraw[radius=2pt,black] (0.5,0.5) circle[];
		\filldraw[radius=2pt,black] (0.5,3.5)circle[];
		\filldraw[radius=2pt,black] (3.5,3.5)circle[];
		\filldraw[radius=2pt,black] (3.5,0.5)circle[];
		\draw (0,0) -- (0,4);
		\draw (0, 0) grid[step=0.5cm] (4,4);
		\filldraw[radius=1pt,black]
		\foreach \x in {0.5,1, ...,3.5} {
			\foreach \y in {0.5,1, ..., 3.5} {
				(\x, \y) circle[]		
			}	
		};
		\draw[rounded corners=3pt,fill=gray!70,opacity = 0.2]
		(12.5,0.5)--(12.5,3.5)--(15.5,3.5)--(15.5,0.5) ;
		\filldraw[radius=1.5pt,red]
		\foreach \x in {12.5, 15.5} {
			\foreach \y in {0.5,1,...,3.5} {
				(\x, \y) circle[]		
			}
		};
		\filldraw[radius=1.5pt,red]
		\foreach \x in {12.5,13,...,15.5} {
			\foreach \y in {0.5,3.5} {
				(\x, \y) circle[]		
			}
		};
		\draw (12,0) -- (12,4);
		\draw (12, 0) grid[step=0.5cm] (16,4);
		\filldraw[radius=1pt,black]
		\foreach \x in {12.5,13, ...,15.5} {
			\foreach \y in {0.5,1, ..., 3.5} {
				(\x, \y) circle[]		
			}	
		};
		\draw[rounded corners=3pt,fill=blue!60,opacity = 0.2]
		(13.5,1.5)--(13.5,2.5)--(14.5,2.5)--(14.5,1.5) ;
		\draw[->,thick] (14,2) -- (14.5,2); 
		\draw[->,thick] (14,2) -- (14,2.5); 
		\draw[->,thick] (14,2) -- (13.5,2); 
		\draw[->,thick] (14,2) -- (14,1.5); 
		\draw[->,thick] (14,2) -- (14.5,2.5); 
		\draw[->,thick] (14,2) -- (13.5,1.5); 
		\draw[->,thick] (14,2) -- (13.5,2.5); 
		\draw[->,thick] (14,2) -- (14.5,1.5); 
		\draw[rounded corners=3pt,fill=blue!60,opacity = 0.2]
		(1.5,1.5)--(1.5,2.5)--(2.5,2.5)--(2.5,1.5) ;
		\draw[->,thick] (2,2) -- (2.5,2); 
		\draw[->,thick] (2,2) -- (2,2.5); 
		\draw[->,thick] (2,2) -- (1.5,2); 
		\draw[->,thick] (2,2) -- (2,1.5);
		\draw[->,color = red] (0,0) -- (0.5,0);
		\draw[->,color = red] (0,0) -- (0,0.5);
		\node at (0.25,-0.25) {$h$}; 
		\node at (-0.25,0.25) {$h$};
		\filldraw[radius=2pt,blue] (2,2) circle[];
		\filldraw[radius=2pt,blue] (14,2) circle[];
		\node at (8,3) {$\frac{1}{h^2}\begin{bmatrix}
				0 & -1 & 0\\
				-1 &4 &-1 \\ 
				0 &-1 &0
			\end{bmatrix}$};
		\node at (8,1) {$\frac{1}{2h^2}\begin{bmatrix}
				-\vartheta & 2\vartheta-2 & -\vartheta\\
				2\vartheta-2 &8-8\vartheta & 2\vartheta-2 \\ 
				-\vartheta & 2\vartheta-2 &-\vartheta
			\end{bmatrix}$};
		\draw[->,thick] (4.2,3) -- (6,3);
		\draw[->,thick] (11.8,1) -- (10.7,1);
		\node at (5,3.3) {$\vartheta = 0$};
		\node at (11.4,1.25) {$\vartheta \neq 0$};
	\end{tikzpicture}
	\captionsetup{font=footnotesize}
	\caption{Illustration of the rectangular grid $\mathcal{V}^h$ 
			and the interaction 
			domain 
			$\mathcal{V}_{\mathcal{I}}^{\alpha}$ represented by 
			$(\protect\markereight)$ and $(\protect\markerseven)$, respectively, 
			with $\alpha \in \mathcal{F}_{\ol{\mathcal{V}} \times 
				\ol{\mathcal{V}}}$ given by 
			\eqref{eq:Local_Laplace_Dec} for a family of discrete Laplacians 
			$-\Delta^h_{\vartheta}$ 
			proposed in 
			\cite{Welk:2021aa}. \textbf{Left}: Neighborhood 
			$\widetilde{\mathcal{N}}(x)$ 
			specified in terms of the connectivity of the standard
			$5-$point stencil $(\vartheta = 0)$. The corresponding interaction 
			domain 
			is part 
			of the local boundary 
			$\mathcal{V}_{\mathcal{I}}^{\alpha} \subset \partial 
			\mathcal{V}^h$.   
			\textbf{Right}: Analogous construction with the $9$-point stencil 
			$(\vartheta 
			\neq 
			0)$. The interaction domain coincides with 
			the discrete 
			local boundary configuration, i.e.
			$\mathcal{V}_{\mathcal{I}}^{\alpha} = 
			\partial \mathcal{V}^h$. }
	\label{fig:Local_Boundary_Recovery}
\end{figure}
In particular, due to the relation $R_S(-W) 
= - 
R_S(W)$ for $W\in \mathcal{W}$, we conclude that the novel 
approach \eqref{eq:Non_Local_PDE} includes the \textit{local} PDE 
\eqref{eq:Continuous_Formulation_S_Flow} as special case and hence provides 
a \textit{natural nonlocal extension}.

\subsection{Continuous-Domain Nonlocal Diffusion 
Processes}\label{sec:subsec_4_1} We follow \cite{Andreu-Vaillo:2010aa}. 
Consider a bounded domain $\mathcal{D} \subset \R^d$ and let 
$J:\R^d \to \R_{+}$ be a radial continuous function satisfying 
\begin{equation}\label{eq:Cont_Prob_Kernel}
	\int_{\R^d}J(x-y)\text{d}y = 1, \quad J(0) > 0 \qquad \forall x \in \R^d.
\end{equation}
The term $J(x-y)$ in \eqref{eq:Cont_Prob_Kernel} may be interpreted as a 
probability density governing jumps from position $y \in \R^d$ to $x \in \R^d$.  
The authors of \cite{Andreu-Vaillo:2010aa} introduced the integral operator 
\begin{align}
		\mathcal{L} f(x) =
		\int_{\R^d}J(x-y)f(y,t)\text{d}y-f(x,t), \quad x \in \R^d
\end{align}
acting on $f\in C(\R^d,\R_{+})$ 
and studied \textit{nonlocal linear} diffusion processes of the form 
\begin{subequations}\label{eq:Cont_Form_Non_local}
	\begin{align}
	\partial_{t} f(x,t) = \mathcal{L} f(x,t) \qquad &\text{on } \quad  
	\mathcal{D} 
	\times 
	\R_+\\ \label{eq:Cont_Form_Non_local-b}
	f(x,t) = g(x) \qquad &\text{on } \quad \R^d\setminus \mathcal{D} \times 
	\R_+,\\
	f(x,0) = \overline{f}_0 \qquad &\text{on} \quad \R^d \times \R_+,
	\end{align}
\end{subequations}   
where $f_0 \in C(\mathcal{D},\R_{+})$ and $g \in C(\R^d\setminus \mathcal{D},\R_{+})$ 
specify the initial state and the nonlocal boundary condition of the system 
\eqref{eq:Cont_Form_Non_local}, respectively. We compare this 
system with our 
model \eqref{eq:Non_Local_PDE} and introduce, as in Section \ref{sec:subsec_4_1}, the 
weighted 
undirected graph $(\mathcal{V}^h,\Omega^h)$ with 
a Cartesian mesh $\mathcal{V}^h$, with boundary $\partial \mathcal{V}^h$  and 
neighborhoods 
\eqref{eq:def-neighborhoods}, and with $\Omega^h$ defined by 
\eqref{eq:Omega-by-alpha} through
\begin{align}
	\Theta(x,y) = 
	\begin{cases}
	0, &\quad \text{for } x,y \notin \mathcal{V}^h, \\
	J(0)-1, &\quad \text{for } x = y, \\
	1, &\quad \text{else},
	\end{cases}
	\qquad \alpha^2(x,y) = J(x-y).
\end{align} 
Then, for the particular 
case $g = 0$ in \eqref{eq:Cont_Form_Non_local-b} and using Equation 
\eqref{eq:Reform_Av} with $\lambda(x)$ defined 
by \eqref{eq:def-lambda-x}, the spatially 
discrete counterpart of 
\eqref{eq:Cont_Form_Non_local} is the \textit{linear nonlocal scalar-valued} diffusion process
 \begin{subequations}\label{eq:Eq_Rel}
	\begin{align}
	\partial_{t} f(x,t) &= \frac{1}{2}\mathcal{D}^{\alpha}(\Theta 
	\mathcal{G}^{\alpha}f)(x,t)+\lambda(x)f(x,t) \qquad &\text{on } &\quad  
	\mathcal{V} 
	\times \R_+, \\
	f(x,t) &= 0 \qquad &\text{on } &\quad \mathcal{V}_{\mathcal{I}}^{\alpha} 
	\times 
	\R_+, \\
	f(x,0) &= \overline{f}_0 \qquad &\text{on } &\quad \ol{\mathcal{V}} \times 
	\R_+.
	\end{align}
\end{subequations}
System \eqref{eq:Eq_Rel} possess a structure which resembles the structure of 
nonlinear system 
\eqref{eq:Non_Local_PDE} after dropping the replicator mapping $R_S$ and 
assuming $S(x)\in \R$ to be a scalar-valued rather than simplex-valued 
$S(x)\in\mc{S}$, as in our approach. 

This comparison shows by virtue of the 
structural similarity that assignment flows may be characterized as genuine 
nonlocal diffusion processes. Essential differences, i.e.~simplex-valued 
variables and the underlying geometry, reflect the entirely different scope of 
this process, however: labeling metric data on graphs. 
 
\subsection{Nonlocal Variational Models in Image Analysis}
We relate the 
system \eqref{eq:Eq_Rel} to variational approaches  presented 
in \cite{Gilboa:2009aa} and to graph-based nonlocal PDEs proposed by \cite{Elmoataz:2008aa,Elmoataz:2015aa}.

Based on a scalar-valued positive function 
$\phi(t)$ which is convex in $\sqrt{t}$ with $\phi(0) = 0$, Gilboa et al. 
\cite{Gilboa:2009aa} 
studied \textit{isotropic} and \textit{anisotropic nonlocal regularization 
functionals} on a continuous spatial domain $\mathcal{D} \subset \R^d$ defined 
in terms of a nonnegative symmetric mapping $\omega:\mathcal{D} \times 
\mathcal{D} \to \R_+$:
\begin{subequations}\label{eq:Gilboa_Functionals} 
\begin{align}
	J_i^{\phi}(f) &= 
	\int_{\mathcal{D}}\phi(|\nabla_{\omega}(f)(x)|^2)\text{d}x, 
	\qquad 
	&\text{\textbf{(isotropic)}}\label{eq:Func_Gilboa_isotrop}\\
	J_a^{\phi}(f) 
	&=\int_{\mathcal{D}}\int_{\mathcal{D}}\phi(f(y)-f(x))^2\omega(x,y)\text{d}y\text{d}x.
		 \qquad 
	&\text{\textbf{(anisotropic)}}
\end{align}
\end{subequations} 
\eqref{eq:Func_Gilboa_isotrop} involves the nonlocal graph-based gradient 
operator which for given neighborhoods $\mathcal{N}(x)$ reads  
\begin{align}\label{eq:Graph_Based_Gradient_Rel_Work}
	\nabla_{\omega} f(x) = \big(\dots,(f(y)-f(x))\sqrt{\omega(x,y)},\dots 
	\big)^T,\qquad y\in 
	\mathcal{N}(x).
\end{align}
Given an initial real valued function $f_0(x)$ on $\Omega$, the variational 
models of \eqref{eq:Gilboa_Functionals} define dynamics in terms of
the steepest descent flows 
\begin{align}\label{eq:Dynamical_System_Gilboa}
\partial_t f(x,t) = -\partial_f J^{\phi}_i(f)(x,t), \qquad \partial_t f(x,t) = 
-\partial_f J^{\phi}_a(f)(x,t), \qquad f(x,0) = f_0(x),
\end{align}
where the variation with respect to $f$ on right hand side of 
\eqref{eq:Dynamical_System_Gilboa} is expressed in 
terms of \eqref{eq:Graph_Based_Gradient_Rel_Work} via
\begin{align}\label{eq:divergence_expression_rel_work}
	\partial_f J^{\phi}_i(f)(x,t) &= -2\int_{\mathcal{D}}(f(y,t)-f(x,t)) 
	\omega(x,y)\Big(\phi^{'}(|\nabla_{\omega}f(y,t)|^2)(y)+\phi^{'}(|\nabla_{\omega}f(x,t)|^2)(x)\Big)\text{d}y, \\
	\partial_f J^{\phi}_a(f)(x,t) &= -4\int_{\mathcal{D}}\big( f(y,t)-f(x,t) 
	\big)\omega(x,y)\phi^{'}\big((f(y,t)-f(x,t))^2\omega(x,y)\big)\text{d}y.
\end{align} 
Then, given a 
graph $(\mathcal{V},\mathcal{E},\omega)$ with neighborhoods as in Section 
\ref{sec:Nonlocal-Calculus}, the discrete counterparts of the dynamical systems  
\eqref{eq:Dynamical_System_Gilboa} on $\mathcal{V}$ read
\begin{align}\label{eq:Discrete_Time_Dynamical_System:Gilboa}
	\dot{f}(x,t) = \sum_{y\in 
	\mathcal{N}(x)}A^{\phi}_{\omega,f}(x,y)f(y), 
	\qquad 
	\dot{f}(x,t) = \sum_{y\in 
	\mathcal{N}(x)}B^{\phi}_{\omega,f}(x,y)f(y),
\end{align}
where the mappings $A^{\phi}_{\omega,f},B^{\phi}_{\omega,f} \in 
\mathcal{F}_{\mathcal{V}\times 
\mathcal{V}}$ represent explicit expressions of the right-hand sides of 
\eqref{eq:Dynamical_System_Gilboa} on $\mathcal{V}$ 
\begin{subequations}\label{eq:Omega_Mappings}
\begin{align}
	A^{\phi}_{\omega,f}(x,y) &=
	\begin{cases}
	2 
	\omega(x,y)\Big(\phi^{'}(|\nabla_{\omega}f(y,t)|^2)(y)+\phi^{'}(|\nabla_{\omega}f(x,t)|^2)(x)\Big)\quad
	&x\neq y,\\
	-2\sum \limits_{\substack{z \in \mathcal{N}(x)\\z \neq 
			x}}\omega(x,z)\Big(\phi^{'}(|\nabla_{\omega}f(z,t)|^2)(z)+\phi^{'}(|\nabla_{\omega}f(x,t)|^2)(x)\Big)
				\quad &x = y,
	\end{cases} \\
	B^{\phi}_{\omega,f}(x,y) &=
	\begin{cases}
	4 \omega(x,y)\phi^{'}\big((f(z,t)-f(x,t))^2\omega(x,y)\big) \quad &x\neq 
	y,\\
	-4 \sum \limits_{\substack{z \in \mathcal{N}(x)\\z \neq 
	x}}\omega(x,z)\phi^{'}\big((f(z,t)-f(x,t))^2\omega(x,y)\big), \quad &x = y.
	\end{cases}
\end{align}
\end{subequations}
Depending on the specification of $\phi(t)$, the dynamics 
governed by the systems 
\eqref{eq:Discrete_Time_Dynamical_System:Gilboa} define nonlinear 
nonlocal diffusion processes with various smoothing properties according to the 
mappings 
\eqref{eq:Omega_Mappings}. Specifically, for $\phi(t) = t$, the functionals 
\eqref{eq:Gilboa_Functionals} coincide as do the systems \eqref{eq:Discrete_Time_Dynamical_System:Gilboa}, since the mappings 
\eqref{eq:Omega_Mappings} do not depend on $f(x,t)$,  but only on 
$\omega$ which is symmetric and nonnegative, and hence agree. Invoking Lemma \ref{Help_Lemma} 
with $\Omega \in \mathcal{F}_{\mathcal{V}\times \mathcal{V}}$ defined through 
\eqref{eq:Omega_Mappings}, setting $\Theta,\alpha \in 
\mathcal{F}_{\mathcal{V}\times \mathcal{V}}$ by $\Theta(x,y) = 
1,\alpha^2(x,y) = 4\omega(x,y), x \neq y$ and $\Theta(x,x) = 
-4\sum_{y\in\mathcal{N}(x)}\omega(x,y),x \in \mathcal{V}$ yields the 
 decomposition \eqref{eq:Avaraging_Matrix} which characterizes 
 \eqref{eq:Graph_Based_Gradient_Rel_Work} in terms of the nonlocal operators 
 from Section \ref{sec:Nonlocal-Calculus} if $f_{|\mathcal{V}_{\mathcal{I}}^{\alpha}} = 0$ 
 holds, by means of relation 
 \eqref{eq:Reform_Av}. 
 Consequently,
\eqref{eq:Discrete_Time_Dynamical_System:Gilboa} admits the representation 
by \eqref{eq:Eq_Rel} for the particular case of zero nonlocal Dirichlet 
conditions.
  
While the above approaches are well suited for 
image denoising and inpainting, our \textit{geometric} approach performs 
 \textit{labeling} of arbitrary metric data on arbitrary graphs.

\subsection{Nonlocal Graph Laplacians} Elmoataz et. al 
\cite{Elmoataz:2015aa} studied discrete nonlocal differential operators 
on weighted graphs 
$(\mathcal{V},\mathcal{E},\omega)$. Specifically, based on the nonlocal gradient 
operator \eqref{eq:Graph_Based_Gradient_Rel_Work},  a class of 
Laplacian operators acting on functions $f\in \mathcal{F}_{\mathcal{V}}$ was defined by  
\begin{subequations}\label{eq:Elmoataz_Model}
\begin{align}
\hskip -0.2cm \mathcal{L}_{\omega,p}f(x) &=
	 \hskip  -0.1cm\begin{cases}
		\beta^+(x)	\hskip -0.2cm\sum \limits_{y \in 
		\mathcal{N}^{+}(x)}\big(\nabla_{\omega}f(x,y)\big)^{p-1}+\beta^-(x)\hskip
		 -0.2cm\sum
		\limits_{y \in 
		\mathcal{N}^{-}(x)}(-1)^{p}\big(\nabla_{\omega}f(x,y)\big)^{p-1},\quad 
		&p \in 
		[2,\infty)\\
			\beta^+(x)\max \limits_{y \in 
			\mathcal{N}^{+}(x)}\big(\nabla_{\omega}f(x,y)\big)+\beta^-(x)\max
		\limits_{y \in 
			\mathcal{N}^{-}(x)}(-1)^{p}\big(\nabla_{\omega}f(x,y)\big),
		 \quad &p = \infty,
	\end{cases}
\intertext{where}
\mathcal{N}^{+}(x) &= \{y \in \mathcal{N}(x):f(y)-f(x)>0\}, \qquad 
\mathcal{N}^{-}(x) = \{y \in \mathcal{N}(x):f(y)-f(x)<0\}.
\end{align}
\end{subequations}
As detailed in \cite[Section 4]{Elmoataz:2015aa} depending on the weighting 
function $\omega \in \mathcal{F}_{\mathcal{V}\times 
\mathcal{V}}$ and on the positive functions $\beta^+,\beta^- \in 
\mathcal{F}_{\mathcal{V}}$  satisfying $\beta^+(x)+\beta^-(x) = 1, x \in 
\mathcal{V}$, the Laplacians \eqref{eq:Elmoataz_Model} enable to generalize a 
broad class of 
variational approaches including \cite{Elmoataz:2008aa} whose Euler Lagrange 
equations involve graph Laplacians. 

In the following, we focus on undirected graphs 
$(\mathcal{V},\mathcal{E},\omega)$ with $\omega(x,y) = \omega(y,x)$.  
Then, for the purpose of data inpainting and following \cite{Elmoataz:2015aa}, 
given a vertex set $\mathcal{A}\subset 
\mathcal{V}$ together with a function 
$g\in \mathcal{F}_{\partial 
\mathcal{A},\R^c}$ specifying the 
boundary condition imposed on 
\begin{align}
\partial \mathcal{A} = \{ x\in \mathcal{V}\setminus \mathcal{A}:\exists y 
\in \mathcal{A} \text{ with } y \in \mathcal{N}(x)\}, 
\end{align}
the nonlocal Laplacian \eqref{eq:Elmoataz_Model} generates a family of nonlocal discrete diffusion 
processes of the form
\begin{subequations}\label{eq:Elmoataz_Diffusion_Rel_Work}
\begin{align}
	\partial_{t} f(x,t) &= \mathcal{L}_{\omega,p}f(x,t) &\quad &\text{on } 
	\mathcal{A}\times \R_{+}, \\
	f(x,t) &= g(x,t) &\quad &\text{on } \partial \mathcal{A}\times \R_{+},\\
	f(x,0) &= f_0(x) &\quad &\text{on } \mathcal{A}.  
\end{align}
\end{subequations} 
To establish a comparison with the proposed nonlocal formulation 
\eqref{eq:Non_Local_PDE}, we 
represent the model \eqref{eq:Elmoataz_Diffusion_Rel_Work} with $g = 0$ on 
$\partial \mathcal{A}$ in terms of the 
operators introduced in Section \ref{sec:Nonlocal-Calculus}. Following 
\cite[Section 5]{Elmoataz:2015aa} and setting the 
weighting function     
\begin{align}\label{eq:Elmoataz_Omega}
	\alpha^f(x,y) = 
 \begin{cases}
 	\beta^{+}(x)\sqrt{\omega(x,y)}^{p-1}\big(\nabla_{\omega}f(x,y)\big)^{p-2}, 
 	\quad &\text{if } f(y)>f(x),\\
 	\beta^{-}(x) \sqrt{\omega(x,y)}^{p-1}\big(\nabla_{\omega}f(y,x)\big)^{p-2}, 
 	\quad &\text{if } f(y)<f(x),
 \end{cases} 
\end{align}
the particular case $p = 2$ simplifies to a linear diffusion process  
\eqref{eq:combinatorial_Laplacian} with \eqref{eq:Elmoataz_Omega} 
directly given in terms of weights $\omega(x,y)$ prescribed by the adjacency relation of the
graph $\mathcal{V}$. 
Moreover, if at each vertex $x\in \mathcal{V}$ the equation $\beta^+(x) = 
\beta^-(x) = 
\frac{1}{2}$ holds, then for any $p\in 
[2,\infty)$ the mapping \eqref{eq:Elmoataz_Omega} is nonnegative and symmetric. 
As a consequence, $\alpha^{f}$ from \eqref{eq:Elmoataz_Omega} can 
substitute $\w(x,y)$ in \eqref{eq:combinatorial_Laplacian} and hence 
specifies a representation of the form \eqref{eq:Non_Local_Dif} when choosing 
the antisymmetric mapping $\alpha \in \mathcal{F}_{\mathcal{V}\times 
\mathcal{V}}$ to satisfy $2\alpha^2(x,y) = \alpha^f(x,y)$. 
Finally, specifying the symmetric mapping 
$\Theta \in \mathcal{F}_{\mathcal{V}\times \mathcal{V}}$ as $\Theta(x,y) = 1$ 
if $x \neq y$ and $\Theta(x,x) = -\sum_{y \in \mathcal{N}(x)} 
\alpha^2(x,y)$, expresses the system \eqref{eq:Elmoataz_Diffusion_Rel_Work} 
through \eqref{eq:Eq_Rel} with $\mathcal{V}$ and 
$\mathcal{V}_{\mathcal{I}}^{\alpha}$ given by $\mathcal{A}$ and $\partial 
\mathcal{A}$, respectively. 

We conclude with a comment similar to the previous sections. While the 
similarity of the above mathematical structures to our approach is evident from 
the viewpoint of diffusion processes, the scope of our approach, data labeling, 
differs and is not directly addressed by established diffusion-based 
approaches.  We further point out the different role
of interaction domain \eqref{eq:Interaction_Dom}. While for model 
\eqref{eq:Elmoataz_Diffusion_Rel_Work} we set $\alpha$ through  
\eqref{eq:Elmoataz_Omega} to satisfy $\mathcal{V}_{\mathcal{I}}^{\alpha} = 
\partial \mathcal{A}$ which is 
subset of given set of vertices $\mathcal{V}$, i.e. $\ol{\mathcal{V}} = \mathcal{V}$ as 
illustrated by the right panel of 
\ref{fig:Nonlocal_Boundary_Ill}), 
we focus in our work on mappings $\alpha$ that lead to an \textit{extension} of 
$\mathcal{V}$ by vertices in $\Z^d\setminus \mathcal{V}$, as presented by the 
left panel of Figure 
\ref{fig:Nonlocal_Boundary_Ill}.
\begin{figure}[ht!]
	\centering
	\begin{tikzpicture}
		\filldraw[radius=1.5pt,red]
		\foreach \x in {0.75,3.25} {
			\foreach \y in {0.75,1,...,3.25} {
				(\x, \y) circle[]		
			}
		};
		\filldraw[radius=1.5pt,red]
		\foreach \x in {0.75,1,...,3.25} {
			\foreach \y in {0.75,3.25} {
				(\x, \y) circle[]		
			}
		};
		\filldraw[radius=1.5pt,red]
		\foreach \x in {0.5,3.5} {
			\foreach \y in {0.5,0.75,...,3.5} {
				(\x, \y) circle[]		
			}
		};
		\filldraw[radius=1.5pt,red]
		\foreach \x in {0.75,1,...,3.25} {
			\foreach \y in {0.5,3.5} {
				(\x, \y) circle[]		
			}
		};
		\draw[rounded corners=3pt,fill=gray!70,opacity = 0.2]
		(0.75,0.75)--(0.75,3.25)--(3.25,3.25)--(3.25,0.75) ;
		\draw[rounded corners=3pt,fill=red!70,opacity = 0.2]
		(0.5,0.5)--(0.5,3.5)--(0.75,3.5)--(0.75,0.5);
		\draw[rounded corners=3pt,fill=red!70,opacity = 0.2]
		(3.25,0.5)--(3.5,0.5)--(3.5,3.5)--(3.25,3.5);
		\draw[rounded corners=3pt,fill=red!70,opacity = 0.2]
		(0.75,3.5)--(3.25,3.5)--(3.25,3.25)--(0.75,3.25);
		\draw[rounded corners=3pt,fill=red!70,opacity = 0.2]
		(0.75,0.5)--(3.25,0.5)--(3.25,0.75)--(0.75,0.75);
		\draw (0,0) -- (0,4);
		\draw (0, 0) grid[step=0.25cm] (4,4);
		\filldraw[radius=1pt,black]
		\foreach \x in {0.5,0.75, ...,3.5} {
			\foreach \y in {0.5,0.75, ..., 3.5} {
				(\x, \y) circle[]		
			}
		};
		
		\filldraw[radius=1.5pt,red]
		\foreach \x in {9.5, 10.5} {
			\foreach \y in {1.5,1.75,...,2.5} {
				(\x, \y) circle[]		
			}
		};
		
		\filldraw[radius=1.5pt,red]
		\foreach \x in {9.5,9.75,...,10.25} {
			\foreach \y in {1.5,1.75,...,2.5} {
				(\x, \y) circle[]		
			}
		};
		
		\draw (8,0) -- (8,4);
		\draw (8, 0) grid[step=0.25cm] (12,4);
		\filldraw[radius=1pt,black]
		\foreach \x in {8.5,8.75, ...,11.5} {
			\foreach \y in {0.5,0.75, ..., 3.5} {
				(\x, \y) circle[]		
			}
			
		};
		
		\draw[rounded corners=3pt,fill=gray!70,opacity = 0.2]
		(8.5,0.5)--(8.5,3.5)--(11.5,3.5)--(11.5,0.5) ;

\filldraw[radius = 1pt,black] (4.4,3.8) circle[];
\filldraw[radius = 2pt,red] (4.4,3) circle[];
\filldraw[radius = 1pt,black] (4.4,3) circle[];
\filldraw[radius = 1pt,black] (4.3,2.2) circle[];
\filldraw[radius = 2pt,red] (4.5,2.2) circle[];
\filldraw[radius = 1pt,black] (4.5,2.2) circle[];
\node at (5,3.8) {$\in \mathcal{V}$};
\node at (5.1,3)   {$\in \mathcal{V}_{\mathcal{I}}^\mathcal{\alpha} $};
\node at (5,2.3) {$\in \ol{\mathcal{V}}$};		
		\filldraw[radius = 1pt,black] (12.3,3.8) circle[];
		\filldraw[radius = 2pt,red] (12.3,3) circle[];
		\filldraw[radius = 1pt,black] (12.3,3) circle[];
		\filldraw[radius = 1pt,black] (12.2,2.2) circle[];
		\filldraw[radius = 2pt,red] (12.4,2.2) circle[];
		\filldraw[radius = 1pt,black] (12.4,2.2) circle[];
		\node at (12.71,2.2) {$\in $};
		\node at (12.9,3.8)   {$\in \mathcal{A} $};
		\node at (13,3) {$\in \partial \mathcal{A}$};
		\node at (2,-0.5) {\textbf{nonlocal G-PDE} \eqref{eq:Non_Local_PDE}};
		\node at (10.3,-0.5) {\textbf{nonlocal approach} \cite{Elmoataz:2015aa}};
		\node at (13.1,2.2) {$\mathcal{V}$};
		\draw[rounded corners=3pt,fill=red!70,opacity = 0.2]
		(9.5,1.5)--(9.5,2.5)--(10.5,2.5)--(10.5,1.5);
	\end{tikzpicture}
	\captionsetup{font=footnotesize}
	\caption{ Schematic illustration of two different 
	instances of 
		$\mathcal{V}_{\mathcal{I}}^{\alpha}$. Nodes 
		$(\protect\markerseven)$ and 
		$(\protect\markereight)$ 
		represent points of the interaction 
		domain 
		$\mathcal{V}_{\mathcal{I}}^{\alpha}$ and the vertex set 
		$\mathcal{V}$,   
		respectively, in terms of the mapping 
		$\alpha \in \mathcal{F}_{\ol{\mathcal{V}} \times 
			\ol{\mathcal{V}}}$.  
		\textbf{Left}: Boundary configuration for the nonlocal G-PDE 
		\eqref{eq:Non_Local_PDE} introduced in this paper. 
		Nonzero interaction of nodes in $\mathcal{V}$ with nodes outside the 
		graph
		$\Z^d\setminus \mathcal{V}$ results in an extended domain 
		$\ol{\mathcal{V}}$ according to \eqref{eq:def-olV}. 
		\textbf{Right}: 
		Boundary configuration for the task of inpainting as proposed in 
		\cite{Elmoataz:2015aa}. The parameter $\alpha$ is specified entirely on 
		$\mathcal{V}$ resulting in a disjoint decomposition $\mathcal{V} = 
		\mathcal{A}\dot{\cup}\partial \mathcal{A}$ where 
		 now $\mathcal{V}_{\mathcal{I}}^{\alpha}$ satisfies 
		 $\mathcal{V}^{\alpha}_{\mathcal{I}} = \partial \mathcal{A}$ to 
		 represent the set of all 
		 nodes with missing information $\mathcal{V}\setminus 
		 \mathcal{A}$.} 
	\label{fig:Nonlocal_Boundary_Ill}
\end{figure}
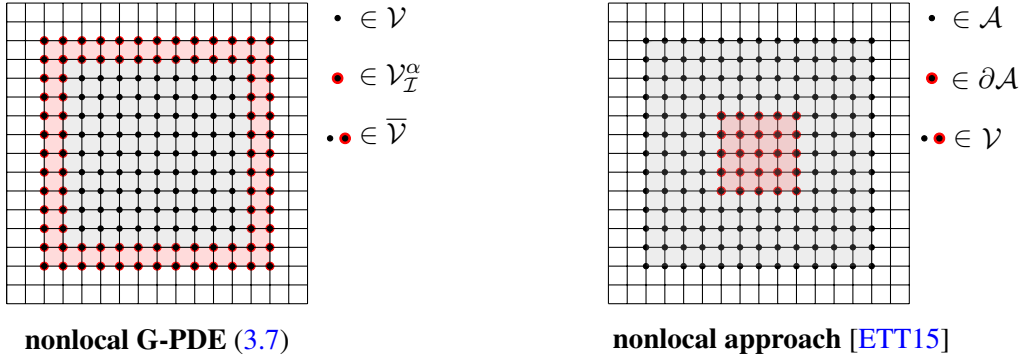
\section{Nonconvex Optimization by Geometric 
Integration}\label{sec:Non_Convex_Optimization}
We show in Section \ref{sec:first-order-geometric-DC} how geometric integration provides a numerical scheme 
for solving the nonlocal partial difference equation \eqref{eq:Non_Local_PDE} 
on a regular discrete grid $\mathcal{V}$ by generating a sequence of states on $\mathcal{W}$ that
monotonically decrease the energy objective \eqref{eq:S-Flow_Pot}. In particular, we show that the geometric Euler scheme is equivalent to the basic two-step iterative approach provided by \cite{Hoai-An:2005aa} for solving nonconvex optimization problems in \textit{DC (difference of convex functions)} format.

In Section \ref{sec:higher-order-geometric},  
we prove the monotonic decrease property for a \textit{novel} class of geometric \textit{multistage} 
integration schemes that speed up convergence and show the relation of this 
class to the nonconvex optimization framework 
presented in  \cite{Fukushima:1981aa,Aragon:2018aa} 

Figure \ref{fig:Overview_Algorithm} provides a schematic overview over key components of the two proposed algorithms, including references 
to the corresponding subsections. Proofs are provided in Appendix 
\ref{app:Geometric-Integration} to enable efficient reading.

\begin{figure}
	\captionsetup{font=footnotesize}
	\centering
\scalebox{0.95}{\includegraphics[width=\textwidth]{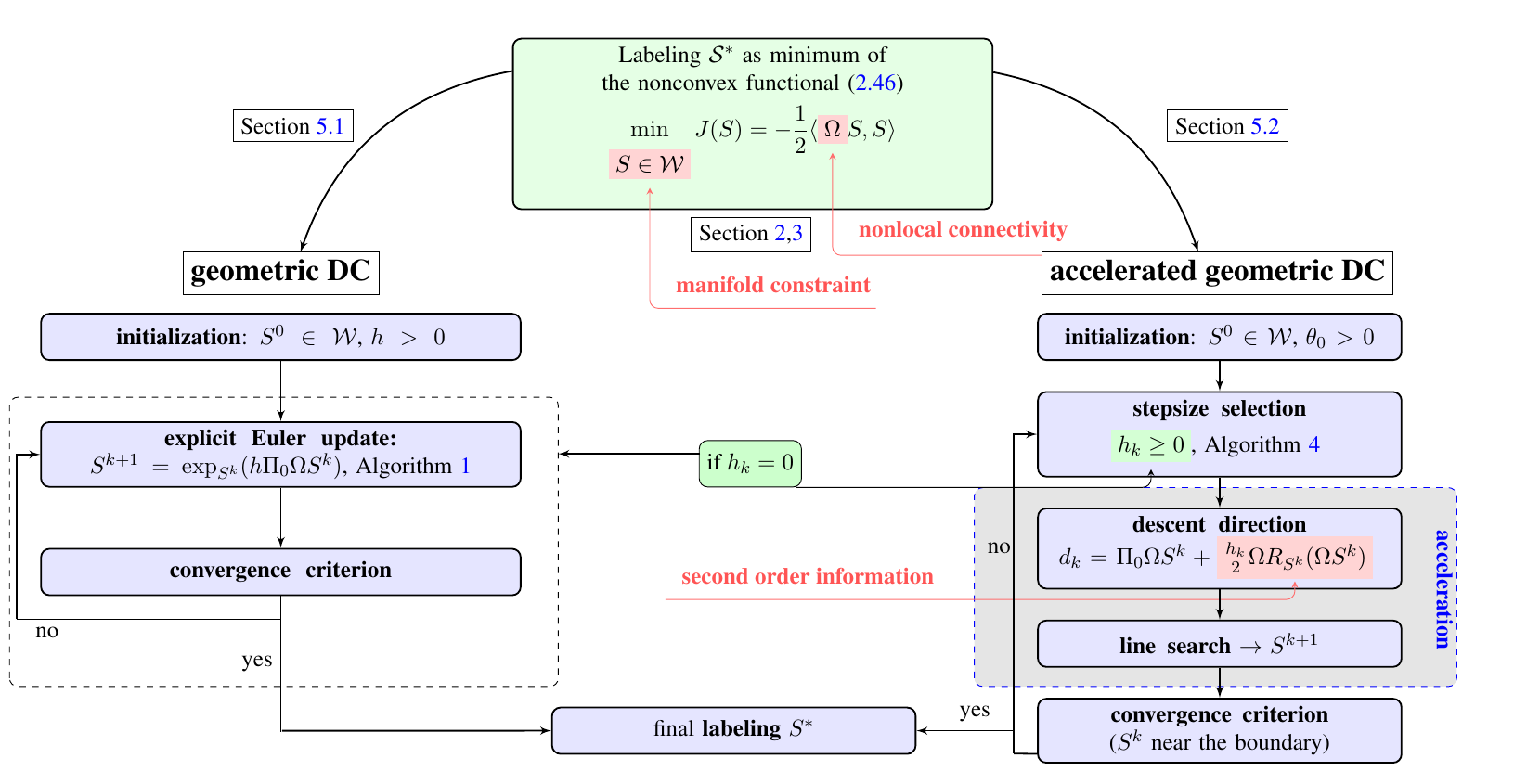}}
	\caption{Sketch of the two algorithmic schemes, Algorithm \ref{Geometric_Explicit_Euler} and Algorithm \ref{Geometric_Two_Stage}, developed in Section \ref{sec:Non_Convex_Optimization}. Common basic components as well as essential differences are highlighted. The major difference corresponds to the acceleration of the basic numerical scheme by geometric integration for solving the nonconvex DC program displayed in the top box. }
	\label{fig:Overview_Algorithm}
\end{figure}

\subsection{First-Order Geometric Integration and DC-Programming}\label{sec:first-order-geometric-DC}
We focus on an one-stage iterative numerical scheme derived by discretizing 
the explicit geometric Euler 
integration \eqref{eq:geometric_euler} in time with a fixed time-step size  
$h> 0$. In this specific case, \eqref{eq:geometric_euler} generates the 
sequence of iterates for approximately solving \eqref{eq:S-flow-S}  
given by
\begin{equation}\label{eq:explicit-Euler}
(S^k)_{k \geq 1} \subset 
\mathcal{F}_{\mathcal{V},\mathcal{W}},\quad
S^{k+1}(x) = \exp_{S^k(x)}\big(h(\Omega S)(x)\big), \quad S^0(x) = 
\exp_{\mathbb{1}_c}\Big(-\frac{D_{\mathcal{X}}(x)}{\rho}\Big),\quad
x\in\mc{V},
\end{equation}
where the index $k$ represents the point in time $kh$.      
We next show that the sequence \eqref{eq:explicit-Euler} locally minimizes the 
potential \eqref{eq:S-Flow_Pot} 
and hence, based on the formulation derived in Proposition 
\ref{Prop:Non_Local_PDE}, how geometric integration provides a finite 
difference scheme for numerically solving the nonlocal G-PDE 
\eqref{eq:Non_Local_PDE} for the particular case of zero nonlocal boundary 
conditions.
\begin{proposition}\label{prop:Numerical_Scheme}
	Let $\alpha,\Theta\in \mathcal{F}_{\ol{\mathcal{V}}\times 
	\ol{\mathcal{V}}},\lambda \in \mathcal{F}_{\mathcal{V}}$ and $\Omega \in 
	\mathcal{F}_{\mathcal{V}\times \mathcal{V}}$ be given as in 
	Lemma~\ref{Help_Lemma}. Then the sequence \eqref{eq:explicit-Euler} 
	satisfies
	\begin{align}\label{eq:PDE_Euler_Update}
	S^{k+1}(x) = 
	\exp_{S^{k}(x)}\bigg(h \Big(\frac{1}{2}\mathcal{D}^{\alpha}\big(\Theta 
	\mathcal{G}^{\alpha}(h \ol{S}^{k})\big) + \lambda \ol{S}^{k}\Big)(x)\bigg), 
	\qquad x 
	\in \mathcal{V},
	\end{align}  
	where the zero extension $\ol{S}^k$ of $S^k$ to $\ol{\mathcal{V}}$ is a 
	discrete approximation $S(h k)$ of the continuous time 
	solution to the
	system \eqref{eq:Non_Local_PDE}, initialized by 
	$S^0(x)$ \eqref{eq:explicit-Euler}
	with imposed zero nonlocal boundary conditions. 
	In addition, if 
	\begin{equation}\label{eq:h-lambda-min}
	h \leq 
	\frac{1}{|\lambda_{\min}(\Omega)|}, 
	\end{equation}
	where 
	$\lambda_{\min}(\Omega)$ denotes the smallest eigenvalue of 
	 $\Omega$, then the sequence $(S^{k})$ achieves the 
	monotone decrease property
	\begin{equation}\label{eq:Dec_Seq_Euler}
	J(S^{k+1}) \leq J(S^{k}),\qquad k\in\N
	\end{equation}  
	for the potential function \eqref{eq:S-Flow_Pot}.
\end{proposition}
\noindent
\textit{Proof.}  Appendix 
\ref{app:first-order-geometric-DC}.

\vspace{0.3cm}
Recent work \cite{Zern:2020aa} on the convergence  
of \eqref{eq:S-flow-S} showed that, up to negligible situations that cannot 
occur when working with real data, limit points $S^{\ast}=\lim_{t\to\infty}S(t)$ of \eqref{eq:S-flow-S} are \textit{integral} assignments 
$S^{\ast}\in\ol{\mc{W}}$. Proposition \ref{prop:Numerical_Scheme} says that for stepsizes 
$h < 1$ the geometric integration step \eqref{eq:explicit-Euler} yields a 
\textit{descent} direction for moving $S(t) \in \mathcal{W}$ to $S(t+h) \in 
\mathcal{W}$ and therefore sufficiently approximates the integral curve corresponding to \eqref{eq:S-flow-S} at time $t+h$. We conclude that 
the fixed point determined by Algorithm \ref{Geometric_Explicit_Euler} listed 
below solves the nonlocal G-PDE \eqref{eq:Non_Local_PDE}.

\vspace{0.2cm}
\begin{algorithm}[H]
	\textbf{Initialization}: $\gamma>|\lambda_{\min}(\Omega)|$ (DC-decomposition 
	parameter, see proof ~ 
	\proofref{prop:Numerical_Scheme}) \\
	$S^{0}=S(0)\in \mathcal{W}$ (initial point by \eqref{eq:S-flow-S})\\
	$\epsilon>0$ (termination threshold)\\
	$\epsilon_0 = \|\ggrad_g J(S^{0})\|$ \quad ($\ggrad_{g}J(S)=R_{S}(\partial_{S}J(S))$) \\
	$k=0$ \\
	\While{$\epsilon_{k}>\epsilon$}{
		$\wt{S}^{k} = \Omega S^{k} + \gamma \log S^{k}$\\
		\text{compute: }\text{$S^{k+1}  = \argmin \limits_{S \in 
		\overline{\mathcal{W}}}\{\gamma S\log S -\langle \wt{S}^k,S 
		\rangle\}$} \text{given by} \eqref{eq:explicit-Euler} 
resp.~\eqref{eq:PDE_Euler_Update}\label{eq:def-next-iterate-alg1} 
with $h = 
\frac{1}{\gamma}$ \\
		$\epsilon_{k} = \|\ggrad_g J(S^{k+1})\|$ \label{eq:grad_eps_alg1}\\
		$k \leftarrow k+1$}
	\caption{Geometric DC-Programming Scheme.}
	\label{Geometric_Explicit_Euler}
\end{algorithm}

\subsection{Higher-Order Geometric Integration}\label{sec:Geometric-Integration}\label{sec:higher-order-geometric}
In this section we show how 
higher-order geometric integration schemes can be used and enhance the 
first-order method of the previous section.

We continue the discussion of the numerical integration of the assignment flow \eqref{eq:S-flow-S} by employing the 
tangent space parameterization \eqref{eq:T0_parameterization}. For a 
discussion of relations to the geometry of $\mc{W}$, we refer to 
\cite{Zeilmann:2020aa}. In what follows, we drop the argument $x\in\mc{V}$ and just work with matrix products -- cf.~\eqref{eq:def-Omega-S-matrix} -- besides the lifting map $\exp_{S}$ that acts row-wise as defined by \eqref{eq:Small_exp}.

Our starting point is the explicit geometric Euler scheme 
\eqref{eq:geometric_euler} and \eqref{eq:explicit-Euler}, respectively, 
\begin{equation}
S(t+h) \approx \exp_{S^0}\big( V(t) + h\dot V(t) \big)
= \exp_{S(t)}\big( h(\Omega S)(t) \big).
\end{equation}
Now compute the second-order derivative of all component functions on $\mc{T}_{0}$  
\begin{equation}\label{eq:second_derivative}
\ddot V(t) \overset{\eqref{eq:T0_parameterization2}}{=} \Pi_{0}\Omega \frac{\dd}{\dd t}\exp_{S^0}\big(V(t)\big)
\overset{\substack{\eqref{eq:T0_parameterization}\\ \eqref{eq:dt-S-exp-V}}}{=} \Pi_{0}\Omega R_{\exp_{S^0}(V(t))} \dot V(t)
\overset{\eqref{eq:T0_parameterization}}{=} \Pi_{0}\Omega R_{S(t)}\big(\Omega S(t)\big).
\end{equation}
Then the second-order expansion $V(t+h) = V(t) + h\dot V(t) + \frac{h^2}{2}\ddot 
V(t) + \mathcal{O}(h^3)$ in $\mc{T}_{0}$ leads to the second-order geometric integration scheme
\begin{subequations}\label{eq:second_order_scheme}
	\begin{align}
	S(t+h) 
	&\approx \exp_{S(t)}\Big( h\dot V(t) + \frac{h^2}{2}\ddot V(t) 
	\Big)\label{eq:second_order_scheme1}\\
	&= \exp_{S(t)}\Big( h\Omega S(t) + \frac{h^2}{2}\Omega R_{S(t)} (\Omega S(t))
	\Big),\label{eq:second_order_scheme2}
	\end{align}
\end{subequations}
which may be read due to \eqref{eq:exp-S-group} as the \textit{two-stage iterative algorithm}
\begin{subequations}\label{eq:two_stage_scheme}
	\begin{align}
	\wt S(t)   &= \exp_{S(t)}\big( h\Omega S(t) \big), 
	\label{eq:two_stage_scheme-a} \\ 
	\label{eq:two_stage_scheme-b}
	S(t+h) &= \exp_{\wt S(t)}\Big( \frac{h^2}{2}\Omega R_{S(t)}(\Omega 
	S(t)) \Big).
	\end{align}
\end{subequations}
Below, we set in view of \eqref{eq:T0_parameterization}
\begin{equation}\label{eq:def-JV}
J(V) := J(S)|_{S=\exp_{S^{0}}(V)} = J\big(\exp_{S^{0}}(V)\big)
\end{equation}
to simplify the notation. 
The following lemma prepares our main result.
\begin{lemma}\label{lem:Grad_Tangent}
	Based on the parametrization \eqref{eq:T0_parameterization}, the Euclidean 
	gradient of the function $V\mapsto J(V)$ is given by 
	\begin{align}\label{eq:Riemannian_Grad_Section}
		\partial J(V) = -R_{\exp_{S^0}(V)}\big(\Omega \exp_{S^0}(V)\big)
		=\ggrad_{g} J(S),
	\end{align} 
that is by the Riemannian gradient of the potential \eqref{eq:S-Flow_Pot}.
\end{lemma}
\noindent
\textit{Proof.}  Appendix \ref{app:Geometric-Integration}.

\vspace{0.3cm}
The next proposition asserts that applying the second-order geometric 
integration scheme \eqref{eq:two_stage_scheme} leads to a sufficient 
decrease of the sequence of values $(J(S^{k}))_{k\in\N}$, if at each iteration the step sizes are chosen according to a \textit{Wolfe rule} like line search procedure 
\cite{Dai:1999aa,Nocedal:2006aa}. Specifically, the step sizes $h$ and $h^{2}$ 
in \eqref{eq:two_stage_scheme-a} and \eqref{eq:two_stage_scheme-b}, 
respectively, are replaced by step size sequences $(\theta_{k})_{k\geq 0}$ and 
$(h_{k}\theta_{k})_{k\geq 0}$. In addition, the proposition 
reveals that, under mild assumptions on the sequence $(h_k)_{k \geq 0}$, the norm 
of the Riemannian gradient \eqref{eq:Riemannian_Grad_Section} becomes 
arbitrarily small. The proposition is proved in Appendix 
\ref{app:Geometric-Integration}.
\begin{proposition}\label{prop:Numerical_Scheme_Second_Order}
	Let $\Omega(x,y)$ be as in 
	Lemma~\ref{Help_Lemma} and let $d:\mathcal{W}\times \R_{+}\to 
	\mathcal{T}_{0}$ be a mapping given by 
	\begin{align}\label{eq:Descent_Dir_Two_Stage}
	d(S,h) = \Pi_{0}\Big(\Omega S +\frac{h}{2}\Omega R_{S}(\Omega S) \Big), 
	\qquad 
	S\in \mathcal{W},\quad h\in \R_{+}.
	\end{align}
	Then the following holds:
	\begin{itemize}
	\item[(i)] There exist sequences $(h_k)_{k\geq  
	0}, (\theta_k)_{k\geq 0}$ and constants $0<c_1<c_2<1$ such that setting 
	\begin{subequations}\label{eq:Second_Order_Scheme}
		\begin{align}
		S^{k+\frac{1}{2}} &= \exp_{S^k}(\theta_k\Omega S^k),\\
		S^{k+1} &= \exp_{S^{k+\frac{1}{2}}}\Big(\frac{h_k\theta_k}{2}\Omega 
		R_{S^k}(\Omega S^k)\Big) 
\label{eq:second_order_second},
		\end{align}
	\end{subequations}
	and 
\begin{equation}\label{eq:dk-accelerate}
	d^k \coloneqq d(S^k,h_k) \in \mathcal{T}_{0}
\end{equation}
	yields iterates 
\begin{equation}\label{eq:Sk+1-Alg-3}
S^{k+1}=\exp_{S^{k}}(\theta_kd^{k}),\qquad k\in\N
\end{equation}
	satisfying 
	\begin{subequations}\label{eq:Sufficient_Decrease_Properties}
	\begin{align}
		J(S^{k+1})-J(S^k) &\leq c_1 \theta_k \langle 
		\ggrad_{g} 
		J(S^k),R_{S^k}(d^k) 
		\rangle_{S^k}, \quad (\normalfont{\text{Armijo 
		condition}})\label{eq:Sufficient_Decrease_Properties_a}\\
		|\langle \ggrad_{g} J(S^{k+1}),R_{S^k}(d^k) 
		\rangle_{S^{k}}| &\leq c_2 |\langle \ggrad_{g} 
		J(S^k),R_{S^k}(d^k)
		\rangle_{S^k}|,\quad (\normalfont{\text{curvature condition}})
		\label{eq:Sufficient_Decrease_Properties_b} 
	\end{align} 
	\end{subequations}
and (recall \eqref{eq:Fisher_Rau})
	\begin{equation}\label{eq:FischerRao-matrix}
	\la U, V \ra_{S} = \sum_{x\in\mc{V}}g_{S(x)}\big(U(x),V(x)\big),
	\qquad U,V\in\mc{T}_{0},\quad S\in\mc{W}.
	\end{equation}
	\item[(ii)] Suppose the limit point $\gamma_{\ast}$ of $(\theta_k)_{k\geq 
	0}$ is bounded away from zero, i.e.~$\gamma_{\ast} = \lim 
	\limits_{k \to \infty} \theta_k
	>0$. Then any limit point $S^{\ast} \in 
	\overline{\mathcal{W}}$ 
	of the 
	sequence \eqref{eq:Second_Order_Scheme} is an 
	equilibrium of the flow \eqref{eq:S-flow-S}. 
	\item[(iii)] If $S^{\ast}$ is a limit point of 
	\eqref{eq:Second_Order_Scheme} which locally minimizes $J(S)$, 
	with sequences  
	$(\theta_k)_{k \geq 0},(h_k)_{k \geq 0}$ as in $\mrm{(ii)}$, then $S^{\ast}\in 
	\ol{\mathcal{W}}\setminus \mathcal{W}$. 
	\item[(iv)] If additionally   
	$\sum_{k \geq 
	0}h_k = 0$ holds in $(ii)$, then the sequence $(\epsilon_{k})_{k \geq 0}$ with 
	$\epsilon_{k} \coloneqq \|\text{\normalfont{grad}}_g J(S^k)\|$ is a zero 
	sequence.
\end{itemize} 
\end{proposition}
\noindent
\textit{Proof.}  Appendix \ref{app:Geometric-Integration}.

\vspace{0.3cm}
Given a state $S^k \in \mathcal{W}$, Proposition 
\ref{prop:Numerical_Scheme_Second_Order} asserts the existence of step size sequences 
$(h_k)_{k\geq 0}, (\theta_k)_{k\geq 0} \subset \R_{+}$ that guarantee a sufficient 
decrease of the objective \eqref{eq:S-Flow_Pot} through 
\eqref{eq:Sk+1-Alg-3} while still remaining numerically 
efficient by avoiding too small step sizes through 
\eqref{eq:Sufficient_Decrease_Properties}. A corresponding proper stepsize selection procedure is summarized as Algorithm \ref{alg:Step_Size_Selection} that calls Algorithm \ref{alg:Search_Aux} as a subroutine. Based on Algorithm \ref{alg:Step_Size_Selection}, the two-stage geometric integration scheme \eqref{eq:two_stage_scheme} that \textit{accelerates}  Algorithm
\ref{Geometric_Explicit_Euler} is listed as Algorithm \ref{Geometric_Two_Stage}. Acceleration is accomplished by utilizing 
at each $S^k$ descent directions $d_k$ given by \eqref{eq:dk-accelerate}, based on second-order 
information provided by the vector field 
\eqref{eq:second_derivative}.

\begin{algorithm}[ht]
	\textbf{Input: } current iterate: $S^k \in \mathcal{W}$, initial step size 
	$\theta_k>0$,\\ descent direction $d_k$ with $\langle 
	\text{\normalfont{grad}}_g J(S^k),R_{S^k}d_k \rangle_{S^k} < 0$,\\
	$k = 1$. \\
	\Repeat{$\theta_k$ satisfies \eqref{eq:Sufficient_Decrease_Properties}}{
		$S^{k+1} = \exp_{S^k}(\theta_{k}d^k)$\\
		\If{$J(S^{k+1})-J(S^k) > \theta_k c_1 \langle 
			\text{\normalfont{grad}}_g J(S^k),R_{S^k}d_k \rangle_{S^k}$}{
		$a = a, b = \theta_{k}$} 
		\Else{
		\If{$|\langle \text{\normalfont{grad}}_g J(S^{k+1}),R_{S^k} d_k 
			\rangle_{S^{k}}| \leq |c_2\langle \text{\normalfont{grad}}_g 
			J(S^k),R_{S^k}d_k
			\rangle_{S^k}|$}{\textbf{stop}}
		$a = \theta_{k},b = b,\theta_{k+1} = \frac{a+b}{2}.$}
		$k \leftarrow k+1.$}
\textbf{Return: } $S^k,\theta_{k}$
\caption{\textrm{Search}\,($S^{k},\theta_k,d_k,c_1,c_2,a,b$).}
\label{alg:Search_Aux}
\end{algorithm}
\begin{algorithm}[h]
	\textbf{Input: } current iterate: $S^k \in \mathcal{W}$, initial step size 
	$\theta_k>0$,\\ descent direction $d_k$ with $\langle 
	\text{\normalfont{grad}}_g J(S^k),R_{S^k}d_k \rangle_{S^k} < 0$,\\
	smallest eigenvalue of $\Omega$, $\lambda_{\min}(\Omega)$ 
	$c_1,c_2\in (0,1)$ with $c_2 \in (c_1,1)$,\\
	\text{initial search interval:} $a_1 = \theta_{k}, 
	b_1 = 
	\frac{1}{|\lambda_{\text{min}}(\Omega)|} \text{ with }a_1 < b_1$,\\
	$k = 1$. \\
	\Repeat{$\theta_k$ satisfies \eqref{eq:Sufficient_Decrease_Properties_a}}{
	\text{$\theta_k = \frac{a_k+b_k}{2},S^{k+1} = \exp_{S^k}(\theta_k 
	d_k)$},\\
	\If{$J(S^{k+1})-J(S^k) > \theta_k c_1 \langle 
		\text{\normalfont{grad}}_g J(S^k),R_{S^k}d_k \rangle_{S^k}$ 
 	}
	{$S^{k+1},\theta_{k+1} \leftarrow 
	\text{Search}(S^{k},\theta_k,c_1,c_2,a_k,b_k)$ (Algorithm \ref{alg:Search_Aux}), 
	\textbf{stop}}
	\Else{
		\If{$|\langle \text{\normalfont{grad}}_g J(S^{k+1}),R_{S^k} d_k 
		\rangle_{S^{k}}| \leq |c_2\langle \text{\normalfont{grad}}_g 
		J(S^k),R_{S^k}d_k
		\rangle_{S^k}|$}
		{\textbf{stop}}
		\Else{
		$a_{k+1} = \theta_{k+1},b_{k+1} = b_k$.}
	}
	$k \leftarrow k+1.$
	}
	\textbf{Return: } $S^k$
	\caption{\textrm{Step}\,($S^k,\theta_k,d_k,c_1,c_2,\lambda_{\text{min}}(\Omega)$).}
	\label{alg:Step_Size_Selection}
\end{algorithm}

In Section \ref{sec:Convergence}, we show that 
Algorithm \ref{Geometric_Two_Stage} converges. This implies, in particular, that Algorithm \ref{Geometric_Explicit_Euler} and Algorithm \ref{Geometric_Two_Stage} terminate after a finite number of steps for any termination parameter $\veps$ with respect to the entropy of the assignment vectors, which measures closeness to an integral solution. Theorem \ref{thm:Existance-Epsilon} asserts the existence of basins of attraction around integral solutions from which the sequence $(S^{k})$ can never escape once it has reached such a region.

We elaborate in terms of Theorem \ref{theorem:convergence-1} a theoretical guideline 
for choosing a sequence $(h_{k})_{k \geq 0}$ which meets the condition of 
Proposition \ref{prop:Numerical_Scheme_Second_Order} (iv). In practice, to 
achieve an acceleration by Algorithm \ref{Geometric_Two_Stage} in comparison with Algorithm \ref{Geometric_Explicit_Euler}, we choose a large value of the step size parameter 
$h_k$ in the beginning and monotonically decrease $h_k$ to zero after a fixed 
number of iterations. One particular step size selection strategy that we used for the numerical experiments will be 
highlighted in Section \ref{sec:Exp_Dis}. 

\begin{algorithm}[h]
	\textbf{Initialization}: (DC-decomposition 
	parameter, 
	see the proof of Prop.~\ref{prop:Numerical_Scheme}),\\
	$S^0= S(0) \in \mathcal{W}$, (initial iterate 
	\eqref{eq:S-flow-S}),\\
	$\epsilon>0$, (termination threshold),\\
	$\lambda_{\min}(\Omega)$, (smallest 
	eigenvalue of $\Omega$), \\
	$c_1,c_2 \in (0,1)$, (cf.~Prop.~ 
	\ref{prop:Numerical_Scheme_Second_Order}),\\
	$\epsilon_0 = \|\text{\normalfont{grad}}_g J(S^0)\|$, 
	$\theta_0 = \frac{1}{\gamma}$ (cf.~\eqref{eq:J-DC}) \\
	$k = 0$. \\
	\While{$\epsilon_{k}>\epsilon$}{
		\text{Choose:} $h_k \in \Big(0,\frac{\|R_{S^k}(\Omega S^k) 
		\|^2_{S^k}}{|\langle 
		R_{S^k}(\Omega S^k),\Omega R_{S^k}(\Omega S^k) \rangle|}\Big)$ 
		\label{eq:alg:step_size_int}\\
		$d_k = \Pi_0\Omega S^{k}+\frac{h_k}{2}\Omega R_{S^k}(\Omega S^k)$ 
		\text{(descent direction by \eqref{eq:dk-accelerate},\eqref{eq:Descent_Dir_Two_Stage})}\\
		\If{$\theta_{k}$ satisfies 
		\eqref{eq:Sufficient_Decrease_Properties}}
		{\text{Set: } $\wt{S}^{k} = 
			\frac{1}{\theta_{k}}\log(\frac{S^k}{\mathbb{1}_c})+d_k$\\
		\text{Compute: }\text{$S^{k+1}  = \argmin \limits_{S \in 		
				\overline{\mathcal{W}}}\{\frac{1}{\theta_{k}}S\log
					 S -\langle 
					\wt{S}^k,S 
				\rangle\}$},\label{alg-final-upd} by\\
				 $S^{k+1} = \exp_{S^k}(\theta_{k}d^k)$
		 \label{alg:tilde_S}	
		}
		\Else{$S^{k+1} \leftarrow$ 
		\textrm{Step}$\big(S^k,\theta_k,d_k,c_1,c_2,\lambda_{\text{min}}(\Omega)\big)$
		 by \text{Algorithm} 
		\ref{alg:Step_Size_Selection}. \label{alg:Update-DC-par-line}}
		$\epsilon_{k+1} = \|\text{\normalfont{grad}}_g J(S^{k+1})\|$,\\
		$k \leftarrow k+1$. }
	\textbf{Returns: } $S^k \approx S^{\ast}$
	\caption{Accelerated Geometric DC Optimization}
	\label{Geometric_Two_Stage}
\end{algorithm}

\vspace{0.5cm}
The following remark  
clarifies how the line search procedure formulated as Algorithm \ref{alg:Step_Size_Selection}, that is used in Algorithm \ref{Geometric_Two_Stage}, differs from the common line search accelerated DC-programming schemes proposed by 
\cite{Fukushima:1981aa} and \cite{Aragon:2018aa}.
\begin{remark}[\textbf{directly related work}]
	Using the notation of Proposition \ref{prop:Numerical_Scheme} and its proof, 
	the step iterated by Algorithm \ref{Geometric_Explicit_Euler} at $S^k \in \mathcal{W}$ reads
	\begin{subequations}\label{eq:DC-not-ac-rem}
	\begin{align}
	\wt{S}^k &= \argmin_{S \in \R^n}\left\{ h^{\ast}(S)-\langle S^{k},S 
	\rangle 
	\right\}, &\quad &\text{with} &\qquad &h(S) = \langle S,\Omega S 
	\rangle+\gamma S \log S,\\
	S^{k+1} &= \argmin_{S \in \R^n}\left\{ 
	g(S)-\langle 
	S,\wt{S}^k \rangle \right\},&\quad &\text{with} &\qquad &g(S) = 
	\delta_{\overline{\mathcal{W}}}(S)+\gamma S \log S, \label{eq:rem_final_up}
	\end{align}
	\end{subequations}
	where $h^{\ast}$ is the conjugate of the convex function $h$.
	Motivated by the work \cite{Fukushima:1981aa}, Arag{\'o}n et al.~\cite{Aragon:2018aa} proposed an accelerated version of the above scheme  
	 by performing an additional line search step along the
	descent direction 
\begin{equation}
	\wt{d}^k = S^{k+1}-S^{k}
\end{equation}
	in 
	\eqref{eq:rem_final_up} for scenarios, where the primary variable $S$ to be determined is \underline{not} manifold-valued. 
	
	The direct comparison with Algorithm 
	\ref{Geometric_Explicit_Euler} reveals that for the 
	specific choice $h_k = 0, k \in \N$ in \eqref{eq:dk-accelerate}, \eqref{eq:Descent_Dir_Two_Stage}, 
line search is performed along the descent direction 
\begin{equation}
d^k = \Pi_{0} \Omega S^{k} = V^{k+1}-V^{k} \in\mc{T}_{0},
\end{equation}
where the last equation follows from applying the parametrization \eqref{eq:T0_parameterization} to \eqref{eq:Second_Order_Scheme} while taking into account \eqref{eq:def-lifting-map} and $R_{S}=R_{S}\Pi_{0}$ for $S\in\mc{W}$.

Comparing $\wt{d}^{k}$ and $d^{k}$ shows the geometric nature of our algorithm in order to handle properly the manifold-valued variable $S$ and the more general descent directions $d^{k}$ with step sizes $h_{k}>0$ in Algorithm \ref{Geometric_Two_Stage}.
\end{remark}

\subsection{Influence of Nonlocal Boundary 
Conditions}\label{sec:Influence_Nonlocal_Boundary}
We conclude this section by explaining in more detail the 
effect of imposing in  
\eqref{eq:Non_Local_PDE} the
zero nonlocal boundary condition on the nonempty interaction domain,  
on the stepsize selection 
procedure presented as Algorithm  \ref{alg:Step_Size_Selection}. This 
explanation is formulated as Remark \ref{rem:-Step_Size} below after the 
following proposition, that states a result analogous to \cite[Proposition 
2.3]{Andreu-Vaillo:2010aa}. The proposition is proved in Appendix
\ref{app:Influence_Nonlocal_Boundary}.
\begin{proposition}\label{lem:Omega_Invertibility}
	For mappings $\Theta,\alpha \in 
	\mathcal{F}_{\overline{\mathcal{V}}\times\overline{\mathcal{V}}}$, let 
	$\Omega \in \mathcal{F}_{\mathcal{V}\times \mathcal{V}}$ and $\lambda 
	 \in \mathcal{F}_{\mathcal{V}}$ be given as in Lemma  
	\ref{Help_Lemma} such that property \eqref{eq:Omega_AF} holds 
	and $\lambda = 1,x\in \mathcal{V}$ in \eqref{eq:def-lambda-x} is achieved. 
	Assume 
	further that the weighted graph $(\mathcal{V},\mathcal{E},\Omega)$ in 
	\eqref{def:weighted-graph} is connected. Then the  following holds:
	\begin{itemize}
	\item[(i)] The smallest Dirichlet eigenvalue of the nonlocal operator 
	\eqref{eq:Non_Local_Dif}
	\begin{align}\label{eq:Dirichlet_Eigvals}
	 \lambda_{1}^{D} = \inf_{f \neq 0} -\frac{\frac{1}{2}\langle f, 
	 \mathcal{D}^{\alpha}(\Theta 
	 	\mathcal{G}^{\alpha}f) 
	 	\rangle_{\ol{\mathcal{V}}}}{\langle f,f 
	 	\rangle_{\ol{\mathcal{V}}}}, \qquad f \in 
			\mathcal{F}_{\ol{\mathcal{V}}}, \quad  
			f_{|\mathcal{V}_{\mathcal{I}}^{\alpha}} = 0,
	\end{align}
	is bounded away from zero and admits the equivalent expression 
		\begin{align}\label{eq:Lemma_Direchtlet_Eigenvals}
		0 < \lambda_1^{D} = \inf_{f \neq 0} \frac{\langle 
			f,(\Lambda- \Omega) f \rangle_{\mathcal{V}}}{\langle f,f 
			\rangle_{\mathcal{V}}},
		\end{align}
		where 
\begin{equation}\label{eq:def-Lambda-lem}
		\Lambda = \Diag ( \lambda ),\qquad
		\lambda = (\dots,\lambda(x) ,\dots )^{\T}
\end{equation}
		with $\lambda(x)$ given by \eqref{eq:def-lambda-x}. 
	\item[(ii)] One has $\lambda_{\min}(\Omega) > -1$.    
	\end{itemize}
\end{proposition}
\noindent
\textit{Proof.}  Appendix 
\ref{app:Influence_Nonlocal_Boundary}.

\vspace{0.3cm}
We are now in the position to characterize the effect of imposing the 
zero nonlocal boundary condition on the step size selection procedure 
(Algorithm \ref{alg:Step_Size_Selection}). 
\begin{remark}[\textbf{parameter selection}]\label{rem:-Step_Size}
	Recalling the proof of Proposition  \ref{prop:Numerical_Scheme}, the update 
	\eqref{eq:PDE_Euler_Update} amounts to perform at 
	each 
	step $k\in\N$ 
	one iteration of a basic DC programming scheme \cite{Hoai-An:2005aa} 
	with respect to the suitable DC decomposition \eqref{eq:J-DC} of 
	\eqref{eq:S-Flow_Pot}, with $\Omega$ satisfying 
	\eqref{eq:Omega_Mat}, \eqref{eq:Omega_AF} by choosing parameter $\gamma 
	>0$ such that $\lambda_{\min}\big(\Omega+\gamma 
	\Diag(\frac{\eins}{S})\big)>0$.   
	In the case of a 
	\textit{nonzero} interaction domain 
	\eqref{eq:Interaction_Dom} with $\Omega,\alpha,\Theta$ as in Lemma 
	\ref{lem:Omega_Invertibility}, Proposition 
	\ref{lem:Omega_Invertibility}(ii) and estimate 
	\eqref{eq:proof-prop-num} yield for 
	$S \in \mathcal{W}$ 
	\begin{subequations}\label{eq:rem-Step_Size} 
	\begin{align}
	\hskip -0.6cm
	\lambda_{\min}\Big(\Omega+\gamma \Diag\big(\frac{\eins}{S}\big)\Big) 
	&> -1+\beta 
	+\gamma > 0 
	\qquad\text{\normalfont{for}}\qquad 
	\gamma > 1-\beta, 
	\\
	\beta &= 
	\sum_{x\in\mathcal{V}_b}\sum_{y \in 
		\mathcal{V}_{\mathcal{I}}^{\alpha}}\Theta(x,y)\alpha^2(x,y)f^2(x).
	\end{align}  
	\end{subequations} 
	In particular, following the steps in proof of Lemma 
	\ref{prop:Numerical_Scheme}, relation $h = \frac{1}{\gamma}$ in 
	connection 
	with \eqref{eq:rem-Step_Size} 
	accounts for bigger step sizes in Algorithm 
	\ref{Geometric_Explicit_Euler} for integrating
	\eqref{eq:Non_Local_PDE} with nonzero interaction domain 
	\eqref{eq:Interaction_Dom}. This will be numerically 
	validated in Section \ref{sec:Exp_Dis} (see Figure 
	\ref{fig:Step_Size_Algorithm1}). 
\end{remark}

We conclude this section with a final comment on the lower bound of the objective 
\eqref{eq:S-Flow_Pot}. 
\begin{remark}(\textbf{global minimizer of} 
\eqref{eq:S-Flow_Pot})\label{rem:constant_labeling}
	 Recalling the terms involved in the objective \eqref{eq:S-Flow_Pot}, the 
	 lower 
	 bound is attained precisely when the first term $\sum_{x \in 
	 	\mc{V}}\sum_{y \in \mathcal{N}(x)} \Omega(x,y) 
	 \|S(x)-S(y)\|^2$ is minimal and the last term $-\frac{1}{2}\|S\|_{F}^{2}$ is 
	 maximal. Therefore the global minimizers of $J(S)$ are given by the set of 
	 spatially constant assignments, where to each node in graph $\mathcal{V}$ 
	 the same prototype $X_j^{\ast} \in \mathcal{X}$ is assigned.   
\end{remark}

\section{Convergence Analysis}\label{sec:Convergence}

This section is devoted to the convergence analysis of Algorithm 
\ref{Geometric_Two_Stage} that performs accelerated geometric integration of 
the Riemannian descent flow \eqref{eq:S-flow-S}. The main results are stated as 
Theorem \ref{theorem:convergence-1} and Theorem \ref{thm:Existance-Epsilon} in 
Section \ref{sec:convergence-main}. The lenghty proofs have been relegated to 
Appendix \ref{app:prep_lemmata}.

\subsection{Preparatory Lemmata}\label{sec:prep_lemmata} 

\begin{lemma}\label{lem:convergence}
	For a nonnegative, symmetric mapping $\Omega \in 
	\mathcal{F}_{\mathcal{V}\times \mathcal{V}}$, let the sequences $(S^k)_{k 
	\geq 
	0},(\theta_k)_{k \geq 0},(h_k)_{k \geq 0}$ be recursively defined by 
	Algorithm \ref{Geometric_Two_Stage} and 
	let $\Lambda$ denote the set of all limit points of the sequence
	$(S^{k})_{k \geq 0}$,   
	\begin{equation}\label{eq:limit_set-theorem-convergence}
	\Lambda = \{ S \in \overline{\mathcal{W}}: \exists (S^{k_l})_{l \geq 0} 
	\text{ 
		with } S^{k_l} \to S \text{ for } l \to \infty \}.
	\end{equation}
	Then there exists $J^{\ast} \in \R$ 
		with $\lim \limits_{k \to \infty} J(S^{k}) = J^{\ast}$, i.e. $J(S)$ is 
		constant on $\Lambda$.
\end{lemma}  
\noindent
\textit{Proof.} \textcolor{blue}{ Appendix \ref{app:prep_lemmata}}.

\vspace{0.3cm}
Next, we inspect the behavior of the iterates generated by Algorithm 
\ref{Geometric_Two_Stage} near a limit point $S^{\ast}\in \ol{W}$. To this end, the 
following index sets are considered at each node $x \in \mathcal{V}$:
\begin{subequations} \label{eq:index-sets-proof-conv-theorem}
\begin{align}
	J_+(S^{\ast}(x)) &= \{j \in [c] \colon (\Omega S^{\ast})_j(x) - \langle 
	S^{\ast}(x),(\Omega S^{\ast})(x)\rangle <0  \}, 
	\label{eq:index-sets-proof-conv-theorem-a} \\
	J_-(S^{\ast}(x)) &= \{j \in [c] \colon (\Omega S^{\ast})_j(x) - \langle 
	S^{\ast}(x),(\Omega S^{\ast})(x)\rangle >0  \}, \\
	J_0(S^{\ast}(x)) &= \{j \in [c] \colon (\Omega S^{\ast})_j(x) - \langle 
	S^{\ast}(x),(\Omega S^{\ast})(x)\rangle =0  \}.
\end{align} 
\end{subequations}
\begin{lemma}\label{lem:Converegnce-Two-Stage}
	Let $\Omega \in 
	\mathcal{F}_{\mathcal{V}\times \mathcal{V}}$ and $(S^k)_{k \geq 
		0},(\theta_k)_{k \geq 0},(h_k)_{k \geq 0}$ be as in Proposition 
	\ref{prop:Numerical_Scheme_Second_Order} $\text{\normalfont{(iv)}}$ with a sequence 
	$(\theta_k)_{k\geq 0}$ bounded by $\theta_k \in [\theta_{\text{min}} 
	,\theta_{\text{max}}]$. Let  
	$S^{\ast}\in \ol{\mathcal{W}}$ be a limit point of $(S^k)_{k \geq 0}$. 
	Then, for the positive function 
$Q(S) = \sum\limits_{x \in \mathcal{V}}\sum \limits_{j \in 
J_+(S^{\ast}(x))}S_j(x)$,
	there are constants $\varepsilon> 0$, $M^{\ast} > 1$ and an index $k_0$ 
	such 
	that 
	for all $k\geq 
	k_0$ with
	$\|S^{\ast}-S^{k}\|<\varepsilon$ the inequality  
	\begin{equation}\label{eq:Estimate-Lemma-Conv}
		Q(S^{k+1})-Q(S^{k})<\frac{\theta_{k}}{M^{\ast}}\sum_{x \in 
		\mathcal{V}}\sum_{j \in 
		J_+(S^{\ast}(x))}\hskip -0.5cm S^k_j(x)((\Omega 
S^{\ast})_j(x)-\langle 
\Omega 
S^{\ast}(x),S^{\ast}(x) \rangle ) < 0
	\end{equation}  
is satisfied.
\end{lemma}
\noindent
\textit{Proof.} \textcolor{blue}{ Appendix \ref{app:prep_lemmata}}.

\subsection{Main Results}\label{sec:convergence-main}

This section provides the main results of our convergence analysis: convergence of the accelerated Algorithm \ref{Geometric_Two_Stage} (Theorem \ref{theorem:convergence-1}) and an estimate of the basins of attraction around equilibria that enable early stopping of Algorithm \ref{Geometric_Two_Stage} (Theorem \ref{thm:Existance-Epsilon}).

\begin{definition}[\textbf{convex functions of Legendre type {\cite[Chapter 
26]{Rockafellar:1970aa}}}]\label{def:Legendre_Type_Fuction}
	Let $f:X \to (-\infty,\infty]$ be a lower-semicontinuous proper convex 
	function with nonempty open domain $C = \intr(\text{dom}f) \neq\emptyset$. Then $f$ is called  
	\begin{enumerate}[(i)]
		\item \textit{essentially smooth}, if $f$ is differentiable on 
		$C$ and for every sequence 
		$(x_k)_{k \in \N} \subset C$ with $x_k \to x^{\ast} \in 
		\overline{C} \setminus C$ converging to a boundary point for $k \to 
		\infty$, it follows $\| \nabla f 
		(x_k)\| \to \infty$;
		\item \textit{Legendre type function}, if $h$ is essentially 
		smooth and strictly convex on $C$.
	\end{enumerate}
\end{definition} 
Convex functions $f$ of Legendre type
yield a class of \textit{Bregman divergence functions} $D_{f}$ through   
\begin{equation}\label{eq:Bregman_Divergence}
\begin{aligned}
D_f \colon &\overline{C}\times C   \to  \R_{+},  \\
&\hspace{0,2cm} (x,y)   \mapsto  f(x) - f(y)-\langle \nabla f(y),x-y 
\rangle,
\end{aligned}
\end{equation}
see, e.g., \cite{Bregman:1967aa,Bauschke:1997aa} for a detailed 
exposition. 
Strict convexity of $f$ and Jensen's inequality imply 
\begin{equation}
\forall (x,y)\in\ol{C}\times C\colon\qquad
D_f(x,y) \geq 0
\quad\text{and}\quad
(D_f(x,y) = 0) \;\Leftrightarrow\; (x = y).
\end{equation} 
In the following, we will use the \textit{Kullback-Leibler (KL) divergence} (a.k.a.~\textit{relative entropy, information divergence}) $D_{\KL}=D_{f}$,
\begin{equation}
D_{\KL}\colon \ol{\mc{S}}\times\mc{S}\to\R_{+},\qquad
D_{\KL}(s,p) = \Big\la s,\log\frac{s}{p}\Big\ra,
\end{equation}
induced by the negative discrete entropy function 
\begin{equation}\label{eq:def-f-KL}
f = \la s,\log s\ra + \delta_{\ol{S}}(s)
\end{equation}
(with the convention $0\cdot\log 0 = 0$). Accordingly, we define with abuse of notation
\begin{equation}
D_{\KL}\colon\ol{\mc{W}}\times\mc{W}\to\R_{+},\qquad
D_{\KL}(S,P) = \sum_{x\in\mc{V}}D_{\KL}\big(S(x),P(x)\big).
\end{equation} 

\begin{theorem}[\textbf{convergence of Algorithm \ref{Geometric_Two_Stage}}]\label{theorem:convergence-1}
 Let $(S^k)_{k\geq 0}$ be a sequence generated by Algorithm 
 \ref{Geometric_Two_Stage}, where the sequences of step sizes 
 $(\theta_k)_{k\geq 
 0},(h_k)_{k \geq 0}$ additionally satisfy the assumptions of Lemma 
 \ref{lem:Converegnce-Two-Stage} and Proposition \ref{prop:Numerical_Scheme_Second_Order}, respectively. If there exists an index $K \in \N$ 
 such that the sequence $(h_k)_{k\geq K}$ satisfies  \begin{subequations}\label{eq:h_k-seq-assumption-theorem-conv}
\begin{align}
 	h_k &\leq C(\Omega) \frac{\|\text{\normalfont{grad}}_g 
 	J(S^k)\|^2_{S^k}}{n} 
	\\ \label{eq:def-lambda-Omega} &\qquad
	\text{with} \qquad C(\Omega) \coloneqq 
 	2\frac{ \theta_{\min} c_1 }{\lambda^2(\Omega)}, \qquad \lambda(\Omega) = 
 	\max\{|\lambda_{\min}(\Omega)|,|\lambda_{\max}|(\Omega)\},
\end{align}
\end{subequations}
then the set $\Lambda = \{S^{\ast}\}$ defined by \eqref{eq:limit_set-theorem-convergence} is a 
 singleton and $\lim_{k \to \infty} D_{\KL}(S^{\ast},S^k) = 0$ holds, i.e. 
 the sequence $(S^{k})_{k\geq 0}$ converges to a unique $S^{\ast}\in \ol{\mathcal{W}}$ 
 which is an equilibrium of \eqref{eq:S-flow-S}.  
\end{theorem}
\noindent
\textit{Proof.} \textcolor{blue}{ Appendix \ref{app:convergence-main}}.
 
\vspace{0.2cm}
According to Proposition 
\ref{prop:Numerical_Scheme_Second_Order} (iii),(iv) the sequence $(S^k)_{k \geq 0}$ converges to a critical point $S^{\ast} \in 
\ol{\mathcal{W}}\setminus\mc{W}$ on the boundary of convex set $\ol{\mathcal{W}}$. Since both functions $g, h$ of the DC-decomposition \eqref{eq:J-DC} have been regularized by the negative entropy, global Lipschitz continuity of the derivatives does $\textit{not}$ hold and hence does not allow to study the convergence rate of Algorithm \ref{Geometric_Two_Stage} along the lines pursued 
in \cite{Aragon:2018aa}, \cite{Bolte:2018aa},  \cite{Nhat_Phan:2018aa}. Therefore, we confine ourselves to establish a \textit{local linear} rate of convergence $S^{k}\to S^{\ast}$ within a suitably define basin of attraction in $\mc{W}$ around $S^{\ast}$. To this end, we adopt the following basic \\[0.2cm]
\textbf{Assumption:} Any stationary point $S^{\ast} \in 
\overline{\mathcal{W}}$ of the sequence $(S^{k})$ generated by Algorithm \ref{Geometric_Two_Stage} is a stable 
equilibrium of the flow \eqref{eq:S-flow-S}:
\begin{equation}
	(\Omega S^{\ast})_{j}(x)-(\Omega 
	S^{\ast})_{j^{\ast}(x)}(x) < 0, \qquad  j \in [c]\setminus {j^{\ast}(x) = 
	\argmax \limits_{l 
	\in 
	[c]} 
	S^{\ast}_l(x)}, \qquad \forall x \in 
	\mathcal{V}.\label{eq:Assumpotion_Convergence}
\end{equation}
\begin{remark}\label{rem:Assumption_Eq_Two_Stage}
	As worked out in \cite[Section 2.3.2]{Zern:2020aa}, the set of initial points $S(0)$ of the flow \eqref{eq:S-flow-S} for which Assumption 
	\eqref{eq:Assumpotion_Convergence} is not satisfied has measure zero. Hence Assumption \eqref{eq:Assumpotion_Convergence} holds in all practically relevant cases.  
\end{remark}

Based on Assumption \ref{eq:Assumpotion_Convergence}, we adopt the results 
reported in \cite[Section 2.3.3]{Zern:2020aa} by defining the open convex 
polytope for each integral equilibrium $S^{\ast} \in 
\mathcal{W}^{\ast}$ as
\begin{equation}\label{eq:AS-ast}
	A(S^{\ast}) \coloneqq \bigcap \limits_{x \in \mathcal{V}} \bigcap 
	\limits_{j \neq j^{\ast}(x)} \{ S \in \mathcal{F}_{\R^{n\times c}} \colon 
	(\Omega S)_j(x) < (\Omega S)_{j^{\ast}(x)}(x) 
	\},
\end{equation}  
and by introducing the \textit{basins of attraction}
\begin{align}\label{eq:def-basins-of-attraction}
	B_{\varepsilon}(S^{\ast}) \coloneqq \{ S \in \overline{W}\colon \max 
	\limits_{x 
	\in \mathcal{V}} \|S(x)-S^{\ast}(x)\|_1 < \varepsilon \} \subset 
	A(S^{\ast})\cap 
	\overline{W},
\end{align} 
where $\varepsilon> 0$ is small enough such that the inclusion in 
\eqref{eq:def-basins-of-attraction} holds.
Due to \cite[Proposition 2.3.13]{Zern:2020aa} a sufficient upper bound $\veps\leq \veps^{\ast}$ for 
the inclusion \eqref{eq:def-basins-of-attraction} to hold is
\begin{equation}\label{eq:Suff-Eps-Artjom}
	\varepsilon^{\ast} = \min_{x \in \mathcal{V}} \min_{j \in [c] \setminus 
	j^{\ast}(x)} \frac{2\big((\Omega S^{\ast})_{j^{\ast}(x)}-(\Omega 
	S^{\ast})_{j}\big)(x)}{\sum \limits_{y \in 
	\mathcal{N}(x)}\Omega(x,y)+\big((\Omega S^{\ast})_{j^{\ast}(x)}-(\Omega 
	S^{\ast})_{j}\big)(x)} >0.
\end{equation}
  The following theorem asserts that a modified criterium applies to the 
  sequence generated by Algorithm \ref{Geometric_Two_Stage}, together with a 
  linear convergence rate $S^{k}\to S^{\ast}$, whenever the sequence $(S^{k})$ 
  enters a basin on attraction $B_{\veps}(S^{\ast})$.  
 
 \begin{theorem}[\textbf{basins of attraction}]\label{thm:Existance-Epsilon}
 	For $\Omega \in 
 	\mathcal{F}_{\mathcal{V}\times \mathcal{V}}$ as in Lemma \ref{Help_Lemma}, 
 	let $(S^k)_{k \geq 0}$ be a sequence generated by Algorithm 
 	\ref{Geometric_Two_Stage}. Let $S^{\ast} \in \overline{\mc{W}}$ be a 
 	limiting point $(S^k)_{k \geq 0}$ that fulfills Assumption 
 	\ref{eq:Assumpotion_Convergence} and let $\varepsilon^{\ast}>0$ be as in 
 	\eqref{eq:Suff-Eps-Artjom}. Then, introducing the positive constants 
 	\begin{align}\label{eq:Lemma_Epsilon_Bound}
 		\overline{h} = 
 		\max_{k \in \N}h_k, \quad \rho^\ast = \max_{S \in 
 			\overline{\mathcal{W}}} \Big( \max_{\mathclap{\substack{ x \in 
 			\mathcal{V},\\j \in [c]\setminus 
 			j^{\ast}(x)}}}\big( 
 			(\Omega S)_{j^{\ast}(x)}-(\Omega S)_{j}\big)(x) \Big), \quad N = 
 			\max 
 			\limits_{y \in \mathcal{V}} 
 			|\mathcal{N}(y)|,
 	\end{align}
 	for all $\varepsilon > 0$ small enough such that  
 	\begin{equation}\label{eq:Lemma_Epsilon_Upper_bound}
 		\varepsilon \leq \min \limits_{x \in \mathcal{V}}\min \limits_{j \in 
 			[c]\setminus 
 			j^{\ast}(x)} \frac{2 \cdot\big((\Omega
 		S^{\ast})_{j^{\ast}(x)}-(\Omega 
 		S^{\ast})_{j}\big)(x)}{1+C \cdot\rho^{\ast}+\big((\Omega
 		S^{\ast})_{j^{\ast}(x)}-(\Omega 
 		S^{\ast})_{j}\big)(x)}, \quad C = 
 	\overline{h} \cdot c\cdot N,
 	\end{equation}
 	the following applies:
 	If for some index $k_0 \in \N$ it holds that $S^{k_0} \in 
 	B_{\varepsilon}(S^{\ast}) \subset B_{\varepsilon^{\ast}}(S^{\ast})$, 
 	then 
 	for all $k \geq k_0$ there exists a mapping $\xi\in 
 	\mathcal{F}_{\mathcal{V}}$ with $\xi(x) \in (0,1),\,\forall x\in\mc{V}$, such that 
 	\begin{equation}\label{eq:Lemma_Inequality_Convergence}
 		\|S^{k}(x)-S^{\ast}(x)\|_1 < \xi^{k-k_0}(x) 
 		\|S^{k_0}(x)-S^{\ast}(x)\|_1, 
 		\qquad 
 		\forall x\in\mathcal{V}.
 	\end{equation}    
 \end{theorem}
 \noindent
 \textit{Proof.} See Appendix \ref{app:convergence-main}.

\section{Experiments and Discussion}\label{sec:Exp_Dis}
In this section, we report numerical results obtained with the algorithms 
introduced in Section \ref{sec:Non_Convex_Optimization}. Details of the 
implementation and parameters settings are provided in Section 
\ref{sec:Exp-subsec-1}. Section \ref{sec:Exp-subsec-2} deals with the impact of the
nonlocal boundary conditions of system \eqref{eq:Non_Local_PDE_T0} on 
properties of averaging matrices $\Omega$ (see Section 
\ref{sec:Non_Local_PDE}), and how this effects the selection of the step size 
parameter $h>0$ in Algorithm \ref{Geometric_Explicit_Euler}. Section 
\ref{sec:Exp-subsec-3} reports results obtained by computing the assignment 
flow with Algorithm \ref{Geometric_Explicit_Euler} and different constant step 
sizes $h > 0$ using the nonlocal G-PDE parametrization 
\eqref{eq:Non_Local_PDE_T0}. In addition, we studied numerical consequences of 
nonlocal boundary conditions \eqref{S_Flow_boundary1}, \eqref{S_Flow_boundary2} 
using the maximal allowable step size \eqref{eq:h-lambda-min} 
according to Proposition \ref{prop:Numerical_Scheme}. Finally, in Section 
\ref{sec:Exp-subsec-4}, we 
compare Algorithm \ref{Geometric_Explicit_Euler} 
and the accelerated Algorithm \ref{Geometric_Two_Stage} by evaluating their 
respective convergence rates to an integral solution of the assignment flow corresponding to a stationary point of the potential \eqref{eq:S-Flow_Pot}, 
for various nonlocal connectivities.
\subsection{Implementation Details}\label{sec:Exp-subsec-1}
All evaluations were performed using the noisy image data depicted by Figure 
\ref{fig:Regularization_comparison} (b). System \eqref{eq:Non_Local_PDE} 
was initialized by $S^0 = L(\eins_{\mc{W}}) \in \mathcal{W}$ with 
$\rho=1$, as specified by \eqref{schnoerr-eq:def-Li}. 
Since the iterates 
$(S^{k})$ converge in all cases to integral solutions which are located 
at vertices on the boundary $\partial \mc{W}$ of $\mc{W}$, whereas the 
numerics is designed for evolutions \textit{on} $\mc{W}$, we applied 
the renormalization routine adopted in \cite[Section 3.3.1]{Astrom:2017ac} with 
$\varepsilon = 10^{-10}$ whenever the sequence $(S^{k})_{k\geq 0}$ came that 
close to $\partial\mc{W}$ on its path to the vertex. 

The averaging matrix $\Omega$ was assembled in two ways as specified in Section \ref{sec:Num_Ex} as items (i) and (ii), called \textit{uniform} and \textit{nonuniform} averaging in this section. In the latter case, the parameter values $\sigma_s = 1,\sigma_p 
= 5$ in were chosen \eqref{eq:nonuniform-params}, as for the experiments reported in Section \ref{sec:Num_Ex}.
The iterative algorithms were terminated at step $k$ when the averaged gradient norm
\begin{equation}\label{eq:Iter_Crit_Av_grad}
	\epsilon_{k} = \frac{1}{n}\sum \limits_{x \in \mathcal{V}} \| 
	R_{S^k(x)}(\Omega S^k(x))\| \leq \epsilon
\end{equation}  
reached a threshold $\epsilon$ which when chosen sufficiently small to satisfy 
bound \eqref{eq:Lemma_Epsilon_Upper_bound} that guarantees a linear 
convergence rate as specified in Theorem \ref{thm:Existance-Epsilon}. 

We point out that during the evaluation and discussion of realized 
experiments our focus was \textit{not} on assessing a comparison of
computational speed in term of absolute runtimes, but on the numerical
behavior of the proposed schemes with regard to number of iterations required 
to solve system  
\eqref{eq:Non_Local_PDE_T0} and in terms of the labeling performance. Thus, we 
did not confine ourselves to impose any restriction on the minimum time step 
size and the maximum
number of iterations and instead appropriately adjusted the parameter 
\eqref{eq:Iter_Crit_Av_grad} to stop the algorithm when a stationary point at 
the boundary of $\mathcal{W}$ was reached.

Since 
$S^{\ast}$ is unknown, we can not directly 
access the 
exact bound in \eqref{eq:Lemma_Epsilon_Upper_bound} beforehand and therefore it 
is not evident how to set $\epsilon$ in practice. However, based on experimental 
evidence, setting the termination 
threshold by $\epsilon = 10^{-7}$ in \eqref{eq:Iter_Crit_Av_grad} serves as 
good estimate, see Figures \ref{fig:Conv_Rates_Algorithm1} and 
\ref{fig:Conv_Rates_Algorithm_2}.   
Algorithm \ref{alg:Step_Size_Selection} requires to specify two parameters 
$c_{1}, c_{2}$ (see line 3). We empirically found that using $c_1 = 0.4, c_2 = 0.95$ is a good choice that we used in all experiments.  

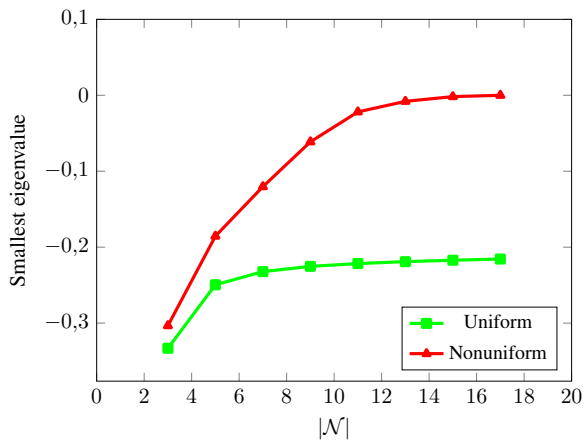
\begin{figure}
	\centering
	\scalebox{0.75}{\begin{tikzpicture}
	\begin{axis}[
		every other node near coord/.append style={font=\tiny},
		legend columns=1,
		legend style={row sep=0.1cm},
		height=8cm,
		width = 10cm,
		legend entries={{\small Uniform},{\small Nonuniform}},
		legend pos= south east,
		mark repeat=1,
		xlabel = $|\mathcal{N}|$,
		ylabel = Smallest eigenvalue,
		xmin=0,
		xmax=20,
		ymax = 0.1,
		scaled ticks=true,
		/pgf/number format/.cd,
		use comma,
		1000 sep={}
		]
		\addplot[line join=round,mark = square*,color = green,ultra thick]
		coordinates {(3,-0.3331411) (5,-0.24957526) (7,-0.23223186) 
		(9,-0.2253046) (11,-0.22157785) (13,-0.21910794) (15,-0.21720773) 
		(17,-0.2156241)};	
		\addplot[line join=round,mark = triangle*,color = red,ultra thick]
		coordinates {(3,-0.3033669) (5,-0.18550295) 
		(7,-0.1203382) (9,-0.06131989) (11,-0.02190208) (13,-0.00800925)
		(15,-0.00177068) (17,-3.7908554e-05)};	
		]
	\end{axis}
\end{tikzpicture}}
	\captionsetup{font=footnotesize}
	\caption{Effect of imposing nonlocal boundary conditions. The green 
	(\protect\markerone) and the 
		red (\protect\markertwo) curves 
		plot the smallest eigenvalues $\lambda_{\min}(\Omega)$ of the parameter matrix 
		\eqref{eq:Avaraging_Matrix} for uniform and nonuniform averaging, respectively, and for different neighborhood sizes $|\mathcal{N}|$. Choosing 
		larger neighborhoods \eqref{eq:def-neighborhoods} increases the 
		smallest eigenvalue and consequently, by \eqref{eq:h-lambda-min}, enables to choose bigger step sizes in 
		Algorithm \ref{Geometric_Explicit_Euler} that achieve the monotone decrease property \eqref{eq:Dec_Seq_Euler}.}
	\label{fig:Step_Size_Algorithm1}
\end{figure}   

\begin{figure}[h]
	\centering
	\scalebox{0.75}{\begin{tikzpicture}
	\begin{axis}[
		every other node near coord/.append style={font=\tiny},
		legend columns=1,
		legend style={row sep=0.1cm},
		height=8cm,
		width = 10cm,
		legend entries={{\small $|\mathcal{N}| = 5\times 5$},{\small 
		$|\mathcal{N}| = 
			7\times 7$},{\small $|\mathcal{N}| = 
			9\times 9$}},
		legend pos= north east,
		mark repeat=1,
		xlabel = Step Size $h$,
		ylabel = Iterations,
		xmin=0,
		xmax=50,
		xmode=log,
		ymax = 400,
		ymin=8,
		scaled ticks=true,
		/pgf/number format/.cd,
		use comma,
		1000 sep={}
		]
		\addplot[line join=round,mark = rectangle*,color = Layer_4,ultra thick]
	coordinates {((1,248) (2,169) 
	(3,128) (4,91) (5,76) 
	(7,96) (10,52) (12,61) (15,49) 
	(20,43) (25,37) (30,48) (40,48) (50,35)};
		\addplot[line join=round,mark = triangle*,color = Layer_1,ultra thick]
		coordinates {((1,231) (2,182) 
			(3,135) (4,100) (5,70) 
			(7,79) (10,81) (12,86) (15,50) 
			(20,45) (25,60) (30,38) (40,48) (50,43)};
		\addplot[line join=round,mark = diamond*,color = green,Layer_6,ultra 
		thick]
coordinates {(1,329) (2,184) 
		(3,156) (4,96) (5,84) 
		(7,62) (10,58) (12,81) (15,50) 
		(20,68) (25,56) (30,37) (40,48) (50,44)};
		\draw[dashed,color = Layer_4] (5.4,0) -- (5.4,399);
		\draw[dashed,color = Layer_1] (8.3,0) -- (8.3,399);
		\draw[dashed,color = Layer_6] (16.6,0) -- (16.6,399);
		]
	\end{axis}
\end{tikzpicture}}
	\hspace{1cm}
	\scalebox{0.75}{\begin{tikzpicture}
	\begin{axis}[
		every other node near coord/.append style={font=\tiny},
		legend columns=1,
		legend style={row sep=0.1cm},
		height=8cm,
		width = 10cm,
		legend entries={{\small $|\mathcal{N}| = 5\times 5$},{\small 
		$|\mathcal{N}| = 
		7\times 7$},{\small $|\mathcal{N}| = 
		9\times 9$}},
		legend pos= north west,
		mark repeat=1,
		xlabel = Step Size $h$,
		ylabel = Labeling Error in \%,
		xmin=0,
		xmax=50,
		xmode=log,
		ymax = 11,
		ymin=7,
		scaled ticks=true,
		/pgf/number format/.cd,
		use comma,
		1000 sep={}
		]
		\addplot[line join=round,mark = square*,color = Layer_4,ultra thick]
		coordinates {(1,5116/65536*100) (2,5098/65536*100) (3,5069/65536*100) 
			(4,5075/65536*100) (5,5104/65536*100) (7,5170/65536*100) 
			(10,5285/65536*100) (12,5342/65536*100) (15,5481/65536*100) 
			(20,5594/65536*100) 
			(25,5786/65536*100) (30,5979/65536*100) (35,6072/65536*100) 
			(40,6213/65536*100) 
			(50,6582/65536*100)};	
		\addplot[line join=round,mark = triangle*,color = Layer_1,ultra thick]
		coordinates {(1,5038/65536*100) (2,5043/65536*100) (3,5044/65536*100) 
			(4,5035/65536*100) (5,5027/65536*100) (7,5062/65536*100) 
			(10,5209/65536*100) (12,5214/65536*100) (15,5323/65536*100) 
			(20,5472/65536*100) (25,5623/65536*100) (30,5806/65536*100) 
			(35,5929/65536*100) 
			(40,6046/65536*100) 
			(50,6387/65536*100)};
		\addplot[line join=round,mark = diamond*,color = Layer_6,ultra thick]
		coordinates {(1,4970/65536*100) (2,4936/65536*100) (3,4920/65536*100) 
			(4,4886/65536*100) (5,4877/65536*100) (7,4852/65536*100) 
			(10,4879/65536*100) (12,4897/65536*100) (15,4911/65536*100) 
			(20,5006/65536*100) (25,5075/65536*100) (30,5171/65536*100) 
			(35,5311/65536*100) 
			(40,5417/65536*100) 
			(50,5584/65536*100)};
		\draw[dashed,color = Layer_4] (5.4,0) -- (5.4,11);
		\draw[dashed,color = Layer_1] (8.3,0) -- (8.3,11);
		\draw[dashed,color = Layer_6] (16.6,0) -- (16.6,11);
		]
	\end{axis}
\end{tikzpicture}}
	\captionsetup{font=footnotesize}
	\caption{Effects of selecting the step size $h$ in Algorithm 
	\ref{Geometric_Explicit_Euler} for various 
	neighborhood sizes $|\mathcal{N}|$. Dashed vertical lines indicate the step size 
	upper bound $\frac{1}{|\lambda_{\min}(\Omega)|}$ that guarantees the monotone decrease property (Proposition \ref{prop:Numerical_Scheme}). \textbf{Left:} Number of iterations required to satisfy the termination criterion \eqref{eq:Iter_Crit_Av_grad}. Larger step sizes decrease the number of iterations but yield unreliable numerical computation when $h$ exceeds the upper bound (see text). 
\textbf{Right:} Pixel-wise labeling error compared to ground truth. Labeling accuracy quickly deteriorates when $h$ exceeds the upper bound.}
	\label{fig:Stopping_Criterium_Algorithm1}
\end{figure}
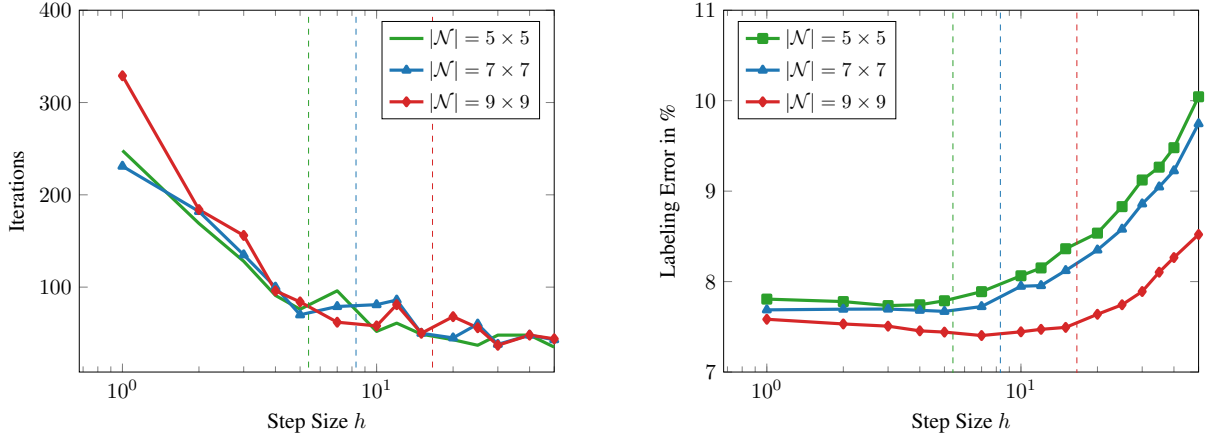 

\subsection{Step Size Selection}\label{sec:Exp-subsec-2}
This section reports results of several experiments that highlight aspects of imposing nonlocal boundary conditions 
\eqref{S_Flow_boundary1}, \eqref{S_Flow_boundary2}  and their influence on the selection of step sizes in Algorithms 
\ref{Geometric_Explicit_Euler} and 
\ref{Geometric_Two_Stage}.
 
To demonstrate these effects we used two different  parameter matrices $\Omega$ defined in accordance with Lemma 
\ref{Help_Lemma}, with $\Theta,\alpha$ given as in Section \ref{sec:Num_Ex}, called  
\textit{uniform} and \textit{nonuniform} averaging,  respectively. To 
access the maximal bound \eqref{eq:h-lambda-min} for the step size $h>0$, as derived in Proposition 
\ref{prop:Numerical_Scheme} in order to achieve the monotone decrease property \eqref{eq:Dec_Seq_Euler}, we 
directly approximated the exact smallest eigenvalue 
$\lambda_{\min}(\Omega)$ 
using available software \cite{Lehoucq:1998aa}. 

Figure \ref{fig:Step_Size_Algorithm1} displays values of the smallest eigenvalue for uniform and nonuniform averaging, respectively, and different sizes of the nonlocal neighborhoods \eqref{eq:def-neighborhoods}: Increasing the size $|\mathcal{N}|$  
decreases the value of $\lambda_{\text{min}}(\Omega)$ and consequently, by virtue of relation $h \geq \frac{1}{|\lambda_{\min}(\Omega)|}$ 
in Proposition \ref{prop:Numerical_Scheme}, to a larger upper bound for setting the 
step size $h$ in Algorithm \ref{Geometric_Explicit_Euler}. This confirms our observation and statement formulated as Remark \ref{rem:-Step_Size}. 

In practice, however, it is too expensive to compute $\lambda_{\min}$ numerically for choosing the step size $h$. Figure \ref{fig:Stopping_Criterium_Algorithm1} shows for three sizes of neighborhoods $|\mc{N}|$ and for step sizes $h$ \textit{smaller and larger} than the upper bound \eqref{eq:h-lambda-min} indicated by dashed vertical lines, 
\begin{enumerate}[(i)]
\item the number of iterations required to reach the termination criterion \eqref{eq:Iter_Crit_Av_grad} (Figure \ref{fig:Stopping_Criterium_Algorithm1}, left panel);
\item the labeling accuracy compared to ground truth (Figure \ref{fig:Stopping_Criterium_Algorithm1}, right panel).
\end{enumerate}
The results show that the bound \eqref{eq:h-lambda-min} should be considered as a hard constraint indeed: Increasing the step size $h$ up to this bound (cf.~Fig.~\ref{fig:Stopping_Criterium_Algorithm1}, left panel) decreases the required number of iterations, as to be expected. But exceeding the bound yields unreliable computation, possibly caused by a too small DC decomposition parameter $\gamma < |\lambda_{\min}(\Omega)|$ which compromises the convexity and hence convergence of the auxiliary optimization problems in Algorithm \ref{Geometric_Explicit_Euler}, line  \eqref{eq:def-next-iterate-alg1}). Likewise, Fig.~\ref{fig:Stopping_Criterium_Algorithm1}, right panel, shows that labelings quickly become inaccurate once the step size exceeds the upper bound. Figure \ref{fig:Regulrization_Plot_Veersus_Stepsizes} visualizes examples.

Overall, these results show that a wide range of save choices of the step size 
parameter $h$ exists, and that choosing the ``best'' value depends on how 
accurate $\lambda_{\min}(\Omega)$ is known beforehand.

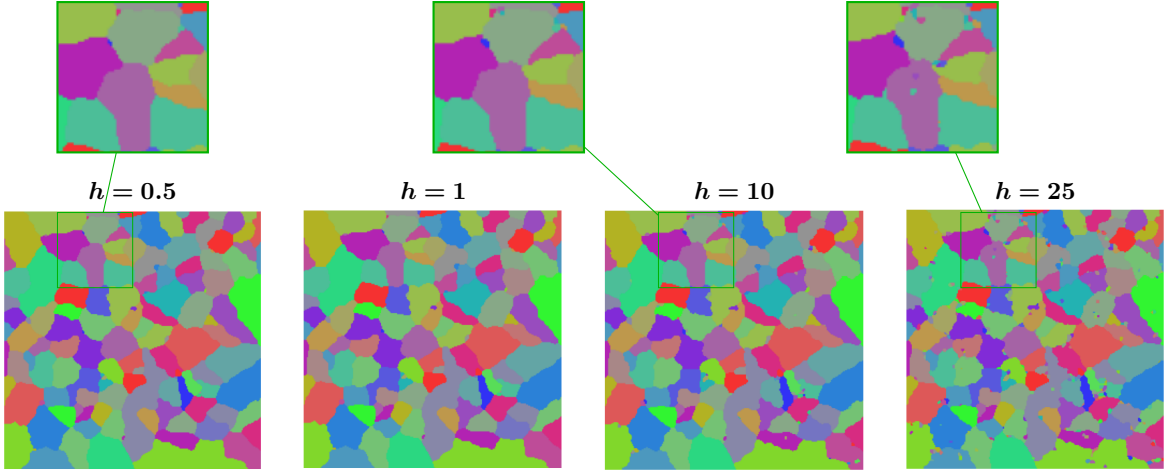
\begin{figure}
	\begin{tikzpicture}[spy using outlines={rectangle, magnification=2, 
			size=1cm, 
			connect spies}]
		\centering
		\node[scale=0.7] at (2,-9.5) 
		{\includegraphics[width=5cm,height=5cm]{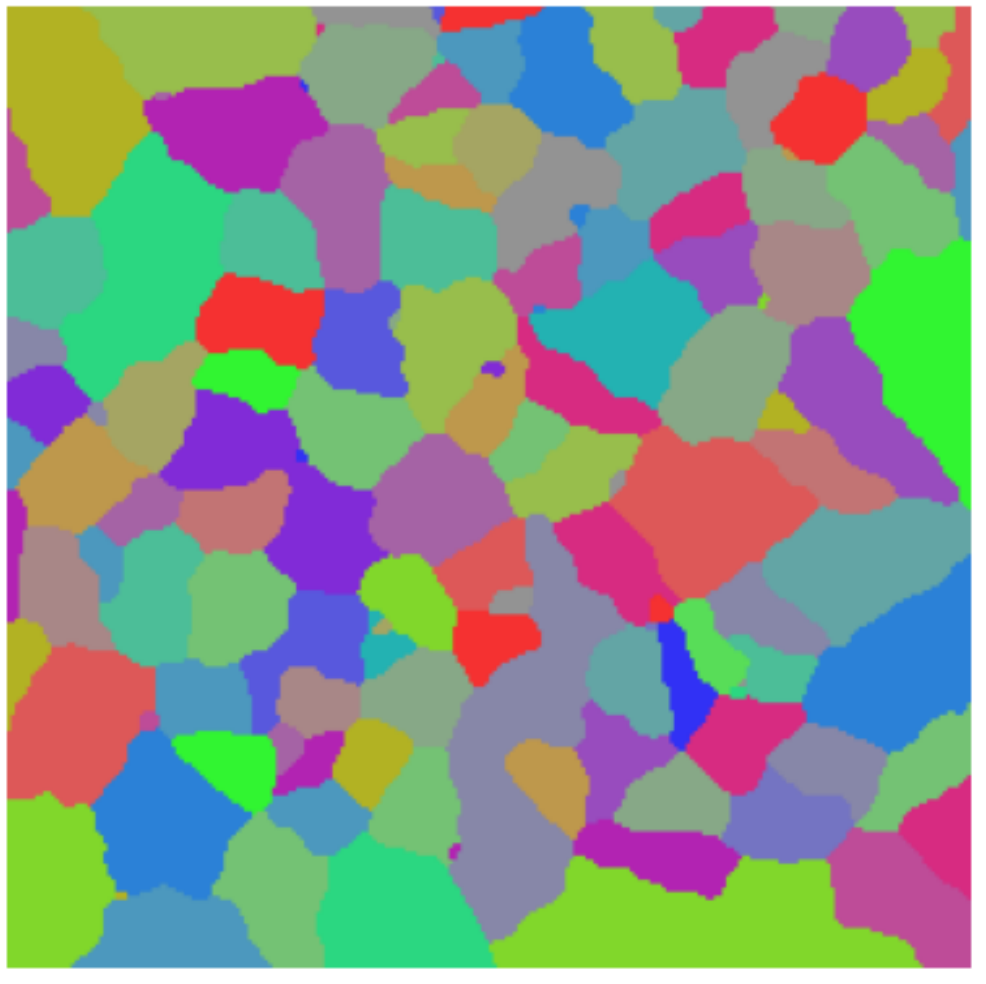}};
		\node[scale=0.7] at (6,-9.5) 		
		{\includegraphics[width=5cm,height=5cm]{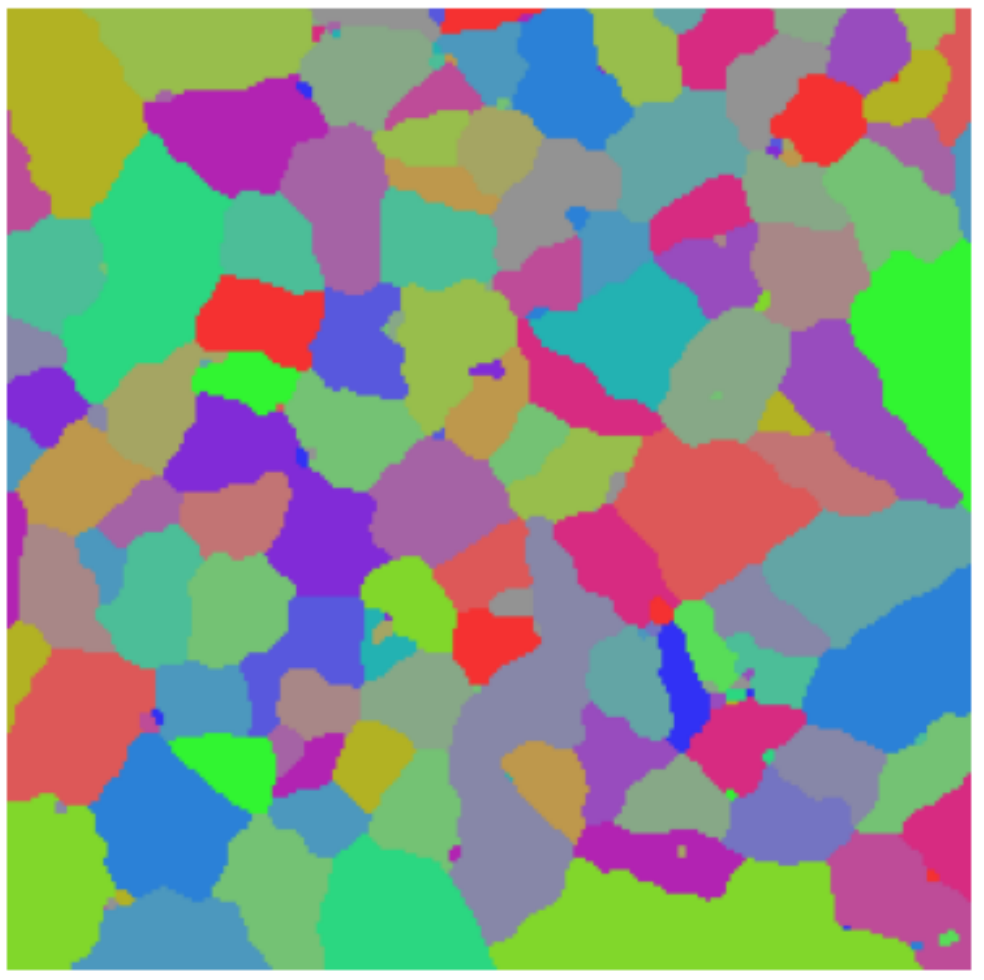}};
		\node[scale=0.7] at (10,-9.5) 
		{\includegraphics[width=5cm,height=5cm]{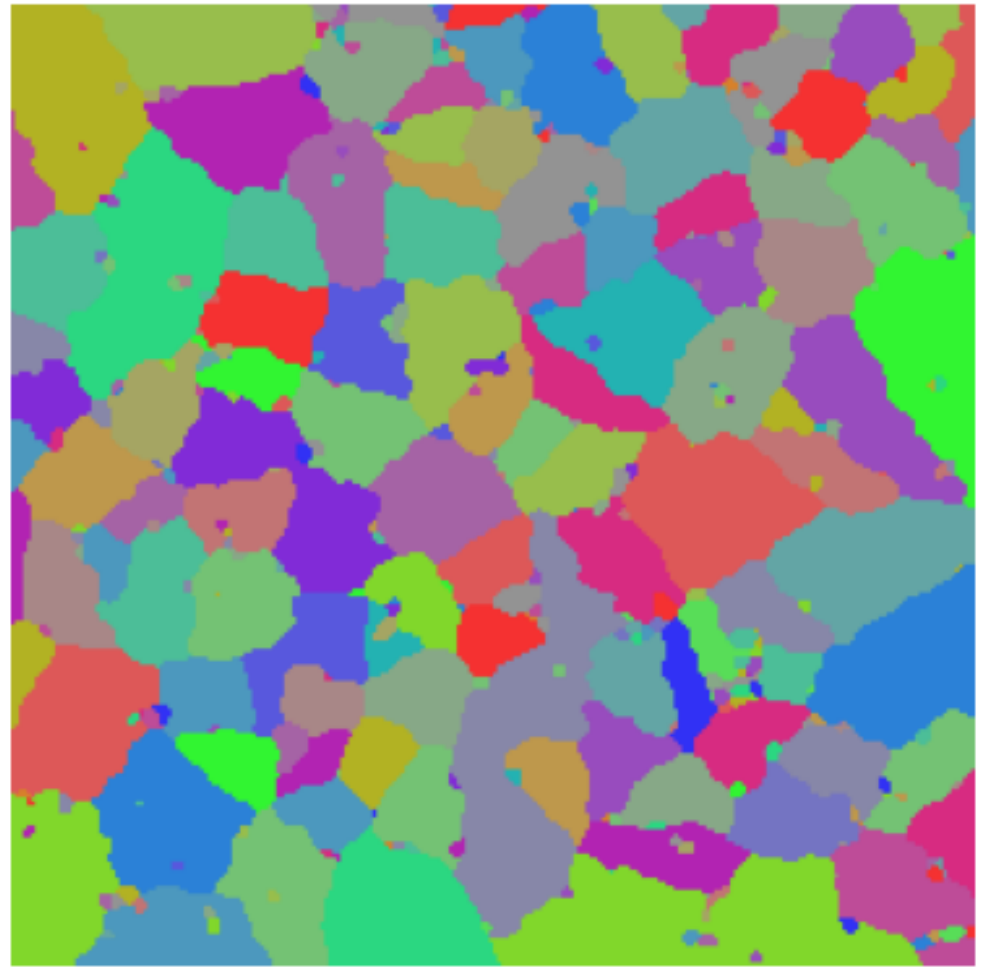}};
		\node[scale=0.7] at (-2,-9.5) 
		{\includegraphics[width=5cm,height=5cm]{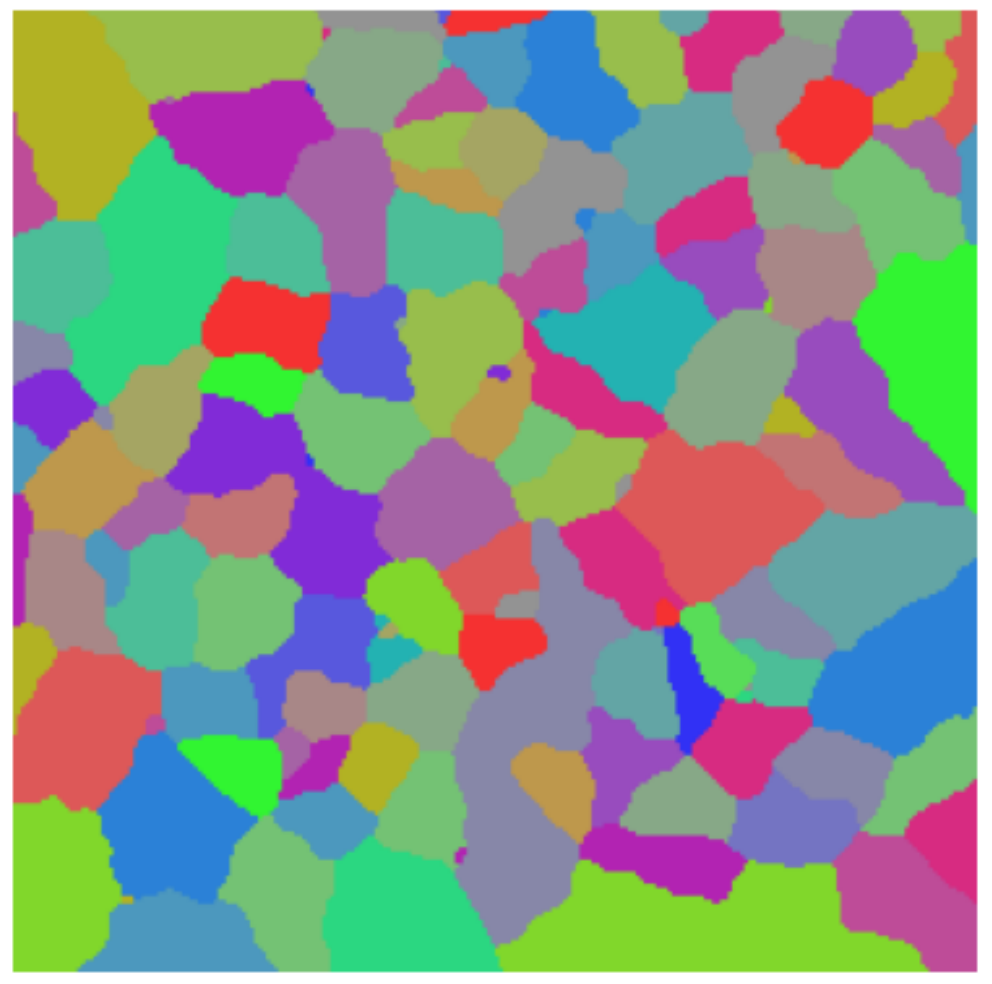}};
		\node[scale=0.85] at (2,-7.5) {$\boldsymbol{h = 1$}};
		\node[scale=0.85] at (6,-7.5) {$\boldsymbol{h = 10$}};
		\node[scale=0.85] at (10,-7.5) {$\boldsymbol{h = 25$}};
		\node[scale=0.85] at (-2,-7.5) {$\boldsymbol{h = 0.5}$};
		\begin{scope}
			\spy[green!70!black,size=2cm] on (-2.5,-8.3) in node [fill=white] 
			at 
			(-2,-6);
		\end{scope}
		\begin{scope}
			\spy[green!70!black,size=2cm] on (5.49,-8.3) in node [fill=white] 
			at 
			(3,-6);
		\end{scope}
		\begin{scope}
			\spy[green!70!black,size=2cm] on (9.51,-8.3) in node 
			[fill=white] 
			at (8.5,-6);
		\end{scope}
	\end{tikzpicture}
	\captionsetup{font=footnotesize}
	\caption{Visualization of 
	regularization 
	impacts when increasing the step 
		size $h$ corresponding to the results in Figure 
		\ref{fig:Stopping_Criterium_Algorithm1}. Labeling results for various step sizes and the neighborhood size $|\mc{N}|=9\times 9$. Conforming to Figure \ref{fig:Stopping_Criterium_Algorithm1}, right panel, labeling accuracy quickly deteriorates once $h$ exceeds the upper bound \eqref{eq:h-lambda-min} (rightmost panel).}
	\label{fig:Regulrization_Plot_Veersus_Stepsizes}
\end{figure}

\begin{figure}
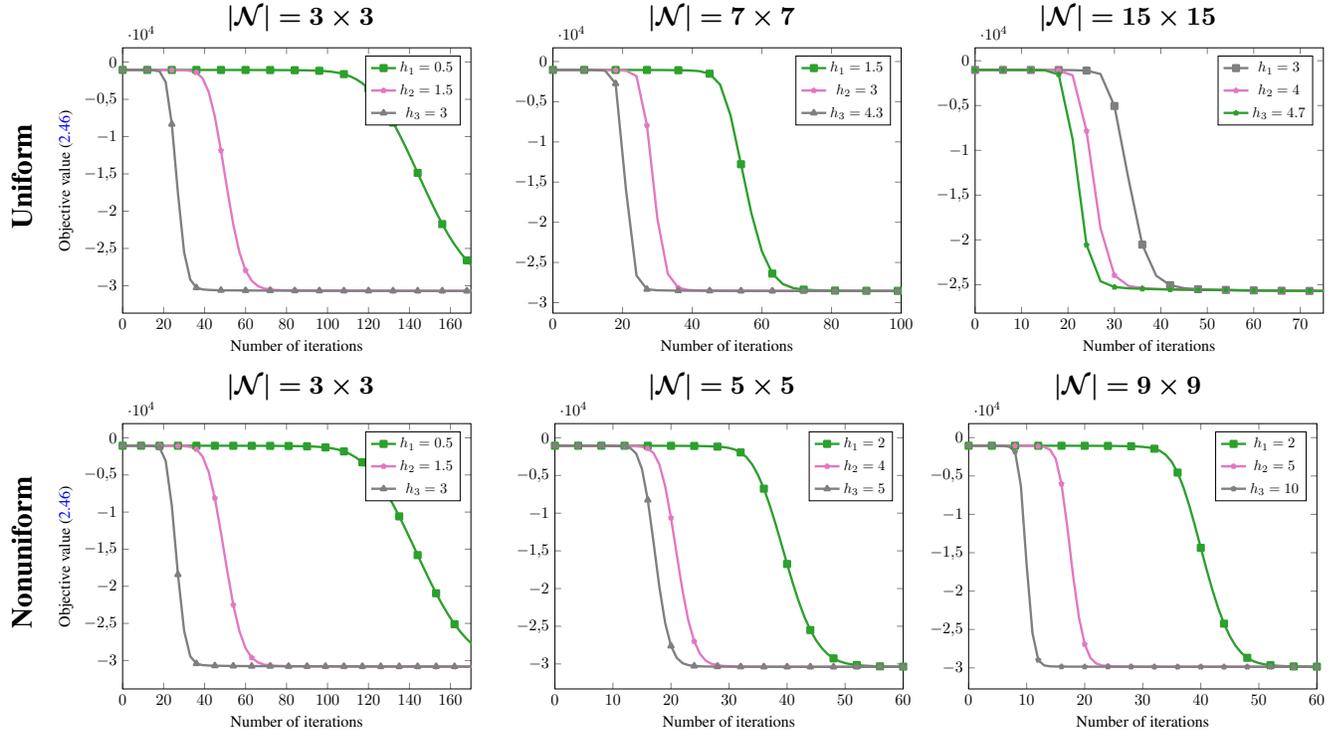

	\centering
	\begin{tikzpicture}
	\node[scale=0.9] at (0.5,2.3) {\boldsymbol{$|\mathcal{N}| = 3\times 
	3$}}; 
	\node[scale=0.9] at (6.1,2.3) {\boldsymbol{$|\mathcal{N}| = 7\times 
	7$}}; 
	\node[scale=0.9] at (11.5,2.3) {\boldsymbol{$|\mathcal{N}| = 
	15\times 15$}};
	\node[rotate=90] at (-3.2,0.2) {\textbf{Uniform}}; 
	\node[rotate=90] at (-3.2,-4.8) {\textbf{Nonuniform}};
	\node[scale=0.9] at (0.5,-2.6) {\boldsymbol{$|\mathcal{N}| = 
	3\times 3$}}; 
	\node[scale=0.9] at (6.1,-2.6) {\boldsymbol{$|\mathcal{N}| = 
	5\times 5$}}; 
	\node[scale=0.9] at (11.5,-2.6) {\boldsymbol{$|\mathcal{N}| = 
	9\times 9$}}; 
	\node[scale = 0.55] at 
	(0,0){\input{./Graphics/First_Order_Step_Size_Uniform_3.tex}};
	\node[scale=0.55] at (6,0) {\input{./Graphics/First_Order_Step_Size.tex}};
	\node[scale = 0.55] at (11.5,0)  
	{\input{./Graphics/First_Order_Step_Size15.tex}};
	\node[scale = 0.55] at (0,-5) 
	{\input{./Graphics/First_Order_Step_Size_Non_Local_3.tex}};
	\node[scale = 0.55] at (6,-5) 
	{\input{./Graphics/First_Order_Step_Size_Non_Local_5.tex}};
	\node[scale = 0.55] at (11.5,-5) 
	{\input{./Graphics/First_Order_Step_Size_Non_Local_9.tex}};
	\end{tikzpicture}
	\captionsetup{font=footnotesize}
	\caption{Minimization of the nonconvex potential \eqref{eq:S-Flow_Pot} by Algorithm 
		\ref{Geometric_Explicit_Euler} for various neigborhood sizes 
		$|\mc{N}|$, for uniform averaging (top row) and nonuniform averaging 
		(bottom row), and for three constant step sizes $0<h_{1}<h_{2}<h_{3}$, 
		where in each experiment $h_{3}$ was chosen smaller than the upper 
		bound discussed in Section \ref{sec:Exp-subsec-2} that guarantees a 
		monotonously decreasing sequence of potential values (Proposition 
		\ref{prop:Numerical_Scheme}). All experiments illustrate this property 
		and that the largest admissible step size $h_{3}$ is most effective. 
		\protect\footnotemark  } 
	\label{fig:Objective_Evaluation_Alg1}
\end{figure}

\footnotetext[1]{ The plotted curves in the 
figure illustrate  
progressing objective 
values of $J(S)$ stagnating near a \textit{local minimizer} $S^{\ast}$. In 
particular, the depicted stagnating value is \textit{not} the  
lower bound of $J(S)$ 
on $\mathcal{W}$ that is given by $J(S^{\ast}) = -\frac{|\mc{V}|}{2}$ and 
attained at the  \textit{global} minimizer $S^{\ast}$, that is always a 
constant labeling and therefore of no interest, see 
Remark \ref{rem:constant_labeling}}.
\subsection{First-Order Optimization}\label{sec:Exp-subsec-3}

This section is devoted to the evaluation of Algorithm 
\ref{Geometric_Explicit_Euler}. We examine how effectively this algorithm 
converges to an integral solution (labeling) for both uniform and nonuniform 
averaging, for different sizes of nonlocal neighborhoods $|\mc{N}|$, and for 
different admissible step sizes $h$ based on the insights gained in Section 
\ref{sec:Exp-subsec-2}: the largest admissible step size increases with the 
neighborhood size $|\mc{N}|$ and when using nonuniform, rather than uniform, 
averaging.  

Figure \ref{fig:Objective_Evaluation_Alg1} displays the corresponding values of 
the objective function \eqref{eq:S-Flow_Pot} as a function of the iteration 
counter. We observe that this first-order algorithm minimizes quite effectively 
the non-convex objective function during the first few dozens of iterations. 

Figure \ref{fig:Conv_Rates_Algorithm1} displays the same information, this time in term of the function $k\mapsto \frac{1}{n}\|S^k-S^{\ast}\|_1$, however. We observe two basic facts: (i) Due to using admissible step sizes, the sequences $(S^{k})_{k\geq 0}$ always converge to the integral solution $S^{\ast}$. (ii) In agreement with Theorem 
\ref{thm:Existance-Epsilon}, the order of convergence increases whenever the 
sequence $(S^{k})_{k\geq 0}$ reaches the basin of attraction.

\begin{figure}
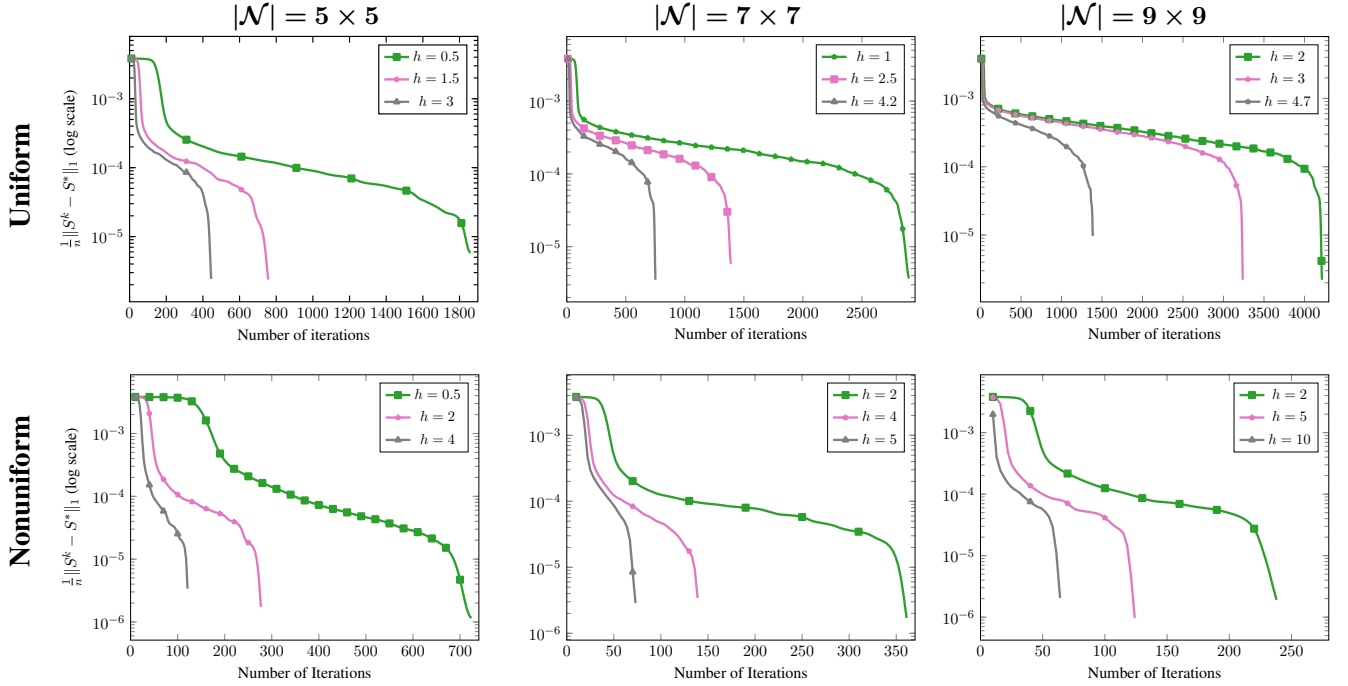

\begin{tikzpicture}
	\node[scale=0.55] at (0,0) 
	{\input{./Graphics/Convergence_Rate_First_Order_Uniform_7.tex}};
		\node[scale=0.55] at (6,0) 
	{\input{./Graphics/Convergence_Rate_First_Order_Nonlocal_5.tex}};
	\node[scale=0.55] at (11.5,0) 
	{\input{./Graphics/Convergence_Rate_First_Order_Nonlocal_9.tex}};
	\node[scale=0.9] at (0.5,2.3) {\boldsymbol{$|\mathcal{N}| = 5\times 
	5$}}; 
	\node[scale=0.9] at (6.1,2.3) {\boldsymbol{$|\mathcal{N}| = 7\times 
	7$}}; 
	\node[scale=0.9] at (11.5,2.3) {\boldsymbol{$|\mathcal{N}| = 
	9\times 9$}};
	\node[scale=0.55] at (0,-4.5) 
	{\input{./Graphics/Convergence_Rates_First_Order_Non_Uniform_3.tex}};
	\node[scale=0.55] at (6,-4.5) 
	{\begin{tikzpicture}
	\begin{axis}[
		every other node near coord/.append style={font=\tiny},
		legend columns=1,
		legend style={row sep=0.1cm},
		height=8cm,
		width = 10cm,
		legend entries={{\small $h = 2$},{\small $h = 4$},{\small 
		$h =5$}},
		legend pos=north east,
		mark repeat=20,
		ymode = log,
		xlabel =Number of Iterations,
		xmin=0,
		xmax=370,
		scaled ticks=true,
		/pgf/number format/.cd,
		use comma,
		1000 sep={}
		]
		\addplot[line join=round,mark = square*,color = Layer_4,ultra thick]
		coordinates {(10,0.003815493249751695)
			(13,0.003808451233083787)
			(16,0.0037990294818647597)
			(19,0.003786024190211133)
			(22,0.0037672595731389224)
			(25,0.003738327460850152)
			(28,0.0036886926605306956)
			(31,0.003587698046865216)
			(34,0.0033511820842354657)
			(37,0.002879066979977192)
			(40,0.0022067559255688593)
			(43,0.0014906773986139888)
			(46,0.0009060077553627329)
			(49,0.000563805623416812)
			(52,0.00040636710876282224)
			(55,0.00033020203683522626)
			(58,0.0002885900676082674)
			(61,0.0002598091594710298)
			(64,0.00023675954452174417)
			(67,0.00021786891252136765)
			(70,0.00020217559766491037)
			(73,0.00018880486115656033)
			(76,0.0001771558873251561)
			(79,0.00016697516185631075)
			(82,0.00015824339030499435)
			(85,0.00015083137177861482)
			(88,0.00014443615953849858)
			(91,0.00013865766000483744)
			(94,0.00013336681701355446)
			(97,0.00012896744101526048)
			(100,0.00012567306169468108)
			(103,0.00012301109308902135)
			(106,0.00012053271958534043)
			(109,0.00011801334179279735)
			(112,0.00011543597153792145)
			(115,0.00011289884678624906)
			(118,0.0001104033362304939)
			(121,0.0001079160817452736)
			(124,0.00010549259465130744)
			(127,0.00010320740253976255)
			(130,0.00010109937534606138)
			(133,9.924213416119842e-05)
			(136,9.768485220080108e-05)
			(139,9.629160436118889e-05)
			(142,9.493105241425125e-05)
			(145,9.364533913937875e-05)
			(148,9.253480507132924e-05)
			(151,9.158538622195391e-05)
			(154,9.070157894027963e-05)
			(157,8.981960939105775e-05)
			(160,8.890212821274074e-05)
			(163,8.787638714424031e-05)
			(166,8.667944204829651e-05)
			(169,8.539214107755686e-05)
			(172,8.425153800200771e-05)
			(175,8.337199165457249e-05)
			(178,8.265136222701834e-05)
			(181,8.196949132137345e-05)
			(184,8.127166887104784e-05)
			(187,8.059738401161149e-05)
			(190,7.999982731107952e-05)
			(193,7.94263748771318e-05)
			(196,7.877203918309486e-05)
			(199,7.795882527112826e-05)
			(202,7.692431879759903e-05)
			(205,7.560956036197969e-05)
			(208,7.398447243184094e-05)
			(211,7.210482361114977e-05)
			(214,7.011472505356004e-05)
			(217,6.812755426491911e-05)
			(220,6.62032645907304e-05)
			(223,6.448802586996402e-05)
			(226,6.323184433814015e-05)
			(229,6.251058645718026e-05)
			(232,6.207664020028243e-05)
			(235,6.168123241079965e-05)
			(238,6.121935296593851e-05)
			(241,6.0642057974295455e-05)
			(244,5.990103743916334e-05)
			(247,5.893020810258319e-05)
			(250,5.7638327231887725e-05)
			(253,5.592051443325625e-05)
			(256,5.374837202113628e-05)
			(259,5.143448171913095e-05)
			(262,4.9654681662395695e-05)
			(265,4.8542503146533786e-05)
			(268,4.769328887767317e-05)
			(271,4.685366207547112e-05)
			(274,4.59905757191258e-05)
			(277,4.513039115685462e-05)
			(280,4.4210501213306954e-05)
			(283,4.308125685352993e-05)
			(286,4.164384250763898e-05)
			(289,4.003872666350548e-05)
			(292,3.8680788441382876e-05)
			(295,3.7700362224798626e-05)
			(298,3.686960764101867e-05)
			(301,3.61362242711077e-05)
			(304,3.554980260677936e-05)
			(307,3.506785889785153e-05)
			(310,3.4579716278860445e-05)
			(313,3.3992997698158284e-05)
			(316,3.324049765383523e-05)
			(319,3.226012837334855e-05)
			(322,3.1002758793023375e-05)
			(325,2.9496466013238832e-05)
			(328,2.794985192904207e-05)
			(331,2.6647427620884002e-05)
			(334,2.5551570000390976e-05)
			(337,2.4358543105804988e-05)
			(340,2.2827498489949402e-05)
			(343,2.079211043964537e-05)
			(346,1.8122323362017467e-05)
			(349,1.4779076031826118e-05)
			(352,1.0894880690479954e-05)
			(355,6.885945341898669e-06)
			(358,3.5948374195736846e-06)
			(361,1.7017171449558679e-06)
		};
	\addplot[line join=round,mark = star,color = Layer_10,ultra thick]
	coordinates {(10,0.003784404724994949)
		(13,0.003738138562420523)
		(16,0.0036070203850190104)
		(19,0.003057929213569691)
		(22,0.0018330986037681878)
		(25,0.000755734478532101)
		(28,0.00038707191562109057)
		(31,0.0002919910657946821)
		(34,0.00024599370696611246)
		(37,0.00021003372901537063)
		(40,0.00018235093084557728)
		(43,0.0001610461855036006)
		(46,0.0001417926679782954)
		(49,0.000126484745262514)
		(52,0.00011713943156612857)
		(55,0.00011107486256684323)
		(58,0.00010548546007669866)
		(61,9.853767555505833e-05)
		(64,9.341342127755303e-05)
		(67,8.881830496671515e-05)
		(70,8.368885198088787e-05)
		(73,7.9075698853464e-05)
		(76,7.384758422620734e-05)
		(79,6.810769043583328e-05)
		(82,6.343133933633176e-05)
		(85,6.0675385860263017e-05)
		(88,5.8182014938323256e-05)
		(91,5.480195742113297e-05)
		(94,5.122329042241162e-05)
		(97,4.87766973402737e-05)
		(100,4.7044166504826676e-05)
		(103,4.5097996739074854e-05)
		(106,4.2296416654243276e-05)
		(109,3.8977122767000724e-05)
		(112,3.602101740751186e-05)
		(115,3.305424802213694e-05)
		(118,2.974410575186709e-05)
		(121,2.6300712509680643e-05)
		(124,2.258277328874444e-05)
		(127,1.95535651159314e-05)
		(130,1.753247497137416e-05)
		(133,1.4616386686028897e-05)
		(136,9.381614000036104e-06)
		(139,3.412617761131684e-06)
	};
		\addplot[line join=round,mark = triangle*,color = Layer_12,ultra thick]
		coordinates {(10,0.003753045547583711)
			(13,0.003613848079787167)
			(16,0.002885770419583213)
			(19,0.001359473527833549)
			(22,0.0004798889650970362)
			(25,0.00029783960228609843)
			(28,0.0002365188795085399)
			(31,0.0001993857947550156)
			(34,0.0001717817056624014)
			(37,0.000149617531933155)
			(40,0.00013155771154649008)
			(43,0.00011547427609072335)
			(46,0.00010231050405133108)
			(49,9.007957168796821e-05)
			(52,7.852915015736454e-05)
			(55,6.70528763992759e-05)
			(58,5.7830495752814777e-05)
			(61,5.045119734590772e-05)
			(64,4.0609072953641856e-05)
			(67,2.629304879519688e-05)
			(70,8.467553569624586e-06)
			(73,2.861807158436924e-06)
		};]
\end{axis}
\end{tikzpicture}};
	\node[scale=0.55] at (11.5,-4.5) 
	{\begin{tikzpicture}
	\begin{axis}[
		every other node near coord/.append style={font=\tiny},
		legend columns=1,
		legend style={row sep=0.1cm},
		height=8cm,
		width = 10cm,
		legend entries={{\small $h = 2$},{\small $h = 5$},{\small $h =10$}},
		legend pos=north east,
		mark repeat=10,
		ymode = log,
		ylabel =,
		xlabel =Number of Iterations,
		xmin=0,
		xmax=280,
		scaled ticks=true,
		/pgf/number format/.cd,
		use comma,
		1000 sep={}
		]
		\addplot[line join=round,mark = square*,color = Layer_4,ultra thick]
		coordinates {(10,0.0038157770398366355)
			(13,0.0038088555574512964)
			(16,0.0037996207251412687)
			(19,0.0037869222788330724)
			(22,0.0037687033042171935)
			(25,0.0037408620388355834)
			(28,0.003693823348257977)
			(31,0.003600182327743038)
			(34,0.003380585542318803)
			(37,0.0029229078869960168)
			(40,0.0022521392250063523)
			(43,0.001527173360531471)
			(46,0.0009263328531868269)
			(49,0.0005724209481245406)
			(52,0.0004124307147592783)
			(55,0.0003387434117246038)
			(58,0.0002993984500795678)
			(61,0.00027199516268952085)
			(64,0.00024998941433811676)
			(67,0.00023168930520231607)
			(70,0.0002160614703851955)
			(73,0.0002022550389942302)
			(76,0.0001896299317667677)
			(79,0.00017809019739109687)
			(82,0.00016763998951413426)
			(85,0.00015841287965724846)
			(88,0.0001501739927156065)
			(91,0.00014260102918232166)
			(94,0.00013584122994144893)
			(97,0.00013006723203419084)
			(100,0.00012508375850025448)
			(103,0.00012066181799642308)
			(106,0.00011663541960350412)
			(109,0.00011283935318831301)
			(112,0.0001090576266717182)
			(115,0.00010512640841562434)
			(118,0.00010106746882726329)
			(121,9.709579474833247e-05)
			(124,9.338629339610249e-05)
			(127,8.982838875479517e-05)
			(130,8.62444683344919e-05)
			(133,8.266215449700025e-05)
			(136,7.943736048168485e-05)
			(139,7.719067207354476e-05)
			(142,7.590538022041883e-05)
			(145,7.500966028723907e-05)
			(148,7.418722320649798e-05)
			(151,7.330395372109843e-05)
			(154,7.224949630181007e-05)
			(157,7.093607793165497e-05)
			(160,6.935821284190115e-05)
			(163,6.764607809150248e-05)
			(166,6.592196205285228e-05)
			(169,6.422350484257628e-05)
			(172,6.269407885472919e-05)
			(175,6.137089717647315e-05)
			(178,6.018584303664746e-05)
			(181,5.910623406997559e-05)
			(184,5.80405296985792e-05)
			(187,5.684880207029717e-05)
			(190,5.542783269240883e-05)
			(193,5.3748706762195685e-05)
			(196,5.1874745041745304e-05)
			(199,4.993432356918854e-05)
			(202,4.7969259684781614e-05)
			(205,4.583829068127173e-05)
			(208,4.338267593205266e-05)
			(211,4.051347098082772e-05)
			(214,3.7075919561883914e-05)
			(217,3.28122973924156e-05)
			(220,2.7443533106191354e-05)
			(223,2.0837962882227927e-05)
			(226,1.4049956089026733e-05)
			(229,9.04974155343653e-06)
			(232,5.5931133948599036e-06)
			(235,3.205901213320741e-06)
			(238,1.9083380280592363e-06)
		};
	\addplot[line join=round,mark = star,color = Layer_10,ultra thick]
	coordinates {(10,0.003754881885539844)
		(13,0.003623088144264165)
		(16,0.0029246193274200145)
		(19,0.0013858002934651144)
		(22,0.00048485649499162054)
		(25,0.00030406051702942084)
		(28,0.000241945402446372)
		(31,0.0002035871022655127)
		(34,0.00017548226818513718)
		(37,0.00015327005333529164)
		(40,0.0001371705708589578)
		(43,0.00012327760629655612)
		(46,0.00011154067702726736)
		(49,0.00010249335999498223)
		(52,9.50071447050508e-05)
		(55,8.92227925876656e-05)
		(58,8.49879744453115e-05)
		(61,8.219967796121205e-05)
		(64,8.034930056224131e-05)
		(67,7.73373042416136e-05)
		(70,7.089438960547978e-05)
		(73,6.273952297823259e-05)
		(76,5.720479098284668e-05)
		(79,5.447542195876743e-05)
		(82,5.3324322555594295e-05)
		(85,5.216588136419158e-05)
		(88,5.090070788440027e-05)
		(91,4.983746852714477e-05)
		(94,4.835441419576555e-05)
		(97,4.556700492521169e-05)
		(100,4.1171390686859155e-05)
		(103,3.6810602741646614e-05)
		(106,3.23521838155238e-05)
		(109,2.8988625873280473e-05)
		(112,2.6636531770559692e-05)
		(115,2.297560898033547e-05)
		(118,1.4355443484235533e-05)
		(121,4.735880618860861e-06)
		(124,9.67954872248955e-07)};
		\addplot[line join=round,mark = triangle*,color = Layer_12,ultra thick]
		coordinates {(10,0.0019831205369286106)
			(13,0.0003939392116969396)
			(16,0.00024277850780436805)
			(19,0.00018710326547842986)
			(22,0.00015170070769115248)
			(25,0.00012958202645277822)
			(28,0.00011681661449930601)
			(31,0.00010695829050758128)
			(34,9.368683811769287e-05)
			(37,8.327458668612982e-05)
			(40,7.535185224932793e-05)
			(43,6.728794597518402e-05)
			(46,6.193397569317719e-05)
			(49,5.899161845846715e-05)
			(52,5.177804912517138e-05)
			(55,4.3160275220048696e-05)
			(58,2.966955338870285e-05)
			(61,1.1136505074149191e-05)
			(64,2.045114125043005e-06)};
		]
\end{axis}
\end{tikzpicture}}; 
	\node[rotate=90] at (-3.3,0.2) {\textbf{Uniform}}; 
	\node[rotate=90] at (-3.3,-4) {\textbf{Nonuniform}}; 
\end{tikzpicture}
\captionsetup{font=footnotesize}
\caption{Norm convergence of the sequence generated by Algorithm \ref{Geometric_Explicit_Euler} towards an integral solution (labeling). Once the basin of attraction of the integral solution has been reached (Theorem \ref{thm:Existance-Epsilon}), the convergence rate increases considerably. }
\label{fig:Conv_Rates_Algorithm1}
\end{figure}

\subsection{Accelerated Geometric Optimization}\label{sec:Exp-subsec-4}
In this section, we report the evaluation of Algorithm \ref{Geometric_Two_Stage} using Algorithm \ref{Geometric_Explicit_Euler} as baseline. The main ingredients of Algorithm \ref{Geometric_Two_Stage} are:
\begin{enumerate}[(i)]
\item The descent direction $d^{k}$ given by \eqref{eq:Descent_Dir_Two_Stage} exploits the second-order term  $\frac{1}{2}\Omega R_{S_k}(\Omega S_k)$ weighted by parameter $h_k$ which, according to line \ref{eq:alg:step_size_int} of Algorithm \ref{Geometric_Two_Stage}, is determined with negligible additional computational cost by 
\begin{equation}\label{eq:second_order_weighting-experiments}
	h_k = \tau \cdot \Big(\frac{\|R_{S^k}(\Omega S^k) \|^2_{S^k}}{|\langle 
	R_{S^k}(\Omega S^k),\Omega R_{S^k}(\Omega 
			S^k) 
			\rangle|}\Big), \qquad \tau \in (0,1).
\end{equation} 
Choosing the parameter $\tau$ is a compromise between making larger steps (large value of $\tau$) and accuracy of labeling (small value of $\tau$). According to our experience, $\tau=0.1$ is a reasonable choice that did never compromise labeling accuracy. This value was chosen for all experiments discussed in this section.
\item Algorithm \ref{Geometric_Two_Stage} calls Algorithm \ref{alg:Step_Size_Selection} which in turn calls Algorithm \ref{alg:Search_Aux} in order to satisfy both conditions \eqref{eq:Sufficient_Decrease_Properties} for sufficient decrease. In order to reduce the computational costs of the inner loop started in line \ref{alg:Update-DC-par-line} of Algorithm \ref{Geometric_Two_Stage}, we only checked the conditions \eqref{eq:Sufficient_Decrease_Properties_a} and \eqref{eq:Sufficient_Decrease_Properties_b} at each iteration up to $K_{\max} = 100$ iterations. Figure \ref{fig:Wolfe_Conditions} illustrates that, while condition \eqref{eq:Sufficient_Decrease_Properties_a} is satisfied throughout all outer loop iterations,  condition \eqref{eq:Sufficient_Decrease_Properties_b} is satisfied too except for a tiny fraction of inner loops, and therefore the validity of \eqref{eq:Sufficient_Decrease_Properties} is still guaranteed up to a negligible part of iteration steps.
\end{enumerate}

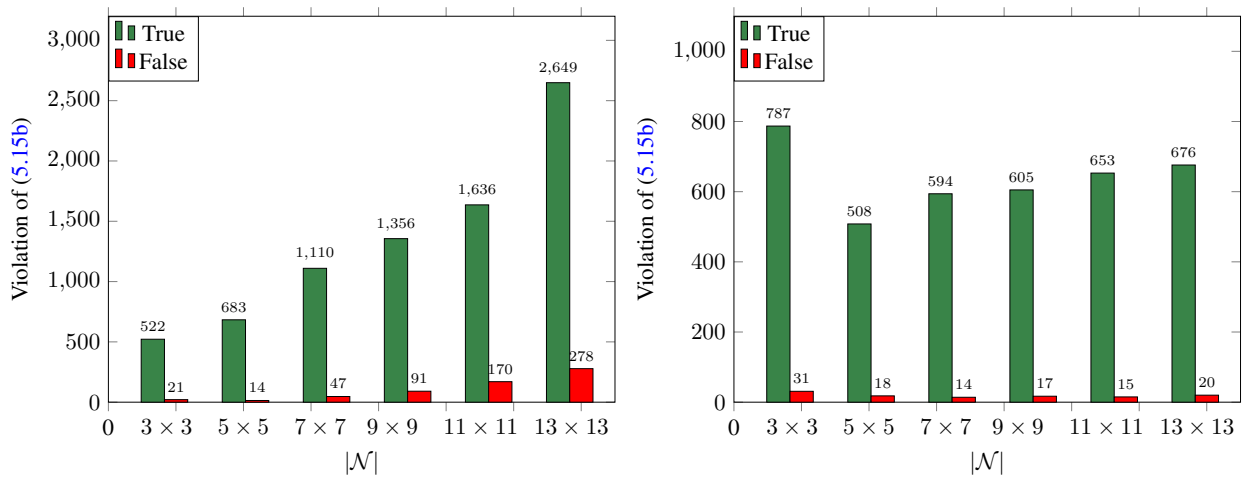
\begin{figure}
	\centering
	\scalebox{0.8}{\begin{tikzpicture}
	\begin{axis}[
		height=8cm,width = 10cm,
		xmin=0, xmax=7.5,
		xticklabels={0,0,$3 \times 3$,$5 \times 5$,$7 \times 7$,$9 
			\times 9$,$11 \times 11$,$13 \times 13$},
		ylabel=Violation of \eqref{eq:Sufficient_Decrease_Properties_b},
		xlabel= {$|\mathcal{N}|$},
		legend style={
			legend columns=1,
			at={(0,1)},
			anchor=north west, ymax = 3200,ymin = 0
		},
		nodes near coords,
		every node near coord/.append style={font=\tiny},
		ybar=0pt,
		x tick label style={xshift=-{(\ticknum==2)*0.4em}},
		x tick label style={xshift={(\ticknum==4)*0.4em}},
		x tick label style={xshift={(\ticknum==5)*0.6em}},
		x tick label style={xshift={(\ticknum==6)*1.5em}},
		x tick label style={xshift={(\ticknum==7)*2.5em}},
		bar width=11pt,
		]
		\addplot[black!20!black,fill=Mea_Deep_Color!40!Layer_4]
		coordinates {(0.83,522) (2.03,683) (3.23,1110) (4.43,1356) 
		(5.63,1636) 
			(6.83,2649) };
		\addplot[black!20!black,fill=red]
		coordinates {(0.83,21) (2.03,14)
			(3.23,47) (4.43,91) (5.63,170) (6.83,278) };
		\legend{True ,False}
	\end{axis}
\end{tikzpicture}}
	\scalebox{0.8}{
\begin{tikzpicture}
	\begin{axis}[
		height=8cm,width = 10cm,
		xmin=0, xmax=7.5,
		xticklabels={0,0,$3 \times 3$,$5 \times 5$,$7 \times 7$,$9 
		\times 9$,$11 \times 11$,$13 \times 13$},
		ylabel=Violation of \eqref{eq:Sufficient_Decrease_Properties_b},
		xlabel= {$|\mathcal{N}|$},
		legend style={	legend columns=1,
			at={(0,1)},
			anchor=north west, ymax = 1100,ymin = 0
		},
		nodes near coords,
		every node near coord/.append style={font=\tiny},
		ybar=0pt,
		bar width=11pt,
		x tick label style={xshift=-{(\ticknum==2)*0.4em}},
		x tick label style={xshift={(\ticknum==4)*0.4em}},
		x tick label style={xshift={(\ticknum==5)*0.6em}},
		x tick label style={xshift={(\ticknum==6)*1.5em}},
		x tick label style={xshift={(\ticknum==7)*2.5em}}
	]
		\addplot[black!20!black,fill=Mea_Deep_Color!40!Layer_4]
		coordinates {(0.83,787) (2.03,508) (3.23,594) (4.43,605) (5.63,653) 
		(6.83,676) };
		\addplot[black!20!black,fill=red]
		coordinates {(0.83,31) (2.03,18)
			(3.23,14) (4.43,17) (5.63,15) (6.83,20) };
		\legend{True ,False}
	\end{axis}
\end{tikzpicture}}
	\captionsetup{font=footnotesize}	\caption{Fraction of inner loops of Algorithm \ref{Geometric_Two_Stage} based on condition \eqref{eq:Sufficient_Decrease_Properties_a} that also satisfied condition \eqref{eq:Sufficient_Decrease_Properties_b} ($\{\protect\markerfour \}$ $=$ True) or not ($\{\protect\markertwo \} $ $=$ False), with initialization $\theta_{0}=0.5$ and uniform averaging (left panel) or nonuniform averaging (right panel). Up to a tiny fraction, condition \eqref{eq:Sufficient_Decrease_Properties_b} is satisfied which justifies to reduce the computational costs of the inner loop by only checking condition \eqref{eq:Sufficient_Decrease_Properties_a} and dispensing with condition \eqref{eq:Sufficient_Decrease_Properties_b} after $K_{\max}$ iterations.
	} 
	\label{fig:Wolfe_Conditions}
\end{figure}

\begin{figure}
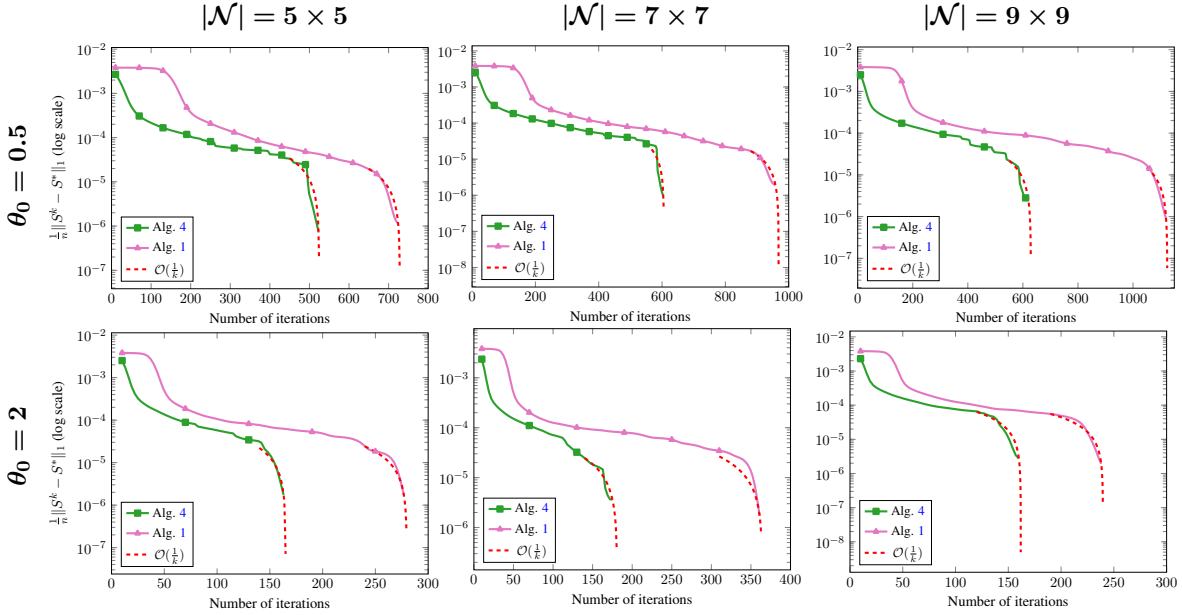

	\begin{tikzpicture}
		\node[scale=0.5] at (0,0) 
		{\input{
				./Graphics/Two_Stage_Convergence_Rates_5.tex}};
		\node[scale=0.5] at (5,0) 
		{\input{./Graphics/Two_Stage_Convergence_Rates_7.tex}};
		\node[scale=0.5] at (10,0) 
		{\input{./Graphics/Two_Stage_Convergence_Rates_3.tex}};
		\node[scale=0.9] at (0.4,2.2) {$\boldsymbol{|\mathcal{N}| = 
		5\times 
		5}$}; 
		\node[scale=0.9] at (5.2,2.2) {$\boldsymbol{|\mathcal{N}| = 
		7\times 
		7}$}; 
		\node[scale=0.9] at (10,2.2) {$\boldsymbol{|\mathcal{N}| = 
		9\times 
		9}$};
		\node[scale=0.5] at (0,-3.8) 
		{\begin{tikzpicture}
	\begin{axis}[
		every other node near coord/.append style={font=\tiny},
		legend columns=1,
		legend style={row sep=0.1cm},
		height=8cm,
		width = 10cm,
		legend entries={{\small 
				Alg. \ref{Geometric_Two_Stage}},{\small Alg. 
				\ref{Geometric_Explicit_Euler}},{\small 
				$\mathcal{O}(\frac{1}{k})$}},
		legend pos=south west,
		mark repeat=20,
		xlabel = Number of iterations,
		ymode = log,
		xmin=0,
		xmax=300,
		ylabel = $\frac{1}{n}\| S^k-S^{\ast} \|_1$ (log scale),
		scaled ticks=true,
		/pgf/number format/.cd,
		use comma,
		1000 sep={}
		]
		\addplot[line join=round,mark = square*,color = Layer_4,ultra thick]
		coordinates {
		(10,0.0025161765070133026)
		(13,0.0017237932239673724)
		(16,0.001047360590253717)
		(19,0.0006177290011279762)
		(22,0.0004203650187535024)
		(25,0.0003252317475904631)
		(28,0.00027409336969112216)
		(31,0.00024065288505843856)
		(34,0.00021473737029793552)
		(37,0.00019509416455534682)
		(40,0.00017898282797177003)
		(43,0.00016496715640505774)
		(46,0.00015142448488056982)
		(49,0.00013970065887778548)
		(52,0.00012906440101489628)
		(55,0.00011985269264701471)
		(58,0.00011178697631387457)
		(61,0.0001048229733203466)
		(64,9.87544423590516e-05)
		(67,9.32572645622071e-05)
		(70,8.884114269547023e-05)
		(73,8.581319899505307e-05)
		(76,8.327184459909999e-05)
		(79,8.077006043292135e-05)
		(82,7.060020623907231e-05)
		(85,6.806880501264879e-05)
		(88,6.582000803217718e-05)
		(91,6.366021286424116e-05)
		(94,6.170827265885735e-05)
		(97,5.9743241831628284e-05)
		(100,5.759449269357574e-05)
		(103,5.549133180684316e-05)
		(106,5.365828516029253e-05)
		(109,5.185578810613856e-05)
		(112,5.018748162781572e-05)
		(115,4.883124606278385e-05)
		(118,4.163886996093677e-05)
		(121,3.903899467933879e-05)
		(124,3.683670428480972e-05)
		(127,3.534297995436722e-05)
		(130,3.4313099115416634e-05)
		(133,3.360380943285742e-05)
		(136,3.293380996044059e-05)
		(139,3.210048435864751e-05)
		(142,2.9338120161866605e-05)
		(145,2.100482245623296e-05)
		(148,1.7734781717261584e-05)
		(151,1.414948179617714e-05)
		(154,1.0387924633527708e-05)
		(157,6.637409114251561e-06)
		(160,3.5029810316125125e-06)
		(163,1.740754928282415e-06)
		};	
		\addplot[line join=round,mark = triangle*,color = Layer_10,ultra thick]
		coordinates {
	(10,0.0038146259179153534)
	(13,0.0038072014356805867)
	(16,0.0037971731754777908)
	(19,0.003783142677218046)
	(22,0.003762480412752006)
	(25,0.003729531593165173)
	(28,0.0036697159092590856)
	(31,0.0035404868521860494)
	(34,0.003253381487528284)
	(37,0.0027508605616794134)
	(40,0.0020875650383667113)
	(43,0.0014115332088473344)
	(46,0.0008765823626280476)
	(49,0.0005497497008965721)
	(52,0.00039590663604277125)
	(55,0.00031248085441818477)
	(58,0.0002679111173846287)
	(61,0.0002402649320644077)
	(64,0.00021862105497490666)
	(67,0.00020084126820336902)
	(70,0.00018580545865799143)
	(73,0.0001727947117826807)
	(76,0.00016105960047069963)
	(79,0.00015047829583787864)
	(82,0.00014110471945317485)
	(85,0.0001326872321044762)
	(88,0.00012562899745075437)
	(91,0.00011987208554368873)
	(94,0.00011522965939689159)
	(97,0.0001109972390692479)
	(100,0.00010629621394470157)
	(103,0.0001018658961346911)
	(106,9.787971145227055e-05)
	(109,9.415980808822844e-05)
	(112,9.118139255276293e-05)
	(115,8.901393223924424e-05)
	(118,8.73505189041771e-05)
	(121,8.601778009465543e-05)
	(124,8.480405072160326e-05)
	(127,8.351436690629071e-05)
	(130,8.206805772206719e-05)
	(133,8.051063450511484e-05)
	(136,7.891092458271089e-05)
	(139,7.715190116720826e-05)
	(142,7.503230630079337e-05)
	(145,7.270257801305106e-05)
	(148,7.033543128012904e-05)
	(151,6.786751211505573e-05)
	(154,6.593608636408052e-05)
	(157,6.47786569417025e-05)
	(160,6.380873939902269e-05)
	(163,6.272555155215695e-05)
	(166,6.149357733413413e-05)
	(169,6.015530202835829e-05)
	(172,5.8777969734512786e-05)
	(175,5.7499428568046956e-05)
	(178,5.6457178843594236e-05)
	(181,5.559441166223039e-05)
	(184,5.4773192616525925e-05)
	(187,5.3952824043248915e-05)
	(190,5.312333322316429e-05)
	(193,5.2184285757182335e-05)
	(196,5.09629353381275e-05)
	(199,4.928331983637341e-05)
	(202,4.706458810424558e-05)
	(205,4.457252510287793e-05)
	(208,4.246980807878429e-05)
	(211,4.121949486570521e-05)
	(214,4.0521931778585184e-05)
	(217,3.9966750498180613e-05)
	(220,3.9327999566134044e-05)
	(223,3.847673100447961e-05)
	(226,3.7299458838089976e-05)
	(229,3.5672796591536786e-05)
	(232,3.3509235774827654e-05)
	(235,3.073136353447378e-05)
	(238,2.7060067546967323e-05)
	(241,2.3031466844176398e-05)
	(244,2.0305549844041496e-05)
	(247,1.8998183603009563e-05)
	(250,1.8370562350544452e-05)
	(253,1.7787649505057862e-05)
	(256,1.7041390415504496e-05)
	(259,1.6028193717135702e-05)
	(262,1.4616409824103482e-05)
	(265,1.2629394112254973e-05)
	(268,9.917483148442381e-06)
	(271,6.64030620704169e-06)
	(274,3.602410360859035e-06)
	(277,1.7171392846610478e-06)
		};	
\addplot[dashed,domain = 240:300,smooth,samples=100,ultra thick, red]
{(4.706458810424558e-05-(4.706458810424558e-05-1.7171392846610478e-06)/(277-202)*(-202+x)};
\addplot[dashed,domain = 140:450,smooth,samples=300,ultra thick, red]
{(5.759449269357574e-05-(5.759449269357574e-05-1.740754928282415e-06)/(163-100)*(-100+x)};
		]
	\end{axis}
\end{tikzpicture}};
		\node[scale=0.5] at (5,-3.8) 
		{\begin{tikzpicture}
	\begin{axis}[
		every other node near coord/.append style={font=\tiny},
		legend columns=1,
		legend style={row sep=0.1cm},
		height=8cm,
		width = 10cm,
		legend entries={{\small 
				Alg. \ref{Geometric_Two_Stage}},{\small Alg. 
				\ref{Geometric_Explicit_Euler}},{\small 
				$\mathcal{O}(\frac{1}{k})$}},
		legend pos=south west,
		mark repeat=20,
		xlabel = Number of iterations,
		ymode = log,
		xmin=0,
		xmax=400,
		scaled ticks=true,
		/pgf/number format/.cd,
		use comma,
		1000 sep={}
		]
		\addplot[line join=round,mark = square*,color = Layer_4,ultra 
		thick]
		coordinates {
			(10,0.002362694363484593)
			(13,0.0015253548903815676)
			(16,0.0008707475701898278)
			(19,0.0005243081255391651)
			(22,0.00038124757390342227)
			(25,0.00031738783910842286)
			(28,0.0002785107848871039)
			(31,0.0002493199635548375)
			(34,0.00022589259492244665)
			(37,0.00020635223527260297)
			(40,0.00018962229045422444)
			(43,0.0001755829628090463)
			(46,0.00016407674371457325)
			(49,0.0001542749440841997)
			(52,0.00014560770661771065)
			(55,0.0001380610197424141)
			(58,0.0001314570485073481)
			(61,0.00012547247043110307)
			(64,0.00012002950187926726)
			(67,0.00011517509207359738)
			(70,0.00011058448086098774)
			(73,0.0001059409215752596)
			(76,0.00010107179903617676)
			(79,9.649782753020317e-05)
			(82,9.265211925700083e-05)
			(85,8.90192551886662e-05)
			(88,8.51023166374639e-05)
			(91,8.080748029989704e-05)
			(94,7.663007952450112e-05)
			(97,7.343219001854677e-05)
			(100,7.119900690667828e-05)
			(103,6.924969559671681e-05)
			(106,6.719856205440978e-05)
			(109,6.496605209269691e-05)
			(112,6.247157788928405e-05)
			(115,5.514244773601624e-05)
			(118,4.5934708646430515e-05)
			(121,4.188688709512465e-05)
			(124,3.8058069988035936e-05)
			(127,3.4813009776274313e-05)
			(130,3.2251318461176605e-05)
			(133,2.9861265359129924e-05)
			(136,2.7542958809549827e-05)
			(139,2.5230632236151132e-05)
			(142,2.2914313014325947e-05)
			(145,2.075731264629785e-05)
			(148,1.903210394696161e-05)
			(151,1.791352233946799e-05)
			(154,1.717956755980631e-05)
			(157,1.6448137792611084e-05)
			(160,1.55463652937026e-05)
			(163,1.4431361111874911e-05)
			(166,6.1107473197858475e-06)
			(169,4.4008042420384335e-06)
			(172,3.6659550912603525e-06)
			(175,3.620260675484032e-06)
		};	
		\addplot[line join=round,mark = triangle*,color = 
		Layer_10,ultra thick]
		coordinates {
			(10,0.003815493249751695)
			(13,0.003808451233083787)
			(16,0.0037990294818647597)
			(19,0.003786024190211133)
			(22,0.003767259573138923)
			(25,0.003738327460850152)
			(28,0.0036886926605306956)
			(31,0.003587698046865216)
			(34,0.0033511820842354657)
			(37,0.002879066979977192)
			(40,0.0022067559255688593)
			(43,0.0014906773986139888)
			(46,0.0009060077553627329)
			(49,0.0005638056234168121)
			(52,0.00040636710876282224)
			(55,0.0003302020368352263)
			(58,0.0002885900676082674)
			(61,0.00025980915947102987)
			(64,0.00023675954452174422)
			(67,0.00021786891252136768)
			(70,0.00020217559766491031)
			(73,0.00018880486115656025)
			(76,0.000177155887325156)
			(79,0.00016697516185631056)
			(82,0.00015824339030499413)
			(85,0.00015083137177861468)
			(88,0.00014443615953849863)
			(91,0.0001386576600048377)
			(94,0.00013336681701355498)
			(97,0.00012896744101526113)
			(100,0.00012567306169468173)
			(103,0.00012301109308902197)
			(106,0.0001205327195853411)
			(109,0.00011801334179279803)
			(112,0.00011543597153792203)
			(115,0.00011289884678624944)
			(118,0.000110403336230494)
			(121,0.00010791608174527354)
			(124,0.00010549259465130755)
			(127,0.00010320740253976292)
			(130,0.000101099375346062)
			(133,9.92421341611995e-05)
			(136,9.768485220080272e-05)
			(139,9.629160436119095e-05)
			(142,9.493105241425355e-05)
			(145,9.36453391393809e-05)
			(148,9.25348050713308e-05)
			(151,9.158538622195468e-05)
			(154,9.070157894027986e-05)
			(157,8.981960939105777e-05)
			(160,8.890212821274078e-05)
			(163,8.787638714424046e-05)
			(166,8.667944204829679e-05)
			(169,8.539214107755706e-05)
			(172,8.42515380020077e-05)
			(175,8.337199165457246e-05)
			(178,8.265136222701843e-05)
			(181,8.19694913213736e-05)
			(184,8.127166887104797e-05)
			(187,8.059738401161152e-05)
			(190,7.99998273110794e-05)
			(193,7.942637487713157e-05)
			(196,7.877203918309455e-05)
			(199,7.795882527112783e-05)
			(202,7.692431879759849e-05)
			(205,7.560956036197897e-05)
			(208,7.398447243184004e-05)
			(211,7.210482361114875e-05)
			(214,7.011472505355912e-05)
			(217,6.812755426491849e-05)
			(220,6.620326459073007e-05)
			(223,6.448802586996396e-05)
			(226,6.323184433814038e-05)
			(229,6.251058645718077e-05)
			(232,6.20766402002832e-05)
			(235,6.168123241080068e-05)
			(238,6.121935296593984e-05)
			(241,6.064205797429721e-05)
			(244,5.9901037439165674e-05)
			(247,5.8930208102586304e-05)
			(250,5.7638327231891885e-05)
			(253,5.592051443326166e-05)
			(256,5.3748372021142754e-05)
			(259,5.143448171913712e-05)
			(262,4.96546816623992e-05)
			(265,4.854250314653389e-05)
			(268,4.769328887767144e-05)
			(271,4.6853662075469014e-05)
			(274,4.59905757191241e-05)
			(277,4.513039115685368e-05)
			(280,4.421050121330661e-05)
			(283,4.3081256853529876e-05)
			(286,4.164384250763899e-05)
			(289,4.003872666350523e-05)
			(292,3.868078844138209e-05)
			(295,3.770036222479748e-05)
			(298,3.6869607641017536e-05)
			(301,3.613622427110686e-05)
			(304,3.5549802606778884e-05)
			(307,3.506785889785128e-05)
			(310,3.457971627886026e-05)
			(313,3.399299769815809e-05)
			(316,3.3240497653834984e-05)
			(319,3.226012837334824e-05)
			(322,3.1002758793022975e-05)
			(325,2.9496466013238324e-05)
			(328,2.794985192904144e-05)
			(331,2.6647427620883216e-05)
			(334,2.5551570000389953e-05)
			(337,2.4358543105803602e-05)
			(340,2.282749848994748e-05)
			(343,2.0792110439642644e-05)
			(346,1.812232336201357e-05)
			(349,1.4779076031820651e-05)
			(352,1.0894880690472919e-05)
			(355,6.88594534189138e-06)
			(358,3.5948374195684656e-06)
			(361,1.701717144953539e-06)
		};	
		\addplot[dashed,domain = 310:460,smooth,samples=100,ultra 
		thick, red]
		{(5.143448171913712e-05-(5.143448171913712e-05-1.701717144953539e-06)/(361-259)*(-259+x)};
		\addplot[dashed,domain = 140:450,smooth,samples=300,ultra 
		thick, red]
		{(2.5230632236151132e-05-(2.5230632236151132e-05-3.620260675484032e-06)/(175-139)*(-139+x)};
		]
	\end{axis}
\end{tikzpicture}};
		\node[scale=0.5] at (10,-3.8) 
		{\begin{tikzpicture}
	\begin{axis}[
		every other node near coord/.append style={font=\tiny},
		legend columns=1,
		legend style={row sep=0.1cm},
		height=8cm,
		width = 10cm,
		legend entries={{\small 
				Alg. \ref{Geometric_Two_Stage}},{\small Alg. 
				\ref{Geometric_Explicit_Euler}},{\small 
				$\mathcal{O}(\frac{1}{k})$}},
		legend pos=south west,
		mark repeat=280,
		xlabel = Number of iterations,
		ymode=log,
		xmin=0,
		xmax=300,
		scaled ticks=true,
		/pgf/number format/.cd,
		use comma,
		1000 sep={}
		]
		\addplot[line join=round,mark = square*,color = Layer_4,ultra thick]
		coordinates {(10,0.002301671651283164)
			(13,0.001455719482247356)
			(16,0.0008198229622107314)
			(19,0.0005065821766600907)
			(22,0.0003828419068482177)
			(25,0.00032585329782224346)
			(28,0.00028871449509944644)
			(31,0.00026012497155346033)
			(34,0.00023636515302186053)
			(37,0.0002160349952318051)
			(40,0.00019898265246585603)
			(43,0.000184935304118119)
			(46,0.00017283285982371997)
			(49,0.0001620406231505651)
			(52,0.000152377038283703)
			(55,0.0001438351332958516)
			(58,0.00013627765959745247)
			(61,0.0001293225768288986)
			(64,0.00012251195018501436)
			(67,0.00011569151409459814)
			(70,0.00010935486312740455)
			(73,0.00010366151536099166)
			(76,9.869894896315005e-05)
			(79,9.45656007342591e-05)
			(82,9.113879275228345e-05)
			(85,8.820884941268298e-05)
			(88,8.565776227929064e-05)
			(91,8.332364042737922e-05)
			(94,8.103069703005996e-05)
			(97,7.872969913055942e-05)
			(100,7.636927226494738e-05)
			(103,7.401784616641624e-05)
			(106,7.202411476762571e-05)
			(109,7.052246683714669e-05)
			(112,6.921011720087398e-05)
			(115,6.794954217888821e-05)
			(118,6.66304886169054e-05)
			(121,6.518919083345546e-05)
			(124,5.928499755819865e-05)
			(127,5.736264899269628e-05)
			(130,5.2249156619677914e-05)
			(133,4.810524989792245e-05)
			(136,4.600175866510586e-05)
			(139,3.85236544800719e-05)
			(142,2.5849019902013756e-05)
			(145,2.033338203397134e-05)
			(148,1.4896935914416853e-05)
			(151,9.717492510009223e-06)
			(154,5.671898723709752e-06)
			(157,3.4491585962954486e-06)
			(160,2.7453560655146888e-06)
		};			
		\addplot[line join=round,mark = triangle*,color = Layer_10,ultra thick]
		coordinates {(10,0.0038157770398366355)
			(13,0.0038088555574512964)
			(16,0.0037996207251412687)
			(19,0.0037869222788330724)
			(22,0.0037687033042171935)
			(25,0.0037408620388355834)
			(28,0.003693823348257977)
			(31,0.003600182327743038)
			(34,0.003380585542318803)
			(37,0.0029229078869960168)
			(40,0.0022521392250063523)
			(43,0.001527173360531471)
			(46,0.0009263328531868269)
			(49,0.0005724209481245406)
			(52,0.0004124307147592783)
			(55,0.0003387434117246038)
			(58,0.0002993984500795678)
			(61,0.00027199516268952085)
			(64,0.00024998941433811676)
			(67,0.00023168930520231607)
			(70,0.0002160614703851955)
			(73,0.0002022550389942302)
			(76,0.0001896299317667677)
			(79,0.00017809019739109687)
			(82,0.00016763998951413426)
			(85,0.00015841287965724846)
			(88,0.0001501739927156065)
			(91,0.00014260102918232166)
			(94,0.00013584122994144893)
			(97,0.00013006723203419084)
			(100,0.00012508375850025448)
			(103,0.00012066181799642308)
			(106,0.00011663541960350412)
			(109,0.00011283935318831301)
			(112,0.0001090576266717182)
			(115,0.00010512640841562434)
			(118,0.00010106746882726329)
			(121,9.709579474833247e-05)
			(124,9.338629339610249e-05)
			(127,8.982838875479517e-05)
			(130,8.62444683344919e-05)
			(133,8.266215449700025e-05)
			(136,7.943736048168485e-05)
			(139,7.719067207354476e-05)
			(142,7.590538022041883e-05)
			(145,7.500966028723907e-05)
			(148,7.418722320649798e-05)
			(151,7.330395372109843e-05)
			(154,7.224949630181007e-05)
			(157,7.093607793165497e-05)
			(160,6.935821284190115e-05)
			(163,6.764607809150248e-05)
			(166,6.592196205285228e-05)
			(169,6.422350484257628e-05)
			(172,6.269407885472919e-05)
			(175,6.137089717647315e-05)
			(178,6.018584303664746e-05)
			(181,5.910623406997559e-05)
			(184,5.80405296985792e-05)
			(187,5.684880207029717e-05)
			(190,5.542783269240883e-05)
			(193,5.3748706762195685e-05)
			(196,5.1874745041745304e-05)
			(199,4.993432356918854e-05)
			(202,4.7969259684781614e-05)
			(205,4.583829068127173e-05)
			(208,4.338267593205266e-05)
			(211,4.051347098082772e-05)
			(214,3.7075919561883914e-05)
			(217,3.28122973924156e-05)
			(220,2.7443533106191354e-05)
			(223,2.0837962882227927e-05)
			(226,1.4049956089026733e-05)
			(229,9.04974155343653e-06)
			(232,5.5931133948599036e-06)
			(235,3.205901213320741e-06)
			(238,1.9083380280592363e-06)
		};	
\addplot[dashed,domain = 190:300,smooth,samples=500,ultra thick, red]
{(5.542783269240883e-05-(5.542783269240883e-05-1.9083380280592363e-06)/(238-190)*(-190+x)};
\addplot[dashed,domain = 120:200,smooth,samples=1000,ultra thick, red]
{(6.66304886169054e-05-(6.66304886169054e-05-2.7453560655146888e-06)/(160-118)*(-118+x)};
		]
	\end{axis}
\end{tikzpicture}};
		\node[scale=0.9,rotate=90] at (-3,0.1) {$\boldsymbol{\theta_0 = 
		0.5}$};
		\node[scale=0.9,rotate=90] at (-3,-3.5) {$\boldsymbol{\theta_0 
		= 2$}};     
	\end{tikzpicture}
	\captionsetup{font=footnotesize}
	\caption{Comparison of the convergence of 
	Algorithm 
	\ref{Geometric_Explicit_Euler} (\protect\markersix) and Algorithm 
	\ref{Geometric_Two_Stage} (\protect\markerfour) 
	towards integral solutions (labelings) for various sizes $|\mc{N}|$ of 
	neighborhoods and nonuniform averaging. For all parameter settings 
	Algorithm \ref{Geometric_Two_Stage} terminates after a smaller number of 
	iterations.
}
	\label{fig:Conv_Rates_Algorithm_2}
\end{figure}

Parameter $\theta_{k}$ of Algorithm \ref{Geometric_Two_Stage} corresponds to the step size parameter $h_{k}$ of Algorithm \ref{Geometric_Two_Stage}. According to the discussion of proper choices of $h_{k}$ in Section \ref{sec:Exp-subsec-2}, parameters $\theta_{k}$ was initialized by values $\theta_{0}\in\{\frac{1}{2},2\}$ and the adaptive search of $\theta_{k}$ was not allowed to exceed the upper bound $\theta_{\max}=10$.

Like Algorithm \ref{Geometric_Explicit_Euler}, Algorithm 
\ref{Geometric_Two_Stage} terminated when condition 
\eqref{eq:Iter_Crit_Av_grad} was satisfied with $\epsilon=10^{-7}$.

\vspace{0.5cm}
Figure \ref{fig:Conv_Rates_Algorithm_2} illustrates the convergence of Algorithms \ref{Geometric_Explicit_Euler} and \ref{Geometric_Two_Stage} towards labelings for the two initial step sizes $\theta_{0}\in\{\frac{1}{2},2\}$ corresponding to the fixed step size $h\in\{\frac{1}{2},2\}$ of Algorithm \ref{Geometric_Explicit_Euler}), and for different sizes $|\mc{N}|$ of neighborhoods with nonuniform averaging. Throughout all experiments, we observed that due to using adaptive step sizes $\theta_{k}$ and second-order information for determining the search direction, Algorithm \ref{Geometric_Two_Stage} terminates after a smaller number of iterations. In particular, the fast convergence of Algorithm \ref{Geometric_Explicit_Euler} within the basins of attraction is preserved.

Table \ref{tab:Iterations} compares Algorithms \ref{Geometric_Explicit_Euler} 
and \ref{Geometric_Two_Stage} in terms of factors of additional iterations 
required by Algorithm \ref{Geometric_Explicit_Euler} to terminate. We observe 
that the efficiency of Algorithm \ref{Geometric_Two_Stage} is more pronounced 
when larger neighborhood sizes $|\mc{N}|$ or uniform averaging are used.

\begin{table*}
	\centering
	\begin{tabular}{lll|l|ll|l}
		\toprule & \hspace{0.7cm} \textbf{    Uniform} &\hskip -2cm &  
		&\hspace{0.7cm} 
		\textbf{  Nonuniform}& 
		\\
		\midrule
		$|\mc{N}|$ & Alg. \ref{Geometric_Explicit_Euler} & 
		Alg. \ref{Geometric_Two_Stage} &Acc. & Alg. 
		\ref{Geometric_Explicit_Euler} &\hskip -0.6cm Alg. 
		\ref{Geometric_Two_Stage} &Acc.\\ 
		\midrule
		$3\times 3$ &  828 &\textbf{543} &1.52 & 760
		&\hskip -0.6cm \textbf{557} &1.36 \\
		$5 \times 5$ &1860 & \textbf{697} &2.66 &726 
		&\hskip -0.6cm \textbf{526} &1.38\\
		$7\times 7$ &3465 & \textbf{1158} &3  &961 &\hskip -0.6cm 
		\textbf{608} &1.58\\
		$9\times 9$ &4707  &\textbf{1447} &3.25 &1123  &\hskip -0.6cm 
		\textbf{622} &1.81 \\
		$11 \times 11$ &9216 &\textbf{1806} &5.10 &1402 &\hskip -0.6cm 
		\textbf{668} &2.1 \\
		$13 \times 13$ &9957 &\textbf{2927} &3.40 &1510 &\hskip -0.6cm 
		\textbf{696} &2.17 \\
		\bottomrule
	\end{tabular} 
	\captionsetup{font=footnotesize}
	\caption{Number of iterations required by Algorithms 
		\ref{Geometric_Explicit_Euler} and \ref{Geometric_Two_Stage} until convergence to a solution of the nonlocal 
	PDE \eqref{eq:Non_Local_PDE}, for \textit{uniform} and \textit{nonuniform} averaging and various neighborhood sizes $|\mc{N}|$. 
Columns Acc.~list the additional factor of iterations required by Algorithm \ref{Geometric_Explicit_Euler} relative to Algorithm \ref{Geometric_Two_Stage}.} 
	\label{tab:Iterations}
\end{table*}

\section{Conclusion and Future Work}\label{sec:Conclusion}


\textbf{Conclusion.} Using established nonlocal calculus, we devised a novel nonlocal PDE with nonlocal boundary conditions on weighted graphs. Our work adds a novel approach to the literature on PDE-based image analysis that extends the scope from denoising and inpainting to image labeling. An in-depth discussion (Section \ref{sec:Related-Work-PDE}) clarified common aspects and differences to related nonlocal approaches from the mathematical viewpoint. Our work has been motivated by the assignment flow approach 
\cite{Astrom:2017ac,Schnorr:2019aa} to metric data labeling which was shown to constitute a special instance of our general approach introduced in this paper. In particular, our PDE contains the local PDE derived in \cite{Savarino:2019ab} as special case and thus provides a natural nonlocal generalization.


The second major contribution of our work rest upon the reparametrization introduced in \cite{Savarino:2019ab} that turns the assignment flow into a Riemannian descent flow with respect to a nonconvex potential. We established in the present paper two relations to numerical schemes \cite{Zeilmann:2020aa} for the geometric integration of the assignment flow: (i) Geometric integration can be applied to solve the novel nonlocal PDE. (ii) We showed that the basic geometric Euler integration scheme corresponds to the basic DC-algorithm of DC programming \cite{Hoan-An:2018aa}. Moreover, the geometric viewpoint reveals how second-order information can be used in connection with line-search in order to accelerate the basic DC-algorithm for nonconvex optimization.

A range of numerical results were reported in order to illustrate properties of the approach and the theoretical convergence results. This includes, in particular a linear convergence rate whenever a basin of attraction corresponding to an integral labeling solution is reached, whose existence was establised in \cite{Zern:2020aa}.

\vspace{0.25cm}
\textbf{Future work.} The assignment flow approach \eqref{eq:assignment-flow} may be considered as a particular ``neural ODE'' from the viewpoint of machine learning that generates layers of a deep network by geometric integration of the flow at discrete points of time. For recent work on learning the parameters from data and on quantifying the uncertainty of label assignments, respectively, we refer to \cite{Huhnerbein:2021th,Zeilmann:2021wt,Zeilmann:2022ul} and \cite{Gonzalez-Alvarado:2021vn}. In the present paper, Lemma \ref{Help_Lemma} characterizes parametrizations for which the theoretical results hold. Uniform and data-driven nonuniform parametrizations were used in the experiments to demonstrate broad applicability. Learning these parameters from data is conceivable but beyond the scope of this paper and hence left for future work. Generalizations of the scalar-valued mappings $\Theta, \alpha$ to tensor-valued mappings are conceivable as well in order to not only model the interaction across the graph but also the interaction between labels. For the specific case of classification entire data sets, rather than labeling individual data points, a first step has been done recently using deep linearized assignment flows \cite{Boll:2022wz}.

Finally, we point out recent work \cite{Savarino:2021ts,Savarino:2021aa} on characterizing assignment flows as critical points of an action functional, provided the nonlocal mapping which specifies the interaction of label assignments across the graph satisfies a certain condition. Reconsidering the PDE \eqref{eq:S-Non_Local_PDE} from this viewpoint defines another problem to be addressed by future work.

\section*{ORCID iDs}

Dmitrij Sitenko    \hspace{1cm}    \orcidaffil{0000-0002-0022-3891}\\
\vskip -0.4cm
Bastian Boll       \hspace{1.48cm} \orcidaffil{0000-0002-3490-3350}\\
\vskip -0.4cm
Christoph Schn\"{o}rr \hspace{0.55cm}  \orcidaffil{0000-0002-8999-2338}

\subsection*{Acknowledgement}  
This work is funded by Deutsche Forschungsgemeinschaft (DFG) under Germany's
Excellence Strategy EXC-2181/1 - 390900948 (the Heidelberg STRUCTURES 
Excellence Cluster).

B.B.~is funded by the Deutsche Forschungsgemeinschaft (DFG), grant SCHN 457/17-1, within the priority programme SPP 2298: “Theoretical Foundations of Deep Learning”.

\bibliographystyle{amsalpha}
\bibliography{Non-Local-S-Flow,CS}

\appendix
\section{Proofs}
\label{sec:Appendix}
\subsection{Proofs of Section 
\ref{sec:Nonlocal_PDE_Formulation}}\label{app:Nonlocal_PDE_Formulation}

\begin{proof}[Proof of Lemma \ref{Help_Lemma}]
	In order to show \eqref{eq:Reform_Av}, we directly compute using assumption 
	\eqref{eq:Mapping_assumption} and the parametrization 
	\eqref{eq:Avaraging_Matrix}, for any $x \in 
	\mathcal{V}$, 
	\begin{subequations} \allowdisplaybreaks
		\begin{align}
			\sum_{y \in \mathcal{V}}  
			\Omega(x,y)f(y) &\overset{\eqref{eq:Avaraging_Matrix}}{=} \sum_{y 
			\in \mc{N}(x)}\Theta(x,y)\alpha^{2}(x,y)
			f(y)+\Theta(x,x)f(x)			
			\\
			&\overset{\phantom{\eqref{eq:Avaraging_Matrix}}}{=}\sum_{y \in 
			\mc{N}(x)
			}\Theta(x,y)\alpha^{2}(x,y)f(y)+\Theta(x,x)f(x) 
			+\big(\lambda(x)-\lambda(x)\big)f(x)\\
			&\overset{\eqref{eq:def-lambda-x}}{=} \sum_{y \in 
			\mc{N}(x)}\Theta(x,y)\alpha^{2}(x,y)\big(f(y)-f(x)\big)+
			\lambda(x)f(x)\\
			&\overset{f|_{\mathcal{V}^{\alpha}_{\mathcal{I}}} = 
				0}{\hspace{-0.35cm}=} -\sum_{y 
				\in 
				\ol{\mathcal{V}}}\Theta(x,y)\alpha^{2}(x,y)\Big(-\big(f(y)-f(x)\big)\Big)+
			\lambda(x)f(x)\\
			&\overset{\eqref{eq:NL-Grad_op}}{=}
			-\sum_{y \in 
				\overline{\mathcal{V}}}\Theta(x,y)\big((\mathcal{D}^{\alpha})^{\ast}(f)(x,y)\big)
			\alpha(x,y)+\lambda(x)f(x)
			\\
			&\overset{\phantom{\eqref{eq:Avaraging_Matrix}}}{=}\sum_{y \in 
				\overline{\mathcal{V}}}\frac{1}{2}\Theta(x,y)\big(-2(\mathcal{D}^{\alpha})^{\ast}(f)(x,y)
			\alpha(x,y)\big)+\lambda(x)f(x) 
			\\
			&\overset{\eqref{eq:def-nonl-grad}}{=}
			\sum_{y \in 
				\overline{\mathcal{V}}}\frac{1}{2}\Theta(x,y)\big(2\mathcal{G}^{\alpha}(f)(x,y)
			\alpha(x,y)\big)+\lambda(x)f(x) \\
			&\overset{\eqref{eq:Non_Local_Dif}}{=}
			\frac{1}{2}\mc{D}^{\alpha}\big(\Theta\mc{G}^{\alpha}(f)\big)(x) + 
			\lambda(x)f(x)
		\end{align}
	\end{subequations}
	which proves \eqref{eq:Reform_Av}. 
	
	Assume that $\lambda(x) \leq 1$ for all $x \in \mathcal{V}$. Then,
	properties \eqref{eq:Omega_Mat} easily 
	follows from the nonnegativity of $\Theta \in 
	\mathcal{F}_{\overline{\mathcal{V}}\times \overline{\mathcal{V}}}$ and  
	definition \eqref{eq:def-lambda-x}. In addition, if $\Omega$ is given by
	\eqref{eq:Avaraging_Matrix} and also satisfies \eqref{eq:Omega_AF}, then 
	equality in 
	\eqref{eq:def-lambda-x} is achieved: 
	\begin{equation}
		1 = \sum_{y \in \mathcal{V}}\Omega(x,y) = \sum_{y \in 
			\mathcal{V}}\Theta(x,y)\alpha^2(x,y)+\Theta(x,x) = 
		\lambda(x)-\hspace{-0.2cm} \underbrace{\sum_{y\in 
				\mathcal{V}_{\mathcal{I}}^{\alpha}}\Theta(x,y)\alpha^2(x,y)}_{\geq
			0} \leq 
		\lambda(x) \overset{\eqref{eq:def-lambda-x}}{\leq} 1.
	\end{equation}
\end{proof}

\begin{proof}[Proof of Proposition \ref{Prop:Non_Local_PDE}]
	Recalling  definition 
	\eqref{eq:RS-Omega}, we directly compute
	\begin{subequations}
		\begin{align}
			R_{S(x,t)}\big((\Omega S)(x,t)\big) 
			&\overset{\phantom{\eqref{S_Flow_boundary1}}}{=} R_{S(x,t)}\Big( 
			\sum_{y\in 
				\mathcal{V}}\Omega(x,y)S(y,t) 
			\Big) \\
			&\overset{\eqref{eq:Reform_Av}}{=} R_{S(x,t)}\Big( 
			\frac{1}{2}\mathcal{D}^{\alpha}\big(\Theta 
			\mathcal{G}^{\alpha}(S)\big)(x)+\lambda(x)S(x)
			\Big).  
		\end{align}  
	\end{subequations}
\end{proof}

\subsection{Proofs of Section 
\ref{sec:Non_Local_PDE_Bal_Law}}\label{app:sec:Non_Local_PDE_Bal_Law}

\begin{proof}[Proof of Proposition \ref{Prop:Non_Local_Bal_Law}]
	For brevity, we omit the argument $t$ and simply write $S = S(t), V = 
	V(t)$. Recall the componentwise operation $\odot$ defined by 
	\eqref{eq:def-odot}, e.g.~$(S\odot V)_{j}(x)=S_{j}(x)V_{j}(x)$ for 
	$j\in[c]$, and $S^{2}(x) = (S\odot S)(x)$.
	
	Multiplying both sides of \eqref{eq:balance-PDE-a} with 
	$S(x) = \exp_{S^0}(V(x))$ and summing over  $x \in \mathcal{V}$ yields  
	\begin{align}\label{eq:proof}
		\sum_{x \in \mathcal{V}}\big(S\odot \dot{V}\big)_j(x) &- \sum_{x 
			\in 
			\mathcal{V}} 
		\frac{1}{2} \Big(S\odot\mathcal{D}^{\alpha}\big(\Theta 
		\mathcal{G}^{\alpha}(S)\big)\Big)_j(x)= \sum_{x \in 
			\mathcal{V}}\big(\lambda S^{2} \big)_j(x).
	\end{align}
	Applying Greens nonlocal first identity 
	\eqref{eq:Green_First_Identity} with $u(x) = S_j(x)$  
	to the second term on the left-hand side yields with \eqref{eq:def-olV} 
	\begin{subequations}\label{eq:Green_Reformulation}
		\begin{align}
			\sum_{x \in \mathcal{V}}\big(S\odot \dot{V}\big)_j(x) &+ 
			\frac{1}{2}\sum_{ 
				x \in \ol{\mc{V}}}\sum_{y \in 
				\ol{\mc{V}}}\big(\mathcal{G}^{\alpha}(S)
			\odot
			(\Theta \mathcal{G}^{\alpha}(S))\big)_j(x,y) \\
			&+ \sum_{y \in 
				\mathcal{V}_{I}^{\alpha}}S_j(y) \mathcal{N}^{\alpha}\big(\Theta
			\mathcal{G}^{\alpha}(S_j) \big)(y) = \sum_{x \in 
				\mathcal{V}}\big(\lambda S^{2} \big)_j(x).
		\end{align}
	\end{subequations}
	Now, using the
	parametrization \eqref{eq:T0_parameterization} of $S$, we compute at each 
	$x\in\mc{V}$: 
	\begin{subequations}\allowdisplaybreaks
		\begin{align}\label{eq:Ref_Id}
			\dot{S}(x) &\overset{\phantom{\eqref{eq:T0_parameterization2}}}{=} 
			\frac{d}{dt}\exp_{S^0(x)}\big(V(x)\big) \\
			&\overset{\eqref{eq:def-small-exp}}{=} 
			\frac{\big(\frac{d}{dt}\big(S^0(x)\odot e^{V(x)}\big)\big)\langle 
				S^0(x),e^{V(x)} \rangle-\big(\frac{d}{dt}\langle 
				S^0(x),e^{V(x)} \rangle \big) S^0(x)\odot e^{V(x)}}{\langle 
				S^0(x),e^{V(x)} \rangle^2}\\
			&\overset{\phantom{\eqref{eq:T0_parameterization2}}}{=} 
			\frac{\langle 
				S^0(x),e^{V(x)} 
				\rangle (S^0 \odot e^{V})(x) \odot 
				\dot{V}(x)-\langle 
				S^0(x)\odot e^{V(x)} , \dot{V}(x)\rangle (S^0 \odot 
				e^{V})(x)}{\langle 
				S^0(x),e^{V(x)} \rangle^2}\\
			&\overset{\phantom{\eqref{eq:T0_parameterization2}}}{=} (S \odot 
			\dot{V})(x) - \langle S(x), \dot{V}(x) \rangle S(x)\\
			&\overset{\eqref{eq:T0_parameterization2}}{=} (S \odot \dot{V})(x) 
			- \langle S(x), (\Pi_{0} \Omega 
			\exp_{S^0}(V))(x)\rangle S(x)\\
			&\overset{\eqref{eq:Phi_Func}}{=} (S \odot \dot{V})(x) - 
			\phi_{S}(x) S(x).
		\end{align}
	\end{subequations}
	Solving the last equation for $(S \odot \dot{V})(x)$ and substitution into 
	\eqref{eq:Green_Reformulation} yields after taking the sum over $x 
	\in \mathcal{V}$, for 
	each $S_j = \{S_{j}(x)\colon x\in\mc{V}\},\,j\in[c]$  
	\begin{subequations}
		\begin{align}
			\frac{1}{2}\frac{d}{d t} \big(\sum_{x \in \mathcal{V}}S_j(x) \big) 
			&+ 
			\frac{1}{2}\langle 
			\mathcal{G}^{\alpha}(S_j), \Theta \mathcal{G}^{\alpha}(S_j)   
			\rangle_{\overline{\mathcal{V}}\times \overline{\mathcal{V}}}  
			+ \sum_{x \in \mathcal{V}}\phi_{S}(x)  S_j(x)  \\ &+ \sum_{y \in 
				\mathcal{V}_{\mathcal{I}^{\alpha}}} S_j 
			\mathcal{N}^{\alpha}\big(\Theta \mathcal{G}^{\alpha}(S_j)\big)(y) = 
			\sum_{x \in 
				\mathcal{V}}\big(\lambda 
			S^2_j\big)(x),
		\end{align}
	\end{subequations}
	which after rearranging the terms is equal to 
	\eqref{eq:Non_Local_Balance_Law}. 
\end{proof}

\subsection{Proofs of Section 
\ref{sec:first-order-geometric-DC}}\label{app:first-order-geometric-DC}
\begin{delayedproof}{prop:Numerical_Scheme}
	Equation \eqref{eq:PDE_Euler_Update} directly follows from  Proposition 
	\ref{Prop:Non_Local_PDE}, from the specification \eqref{eq:def-Si} of the 
	similarity mapping and from the relation $\exp_{p}=\Exp_{p}\circ R_{p}$ for 
	$p\in\mc{S}$ (cf.~\eqref{eq:Small_exp}, \eqref{eq:def-lifting-map}). 
	Leveraging the 
	parametrization 
	\eqref{eq:Non_Local_PDE_T0} of 
	system \eqref{eq:Non_Local_PDE}, discretization of 
	\eqref{eq:Non_Local_PDE_T0} by forward finite differences with step size 
	parameter $h> 0$ yields for $x\in \mathcal{V}$
	\begin{equation}
		\frac{V^{k+1}(x)-V^k(x)}{h} = \Big( 
		\frac{1}{2}\mathcal{D}^{\alpha}\big(\Theta 
		\mathcal{G}^{\alpha}(\exp_{S^0}(V^k))\big)+\lambda\exp_{S^0}(V^k)\Big)(x)
	\end{equation}
	which is \eqref{eq:PDE_Euler_Update} after applying the lifting map 
	\eqref{eq:def-lifting-map} to $V^{k+1}$. Consequently, in view of zero 
	nonlocal 
	boundary conditions, the zero extension of \eqref{eq:PDE_Euler_Update} to 
	$\ol{\mathcal{V}}$ verifies that $\ol{S}^k$ is indeed a first 
	order approximation of solution $S(kh)$ to \eqref{eq:Non_Local_PDE}.
	
	It remains to show that 
	\eqref{eq:explicit-Euler} implies 
	\eqref{eq:Dec_Seq_Euler}. 
	Adding and subtracting a convex negative entropy term 
	\begin{equation}
		\la S , \log S \ra = \sum_{x\in\mc{V}} \la S(x),\log S(x)\ra,\qquad 
		\log S(x) = \big(\log S_{1}(x),\dotsc,\log S_{c}(x)\big)^{\T}
	\end{equation}	
	to the
	potential \eqref{eq:S-Flow_Pot}, we write with the convex constraint 
	$S\in\ol{\mc{W}}$ represented by the delta-function $\delta_{\ol{\mc{W}}}$,
	\begin{equation}\label{eq:J-DC}
		J(S) = \underbrace{\gamma\la S,\log S\ra + \delta_{\ol{W}}(S)}_{g(S)} 
		-  
		\Big(\underbrace{\frac{1}{2}\la S,\Omega S\ra + \gamma\la S,\log 
			S\ra}_{h(S)}\Big),\quad \gamma>|\lambda_{\min}(\Omega)|,
	\end{equation}
	which is a DC-function \cite{Hartman:1959aa} if 
	$\gamma>|\lambda_{\min}(\Omega)|$, i.e.~both $g(S)$ and 
	$h(S)$ are 
	convex. Indeed, while the convexity of $g$ is obvious, the convexity of $h$ 
	becomes apparent when inspecting its Hessian. Writing
	\begin{equation}\label{eq:vec-def-proof}
		s = \vvec_{r}(S)
	\end{equation}
	with the row-stacking mapping $\vvec_{r}$, we have ($\otimes$ denotes the 
	Kronecker matrix product)
	\begin{subequations}
		\begin{align}
			\la S, \Omega S\ra
			&= \la s, (\Omega\otimes I_{c}) s\ra \\
			\la S,\log S\ra 
			&= \la s, \log s\ra,\qquad \log s = (\dotsc,\log s_{i},\dotsc)^{\T}
		\end{align}
	\end{subequations}
	and hence for any $v\in\R^{n c}$ with $\|v\|=1$
	\begin{equation}\label{eq:proof-prop-num}
		d^{2}h(S)(v,v)
		= \Big\la v,\Big((\Omega\otimes I_{c}) + \gamma 
		\Diag\Big(\frac{\eins}{s}\Big)\Big) v\Big\ra
		> \lambda_{\text{min}}(\Omega) + \gamma,
	\end{equation}
	where the last inequality follows from $\lambda \geq 
	\lambda_{\text{min}}(\Omega)$ for any eigenvalue $\lambda$ of the symmetric 
	matrix $\Omega$ (recall \eqref{eq:Omega_Mat}, \eqref{eq:Omega_AF}), 
	$\lambda(A\otimes B)=\lambda_{i}(A)\lambda_{j}(B)$ for some $i,j$ 
	\cite{Graham:1981wj}, and $\lambda_{\min}(\Diag(\frac{\eins}{s}))>1$ if 
	$S\in\ol{\mc{W}}$.
	
	Thus, if $\gamma > |\lambda_{\min}(\Omega)|$ then $h$ is convex and 
	minimizing \eqref{eq:J-DC} is a 
	DC-programming problem 
	\cite{Horst:1999aa,Hoai-An:2005aa}. Using Fenchel's inequality 
	$-h(S^{k})\leq 
	h^{\ast}(\wt{S})-\la S^{k}, \wt{S}\ra,\,\forall \wt{S}$, let $\wt{S}^{k}$ 
	minimize at the current iterate $S^{k}$ the upper bound
	\begin{subequations}
		\begin{gather}
			J(S^{k}) = g(S^{k})-h(S^{k})
			\leq g(S^{k}) + h^{\ast}(\wt{S})-\la S^{k}, \wt{S}\ra,\qquad 
			\forall \wt{S}
			\intertext{with respect to $\wt{S}$, i.e.} \label{eq:DC-proof-wtSk}
			0 = \partial h^{\ast}(\wt{S}^{k})-S^{k} 
			\qquad\gdw\qquad \wt{S}^{k}\in\partial h(S^{k})=\nabla h(S^{k}).
		\end{gather}
	\end{subequations}
	In particular, $-h(S^{k}) = h^{\ast}(\wt{S}^{k})-\la S^{k},\wt{S}^{k}\ra$ 
	and hence
	\begin{equation}
		J(S^{k}) = g(S^{k})+h^{\ast}(\wt{S}^{k})-\la S^{k},\wt{S}^{k}\ra. 
	\end{equation}
	Minimizing in turn the right-hand side with respect to $S^{k}$ guarantees 
	\eqref{eq:Dec_Seq_Euler} and defines the update $S^{k+1}$ by
	\begin{subequations}
		\begin{align}
			S^{k+1}=\arg\min_{S}\{g(S)-\la S,\wt{S}^{k}\ra\}
			\qquad\gdw\qquad 
			0 &= \partial g(S^{k+1})-\wt{S}^{k}
			\\ \gdw\quad 
			\gamma (\log S^{k+1}(x) + \eins) + \partial 
			\delta_{\ol{\mc{S}}}\big(S^{k+1}(x)\big) 
			&\overset{\eqref{eq:DC-proof-wtSk}}{=} 
			\nabla h(S^{k})(x) 
			\\ &=
			(\Omega S^{k})(x) + \gamma (\log S^{k}(x) + \eins).
		\end{align}
	\end{subequations} 
	Solving for $S^{k+1}(x)$ yields \eqref{eq:explicit-Euler} 
	resp.~\eqref{eq:PDE_Euler_Update} with stepsize $h = \frac{1}{\gamma} < 1$ 
	due 
	to $\gamma >|\lambda_{\min}(\Omega)|$.
\end{delayedproof}

\subsection{Proofs of Section 
\ref{sec:Geometric-Integration}}\label{app:Geometric-Integration}
\begin{proof}[Proof of Lemma \ref{lem:Grad_Tangent}]
	Taking into account the parametrization \eqref{eq:T0_parameterization}, we 
	compute the partial derivative of 
	\eqref{eq:S-Flow_Pot} (recall the operation $\odot$ defined by 
	\eqref{eq:def-odot})
	\begin{subequations}
		\begin{align}
			\partial_i J(V) &= -\big\langle \Omega \exp_{S^0}(V),\partial_i 
			\exp_{S^0}(V) 
			\big\rangle \\
			&= -\big\langle \Omega \exp_{S^0}(V), \exp_{S^0}(V) \odot 
			e_i+\exp_{S^0}(V)_i\exp_{S^0}(V)\big\rangle \\
			&=-\big(\Omega \exp_{S^0}(V) \odot \exp_{S^0}(V)\big)_i+  
			\big\langle 
			\Omega 
			\exp_{S}(V),\exp_{S^0}(V)\big\rangle \exp_{S^0}(V)_i\\
			&= -\big(R_{\exp_{S^0}(V)}(\Omega \exp_{S^0}(V))\big)_i
		\end{align}
	\end{subequations}
	and consequently $\partial J(V) = \partial_{V} J(V) = 
	-R_{\exp_{S^0}(V)}(\Omega 
	\exp_{S^0}(V))= R_{S}\partial_{S}J(S)=\ggrad_{g}J(S)$. 
\end{proof}

\begin{proof}[Proof of Proposition 
	\ref{prop:Numerical_Scheme_Second_Order}]$\text{ }$
	\begin{enumerate}[(i)]
		\item
		Using $S^{k}=\exp_{S^{0}}(V^{k})$ and 
		\begin{equation}\label{eq:partial-J}
			\partial J(V^{k})=-R_{S^{k}}(\Omega S^{k})=\ggrad_{g}J(S^{k})
		\end{equation}
		by Lemma \ref{lem:Grad_Tangent} 
		along with the identities (recall that both $R_{S}$ and the orthogonal 
		projection $\Pi_{0}$ act row-wise)
		\begin{equation}\label{eq:R-identities}
			R_{S} = 
			\Pi_0 R_{S} = R_{S}\Pi_{0} = \Pi_0 R_{S}\Pi_{0} = 
			R_{S}|_{\mc{T}_{0}},\quad S\in\mc{W},\qquad
			\Pi_{0}^{2}=\Pi_{0} 
		\end{equation}
		and 
		\begin{equation}\label{eq:R_S_inverse}
			\big(R_{S^k}|_{\mc{T}_{0}}\big)^{-1}V
			= \Big(\dotsc,\Pi_{0}\frac{V(x)}{S^{k}(x)},\dotsc\Big)^{\T},\quad 
			x\in V,\quad V\in\mc{T}_{0},\quad S^{k}\in \mathcal{W}
		\end{equation}
		by \cite[Lemma 3.1]{Savarino:2019ab}, we have
		\begin{subequations}\label{eq:Descent_Direction_proof}\allowdisplaybreaks
			\begin{align}
				\la \partial J(V^k), d^{k}\ra 
				&\overset{\eqref{eq:dk-accelerate}}{=} 
				\langle \partial J(V^k),d(S^{k},h_{k}) \rangle 
				\\
				&\overset{\phantom{\eqref{eq:Fisher_Rau}}}{=} -\langle 
				R_{S^k}(\Omega 
				S^k),\Pi_{0}\Omega S^k \rangle-\frac{h_{k}}{2}\langle \partial 
				J(V^k),\Pi_{0}\Omega \partial J(V^k) \rangle \\
				\hskip -0.25cm
				&\overset{\phantom{\eqref{eq:Fisher_Rau}}}{=}\hskip -0.25cm 
				-\langle 
				R_{S^k}(\Omega 
				S^k),\big((R_{S^k}|_{\mc{T}_{0}})^{-1}R_{S^k}|_{\mc{T}_{0}}\big)
				 \Pi_0 \Omega S^k 
				\rangle-\frac{h_{k}}{2}\langle \partial 
				J(V^k),\Pi_{0}\Omega \partial J(V^k) \rangle\\
				\hskip -0.25cm
				&\hskip 
				-0.75cm\overset{\eqref{eq:FischerRao-matrix},\eqref{eq:R-identities},\eqref{eq:R_S_inverse}}{=}\hskip
				-0.25cm-\langle R_{S^k}(\Omega 
				S^k),R_{S^k}(\Omega S^k) \rangle_{S^k}-\frac{h_{k}}{2}\langle 
				\partial 
				J(V^k),\Pi_{0}\Omega \partial J(V^k) \rangle. 
				\label{eq:Descent_Direction_proof-c}
			\end{align} 
		\end{subequations}
		Since the first term on the right-hand side of 
		\eqref{eq:Descent_Direction_proof-c} 
		is negative on $\mc{T}_{0}$, setting 
		\begin{equation}\label{eq:proof-descent-def-step-size-h}
			h_k 
			\in 
			\bigg(0,\frac{\|R_{S^k}(\Omega S^k) \|^2_{S^k}}{|\langle \partial 
				J(V^k),\Pi_{0}\Omega \partial J(V^k) \rangle|}\bigg)
		\end{equation}
		yields a sequence 
		$(d^k)_{k\geq 1}$ satisfying 
		\begin{equation}\label{eq:dk-descent}
			\langle \partial J(V^k), d^k\rangle < 0,\qquad k \geq 1.
		\end{equation}
		Consider $c_1,c_2 \in (0,1)$ with $c_1 < c_2$ and set 
		\begin{subequations}
			\begin{align}
				G(\gamma) &= J(V^k+\gamma d^k), \\ \label{eq:def-L-gamma}
				L(\gamma) &= J(V^k)+c_1\gamma\langle \partial 
				J(V^k),d^k\rangle
				\quad\text{for}\quad\gamma \geq 0.
			\end{align}
		\end{subequations}
		Due to $c_1 < 1$ and \eqref{eq:dk-descent}, the inequality  
		\begin{equation}
			G'(0) = \langle \partial J(V^k), 
			d^k\rangle < c_1\langle \partial J(V^k), d^k\rangle = L'(0)<0
		\end{equation}
		holds. Hence there is a constant $t_k>0$ such that
		\begin{subequations}\label{eq:proof_1_Descent}
			\begin{align}
				G(\gamma) &< L(\gamma), \quad\gamma \in (0,t_k), \\ 
				\label{eq:proof_1_Descent-b}
				G(t_k) &= L(t_k).
			\end{align}
		\end{subequations}
		Substituting the first-order Taylor expansion
		\begin{subequations}\label{eq:inequality-obj-proof-descent}
			\begin{align}
				G(t_{k})&=J(V^{k}+t_{k}d^{k})
				= G(0)+t_{k} G'(\wt{\gamma}_{k})
				\\
				&= J(V^{k})+t_{k}\la\partial 
				J(V^{k}+\wt{\gamma}_{k}d^{k}),d^{k}\ra,\qquad 
				\wt{\gamma}_{k}\in (0,t_{k})
			\end{align}
		\end{subequations}
		into \eqref{eq:proof_1_Descent-b} yields with \eqref{eq:def-L-gamma}, 
		\eqref{eq:dk-descent} and $0<c_{1}<c_{2}<1$
		\begin{subequations}\label{eq:proof_3_Descent}
			\begin{align}
				\langle \partial J(V^k+\wt{\gamma}_k d^k),d^k \rangle = 
				c_1\langle 
				\partial J(V^k),d^k \rangle \geq c_2\langle \partial J(V^k),d^k 
				\rangle.
			\end{align}
			Therefore, with $\partial 
			J(V^k),d^k \in 
			\mathcal{T}_0$ and using that the restriction 
			$R_{S^k}|_{\mathcal{T}_0}$ of the map $R_{S^k}$ to 
			$\mathcal{T}_0$ is 
			invertible 
			with the inverse
			$(R_{S^k})_{|\mathcal{T}_0}^{-1}$ 
			acting row-wise as specified by \eqref{eq:R_S_inverse}, the 
			right-hand side of \eqref{eq:proof_3_Descent}  
			becomes
			\begin{align}
				c_2\langle \partial J(V^k),d^k 
				\rangle
				&= c_2 \big\langle \partial 
				J(V^k),(R_{S^k}|_{\mathcal{T}_0})^{-1}(R_{S^k}(d^k) ) 
				\big\rangle 
				\\ & 
				\overset{\eqref{eq:FischerRao-matrix},\eqref{eq:R_S_inverse}}{=}
				 c_2 \big\langle 
				\Pi_0 \partial J(V^k), R_{S^k}(d^k)\big\rangle_{S^{k}
				}.
			\end{align}   
		\end{subequations}
		By virtue of \eqref{eq:partial-J} and $\Pi_{0} \partial J(V^k) = 
		\partial J(V^k)$, both sides of \eqref{eq:proof_3_Descent} correspond 
		to the expressions of \eqref{eq:Sufficient_Decrease_Properties_b} 
		between the bars $|\dotsb|$. Since the above derivation shows that both 
		sides of \eqref{eq:proof_3_Descent} are negative, taking the magnitude 
		on both sides proves \eqref{eq:Sufficient_Decrease_Properties_b}. 
		
		Recalling the shorthand \eqref{eq:def-JV}, inequality 
		\eqref{eq:inequality-obj-proof-descent} and setting $\theta_k$ 
		small enough 
		with $\theta_k \leq \wt{\gamma}_k$, the iterates $S^{k+1} = 
		\exp_{S^0}(V^k+\theta_kd_k)$ satisfy 
		\begin{subequations}
			\begin{align}
				J(S^{k+1})-J(S^{k}) 
				&\overset{\eqref{eq:inequality-obj-proof-descent}}{=}t_{k}\la\partial
					 J(V^{k}+\wt{\gamma}_{k}d^{k}),d^{k}\ra\\
				&\overset{\phantom{\eqref{eq:inequality-obj-proof-descent}}}{\leq}
				 \theta_k \la\partial 
				J(V^{k}+\wt{\gamma}_{k}d^{k}),d^{k}\ra \\
				&\hskip -0.05cm 
				\overset{\eqref{eq:proof_3_Descent}}{\leq} \theta_k 
				c_2\langle 
				\partial J(V^k),d^k 
				\rangle \\
				& \overset{\substack{\eqref{eq:partial-J} 
				\\\eqref{eq:proof_3_Descent}}}{=}	
				\theta_kc_{2}\la\ggrad_{g}J(S^{k}),R_{S^{k}}(d^{k})\ra_{S^{k}}
			\end{align}
		\end{subequations}
		which proves inequality \eqref{eq:Sufficient_Decrease_Properties_a} 
		since both sides are non-positive and $c_{1}<c_{2}$.
		\item
		We prove by contradiction: Assume, on the contrary, that there exists a 
		sequence 
		$(S^{k})_{k\geq 
			0}\subset\ol{\mc{W}}$ in the compact set $\ol{\mc{W}}$ and a 
			convergent 
		subsequence $(S^{k_{l}})_{l\geq 0}$ with limit point $\lim 
		\limits_{l\to 
			\infty}S^{k_l} = S^{\ast}$ which is \textit{not} an equilibrium of 
			\eqref{eq:S-flow-S}. 
		Then, since the functional 
		\eqref{eq:S-Flow_Pot} is bounded from below on 
		$\overline{\mathcal{W}}$, 
		taking the sum in  
		\eqref{eq:Sufficient_Decrease_Properties_a} yields
		\begin{equation}
			\sum_{l = 0}^{\infty} c_1 \gamma_{k_l} \langle 
			\text{\normalfont{grad}}_g 
			J(S^{k_l}),R_{S^{k_l}}(d^{k_l}) 
			\rangle_{S^{k_l}} >\sum_{l = 0}^{\infty} 
			\big(J(S^{k_{l+1}})-J(S^{k_l})\big)
			= \underbrace{J(S^{\ast})-J(S^0)}_{> -\infty},  
		\end{equation}    
		and consequently   
		\begin{equation}
			c_1 \gamma_{\ast} \langle 
			\text{\normalfont{grad}}_g 
			J(S^{\ast}),R_{S^{\ast}}(d^{\ast}) 
			\rangle_{S^{\ast}} = 0.
		\end{equation}
		Using $d^{\ast}=d(S^{\ast},h_{\ast})$ given by 
		\eqref{eq:Descent_Dir_Two_Stage} along with $c_{1}>0$ and the 
		assumption 
		$\gamma_{\ast} >0$, we evaluate this equation similarly to 
		\eqref{eq:Descent_Direction_proof} 
		\begin{subequations}\label{eq:proof_contradiction}\allowdisplaybreaks
			\begin{align}
				0 &\overset{\phantom{\eqref{eq:R-identities}}}{=} \la 
				\ggrad_{g}J(S^{\ast}),R_{S^{\ast}}(d^{\ast})\ra_{S^{\ast}}
				\\
				&\overset{\eqref{eq:R-identities}}{=} 
				\Big\la-R_{S^{\ast}}(\Omega S^{\ast}),R_{S^{\ast}}\Big(\Omega 
				S^{\ast} + \frac{h_{\ast}}{2}\Omega R_{S^{\ast}}(\Omega 
				S^{\ast})\Big)\Big\ra_{S^{\ast}}
				\\
				&\hskip - 0.41cm \overset{\eqref{eq:FischerRao-matrix}, 
					\eqref{eq:R-identities}}{=} -\sum_{x\in\mc{V}}\Big\la 
					\Pi_{0} 
				R_{S^{\ast}(x)}(\Omega 
				S^{\ast})(x),\frac{R_{S^{\ast}(x)}\big(\Omega S^{\ast} + 
					\frac{h_{\ast}}{2}\Omega R_{S^{\ast}}(\Omega 
					S^{\ast})\big)(x)}{S^{\ast}(x)}\Big\ra
				\\
				&\overset{\eqref{eq:R_S_inverse}}{=} 
				-\sum_{x\in\mc{V}}\Big\la R_{S^{\ast}(x)}(\Omega S^{\ast})(x), 
				(R_{S^{\ast}(x)}|_{T_{0}})^{-1} R_{S^{\ast}(x)}\Big(\Omega 
				S^{\ast} + \frac{h_{\ast}}{2}\Omega R_{S^{\ast}}(\Omega 
				S^{\ast})\Big)(x)\Big\ra
				\\
				&\overset{\eqref{eq:R-identities}}{=} 
				-\la\Omega S^{\ast},R_{S^{\ast}}(\Omega 
				S^{\ast})\ra-\frac{h_{\ast}}{2}\big\la\Omega 
				S^{\ast},R_{S^{\ast}}\big(\Omega R_{S^{\ast}}(\Omega 
				S^{\ast})\big)\big\ra.
			\end{align}
		\end{subequations}
		Hence
		\begin{subequations}\label{eq:proof-contr-step-size-h-limits}\allowdisplaybreaks
			\begin{align}
				\frac{h_{\ast}}{2}\big\langle \Omega S^{\ast}, & 
				R_{S^{\ast}}\big(\Omega 
				R_{S^{\ast}} (\Omega 
				S^{\ast})\big) \big\rangle 
				= -\langle \Omega S^{\ast}, R_{S^{\ast}}(\Omega 
				S^{\ast}) 
				\rangle 
				\\
				&= -\sum_{x\in\mc{V}}\big\la(\Omega 
				S^{\ast})(x),R_{S^{\ast}(x)}(\Omega S^{\ast})(x)\big\ra
				\intertext{using $R_{p}\eins_{c}=0,\; p \in\mc{S}$}
				&\overset{\phantom{\eqref{eq:def-Rp}}}{=} 
				-\sum_{x\in\mc{V}}\Big\la(\Omega 
				S^{\ast})(x)-\big\la(\Omega 
				S^{\ast})(x),S^{\ast}(x)\big\ra\eins_{c},R_{S^{\ast}(x)}(\Omega 
				S^{\ast})(x)\Big\ra
				\\
				&\overset{\eqref{eq:def-Rp}}{=} 
				-\sum_{x\in\mc{V}}\Big\la(\Omega S^{\ast})(x)-\big\la(\Omega 
				S^{\ast})(x),S^{\ast}(x)\big\ra\eins_{c},
				\\
				&\qquad\qquad
				S^{\ast}(x)\odot\Big((\Omega S^{\ast})(x)-\big\la 
				S^{\ast}(x),(\Omega S^{\ast})(x)\big\ra \eins_{c}\Big)\Big\ra
				\\
				&\overset{\phantom{\eqref{eq:def-Rp}}}{=} -\sum \limits_{x \in 
					\mathcal{V}}\sum \limits_{j  \in  
					[c]}S^{\ast}_j(x)\Big((\Omega S^{\ast})_j(x)-\big\langle 
					(\Omega 
				S^{\ast})(x),S^{\ast}(x) \big\rangle \Big)^2.
			\end{align}
			By \cite[Proposition 5]{Zern:2020aa}, $S^{\ast}$ is an equilibrium 
			of 
			the 
			flow \eqref{eq:S-flow-S} if and only if 
			\begin{equation}\label{eq:equilibrium-condition}
				(\Omega S^{\ast})_{j}(x)=\langle 
				(\Omega S)^{\ast}(x),S^{\ast}(x) \rangle,\qquad \forall x\in 
				\mathcal{V},\qquad
				\forall j\in\supp(S^{\ast}(x)). 
			\end{equation}
			Therefore, by assumption, there exists 
			$\widetilde{x} \in 
			\mathcal{V}$ and $l\in \supp\big(S^{\ast}(\widetilde{x})\big)$ with 
			$(\Omega S^{\ast})_{l}(\wt{x})\neq \langle 
			\Omega S^{\ast}(\wt{x}),S^{\ast}(\wt{x}) \rangle$ and consequently 
			\begin{align}\label{eq:proof-contradition-eq}
				\frac{h_{\ast}}{2}\big\langle \Omega S^{\ast},  
				R_{S^{\ast}}\big(\Omega 
				R_{S^{\ast}} (\Omega 
				S^{\ast})\big) \big\rangle 
				&= -\langle \Omega S^{\ast}, R_{S^{\ast}}(\Omega 
				S^{\ast}) 
				\rangle 
				\\
				&\leq -S_{l}^{\ast}(\widetilde{x})\Big((\Omega 
				S^{\ast})_{l}(\widetilde{x})-\big\langle 
				(\Omega S^{\ast})(\widetilde{x}),S^{\ast}(\widetilde{x}) 
				\big\rangle\Big)^2 
				\\
				&< 0. 
			\end{align}
		\end{subequations}
		Since the first two expressions are strictly negative, this yields the 
		contradiction
		\begin{subequations}
			\begin{align}
				-\frac{1}{2}\langle \Omega S^{\ast}, R_{S^{\ast}}(\Omega 
				S^{\ast}) \rangle
				&= -\frac{1}{2}\frac{\langle \Omega S^{\ast}, 
				R_{S^{\ast}}(\Omega S^{\ast}) \rangle}{|\langle
					\Omega S^{\ast},R_{S^{\ast}}(\Omega 
					R_{S^{\ast}}(\Omega 
					S^{\ast}))  \rangle|}|\langle \Omega 
				S^{\ast},R_{S^{\ast}}(\Omega R_{S^{\ast}}(\Omega 
				S^{\ast}))  \rangle|
				\\
				&\hskip -0.57cm 
				\overset{\eqref{eq:R-identities},\eqref{eq:def-JV}}{=} 
				-\frac{1}{2}\frac{\langle \Omega S^{\ast}, R_{S^{\ast}}(\Omega 
				S^{\ast}) 
					\rangle}{|\la\ggrad_{g}J(S^{\ast}),\Pi_{0}\Omega 
					\ggrad_{g}J(S^{\ast})\ra|}|\langle \Omega 
				S^{\ast},R_{S^{\ast}}(\Omega R_{S^{\ast}}(\Omega 
				S^{\ast}))  \rangle|
				\\
				&\hskip 
				-0.57cm\overset{\eqref{eq:proof-descent-def-step-size-h},\eqref{eq:def-JV}}{\leq}
				-\frac{h_{\ast}}{2}|\langle \Omega 
				S^{\ast},R_{S^{\ast}}(\Omega R_{S^{\ast}}(\Omega 
				S^{\ast}))  \rangle|
				\\
				&\hskip -0.27cm\overset{\eqref{eq:proof-contradition-eq}}{=}
				-\la\Omega S^{\ast},R_{S^{\ast}}(\Omega S^{\ast})\ra
			\end{align}
		\end{subequations}
		which proves (ii).
		\item We prove by contraposition and show that a limit point 
		$S^{\ast}\in\mc{W}$ cannot locally minimize $J(S)$.  
		Let $\ol{S}_{(l)} \in \ol{\mathcal{W}}$ be a constant vector field 
		given for each $x\in 
		\mathcal{V}$ by 
		\begin{equation}\label{eq:def-bar-S-l-fixed}
			\ol{S}_{(l)}(x) = e_{l}=(0,\dots,0,1,0\dots,0)^{\T} \in \R^c, 
		\end{equation}
		for arbitrary $l \in [c]$. 
		Then, for any $S \in \ol{\mathcal{W}}$ with $S(x)\in\Delta_{c}$ for 
		each $x\in\mc{V}$, and with $\Omega(x,y)\geq 0$,
		\begin{subequations}\label{eq:proof-global-min}
			\begin{align}
				\langle S, \Omega S \rangle = \sum_{x \in \mathcal{V}}\sum_{j 
				\in [c]} 
				\sum_{y \in 
					\mathcal{N}(x)}\Omega(x,y)S_j(x) S_j(y) 
				&\leq \sum_{x \in 
					\mathcal{V}}\big(\sum_{y \in 
					\mathcal{N}(x)} \Omega(x,y)\big)\underbrace{\sum_{j \in 
						[c]}S_j(x)}_{ = 1} \\
				&= \sum_{x \in \mathcal{V}}\sum_{j \in [c]} 
				\sum_{y \in 
					\mathcal{N}(x)}\Omega(x,y)\ol{S}_{(l)j}(x) \ol{S}_{(l)j}(y) 
					\\
				&= \langle \ol{S}_{(l)},\Omega \ol{S}_{(l)} \rangle,
			\end{align}
		\end{subequations}
		where the inequality is strict if $S\in \mathcal{W}$.
		Consequently, 
		the constant vector $\ol{S}_{(l)}$ is a 
		global minimizer of the objective function $J(S)$ \eqref{eq:S-Flow_Pot} 
		with minimal value
		$J(\ol{S}_{(l)}) = -\frac{1}{2}\sum \limits_{x\in \mathcal{V}}\sum 
		\limits_{y \in 
			\mathcal{N}(x)}\Omega(x,y)$. Let 
			$B_{\delta}(S^{\ast})\subset\mc{W}$ be the open ball 
		with radius $\delta>0$ containing $S^{\ast}$. By assumption,  
		$S^{\ast}_j(x) > 0,\,\forall x\in\mc{V},\,\forall j \in [c]$ and there 
		exists an $\epsilon > 0$ small enough such that 
		\begin{equation}\label{eq:proof-contradiction-eps-local-min}
			S^{\ast}_{\epsilon} \coloneqq 
			S^{\ast}+\epsilon(\ol{S}_{(l)}-S^{\ast}) \in B_{\delta}(S^{\ast}) 
			\subset
			\mathcal{W}.
		\end{equation}
		Evaluating $J(S)$ at $S^{\ast}_{\epsilon}$ yields
		\begin{subequations}\label{eq:proof-contr-(iii)}\allowdisplaybreaks
			\begin{align}
				J(S^{\ast}_{\epsilon}) 
				&\overset{\eqref{eq:proof-contradiction-eps-local-min}}{=} 
				-\frac{1}{2}\big\langle 
				S^{\ast}+\epsilon(\ol{S}_{(l)}-S^{\ast}), \Omega 
				(S^{\ast}+\epsilon(\ol{S}_{(l)}-S^{\ast}))  \big\rangle \\
				&\overset{\phantom{\eqref{eq:proof-contradiction-eps-local-min}}}{=}
				 J(S^{\ast})-\epsilon \langle 
				S^{\ast},\Omega(\ol{S}_{(l)}-S^{\ast})\rangle-\frac{\epsilon^2}{2}\langle
				 \ol{S}_{(l)}-S^{\ast},\Omega(\ol{S}_{(l)}-S^{\ast}) \rangle\\
				&\hskip -0.1cm \overset{\text{(ii)},\eqref{eq:Omega_Mat}}{=} 
				J(S^{\ast})-\epsilon \Big\langle \langle S^{\ast},\Omega 
				S^{\ast} \rangle 
				\mathbb{1},\ol{S}_l - S^{\ast} 
				\Big\rangle+\frac{\epsilon^2}{2}\Big\langle 
				\langle  
				S^{\ast},\Omega 
				S^{\ast} \rangle 
				\mathbb{1},\ol{S}_{(l)} - S^{\ast} \Big\rangle \\
				&\hspace{5cm}+\epsilon^2\big(J(\ol{S}_{(l)})+\frac{1}{2}\langle 
				\ol{S}_{(l)},\Omega 
				S^{\ast} \rangle \big),
			\end{align}
			and since $\la\eins,\ol{S}_{(l)}-S^{\ast}\ra = 
			\sum_{x\in\mc{V}}\sum_{j\in[c]}(\ol{S}_{(l)j}(x)-S^{\ast}_{j}(x)) 
			\overset{\eqref{eq:def-bar-S-l-fixed}}{=}\sum_{x\in\mc{V}}(1-\sum_{j\in[c]}S^{\ast}_{j}(x))=0$,
			\begin{align}
				= J(S^{\ast})+\epsilon^2\big(J(\ol{S}_{(l)})+\frac{1}{2}\langle 
				\overline{S}_{(l)},\Omega S^{\ast} 
				\rangle 
				\big).
			\end{align}
			It follows from (ii) that $S^{\ast}$ is an equilibrium point. Hence 
			we can invoke condition \eqref{eq:equilibrium-condition} to obtain 
			the identity
			\begin{align}
				\frac{1}{2}\langle \overline{S}_{(l)}, \Omega S^{\ast} 
				\rangle &= \frac{1}{2}\sum_{x\in \mathcal{V}} \sum_{j \in [c]} 
				(\Omega 
				S^{\ast})_j(x)\overline{S}_{(l)j}(x) = \frac{1}{2} \sum_{x\in 
				\mathcal{V}} (\Omega 
				S^{\ast})_l(x) 
				\\
				&\overset{\eqref{eq:equilibrium-condition}}{=} 
				\frac{1}{2}\sum_{x\in \mathcal{V}} \langle 
				S^{\ast}(x),\Omega S^{\ast}(x)  \rangle
				= - J(S^{\ast})
			\end{align}
			and consequently, since $\ol{S}_{(l)}$ was shown above to be a 
			global minimizer of $J$,
			\begin{align}
				J(S_{\epsilon}^{\ast}) = J(S^{\ast}) + \epsilon^2\big( 
				J(\overline{S}_{(l)})- J(S^{\ast})\big) < J(S^{\ast}).
			\end{align}
		\end{subequations}
		By assumption we have $S^{\ast}\in \mathcal{W}$ and using  
		\eqref{eq:proof-global-min} it holds $J(S^{\ast}_{\epsilon})< 
		J(S^{\ast})$. As  
		$\delta > 0$ was chosen arbitrarily subject to the constraint 
		\eqref{eq:proof-contradiction-eps-local-min}, this shows that 
		$S^{\ast}$ cannot 
		be a local minimizer which proves (iii).
		\item Analogous to 
		\eqref{eq:proof-contr-step-size-h-limits} we compute
		\begin{equation}\label{eq:zero-seq-h_k-proof-twostage-iv}
			\begin{aligned}
				-&\frac{h_{k}}{2}\big\langle \Omega S^{k},  
				R_{S^{k}}\big(\Omega 
				R_{S^{k}} (\Omega 
				S^{k})\big) \big\rangle -\langle \Omega S^{k}, R_{S^{k}}(\Omega 
				S^{k}) 
				\rangle 
				\\
				&= -\frac{h_{k}}{2}\big\langle \Omega S^{k}, 
				R_{S^{k}}\big(\Omega 
				R_{S^{k}} (\Omega 
				S^{k})\big) \big\rangle-\sum \limits_{x \in 
					\mathcal{V}}\sum 
				\limits_{j  \in  
					[c]}S^{k}_j(x)\Big((\Omega S^{k})_j(x)-\big\langle (\Omega 
				S^{k})(x),S^{k}(x) \big\rangle \Big)^2 
				\\
				&= -\frac{h_{k}}{2}\big\langle \Omega S^{k}, 
				R_{S^{k}}\big(\Omega 
				R_{S^{k}} (\Omega 
				S^{k})\big) \big\rangle-\sum \limits_{x \in 
					\mathcal{V}}\sum 
				\limits_{j  \in  
					[c]}\frac{1}{S^k_j(x)}\Big(S^{k}_j(x)\big((\Omega 
				S^{k})_j(x)-\big\langle (\Omega 
				S^{k})(x),S^{k}(x) \big\rangle \big)\Big)^2 
				\\
				&= -\frac{h_{k}}{2}\big\langle \Omega S^{k}, 
				R_{S^{k}}\big(\Omega 
				R_{S^{k}} (\Omega 
				S^{k})\big) \big\rangle-\sum \limits_{x \in 
					\mathcal{V}}\Big\langle \frac{\eins}{S^k(x)}, 
				\text{\normalfont{grad}}_g(J(S^k))(x)\odot 
				\text{\normalfont{grad}}_g(J(S^k))(x) \Big\rangle.
			\end{aligned}
		\end{equation}
		Since this expression converges to $0$ for $k\to\infty$, the additional 
		assumption $\sum_{k = 0}^{\infty}h_k < 
		\infty$ implies that the second term on the right hand side is a zero 
		sequence which shows  
		$(\text{\normalfont{iv}})$.  \qedhere
	\end{enumerate}
\end{proof}

\subsection{Proofs of Section 
\ref{sec:Influence_Nonlocal_Boundary}}\label{app:Influence_Nonlocal_Boundary}

\begin{proof}[Proof of Proposition \ref{lem:Omega_Invertibility}]$\text{ }$
	\begin{enumerate}[(i)]
		\item
		Let $D$ be the diagonal degree matrix  
		\begin{equation}\label{eq:def-degree-D}
			D(x,x) = \sum_{y \in 
				\mathcal{V}}\Omega(x,y),\qquad 
		\end{equation}
		and let $f\in 
		\mathcal{F}_{\mathcal{V}}$. Then, using  
		$\sum \limits_{x,y \in \mathcal{V}}f^2(x) 
		= \sum \limits_{x,y \in \mathcal{V}} f^2(y)$, one has
		\begin{subequations}\allowdisplaybreaks\label{eq:Expression_Innerproduct}
			\begin{align}
				\langle f, (D-\Omega)f \rangle_{\mathcal{V}} 
				&\overset{\phantom{\Omega(x,y) = \Omega(y,x)}}{=} \sum_{x\in 
					\mathcal{V}}\sum_{y \in 
					\mathcal{V}} \Omega(x,y)\big((f^2(x)-f(x)f(y)\big) \\
				&\overset{\Omega(x,y) = \Omega(y,x)}{=} \sum_{x\in 
					\mathcal{V}}\sum_{y \in \mathcal{V}} 
				\Omega(x,y)\big((\frac{1}{2}f^2(x)-f(x)f(y)+\frac{1}{2}f^2(y)\big)\\
				&\overset{\phantom{\Omega(x,y) = \Omega(y,x)}}{=} 
				\frac{1}{2}\sum_{x\in 
					\mathcal{V}}\sum_{y \in 
					\mathcal{V}} \Omega(x,y)(f(x)-f(y))^2.
			\end{align}
		\end{subequations}
		Now we directly derive the right-hand side of 
		\eqref{eq:Lemma_Direchtlet_Eigenvals} from
		\eqref{eq:Dirichlet_Eigvals}. 
		\begin{subequations}\allowdisplaybreaks
		\begin{align}
			-\frac{\langle f, \mathcal{D}^{\alpha}(\Theta 
				\mathcal{G}^{\alpha}f) 
				\rangle_{\ol{\mathcal{V}}}}{\langle f,f 
				\rangle_{\ol{\mathcal{V}}}} 
			&\overset{\eqref{eq:Non_Local_Dif},\eqref{eq:def-nonl-grad}}{=} 
			\frac{\sum \limits_{x\in \ol{\mathcal{V}}}f(x)2\big(\sum 
				\limits_{y \in 			
\ol{\mathcal{V}}}\Theta(x,y)\alpha^2(x,y)(f(x)-f(y))\big)}{\sum
				\limits_{x\in 
					\ol{ 
						\mathcal{V}}} 
				f^2(x)} \\
			&\hskip -0.2cm \overset{\eqref{eq:def-olV}, \hspace{0.1cm}
				f|_{\mathcal{V}_{\mathcal{I}}^{\alpha}} = 0}{=} 
			\frac{\sum \limits_{x\in \mathcal{V}}f(x)2\big(\sum 
				\limits_{y \in \mathcal{V} \dot{\cup} 				
\mathcal{V}_{\mathcal{I}}^{\alpha}}\Theta(x,y)\alpha^2(x,y)(f(x)-f(y))\big)}{\sum
				\limits_{x\in  \mathcal{V}} f^2(x)}\\
&\overset{\phantom{\eqref{eq:Non_Local_Dif},\eqref{eq:def-nonl-grad}}}{=}
			\frac{\sum \limits_{x\in \mathcal{V}}\sum \limits_{y \in 
\mathcal{V}}\big(\Theta(x,y)\alpha^2(x,y)(f^2(x)-2f(x)f(y)+f^2(x))\big)}{\sum
				\limits_{x\in \mathcal{V}}f^2(x)} \\
			&\hspace{3cm}+\frac{2\sum \limits_{x\in 
					\mathcal{V}}\big(\sum \limits_{y \in 
\mathcal{V}_{\mathcal{I}}^{\alpha}}\Theta(x,y)\alpha^2(x,y)\big)f^2(x)}{\sum
				\limits_{x\in \mathcal{V}}f^2(x)}
			\intertext{and analogous to \eqref{eq:Expression_Innerproduct}} 
			&\hskip 
-1.8cm\overset{\phantom{\eqref{eq:alpha-boundary-0},\eqref{eq:Avaraging_Matrix}}}{=}
			\frac{\sum \limits_{x\in 
					\mathcal{V}}\sum \limits_{y \in \mathcal{V}} 
				\Theta(x,y)\alpha^2(x,y)(f(x)-f(y))^2 + 
				2\sum \limits_{x\in 
					\mathcal{V}}\big(\sum \limits_{y \in 
\mathcal{V}_{\mathcal{I}}^{\alpha}}\Theta(x,y)\alpha^2(x,y)\big)f^2(x)}{\sum
				\limits_{x \in \mc{V}} f^2(x)}\label{eq:proof-dir-eigval}
			\\
			&\hskip 
			-1.45cm\overset{\substack{\eqref{eq:def-olV} 
			\\\eqref{eq:def-lambda-x} \\ 
					\eqref{eq:Avaraging_Matrix}}}{=} 
			\frac{\sum \limits_{x\in 
					\mathcal{V}}\sum \limits_{y \in \mathcal{V}} 
				\Omega(x,y)(f(x)-f(y))^2 + 
				2\sum \limits_{x\in 
					\mathcal{V}}\big(\lambda(x)-\sum_{y\in 
					\mathcal{V}}\Omega(x,y)\big)f^2(x)}{\sum
				\limits_{x 
					\in 
					\mathcal{V}} 
				f^2(x)}\label{eq:Inequality}\\
			&\overset{\eqref{eq:Expression_Innerproduct}}{=}
			2\frac{\langle f,(D-\Omega) f\rangle_{\mathcal{V}} + 
				\langle f, (\Lambda-D)f \rangle_{\mathcal{V}}}{\langle f,f 
				\rangle_{\mathcal{V}}}
			\\
			\label{eq:proof-lambda-min}
			&\overset{\phantom{\eqref{eq:Expression_Innerproduct}}}{=} 
			2\frac{\langle f,(\Lambda-\Omega) f\rangle_{\mathcal{V}}}{\langle 
			f,f 
				\rangle_{\mathcal{V}}}
		\end{align}
		\end{subequations}
		which proves that the right-hand sides of \eqref{eq:Dirichlet_Eigvals} 
		and \eqref{eq:Lemma_Direchtlet_Eigenvals} are equal.
		
		By virtue of \eqref{eq:def-lambda-x} which is an \textit{equation} by 
		assumption, the matrix $\Lambda - \Omega$ defined by 
		\eqref{eq:def-Lambda-lem} and \eqref{eq:Avaraging_Matrix} is diagonal 
		dominant, i.e.
		\begin{align}
			\big|\big(\Lambda(x,x)-\Omega(x,x)\big) 
			-\sum \limits_{\substack{y \in \mathcal{V} \\ y \neq 
			x}}\Omega(x,y)\big| = 
			\sum_{y \in 
				\mathcal{V}_{\mathcal{I}}^{\alpha}} \Theta(x,y)\alpha^2(x,y) 
				\geq 0,\qquad x \in \mc{V},
		\end{align}
		and therefore positive semidefinite, which shows $\lambda^{D}_{1} \geq 
		0$. 
		In order to show that in fact the strict inequality $\lambda^{D}_{1} > 
		0$ holds, let $f \in \mathcal{F}_{\mathcal{V}}$ be such that equality 
		is 
		achieved in 
		\eqref{eq:Dirichlet_Eigvals}. We distinguish constant and non-constant 
		functions $f$. For constant $f = c\,\mathbb{1},\; c\in\R$, since the 
		set $\mc{V}_{\mc{I}}^{\alpha}$ given by \eqref{eq:Interaction_Dom} is 
		nonempty, there 
		exists an $\wt{x} \in \mathcal{V}$ with $\sum_{y\in 
\mathcal{V}_{\mathcal{I}^{\alpha}}}\Theta(\wt{x},y)\alpha^2(\wt{x},y)
		> 0$. Hence by \eqref{eq:proof-dir-eigval}, 
		\eqref{eq:proof-lambda-min},
		\begin{align}
			\lambda_{1}^D = \frac{\langle f,(\Lambda-\Omega)f 
			\rangle_{\mathcal{V}} 
			}{\langle f,f 
				\rangle_{\mathcal{V}}} > \frac{\sum_{y\in 	
\mathcal{V}_{\mathcal{I}^{\alpha}}}\Theta(\wt{x},y)\alpha^2(\wt{x},y)}{2n}>0.
		\end{align}
		If $f$ is non-constant, then there exist $\wt{x},\wt{y} \in 
		\mathcal{V}$ with 
		$f(\wt{y}) \neq f(\wt{x})$. Hence, since $\mathcal{V}$ is connected, 
		\eqref{eq:proof-dir-eigval}, 
		\eqref{eq:proof-lambda-min} yield 
		\begin{align}
			\lambda_{1}^D = \frac{\langle f,(\Lambda-\Omega)f 
				\rangle_{\mathcal{V}}}{\langle f,f 
				\rangle_{\mathcal{V}}} > \frac{
				\Omega(\wt{x},\wt{y})(f(\wt{x})-f(\wt{y}))^2}{2\max 
				\limits_{x \in \mathcal{V}} f^2(x) } > 0.
		\end{align} 
		\item
		We 
		perform similarly to \eqref{eq:Interaction_Dom} a disjoint 
		decomposition of the vertex set $\mathcal{V}$ and introduce the sets
		\begin{align}\label{eq:def-Vi-Vb}
			\mathcal{V}_i = \{ x \in \mathcal{V} \colon \alpha(x,y) = 0 \text{ 
				for 
			} y 
			\in 
			\mathcal{V}_{\mathcal{I}}^{\alpha} \}, \qquad \mathcal{V}_b = 
			\mathcal{V} \setminus \mathcal{V}_i.
		\end{align}
		Hence $\mathcal{V}_b\neq \emptyset$ if and only if 
		$\mathcal{V}_{\mathcal{I}}^{\alpha} \neq \emptyset$ and
		\eqref{eq:Mapping_assumption}, \eqref{eq:Avaraging_Matrix} yield 
		\begin{align}\label{eq:identities-lemma-proof}
			\forall x \in \mc{V}_{i},\qquad
			\lambda(x)-\sum_{y\in \mathcal{V}} \Omega(x,y) = 0.
		\end{align} 
		Let $f$ be a 
		\textit{normalized} eigenvector to the 
		smallest eigenvalue $\lambda_{\text{min}}(\Omega)$ of $\Omega$. Then, 
		using  
		\eqref{eq:identities-lemma-proof} and the inequality
		\begin{align}\label{eq-lemma-ineq-proof}
			(f(x)-f(y))^2 \leq 2(f^2(x) + f^2(y)), \qquad x,y \in 
			\mathcal{V},f\in 
			\mathcal{F}_{\mathcal{V}}
		\end{align}
		further yields 
		\begin{subequations}\allowdisplaybreaks
			\begin{align}
				\hskip +1cm -\lambda_{\text{min}}(\Omega)
				&\overset{\phantom{\eqref{eq-lemma-ineq-proof}}}{=} -\langle f, 
				\Omega f \rangle_{\mathcal{V}} 
				= \langle f 
				,(D-\Omega) f \rangle_{\mathcal{V}} - \langle f, 
				D 
				f \rangle_{\mathcal{V}} \\
				&\hskip 
-0.4cm\overset{\eqref{eq:def-degree-D},\eqref{eq:Expression_Innerproduct}}{=}
				\frac{1}{2}\sum_{x \in 
					\mathcal{V}}  \sum_{y\in 
					\mathcal{V}} 
				\Omega(x,y)(f(x)-f(y))^2 - \sum_{x \in 
					\mathcal{V}}\sum_{y\in 
					\mathcal{V}}\Omega(x,y)f^2(x) \\
				&\overset{\eqref{eq-lemma-ineq-proof}}{\leq} \sum_{x \in 
				\mathcal{V}}  
				\sum_{y\in 
					\mathcal{V}} 
				\Omega(x,y)f^2(x)\\
				&\overset{\eqref{eq:def-Vi-Vb}}{=}  \sum_{x \in 
					\mathcal{V}_i}  \sum_{y\in 
					\mathcal{V}} 
				\Omega(x,y)f^2(x) + \sum_{x \in 
					\mathcal{V}_b}\sum_{y\in 
					\mathcal{V}}\Omega(x,y)f^2(x) \\
&\hskip-0.27cm\overset{\eqref{eq:Omega_AF},\eqref{eq:Avaraging_Matrix}}{\leq}\sum_{x\in
					\mathcal{V}_i}f^2(x)+\sum_{x\in 
					\mathcal{V}_b}\big(1-\sum_{y\in 
\mathcal{V}_{\mathcal{I}}^{\alpha}}\Theta(x,y)\alpha^2(x,y)\big)f^2(x)
				\\
				&\overset{\phantom{\eqref{eq:def-olV}}}{=} 
				\sum_{x\in\mc{V}}f^{2}(x) - 
				\sum_{x\in \mathcal{V}_b}\sum_{y\in 					
\mathcal{V}_{\mathcal{I}}^{\alpha}}\Theta(x,y)\alpha^2(x,y)f^2(x)
				\\
				&\overset{\eqref{eq:def-olV}}{=} 
				1 - 
\sum_{x\in\mc{V}_{b}}\Big(1-\Theta(x,x)-\sum_{y\in\mc{V}}\Theta(x,y)\alpha^{2}(x,y)\Big)
				 f^{2}(x)
				\\
				&\overset{\eqref{eq:def-lambda-x}}{<} 1.
			\end{align}
		\end{subequations}
	\end{enumerate}
\end{proof}

\subsection{Proofs of Section \ref{sec:prep_lemmata}}\label{app:prep_lemmata}
\begin{proof}[Proof of Lemma \ref{lem:convergence}]
	Since $\overline{\mathcal{W}}\subset \R^{nc}$ is compact, 
	$(S^{k})_{k \geq 0} \subset \ol{\mathcal{W}}$ is bounded and there exists a 
	convergent subsequence $(S^{k_l})_{l \geq 0}$ with 
	$\lim \limits_{l\to\infty} S^{k_l} = S^{\ast}$ and $\Lambda$ nonempty and 
	compact. Due to Proposition 
	\ref{prop:Numerical_Scheme_Second_Order}, the sequence $(J(S^k))_{k 
		\geq 0}$ is nonincreasing and bounded from below with $\lim 
	\limits_{k \to \infty} J(S^{k}) = J^{\ast}$ for 
	some $J^{\ast} > -\infty$. 
	
	In view of the definition \eqref{eq:RS-Omega} of the mapping $S\mapsto 
	R_{S}(\Omega S)$, the right-hand side of \eqref{eq:Descent_Dir_Two_Stage} 
	is bounded for any $S\in\mc{S}$. Hence the subsequence $(d^{k_{l}})_{l\geq 
	0}$ induced by $(S^{k_{l}})_{l\geq 0}$ through 
	\eqref{eq:Descent_Dir_Two_Stage}, \eqref{eq:dk-accelerate} is convergent as 
	well. Consequently, for any limit point $S^{\ast} \in 
	\Lambda$,
	there exists a subsequence $(S^{k_l} )_{l \geq 0}$ with 
	\begin{equation}\label{eq:proof-Lem-6.1-Skl}
		S^{k_{l}} 
		\to S^{\ast} \quad\text{and}\quad 
		d^{k_l} \to d^{\ast} \quad\text{as}\quad l\to\infty. 
	\end{equation}
	It remains to show that 
	$\lim \limits_{l\to\infty} J(S^{k_l}) = J(S^{\ast}) = J^{\ast}$. 
	
	Analogous to the proof of 
	Proposition \ref{prop:Numerical_Scheme}, we adopt 
	the decomposition  
	\eqref{eq:J-DC} of $J(S)$ by
	\begin{subequations}\label{eq:DC-dec-proof-lem-convergence}
		\begin{align}
			J(S) = g(S)-h(S) \quad \text{with} \quad  g(S) &= 
			\delta_{\ol{\mathcal{W}}}(S)+\gamma  \langle S,\log S 
			\rangle, 
			\\ \label{eq:def-h}
			h(S) 
			&= \frac{1}{2}\langle S,\Omega S \rangle + \gamma 
			\langle S,\log S 
			\rangle,
		\end{align}
	\end{subequations}
	with appropriately chosen initial decomposition parameter $\gamma$ in 
	Algorithm 
	\ref{Geometric_Two_Stage}  such 
	that $g,h$ are strictly convex on 
	$\mathcal{W}$.
	By the lower semicontinuity of $J(S)$, we have 
	\begin{equation}\label{eq:proof-convergence-theorem-lower-bound}
		\liminf \limits_{l\to\infty} 
		J(S^{k_l}) \geq J(S^{\ast}).
	\end{equation}
	In addition, by invoking line \ref{alg-final-upd} of 
	Algorithm \ref{Geometric_Two_Stage} defining the iterate $S^{k_l}$ by the 
	inclusion $\gamma  
	\theta_{k_l-1}\widetilde{S}^{k_l-1} \in  \partial g(S^{k_l})$ if $\theta_k$ 
	satisfy the Wolfe conditions, and by line 
	\eqref{alg:Update-DC-par-line} otherwise, we have 
	\begin{align}
		&g(S^{k_l}) - \gamma  
		\theta_{k_l-1}\langle \widetilde{S}^{k_l-1}, S^{k_l}-S^{k_{l}-1} 
		\rangle 
		\leq 
		g(S^{\ast}) 
		-\gamma  
		\theta_{k_l-1}\langle \widetilde{S}^{k_l-1},S^{\ast}-S^{k_{l}-1} 
		\rangle, 
	\end{align}
	which after rearranging reads
	\begin{align}\label{eq:proof-theorem-ineq-g}
		g(S^{k_l}) \leq g(S^{\ast}) -\gamma  
		\theta_{k_l-1}\langle d^{k_l-1}, 
		S^{\ast}-S^{k_{l}} 
		\rangle-\gamma\Big\langle 
		\log\Big(\frac{S^{k_l-1}}{\mathbb{1}_c}\Big),S^{\ast}-S^{k_l} 
		\Big\rangle.
	\end{align}
	Setting
	\begin{equation}\label{eq:def-delta-proof-Lemma}
		\delta = \sum \limits_{x\in\mathcal{V}}\sum 
		\limits_{j \in \supp(S^{\ast}(x))}\log(S^{\ast}_j(x))\cdot S^{\ast}_j(x)
	\end{equation}
	and using  \eqref{eq:proof-Lem-6.1-Skl}, we obtain for the last term 
	\begin{subequations}
		\begin{align}
			\lim \limits_{l \to \infty} & \Big\langle 
			\log\Big(\frac{S^{k_l-1}}{\mathbb{1}_c}\Big), 
			S^{\ast}-S^{k_l}\Big\rangle 
			= 	
			\lim \limits_{l \to \infty} \langle 
			\log(S^{k_l-1}),S^{\ast}-S^{k_l} \rangle \\
			&= \lim \limits_{l \to \infty}\Big( \langle 
			\log(S^{k_l-1})+\log(e^{\theta_{k_l-1}d^{k_l-1}}),S^{\ast}-S^{k_l} 
			\rangle -\theta_{k_l-1}\langle d^{k_l-1},S^{\ast}-S^{k_l} \rangle 
			\Big)
			\\
			&= \lim \limits_{l \to 
			\infty}\Big(\Big\la\log\big(\exp_{S^{k_{l}-1}}(\theta_{k_l-1}d^{k_l-1})\big)
			 + \log\la S^{k_{l}-1},e^{\theta_{k_l-1}d^{k_l-1}}\ra \eins_{c}, 
			S^{\ast}-S^{k_{l}}\Big\ra 
			\\ &\qquad\qquad
			-\theta_{k_l-1}\langle d^{k_l-1},S^{\ast}-S^{k_l} \rangle 
			\Big)
			\intertext{using $\la\eins_{c},S^{\ast}-S^{k_{l}}\ra=1-1=0$}
			&\overset{\eqref{eq:def-delta-proof-Lemma}}{=}
			\underbrace{\lim \limits_{l \to \infty}\langle 
				\log(S^{k_l}),S^{\ast}-S^{k_l} 
				\rangle}_{\to \delta - \delta = 0} - \underbrace{\lim 
				\limits_{l \to 
					\infty}\langle 
				\theta_{k_l-1}d^{k_l-1}, 
				S^{\ast}-S^{k_l} \rangle}_{\to 0} \\
			&= 0.
		\end{align}
	\end{subequations}
	Hence by noticing $\theta_k \in 
	[\theta_0,\frac{1}{|\lambda_{\min}(\Omega)|}]$, the sequence 
	$(\theta_{k_l})$ is bounded and taking the limit in 
	\eqref{eq:proof-theorem-ineq-g} yields 
	\begin{align}\label{eq:limsup-g}
		\limsup \limits _{l\to\infty} g(S^{k_l}) \leq g^{\ast}(S^{\ast}). 
	\end{align}  
	Now, turning to the function $h$ of 
	\eqref{eq:DC-dec-proof-lem-convergence}, lower semicontinuity yields 
	$\liminf \limits_{l\to\infty} h(S^{k_{l}})\geq h(S^{\ast})$ and hence
	\begin{subequations}\label{eq:proof_estimate}
		\begin{align}
			\limsup \limits_{l\to\infty} J(S^{k_l}) 
			&= \limsup \limits_{l\to\infty} 
			\big(g(S^{k_l})-h(S^{k_l})\big) 
			\leq \limsup \limits_{l\to\infty} 
			g(S^{k_l})-\liminf \limits_{l\to\infty} h(S^{k_l}) \\
			&\overset{\eqref{eq:limsup-g}}{\leq} 
			g(S^{\ast})-h(S^{\ast}).
		\end{align} 
	\end{subequations}
	Finally, combining this with 
	\eqref{eq:proof-convergence-theorem-lower-bound} and by uniqueness of the 
	limit $J^{\ast}$, we have $J(S^{\ast}) = J^{\ast}$ for any $S^{\ast} \in 
	\Lambda$, which completes the proof.  
\end{proof}

\begin{proof}[Proof of Lemma \ref{lem:Converegnce-Two-Stage}]
	Throughout the proof we skip 
	the action of projection operator $\Pi_{0}$ in $d^k(x)$ given by 
	\eqref{eq:Descent_Dir_Two_Stage} and \eqref{eq:Sk+1-Alg-3}, due to the 
	invariance of lifting map 
	\eqref{eq:def-lifting-map} by property \eqref{eq:exp-constant}.  
	By definition \eqref{eq:Sk+1-Alg-3} of $S^{k+1}$, it follows for $x \in 
	\mathcal{V}$ and $j \in 
	J_+(S^{\ast}(x))$ that
	\begin{equation}\label{eq:proof-convergence-iterate-decrease}
		\begin{aligned}
			\big(S^{k+1}(x)-S^k(x)\big)_j &= 
			S^k_j(x)\Big(\frac{e^{\theta_k d^k(x)}}{\langle 
				S^k(x),e^{\theta_kd^k(x)} \rangle}-\mathbb{1} \Big)_j \\
			&=\frac{S^k_j(x)}{\langle 
				S^k(x),e^{\theta_k d^k(x)} 
				\rangle}\Big(e^{\theta_k d_j^k(x)}-\langle 
			S^k(x),e^{\theta_k d^k(x)} \rangle\Big)
			\\
			&=\frac{S^k_j(x)}{\langle 
				S^k(x),e^{\theta_k d^k(x)} 
				\rangle}\Big(\sum_{l = 0}^{\infty}\beta^k_{l,j}(x) \Big),\qquad
			\forall J_{+}(S^{\ast}(x)),
		\end{aligned}
	\end{equation}
	where we employed the power series of the exponential function and the 
	shorthand $(\beta^k_{l,j}(x))_{l\geq 0}$  
	\begin{subequations}\label{eq:conv-proof-beta-def}
		\begin{align}
			\beta^k_{l,j}(x) &= 
			\frac{\theta_k^{l}}{l!}\Big((d^k_j(x))^l-\langle 
			S^k(x),(d^k(x))^l \rangle\Big)  
			\\
			&\overset{\eqref{eq:Descent_Dir_Two_Stage}}{=} 
			\frac{\theta_k^{l}}{l!}\Big((\Omega S^k)_j^l(x)-\langle 
			S^k(x),(\Omega S^k)^l(x) \rangle\Big)+\mc{O}(h_k).
		\end{align}
	\end{subequations}
	Let $M:\ol{\mathcal{W}}\times \R_{+}\to \R_{+}$ denote the function
	\begin{equation}\label{eq:MSgamma-ub}
		M(S,\gamma) = \max_{x\in \mathcal{V}}\max_{h \in 
			[0,h_{\text{max}}]} 
		\langle S(x),e^{\gamma d(S,h)(x)}\rangle^2 \leq M^{\ast},\qquad S\in 
		\mathcal{W},
	\end{equation} 
	with $h_{\text{max}} = \max \limits_{k\geq 0}h_k$ and $d(S,h)$ as in 
	\eqref{eq:Descent_Dir_Two_Stage}. Since $M(S,\gamma)$ is a continuous 
	mapping on a 
	compact set $\mathcal{W}\times [\theta_{\text{min}},\theta_{\text{max}}]$,  
	it attains its maximum $M^{\ast}>1$. 
	Due to the equilibrium condition \eqref{eq:equilibrium-condition} there 
	exists an 
	$\varepsilon_1 >0 $ such that, for all $S\in \mathcal{W}$ with 
	$\|S^{\ast}-S\|<\varepsilon_1$, the inequality 
	\begin{equation}\label{eq:Inequality-proof-lemma1}
		-\Big((\Omega S)_j(x)-\langle \Omega S(x), 
		S(x)\rangle \Big) > -\frac{1}{\sqrt{M^{\ast}}}\Big( (\Omega 
		S^{\ast})_j(x)-\langle \Omega S^{\ast}(x), 
		S^{\ast}(x)\rangle\Big)>0.
	\end{equation}
	is satisfied for 
	all indices
	$j \in J_+(S^{\ast}(x))$ given by \eqref{eq:index-sets-proof-conv-theorem} 
	(i.e.~the terms inside the brackets on either side are negative) and $x \in 
	\mathcal{V}$.
	In particular, since $S^{\ast}\in \ol{\mathcal{W}}$ is a limit point of 
	$(S^k)_{k \geq 0}$, 
	there is a convergent subsequence $(S^{k_s})_{s\geq 0}$ with $S^{k_s} \to 
	S^{\ast}$ and consequently $\|S^{k_{s_0}}-S^{\ast}\|<\varepsilon_1$ for 
	some 
	$k_{s_0} \in \N$. 
	Now, using the componentwise inequality $p^l \leq p$ for $l \in \N$ 
	and $p \in \mathcal{S}$, we have 
	\begin{align}\label{eq:conv-proof-beta-ineq-1}
		0  \leq \big\langle \mathbb{1},\big(S^k(x) 
		\odot \Omega S^k(x)\big)^l \big\rangle \leq \big\langle S^k(x),(\Omega 
		S^k(x))^l\big\rangle.
	\end{align} 
	Employing \eqref{eq:conv-proof-beta-ineq-1} in 
	\eqref{eq:conv-proof-beta-def} and using $h^{k_{s}} \to 0$ shows 
	that 
	there 
	exists a smallest index $k_{0} \geq k_{s_0}$ such that 
	\begin{equation}\label{eq:Inequality-proof-lemma2}
		\beta_{l,j}(x) \leq \frac{\theta_k^{l}}{l!}\Big((\Omega 
		S^{k_{s_0}})^l_j(x)-\langle 
		S^{k_{s_0}}(x),(\Omega S^{k_{s_0}}(x)) \rangle^l\Big)+O(h^{k_{s_0}}) 
		<0, \quad \forall j \in J_+(S^{\ast}(x)),\; l\in\N.  
	\end{equation}
	Therefore, setting $\varepsilon_1 \coloneqq \| S^{\ast}-S^{k_{0}} \|$ for 
	all $S^{k}$ satisfying $\|S^k-S^{\ast} \| < \varepsilon$ and $k \geq k_0$ 
	with 
	$\varepsilon 
	\coloneqq \min\{\varepsilon_0,\varepsilon_1\}$, the inequalities 
	\eqref{eq:Inequality-proof-lemma1} and \eqref{eq:Inequality-proof-lemma2} 
	are simultaneously satisfied   
	and using 
	\begin{equation}
		(\Omega S^{k_{s_0}})_j^l(x)
		\overset{\eqref{eq:index-sets-proof-conv-theorem}}{<}\langle (\Omega 
		S^{k_{s_0}})(x), 
		S^{k_{s_0}}(x) \rangle^l,\quad \forall j\in J_+(S^{\ast}(x)),\quad 
		l\in\N
	\end{equation}
	enables to estimate 
	\eqref{eq:proof-convergence-iterate-decrease} by
	\begin{subequations}\label{eq:Lemma-proof-conv-final-estimate}
		\begin{align}
			\big(S^{k+1}(x)-S^k(x)\big)_j 
			&\overset{\phantom{\eqref{eq:Inequality-proof-lemma2}}}{=} 
			\frac{S^k_j(x)}{\langle 
				S^k(x),e^{\theta_kd^k(x)} 
				\rangle}\Big(\sum_{l = 1}^{\infty} \beta^k_{l,j}(x)\Big)
			\\
			&\overset{\eqref{eq:Inequality-proof-lemma2}}{\leq}\frac{S^k_j(x)}{\langle
				S^k(x),e^{\theta_kd^k(x)} 
				\rangle}\Big(\theta_k\big((\Omega 
			S^k)_j(x)-\langle 
			S^k(x),\Omega S^k(x) \rangle\big)  \\
			&\hskip 1cm+\sum_{l = 
				2}^{\infty}\frac{\theta_k^{l}}{l!}\big((\Omega 
			S^k)^l_j(x)-\langle 
			S^k(x),\Omega S^k(x) \rangle^l\big) + \mc{O}(h^{k}) \Big)
			\\
			&\overset{\eqref{eq:Inequality-proof-lemma1}}{\leq} 
			\frac{-S^k_j(x)}{\langle 
				S^k(x),e^{\theta_kd^k(x)} 
				\rangle \cdot \sqrt{M^{\ast}}}\Big(\theta_{k}\big(\langle 
			\Omega 
			S^{\ast}(x),S^{\ast}(x) \rangle 
			-(\Omega 
			S^{\ast})_j(x)\big)\Big)
			\\
			&\overset{\eqref{eq:MSgamma-ub}}{\leq} 
			-\theta_{k} 
			\frac{S^k_j(x)}{M^{\ast}}\big(\langle
			\Omega 
			S^{\ast}(x),S^{\ast}(x) \rangle 
			-(\Omega 
			S^{\ast})_j(x)\big),\qquad
			\forall J_{+}(S^{+}(x)).
		\end{align}
	\end{subequations} 
	Taking the sum over $x\in \mathcal{V}$ shows 
	\eqref{eq:Estimate-Lemma-Conv}. 
\end{proof}

\subsection{Proofs of Section 
\ref{sec:convergence-main}}\label{app:convergence-main}

\begin{proof}[Proof of Theorem \ref{theorem:convergence-1}]
	Let $S^{\ast} \in \Lambda$ be a limiting point of $(S^k)_{k\geq 0}$ with 
	$S^{\ast}(x) \in \ol{\mathcal{S}}\setminus \mathcal{S},\;\forall 
	x\in\mc{V}$, by Proposition \ref{prop:Numerical_Scheme_Second_Order}(iii), 
	and let 
	$\theta_k \in \R_+, 
	S^{k+1} \in 
	\mathcal{W}$ and $\widetilde{S}^{k}$ be determined by Algorithm 
	\ref{Geometric_Two_Stage} (see lines 
	\eqref{alg-final-upd} 
	and \eqref{alg:tilde_S}), respectively. Then, by the well-known 
	\textit{three-point 
		identity}
	\cite[Lemma 3.1]{Chen:1993aa} with respect to 
	$S^{k+1},S^{k}\in\mc{W},\,S^{\ast} 
	\in 
	\ol{\mathcal{W}}$, one has
	\begin{align}\label{eq:three-point-id-proof}
		D_{\KL}(S^{\ast},S^{k+1})-D_{\KL}(S^{\ast},S^{k}) = 
		-D_{\KL}(S^{k+1},S^{k})-\langle \nabla f(S^{k+1})-\nabla 
		f(S^{k}),S^{\ast}-S^{k+1} \rangle.
	\end{align} 
	Recalling step size selection \ref{alg:Step_Size_Selection} it holds 
	$\theta_k \in (\theta_0,\frac{1}{|\lambda_{\min}(\Omega)|})$ and leveraging 
	the DC-decomposition \eqref{eq:DC-dec-proof-lem-convergence} with 
	$\gamma = \frac{1}{\theta_k}$, the inclusion
	$\Omega S^{k}+\frac{1}{\theta_k}\log(\frac{S^k}{\mathbb{1}_c})\in \partial 
	h(S^k)$ and the strict convexity of 
	$h(S)$ on $\mathcal{W}$ imply by the gradient inequality
	\begin{equation}\label{eq:theorem-ineq-convergence-h}
		h(S^{k+1})-h(S^k)-\Big\langle \Omega 
		S^{k}+\frac{1}{\theta_k}\log\Big(\frac{S^k}{\mathbb{1}_c}\Big) 
		,S^{k+1}-S^{k} 
		\Big\rangle> 0
	\end{equation}
	and hence
	\begin{subequations}\allowdisplaybreaks
		\begin{align}
			h(S^{k+1}) & -h(S^k)-\Big\langle \Omega 
			S^k+\frac{1}{\theta_k}\log\Big(\frac{S^k}{\mathbb{1}_c}\Big) 
			,S^{k+1}-S^{k} 
			\Big\rangle
			\\
			&\overset{\eqref{eq:def-h}}{=} 
			\frac{1}{2}\langle S^{k+1},\Omega S^{k+1}\rangle-\frac{1}{2}\langle 
			S^{k},\Omega S^{k}
			\rangle 
			\\ & \qquad
			+\frac{1}{\theta_k} \Big(\langle 
			S^{k+1},\log(S^{k+1})\rangle-\langle 
			S^{k},\log S^{k} \rangle - \Big\langle 
			\log\Big(\frac{S^k}{\mathbb{1}_c}\Big),S^{k+1}-S^k \Big\rangle\Big)
			\\ &\qquad
			-\langle \Omega 
			S^k, 
			S^{k+1}-S^k \rangle \\
			&\hskip -0.3cm\overset{\eqref{eq:S-Flow_Pot}, 
				\eqref{eq:Bregman_Divergence}}{=} 
			J(S^{k})-J(S^{k+1})+\frac{1}{\theta_k} 
			D_{\KL}(S^{k+1},S^k)-\langle  \Omega S^k, 
			S^{k+1}-S^k \rangle.
		\end{align}
	\end{subequations}
	Therefore inequality \eqref{eq:theorem-ineq-convergence-h} is 
	equivalent to
	\begin{equation}\label{eq:proof-conv-thoerem-ineq-h-mod}
		\begin{aligned}
			-D_{\KL}(S^{k+1},S^{k}) \leq \theta_k\big( J(S^k)-J(S^{k+1}) 
			-\langle \Omega S^k, 
			S^{k+1}-S^k \rangle \big).
		\end{aligned}
	\end{equation}
	Combining \eqref{eq:proof-conv-thoerem-ineq-h-mod} and 
	\eqref{eq:three-point-id-proof} yields
	\begin{equation}\label{eq:proof-conv-theorem-ineq-1}
		\begin{aligned}
			D_{\KL}(S^{\ast},S^{k+1})-D_{\KL}(S^{\ast},S^{k}) &\leq 
			\theta_k\big( J(S^k)-J(S^{k+1}) 
			-\langle \Omega S^k, 
			S^{k+1}-S^k \rangle \big) \\
			&\qquad
			-\langle 
			\nabla f(S^{k+1})-\nabla 
			f(S^{k}),S^{\ast}-S^{k+1} 
			\rangle.
		\end{aligned}
	\end{equation}
	Next, in view of Algorithm 
	\ref{Geometric_Two_Stage}, line \eqref{alg:tilde_S}, we rewrite 
	the last term in 
	\eqref{eq:proof-conv-theorem-ineq-1} in the form 
	\begin{subequations}\label{eq:ident-conv-theorem-proof}\allowdisplaybreaks
		\begin{align}
			\langle 
			\nabla f(S^{k+1})-\nabla 
			f(S^{k}),S^{\ast}-S^{k+1} 
			\rangle &\overset{\substack{\eqref{eq:def-f-KL}\\ 
			S^{k},S^{k+1}\in\mc{W}}}{=} \langle 
			\mathbb{1}_c+\log(S^{k+1})-(\mathbb{1}_c+\log(S^k)),S^{\ast}-S^{k+1}
				 			\rangle   \\
			& \overset{\substack{\text{Algorithm \ref{Geometric_Two_Stage}} \\ 
			\text{line  
						\eqref{alg:tilde_S}}}}{=} \langle \log 
						(S^k)+\log(e^{\theta_k 
				d^k})-\log(S^k), S^{\ast} - S^{k+1}\rangle\\
			&\hspace{3cm}-\underbrace{\langle \log(\langle S^k,e^{\theta_k d^k} 
				\rangle)\mathbb{1}_c, S^{\ast} - S^{k+1}\rangle}_{ = 0} \\
			&\overset{\phantom{\substack{\text{Algorithm 
			\ref{Geometric_Two_Stage}} \\ 
						\text{line \eqref{alg:tilde_S}}}}}{=} \theta_k \langle 
						d^k,S^{\ast}-S^{k+1} 
			\rangle.
		\end{align}
	\end{subequations}
	Consequently, \eqref{eq:proof-conv-theorem-ineq-1} becomes 
	\begin{subequations}\label{eq:proof-conv-theorem-ineq-2}
		\begin{align}
			D_{\KL}&(S^{\ast},S^{k+1})-D_{\KL}(S^{\ast},S^{k})
			\\
			&\overset{\phantom{\eqref{eq:S-Flow_Pot}}}{\leq}
			\theta_k\big( 
			J(S^k)-J(S^{k+1})\big)-\theta_k\langle \Omega 
			S^k,S^{\ast}-S^k 
			\rangle
			- \frac{\theta_kh_k}{2}\langle \Omega R_{S^k}(\Omega S^k), 
			S^{\ast}-S^{k+1}  \rangle
			\\
			&\overset{\eqref{eq:S-Flow_Pot}}{=}\theta_k\Big(2\big(J(S^{\ast})-J(S^{k+1})\big)+J(S^{k+1})-J(S^k)
			\\ &\qquad\qquad
			-\frac{h_k}{2}\langle
			\Omega 
			R_{S^k}(\Omega S^k), 
			S^{\ast}-S^{k+1}  \rangle 
			-\langle 
			S^k, \Omega  S^{\ast}\rangle-2J(S^{\ast}) \Big).
		\end{align}
	\end{subequations}
	Using the inequality of Cauchy Schwarz and taking into account 
	$S^{\ast}\in\ol{W},\,S\in \mathcal{W}$, we estimate with $\lambda(\Omega)$ 
	defined by \eqref{eq:def-lambda-Omega} 
	\begin{equation}\label{eq:CS-inequality-second-order-norm}
		|\langle \Omega R_{S}(\Omega S),S^{\ast}-S \rangle|\leq \|\Omega R_{S} 
		(\Omega S)\| \cdot \|S^{\ast} -S \| \leq 
		\frac{\lambda^2(\Omega)}{2} \| S \| \sqrt{n} \leq 
		\frac{\lambda^2(\Omega)\cdot n}{2},
	\end{equation}
	where the factor $\frac{1}{2}$ is due to the fact that the matrices 
	$R_{S(x)}$ given by \eqref{eq:def-Rp} are positive semidefinite with 
	$\lambda_{\max}(R_{S(x)})\leq \frac{1}{2}$, which easily follows from 
	Gershgorin's circle theorem.
	
	Using the descent step based on \eqref{eq:Descent_Dir_Two_Stage} and 
	\eqref{eq:dk-descent}, we consider three further terms of 
	\eqref{eq:proof-conv-theorem-ineq-2}.
	\begin{subequations}\label{eq:conv-proof-ineq-less-zero}\allowdisplaybreaks
		\begin{align}
			J(S^{k+1})-J(S^k)&-\frac{h_k}{2}\langle
			\Omega 
			R_{S^k}(\Omega S^k), 
			S^{\ast}-S^{k+1}  \rangle 	
			\\	&\overset{\eqref{eq:Sufficient_Decrease_Properties_a}}{\leq}
			\theta_{k}c_{1}\underbrace{\la R_{S^{k}}(\Omega S^{k}), 
			R_{S^{k}}(d^{k})\ra_{S^{k}}}_{\leq 0}
			-\frac{h_k}{2}\langle
			\Omega 
			R_{S^k}(\Omega S^k), 
			S^{\ast}-S^{k+1}  \rangle
			\\
			&\overset{\eqref{eq:Descent_Dir_Two_Stage}}{\leq} 
			-\theta_k
			c_1(\langle
			R_{S^k}(\Omega 
			S^k),R_{S^k}(\Omega 
			S^k) \rangle_{S^k}\\
			&+\frac{\theta_k c_1 h_k}{2}|\langle R_{S^k}(\Omega 
			S^k),R_{S^{k}}\Omega 
			R_{S^k}\Omega S^k \rangle_{S^{k}}|) +\frac{h_k}{2}|\langle
			\Omega 
			R_{S^k}(\Omega S^k), 
			S^{\ast}-S^{k+1}  \rangle| 
			\\
			&\hskip -0.4cm 
			\overset{\eqref{eq:proof-descent-def-step-size-h},\eqref{eq:CS-inequality-second-order-norm}}{\leq}
			-\frac{\theta_k c_1}{2}\langle R_{S^k}(\Omega S^k),R_{S^k}(\Omega 
			S^k) \rangle_{S^k}+\frac{\lambda^2(\Omega) n h_k}{4}\\
			&\overset{\phantom{\eqref{eq:Descent_Dir_Two_Stage}}}{=} 
			-\frac{\theta_{k}c_{1}}{2}\|\ggrad J(S^{k})\|_{S^{k}}^{2} + 
			\frac{\lambda^2(\Omega) n h_k}{4}
			\\
			&\overset{\phantom{\eqref{eq:Descent_Dir_Two_Stage}}}{\leq}
			0,
		\end{align}
	\end{subequations}
	where the last inequality is holds due to assumption 
	\eqref{eq:h_k-seq-assumption-theorem-conv}.
	Now we focus on the last remaining term occurring in 
	\eqref{eq:proof-conv-theorem-ineq-2}. Using the index sets 
	\eqref{eq:index-sets-proof-conv-theorem} with respect to the limit point 
	$S^{\ast} \in \ol{\mathcal{W}}$ 
	along with $S^k(x)\in \mathcal{S}$, we get 
	\begin{subequations}\label{eq:proof-conv-reformulation-limit-point}\allowdisplaybreaks
		\begin{align}
			-\langle S^k,\Omega S^{\ast}\rangle-2 J(S^{\ast}) 
			&\overset{\eqref{eq:S-Flow_Pot}}{=} 
			-\sum_{x\in 
				\mathcal{V}} \langle S^{k}(x),\Omega S^{\ast}(x) 
				\rangle+\sum_{x\in 
				\mathcal{V}} \langle S^{\ast}(x),\Omega S^{\ast}(x) \rangle \\
			&\overset{\phantom{\eqref{eq:S-Flow_Pot}}}{=} -\sum_{x\in 
			\mathcal{V}} 
			\sum_{j \in [c]}S^{k}_j(x) (\Omega 
			S^{\ast})_j(x)+\sum_{x\in \mathcal{V}} \underbrace{\sum_{j\in [c]} 
			S^k_j(x)}_{=1}\langle 
			S^{\ast}(x),\Omega 
			S^{\ast}(x) \rangle \\
			&\overset{\phantom{\eqref{eq:S-Flow_Pot}}}{=} -\sum_{x\in 
			\mathcal{V}} 
			\sum_{j \in [c]}S^{k}_j(x)\big((\Omega 
			S^{\ast})_j(x)-\langle 
			S^{\ast}(x),\Omega 
			S^{\ast}(x) \rangle\big) \\
			&\overset{\eqref{eq:index-sets-proof-conv-theorem}}{=} -\sum_{x\in 
				\mathcal{V}}\Big( 
			\sum_{j \in J_-(S^{\ast}(x))}S^{k}_j(x)\big((\Omega 
			S^{\ast})_j(x)-\langle 
			S^{\ast}(x),\Omega 
			S^{\ast}(x) \rangle\big)\\
			&\hspace{1.5cm}+\sum_{j \in 
				J_+(S^{\ast}(x))}S^{k}_j(x)\big((\Omega 
			S^{\ast})_j(x)-\langle 
			S^{\ast}(x),\Omega 
			S^{\ast}(x) \rangle\big)\Big).
		\end{align}
	\end{subequations}
	As a result, combining 
	\eqref{eq:conv-proof-ineq-less-zero} and 
	\eqref{eq:proof-conv-reformulation-limit-point} for all $k \geq K$ and 
	using $J(S^{\ast})-J(S^{k+1})<0$, \eqref{eq:proof-conv-theorem-ineq-2} 
	becomes
	\begin{subequations}\label{eq:final-estimate-divergence-proof-conv}
		\begin{align}
			D_{\KL}(S^{\ast},S^{k+1})&-D_{\KL}(S^{\ast},S^{k}) \leq 
			\theta_k\Big(J(S^{\ast})-J(S^{k+1})-\sum_{x\in 
				\mathcal{V}}\big( 
			\sum_{j \in J_-(S^{\ast}(x))}S^{k}_j(x)\big((\Omega 
			S^{\ast})_j(x)\\
			&-\langle 
			S^{\ast}(x),\Omega 
			S^{\ast}(x) \rangle\big)
			+\sum_{j \in 
				J_+(S^{\ast}(x))}S^{k}_j(x)\big((\Omega 
			S^{\ast})_j(x)-\langle 
			S^{\ast}(x),\Omega 
			S^{\ast}(x) \rangle\big)\big) \Big)
		\end{align}
	\end{subequations}
	By Lemma \ref{lem:Converegnce-Two-Stage}, there exist $\varepsilon >0$ and 
	$k_0 \in \N$ such that for all $S^k\in\mathcal{W}$ with $k\geq k_0$ and 
	$\|S^k-S^{\ast}\| < \varepsilon$ inequality \eqref{eq:Estimate-Lemma-Conv} 
	is satisfied, where $$Q(S)=\sum_{x\in\mc{V}}\sum_{j\in J_{+}(S^{\ast}(x))} 
	S_{j}(x).$$
	Introducing the mapping $$V\colon\mathcal{W}\to \R_{+},\qquad V(S) = 
	D_{\KL}(S^{\ast},S)+M^{\ast} \, Q(S)$$ with $M^{\ast}>1$ as in Lemma 
	\ref{lem:Converegnce-Two-Stage}, we obtain 
	\begin{equation}\label{eq:estimate-proof-conv}
		\begin{aligned}
			V(S^{k+1})&-V(S^{k}) 
			= 
			D_{\KL}(S^{\ast},S^{k+1})-D_{\KL}(S^{\ast},S^{k})+ 
			M^{\ast}\big(Q(S^{k+1})-Q(S^{k})\big) \\
			&\overset{\substack{\eqref{eq:index-sets-proof-conv-theorem-a} \\ 
			\eqref{eq:final-estimate-divergence-proof-conv}}}{\leq} 
			\theta_k\Big( 
			J(S^{\ast})-J(S^{k})-\sum_{x\in 
				\mathcal{V}} 
			\sum_{j \in J_-(S^{\ast}(x))}\hskip -0.4cm S^{k}_j(x)\big((\Omega 
			S^{\ast})_j(x)-\langle 
			S^{\ast}(x),\Omega 
			S^{\ast}(x) \rangle\big).
		\end{aligned} 
	\end{equation}
	By Lemma 
	\ref{lem:convergence} $J(S)$ is constant on the set of limit points of the 
	sequence $(S^{k})$ and the right-hand side of 
	\eqref{eq:estimate-proof-conv} is strictly negative unless $S^k$ is a 
	stationary point of $J(S)$. Consequently,  \eqref{eq:estimate-proof-conv} 
	is 
	\textit{strictly} negative for all $k \geq k_0$ with $\|S^k-S^{\ast} \| < 
	\varepsilon$. Consider $U_{\delta} = \{S \in 
	\ol{\mathcal{W}}\colon V(S) < \delta\}$ with $\delta$ small enough such 
	that $U_{\delta} \subset \{S \in 
	\ol{\mathcal{W}} 
	\colon \| S-S^{\ast} \| < \epsilon  \}$. Then, as $S^{\ast} \in \Lambda$ 
	is a limit point, there exists an index $K \geq k_0$ such that $S^{K} \in 
	U_{\delta}$ and $(S^{k})_{k\geq K}\subset U_{\delta}$ due to $V(S^{K+1}) < 
	V(S^K) < \delta$ by \eqref{eq:estimate-proof-conv}. Therefore, for $k\geq 
	K$ we 
	conclude
	\begin{equation}
		0 \leq D_{\KL}(S^{\ast},S^k) \leq V(S^k) \to 0 \quad \text{ for } \quad 
		k \to \infty, 
	\end{equation}  
	which shows $S^k \to S^{\ast}$.
\end{proof} 
%


\begin{proof}[Proof of Theorem \ref{thm:Existance-Epsilon}]
	For $\varepsilon > 0$ let $k \in \N$ be such that $S^k \in 
	B_{\varepsilon}(S^{\ast})$. Then, with $S^{k+\frac{1}{2}},S^{k+1} \in 
	\mathcal{W}$ given by 
	\eqref{eq:Second_Order_Scheme} and taking into account assumption 
	\eqref{eq:Assumpotion_Convergence}, we have for any $x \in \mathcal{V}$ 
	with $S^{\ast}(x)=e_{j^{\ast}(x)}$
	\begin{subequations}\label{eq:key_estimate}\allowdisplaybreaks
		\begin{align}
			\|S^{k+1}(x)-S^{\ast}(x)\|_1 
			&\overset{\phantom{\eqref{eq:Second_Order_Scheme}}}{=} 
			\sum_{j\in[c]\setminus j^{\ast}(x)} S^{k+1}_{j}(x) + 
			1-S^{k+1}_{j^{\ast}(x)}(x)
			\\
			&\overset{\phantom{\eqref{eq:Second_Order_Scheme}}}{=}  2 - 2 
			S^{k+1}_{j^{\ast}(x)}(x) \\
			&\overset{\eqref{eq:Second_Order_Scheme}}{=}2-2 
			\frac{S^{k}_{j^{\ast}(x)}(x)e^{\theta_{k}(\Omega 
					S^k)_{j^{\ast}(x)}(x)+\frac{\theta_{k}h_k}{2}(\Omega
					R_{S^k}(\Omega S^k))_{j^{\ast}(x)}(x)}}{\langle 
				S^{k}(x),e^{\theta_k(\Omega 
				S^k)(x)+\frac{\theta_{k}h_k}{2}\Omega 
					R_{S^k}(\Omega S^k)(x)} \rangle}\\
			&\overset{\phantom{\eqref{eq:Second_Order_Scheme}}}{=}  2 - 
			\frac{2 S^{k}_{j^{\ast}(x)}(x)}{S^{k}_{j^{\ast}(x)}(x)+\sum
				\limits_{j 
					\neq j^{\ast}(x)}S_j^{k}(x)e^{-\theta_k H_j(x)}},
		\end{align}
	\end{subequations}
	with the shorthand 
	\begin{equation}\label{eq:Proof_Barrior_H_Func}
		H_j(x) := (\Omega 
		S^k)_{j^{\ast}(x)}(x)-(\Omega S^k)_{j}(x)+\frac{h_k}{2}\big( 
		(\Omega R_{S^k}(\Omega S^{k}))_{j^{\ast}(x)}(x)-(\Omega R_{S^k}(\Omega 
		S^{k}))_{j}(x)\big).
	\end{equation}
	We consider the first two terms of the right-hand side of 
	\eqref{eq:Proof_Barrior_H_Func}. 
	Since $S^k(x)\in 
	B_{\varepsilon}(S^{\ast})$, we have
	\begin{equation}\label{eq:epsilon-inequality-proof-lemma}
		S^k_{j^{\ast}(x)}(x) > 1-\frac{\varepsilon}{2}, \qquad S^{k}_{j}(x) 
		< 
		\frac{\varepsilon}{2} \quad \text{for all}  \quad j  \neq j^{\ast}(x)
	\end{equation}
	and get 
	\begin{subequations}\label{eq:inequality_lower_bound-lemma-proof}
		\begin{align}
			(\Omega S)_{j^{\ast}(x)}(x)-(\Omega S)_{j}(x) 
			&\overset{\eqref{eq:def-Omega-S-matrix}}{=} 
			\sum_{\mathclap{y \in 
					\mathcal{N}(x)}}\Omega(x,y)S_{j^{\ast}(x)}(y) - 
					\sum_{\mathclap{y \in 
					\mathcal{N}(x)}}\Omega(x,y)S_{j}(y)
			\\
			&\hspace{-3cm}=\sum_{\mathclap{\substack{y \in 
						\mathcal{N}(x)\\j^{\ast}(y) = 
						j^{\ast}(x)}}}\Omega(x,y)S_{j^{\ast}(x)}(y)+\sum_{\mathclap{\substack{y
						 \in \mathcal{N}(x)\\j^{\ast}(y) \neq 
						j^{\ast}(x)}}}\Omega(x,y)S_{j^{\ast}(x)}(y)- 
			\sum_{\mathclap{\substack{y \in 
						\mathcal{N}(x)\\j^{\ast}(y) = 
						j}}}\Omega(x,y)S_{j}(y)
			-\sum_{\mathclap{\substack{y \in 
						\mathcal{N}(x)\\j^{\ast}(y) \neq 
						j}}}\Omega(x,y)S_{j}(y).
		\end{align}
		Skipping the nonnegative second term and applying the 
		constraint $S_j(y) 
		< 1$ for indices $j^{\ast}(y) = j$, it follows with 
		\eqref{eq:epsilon-inequality-proof-lemma}
		\begin{align}
			(\Omega S)_{j^{\ast}(x)}(x)-(\Omega S)_{j}(x)
			&\overset{\phantom{\eqref{eq:epsilon-inequality-proof-lemma}}}{>} 
			\sum_{\mathclap{\substack{y \in 
						\mathcal{N}(x)\\j^{\ast}(y) = 
						j^{\ast}(x)}}}\Omega(x,y)S_{j^{\ast}(x)}(y) -	
			\sum_{\mathclap{\substack{y \in 
						\mathcal{N}(x)\\j^{\ast}(y) = 
						j}}}\Omega(x,y)
			-\sum_{\mathclap{\substack{y \in 
						\mathcal{N}(x)\\j^{\ast}(y) \neq 
						j}}}\Omega(x,y)S_{j}(y)\\
			&\overset{\eqref{eq:epsilon-inequality-proof-lemma}}{>}
			(1-\frac{\varepsilon}{2})\sum_{\mathclap{\substack{y \in 
						\mathcal{N}(x)\\j^{\ast}(y) = 
						j^{\ast}(x)}}}\Omega(x,y)- 
			\sum_{\mathclap{\substack{y \in 
						\mathcal{N}(x)\\j^{\ast}(y) = 
						j}}}\Omega(x,y)
			-\frac{\varepsilon}{2}\sum_{\mathclap{\substack{y \in 
						\mathcal{N}(x)\\j^{\ast}(y) \neq j}}}\Omega(x,y)
			\intertext{and after rewriting the last sum as 
			$1-\sum_{\substack{y\in\mc{N}(x) \\ j^{\ast}(x)=j}}\Omega(x,y)$ and 
			using $S^{\ast}(x)=e_{j^{\ast}(x)}$}
			&\geq(1-\frac{\varepsilon}{2})\big((\Omega
			S^{\ast})_{j^{\ast}(x)}-(\Omega 
			S^{\ast})_{j}\big)(x)-\frac{\varepsilon}{2}.
		\end{align}
	\end{subequations}
	
	\vspace{0.25cm}
	Now we consider the last two terms of the right-hand side of 
	\eqref{eq:Proof_Barrior_H_Func}, starting with the expression 
	$R_{S^{k}}(\Omega S^{k})$. 
	As $\ol{B_{\varepsilon}(S^{\ast})}$ is compact, the maximum  
	\begin{equation}\label{eq:rho-star-estimate}
		\rho^{\ast} = 
		\max \limits_{S \in \ol{B_{\varepsilon}(S^{\ast})}}\rho(S),\qquad
		\rho(S) = \max_{x \in 
			\mathcal{V}}\max_{l \in [c]\setminus j^{\ast}(x)}\big( 
		(\Omega 
		S)_{j^{\ast}(x)}-(\Omega S)_{l} \big)(x)
	\end{equation}  
	is attained. For $j\in [c]$ with 
	$(R_{S^k}(\Omega 
	S^k)\big)_j(x) < 0$, we get
	\begin{subequations}\label{eq:Replicator-bound-proof-lemma}
		\begin{align}
			\big(R_{S^k}(\Omega S^k)\big)_j(x) &= S_j^k(x)\big( (\Omega 
			S^k)_j(x)-\langle 
			S^k(x),(\Omega S^k)(x)\rangle \big) \\
			&=S_j^k(x)\Big( \sum_{l \neq j}S_{l}^k(x)\big((\Omega 
			S^k)_j(x)-(\Omega S^k)_{l}(x)\big) \Big).
		\end{align}
		Taking into account 
		\eqref{eq:def-basins-of-attraction} for $S^{k} 
		\in 
		B_{\varepsilon}(S^{\ast})$, we have $(\Omega 
		S^k)_{j^{\ast}(x)}(x)>(\Omega 
		S^k)_{l}(x)$ for all $l \in [c]\setminus {j^{\ast}(x)}$ by 
		\eqref{eq:AS-ast} and due to $R_{S^k}(\Omega S^k)_j(x) <0$, we conclude 
		$j\neq j^{\ast}(x)$ in the preceding equation. 
		Consequently, applying the second inequality in 
		\eqref{eq:epsilon-inequality-proof-lemma} further yields
		\begin{align}
			\big(R_{S^k}(\Omega S^k)\big)_j(x)
			&\overset{\eqref{eq:epsilon-inequality-proof-lemma}}{>}\frac{\varepsilon}{2}\sum_{l\neq
				j}S^k_{l}(x)\big((\Omega 
			S^k)_j-(\Omega S^k)_{l}\big)(x)\\
			&\overset{\eqref{eq:Assumpotion_Convergence}}{\geq}\frac{\varepsilon}{2}\sum_{l\neq
				j}S^k_{l}(x)\big((\Omega 
			S^k)_j-(\Omega S^k)_{j^{\ast}(x)}\big)(x)
			\\
			&\overset{\phantom{\eqref{eq:rho-star-estimate}}}{=}  
			\frac{\varepsilon}{2}(1-S_j^k(x))\big((\Omega 
			S^k)_j-(\Omega S^k)_{j^{\ast}(x)}\big)(x)
			\\
			&\overset{\eqref{eq:rho-star-estimate}}{\geq} 
			- \frac{\varepsilon}{2}\rho^{\ast}.
		\end{align}
	\end{subequations}
	In view of the last two terms of the right-hand side of 
	\eqref{eq:Proof_Barrior_H_Func}, we introduce the index sets
	\begin{equation}
		\begin{aligned}
			\mathcal{N}^j_{-}(x) := \{y \in 
			\mathcal{N}(x) \colon \big(R_{S}(\Omega S)\big)_j (y) < 
			\big(R_{S}(\Omega S)\big)_{j^{\ast}(x)}(y) \}, \\  
			\mathcal{N}^j_{+}(x) 
			:= \{y \in 
			\mathcal{N}(x) \colon \big(R_{S}(\Omega S)\big)_j(y) > 
			\big(R_{S}(\Omega S)\big)_{j^{\ast}(x)}(y) \},
		\end{aligned}
	\end{equation}
	and estimate 
	\begin{subequations}\label{eq:proof-conv-theorem-sec-order-bound}
		\begin{align} 
			(\Omega R_{S^k}(\Omega S^{k}))_{j^{\ast}(x)}(x)-(\Omega 
			R_{S^k}(\Omega 
			S^{k}))_{j}(x)&= \hskip -0.3cm\sum_{y \in 
				\mathcal{N}(x)} \hskip -0.2cm \Omega(x,y)\big( 
			R_{S^k}(\Omega S^k)_{j^{\ast}(x)}- R_{S^k}(\Omega 	
			S^k)_{j}\big)(y)\\
			&\geq \hskip -0.3cm \sum_{y \in 
				\mathcal{N}^j_+(x)} \hskip -0.2cm\Omega(x,y)\big( 
			R_{S^k}(\Omega S^k)_{j^{\ast}(x)}- R_{S^k}(\Omega 
			S^k)_{j}\big)(y).
		\end{align}
	\end{subequations}
	Regarding the term $(\dotsb)$ in round brackets, using  $\eins^{\T} 
	R_{S^{k}} = 0^{\T}$ and consequently $\sum_{l \in [c]}(R_{S^k}(\Omega 
	S^k))_l(y) = 0$ for $y \in  \mathcal{N}^j_+(x)$, it follows that 
	\begin{subequations}\label{eq:bound-replicator-proof-conv-lem}
		\begin{align} 
			R_{S^k}(\Omega S^k)_{j^{\ast}(x)}(y)- 
			R_{S^k}(\Omega 
			S^k)_{j}(y) 
			&\overset{\phantom{\eqref{eq:Replicator-bound-proof-lemma}}}{=} 
			2(R_{S^k}(\Omega 
			S^k))_{j^{\ast}(x)}(y)+\sum_{\mathclap{\substack{l\in 
						[c]\\l 
						\notin  \{j^{\ast}(x),j\}}}}(R_{S^k}(\Omega 
						S^k))_{l}(y) \\
			&\overset{\phantom{\eqref{eq:Replicator-bound-proof-lemma}}}{\geq} 
			2c 
			\min_{l\in [c]\setminus j^{\ast}(y)}(R_{S^k}(\Omega 
			S^k))_{l}(y)\\
			&\overset{\eqref{eq:Replicator-bound-proof-lemma}}{>}-\varepsilon c 
			\rho^{\ast}.
		\end{align}
	\end{subequations}
	Consequently, applying \eqref{eq:bound-replicator-proof-conv-lem} and 
	$\Omega(x,y)\leq 1$, inequality 
	\eqref{eq:proof-conv-theorem-sec-order-bound} becomes
	\begin{equation}
		\Big( 
		\big(\Omega R_{S^k}(\Omega S^{k})\big)_{j^{\ast}(x)}-\big(\Omega 
		R_{S^k}(\Omega 
		S^{k})\big)_{j}\Big)(x) > -\varepsilon |\mathcal{N}(x)| c
		\rho^{\ast}.
	\end{equation}
	
	\vspace{0.25cm}
	\noindent
	Substituting this estimate and 
	\eqref{eq:inequality_lower_bound-lemma-proof} into  
	\eqref{eq:Proof_Barrior_H_Func} yields for any $x \in \mathcal{V}$ and $j 
	\in [c]\setminus\{ j^{\ast}(x)\}$
	\begin{equation}
		H_j(x) \geq (1-\frac{\varepsilon}{2})((\Omega
		S^{\ast})_{j^{\ast}(x)}-(\Omega 
		S^{\ast})_{j})(x)-\frac{\varepsilon}{2} - 
		\frac{\overline{h}c}{2}\varepsilon 
		|\mathcal{N}(x)|\rho^{\ast}, \quad \overline{h} = 
		\max_{k\geq 
			k_0}h_k.
	\end{equation}
	
	\vspace{0.25cm}
	\noindent
	Thus, returning to \eqref{eq:key_estimate}, we finally obtain for all 
	$\varepsilon$ satisfying \eqref{eq:Lemma_Epsilon_Upper_bound} and 
	using 
	\begin{equation}\label{eq:proof-H-ast}
		H^{\ast}(x) := \min \limits_{j \neq 
			j^{\ast}(x)} H_j(x) > 
		0
	\end{equation}
	the bound
	\begin{subequations}\allowdisplaybreaks
		\begin{align}
			\|S^{k+1}(x)-S^{\ast}(x)\|_1
			&\leq 2 - 
			\frac{2 S^{k}_{j^{\ast}(x)}(x)}{S^{k}_{j^{\ast}(x)}(x)+\sum
				\limits_{j 
					\neq j^{\ast}(x)}S_j^{k}(x)e^{-\theta_k H^{\ast}(x)}} 
			\\
			&=
			\frac{2\big( 1-S^k_{j^{\ast}(x)}(x)\big)e^{-\theta_k 
					H^{\ast}(x)}}{S^{k}_{j^{\ast}(x)}(x)+\big( 
				1-S^k_{j^{\ast}(x)}(x)\big)e^{-\theta_k H^{\ast}(x)}} \\
			&\overset{S^{k}_{j^{\ast}(x)}(x)=e_{j^{\ast}(x)}}{=} 
			\|S^k(x)-S^{\ast}\|_1	\underbrace{\frac{e^{-\theta_k 
						H^{\ast}(x)}}{S^{k}_{j^{\ast}(x)}(x)+\big( 
					1-S^k_{j^{\ast}(x)}(x)\big)e^{-\theta_k 
					H^{\ast}(x)}}}_{\eqqcolon 
				\xi(x)< 1, \text{ if } H^{\ast}(x) > 0.}
			\\
			&=: \|S^k(x)-S^{\ast}\|_1 \cdot \xi(x) 
		\end{align}
	\end{subequations} 
	with $\xi(x)<1$, since $H^{\ast}(x)>0$ by \eqref{eq:proof-H-ast}. Induction 
	over $k>k_{0}$ yields
	\begin{align}
		\|S^{k+1}(x)-S^{\ast}(x)\|_1 < 
		\xi^{k-k_0}(x)\|S^{k_0}(x)-S^{\ast}(x)\|_1 
	\end{align}
	which proves \eqref{eq:Lemma_Inequality_Convergence}.
\end{proof}

\end{document}